\documentclass{article}

\usepackage{amsmath, amsthm, amssymb, amsfonts}
\usepackage{thmtools}
\usepackage{graphicx}
\usepackage{setspace}
\usepackage{geometry}
\usepackage{theoremref}
\usepackage{float}
\usepackage{hyperref}
\usepackage[utf8]{inputenc}
\usepackage[english]{babel}
\usepackage{textcomp}
\usepackage{newunicodechar}
\newunicodechar{−}{-}
\usepackage{framed}
\usepackage[dvipsnames]{xcolor}
\usepackage{tcolorbox}
\usepackage{mathtools}

\colorlet{LightGray}{White!90!Periwinkle}
\colorlet{LightOrange}{Orange!15}
\colorlet{LightGreen}{Green!15}
\colorlet{LightBlue}{Blue!15}

\newcommand{\HRule}[1]{\rule{\linewidth}{#1}}

\makeatletter
\g@addto@macro\bfseries{\boldmath}
\makeatother

\newcommand{\C}{\mathbb{C}}
\newcommand{\Dy}{\mathcal{D}}

\newcommand{\N}{\mathcal{N}}
\newcommand{\T}{\mathbb{T}}

\newcommand{\Ka}{\mathcal{K}}
\newcommand{\conj}[1]{\overline{#1}}
\newcommand{\D}{\mathbb{D}}

\newcommand{\Po}{\mathcal{P}}

\newcommand{\cD}{\conj{\mathbb{D}}}

\newcommand{\dist}[2]{\text{dist}( #1, #2 ) }

\newcommand{\hil}{\mathcal{H}}

\newcommand{\R}{\mathbb{R}}

\renewcommand{\Dy}{\mathcal{D}}

\renewcommand\Re{\operatorname{Re}}
\renewcommand\Im{\operatorname{Im}}
\newcommand{\supp}[1]{\text{supp}({#1})}

\makeatletter
\newcommand{\proofpart}[2]{%
    \par
  \addvspace{\medskipamount}%
  \noindent\emph{Step #1: #2}\par\nobreak
  \addvspace{\smallskipamount}%
  \@afterheading
}
\makeatother
\DeclarePairedDelimiter{\abs}{\lvert}{\rvert}%
\DeclarePairedDelimiter{\norm}{\lVert}{\rVert}%

\newtheorem{thm}{Theorem}[section]
\newtheorem{lemma}[thm]{Lemma}

\newtheorem{cor}[thm]{Corollary}
\newtheorem{prop}[thm]{Proposition}

\theoremstyle{definition}
\newtheorem{example}[thm]{Example}
\theoremstyle{definition}

\newtheorem{princ}[thm]{Principle}

\begin{document}


\title{ \normalsize \textsc{}
		\\ [2.0cm]
		\HRule{1.5pt} \\
		\LARGE \textbf{\uppercase{The Uncertainty principle in harmonic analysis}
		\HRule{2.0pt} \\ [0.6cm] \LARGE{Lecture Notes on Selected Topics} \vspace*{10\baselineskip}}
		}
\date{}
\author{\textbf{Author} \\ 
		Adem Limani \\
		Centre for Mathematical Sciences \\
        Lund University\\
        Sweden\\
        Spring 2026}

\maketitle
\newpage

\section*{Foreword}

These lecture notes grew out of a PhD course on the many manifestations of the uncertainty principle in harmonic analysis, given at the Centre for Mathematical Sciences at Lund University in the spring semester of 2026. The subject is vast and touches not only harmonic analysis, but also complex analysis and operator theory.

The material presented here reflects a selection of topics from this broad landscape, principally guided by the author’s mathematical taste. The notes are therefore not intended to provide a comprehensive account, nor do they aim to cover the most recent developments. Rather, they document a personal journey toward understanding some of the deeper aspects of the subject. It is my hope that further reading, and future interactions with experts in the field will lead to greater insights and, in time, improvements of these notes. 

To make myself clear, I do not proclaim myself as an expert in the topic, but rather a curious explorer with intentions of expanding his knowledge in the field, and hopefully, bringing more ambitions people on board to this journey.

The idea for this course was initiated and strongly encouraged by my former doctoral supervisor, Sandra Pott. I am grateful to her for suggesting me to embark on this adventure,  and for presenting me with this opportunity to broadening my perspectives in mathematics.

I am also grateful to Bartosz Malman for his friendship and mentorship over the years. Several of the ideas and perspectives reflected in these notes have grown out of our discussions and joint work, and it has been a pleasure to include some of his recent contributions to the topic, in these notes.

I also thank Alexandru Aleman and Erik Wahl\'en for helpful discussions and advice on preparing lecture notes. I have also benefited greatly from stimulating conversations with Alex Bergman and Eskil Rydhe on several of the topics included here.

I would also like to sincerely thank Tomas Persson, for encouraging me to lecture the course despite my notes not being fully prepared at the time. In fact, his encouragement was a decisive motivation for me to finally complete them.

Finally, I am deeply grateful to my family and friends for their constant support. The final parts of these notes were written shortly after the birth of my first child. Despite the associated challenges with becoming a parent, the love and encouragement from my family made it possible to carry this work through. Therefore, I dedicate these notes to my beloved wife and son, Fitore \& Albin.

\newpage

\newpage

\tableofcontents
\newpage

\section{Preface}

\subsection{Principles of uncertainty}

What is meant by a principle of uncertainty? For some, it  may call to mind quantum mechanics, where the limitations on the simultaneous measurement of position and momentum is expressed. More broadly, one may consider systems equipped with complementary pairs of variables, such as time and frequency or position and scale, which behave like conjugate axes and consequently resist simultaneous sharp localization.
\medskip
In analysis, this phenomenon appears through the Fourier transform, where it is reflected in precise mathematical statements relating the localization of a function to that of its transform. Despite these varied manifestations, the common structural theme that emerges is that an uncertainty principle expresses a fundamental trade-off between the simultaneous precision of two complementary descriptions of a system.
\medskip
In what follows, I will outline a (meta)-philosophical viewpoint that, that I personally find reveals a deeper unity between the physical and mathematical aspects of the subject. Let us therefore widen our horizon and consider the following abstract modelistic framework.
\medskip
Let $X$ and $Y$ be sets, and consider a transformation
\[
T: X \to Y
\]
which transports elements of $X$ into $Y$ while preserving some structure, perhaps linearity, continuity, or something else. For instance, one may interpret $T$ as encoding an observation of elements in $X$, hence it is important that $T$ is not too wild, so that information about $x \in X$ is still reflected in $T(x)$.
\medskip
A fundamental problem is to understand how these mappings act on certain elements, so suppose we have a subset $X_0 \subset X$ of elements with some desirable properties, and likewise subset $Y_0 \subset Y$ on the target space of $T$. A natural question that often arises is
\begin{equation}\label{EQ:COMP}
\text{Is} \, \, T(X_0) \cap Y_0  \, \, \text{ non-trivial?}
\end{equation}
Roughly speaking, this abstract model attempts to capture a class of problems of the form: 
\\
\emph{"Can we find objects from $X_0$, and whose $T$-transform also enjoys properties from $Y_0$?"}. \\ 

This leads to the study of \emph{incompatible} pairs $(X_0,Y_0)$, that is, pairs of sets for which the above fails. When the compatibility fails, the obstruction is often reflects a deeper limitation, that certain properties in $X_0$ via the transformation $T$ act as an impediment for certain properties in $Y_0$. When these properties admit quantitative descriptions, such incompatibility gives rise to what one may refer to as an \emph{uncertainty principle}:
\\
\emph{One cannot simultaneously impose the constraints defining $X_0$ and $Y_0$ beyond a certain threshold.}
\\
In some sense, one may interpret that uncertainty principles arise from uniqueness phenomena associated with specials subsets of incompatible pairs.
\medskip
In mathematical analysis, the sets $X$ and $Y$ typically carry some linear and topological structures, and thus one is often led to consider the following refinement the compatibility question:
\[
\text{Is} \, \, T(X_0) \cap Y_0 \quad \text{dense in } Y_0 \quad (\text{or in } Y).
\]
This brings in an additional approximation dimension to the problem, which often lies in the heart of harmonic analysis.
\medskip
The most prominent setting in which such questions arise is when $T$ is the classical Fourier transform, acting on various spaces of functions. Its central role does not only stem from its rich mathematical structure, but also from its interpretation in describing the fabrics of reality, as in quantum mechanics and signal processing. For instance, $X_0$ could be space of compactly supported continuous functions on the real line, and $Y_0$ could be functions which vanish of certain prescribed on certain prescribed set, or perhaps growth at a certain rate at infinity. At this level, the theory already displays a remarkable variety of uncertainty phenomena. 

\medskip 

These notes are devoted to exploring several classical manifestations of the uncertainty principle, primarily on the real line $\R$ and on subintervals of it, often identified with the unit circle $\T \cong [0,2\pi)$. 

We shall not attempt to systematically track which results extend to higher dimensions and which fail there. Typically, such questions are either trivial or extremely difficult, and thus often lie close to the research frontier.

Formally, the Fourier transform of a complex-valued function $f$ on $\R$ is defined by
\[
\widehat{f}(\xi) = \int_{\R} f(x) e^{-ix\xi} \, dx \qquad \xi \in \R,
\]
interpreted in an appropriate sense depending on the regularity of $f$. On the unit circle, that is, for $2\pi$-periodic functions, the Fourier coefficients are given by
\[
\widehat{f}(n) := \int_{\T} f(\zeta)\, \zeta^{-n} \, dm(\zeta) \qquad n \in \mathbb{Z},
\]
where $\zeta = e^{it}$ and $dm(e^{it}) = \frac{dt}{2\pi}$ denotes normalized arc-length measure on $[0,2\pi)$.

\medskip

A guiding principle throughout is that structural properties of a function are reflected, often in subtle and indirect ways, in the behavior of its Fourier transform, and conversely. A fundamental manifestation of this principle is provided by Parseval's identity
\[
\int_{\T} |f(\zeta)|^2 \, dm(\zeta)= \sum_{n\in \mathbb{Z}} |\widehat{f}(n)|^2
\]
and its real-line counterpart, Plancherel's theorem,
\[
\int_{\R} |f(x)|^2 \, dx = \frac{1}{2\pi} \int_{\R} |\widehat{f}(\xi)|^2 \, d\xi.
\]
These identities intrinsically arise from the orthogonality of the exponential functions and reveal the underlying Hilbert space geometry of Fourier analysis. It is precisely this geometric structure that makes the theory fruitful, and once one leaves this Hilbert space realm of, many of these results fail or require substantial modification, often leading to delicate threshold phenomena.

One immediate consequence is that a measure whose Fourier coefficients are square summable, or a function whose Fourier transform lies in $L^2$, cannot be supported on a set of Lebesgue measure zero. This statement is essentially sharp, and its borderline nature will be illustrated by the Ivashev--Musatov theorem in Section~6.

Another recurring theme is that sparsity of Fourier support forces spatial spread. Roughly speaking, a function whose Fourier transform is supported on a thin set must itself be widely spread. This principle appears in many forms, beginning with the work of Paley and Zygmund on lacunary series, which we discuss in Section~2.

In a different direction, the Paley--Wiener theorem shows that the Fourier transform of a compactly supported function extends to an entire function and, in particular, cannot vanish too frequently. This provides a complementary manifestation of the uncertainty principle, where strong spatial localization imposes rigidity in the frequency domain.

We also encounter the classical Heisenberg uncertainty principle, which quantifies the impossibility of simultaneous concentration in both position and frequency.

\medskip

A further key phenomenon is the role of logarithmic integrability. For instance, a classical Theorem by Jensen implies that if $f \in L^1(\T)$ satisfies $\widehat{f}(n) = 0$ for all $n < 0$, then
\[
\int_{\T} \log |f(\zeta)| \, dm(\zeta) > -\infty.
\]
Thus, such a function cannot vanish on a set of positive measure, and moreover its magnitude must be sufficiently distributed so as to be logarithmically integrable. The real-line analogue is determined by the logarithmic integral
\[
\int_{\R} \frac{\log |f(x)|}{1+x^2} dx >- \infty.
\]
These logarithmic integral conditions plays a central role in the theory of Hardy spaces, but its influence extends much further. They will reappear in our study of polynomial approximation, moment problems, and various questions of uniqueness type

\medskip

Finally, we briefly allude to nonlinear analogues of the Fourier transform, such as the map
\[
\mathcal{K}_\mu(f)(z) := 
\frac{\int_{\T} \frac{f(\zeta)\, d\mu(\zeta)}{1 - \overline{\zeta} z}}
{\int_{\T} \frac{d\mu(\zeta)}{1 - \overline{\zeta} z}} \qquad z\in \C \setminus \supp{\mu},
\]
which arises naturally in complex function theory. While we shall not pursue this direction in depth, it serves as a reminder that many of the principles discussed here persist, in altered form, beyond the linear setting.

\subsection{Organization}

These notes should be viewed as a collection of selected topics illustrating, from different angles, manifestations of the uncertainty principle:

\emph{$f$ and its Fourier transform $\widehat{f}$ cannot both be small simultaneously.}

\medskip

The material is divided into seven sections, each consisting of roughly 5--7 subsections. The final subsection of each section is devoted to further results and is intended to guide the reader toward the research frontier.

The sections are, to a large extent, independent of one another, and the notes are not meant to be read in a strictly chronological order. This structure is intended to accommodate PhD students and researchers with limited time, allowing for selective reading.

\medskip

In Sections 2 and 3 we work on the unit circle $\T$, with an emphasis on Fourier coefficients.

\medskip

Section 2 begins with a brief introduction to classical representation and uniqueness problems for Fourier series. We then turn to problems of translates, focusing on results of Wiener and Beurling in the settings of $\ell^1$ and $\ell^2$. This is followed by a discussion of local-to-global properties of lacunary series and Riesz products, which are subsequently used to obtain coarse classification results for Fourier coefficients of measures on $\T$.

\medskip

Section 3 is devoted to the role of the logarithmic integral. We study the Riesz brothers' theorem and its converses, and then proceed to Szeg\"o's theorem. We conclude the section by considering extensions to the unit disc developed by Khrushchev.

\medskip

Section 4 concerns uniqueness problems for the Fourier transform on $L^2(\R)$. We begin with Heisenberg's uncertainty principle and continue with the Paley--Wiener theorem on functions with compactly supported Fourier transform. The final part of the section treats deeper results on measures with spectral gaps, including Pollard's theorem, the Amrein--Berthier theorem, and the Logvinenko--Sereda theorem on Fourier uniqueness pairs.

\medskip

Section 5 studies the role of the logarithmic integral on the real line. The main theme is Bernstein's problem on weighted polynomial approximation and its equivalent reformulations in terms of moment problems and quasi-analyticity. We also discuss a unilateral version due to Beurling and prove the classical theorem of Cartwright--Levinson.

\medskip

Sections 6 and 7 are devoted to more specialized topics related to uniform Fourier majorants for functions with small support.

Section 6 treats the Ivashev--Musatov theorem and some of its consequences. The presentation is largely based on the work of T. W. K\"orner in \cite{korner1986theorem3}, where one can extract a more streamlined proof that avoids oscillatory integral techniques, and requires less regularity assumptions, in contrast to the conventional approach in \cite{havinbook}.

\medskip

Section 7 is devoted to the Beurling--Malliavin multiplier theorem. We present two different proofs of this result and discuss some of its consequences. The first proof follows the classical survey in \cite{mashreghi2006beurling} by V.P. Havin, J. Mashreghi and F. Nazarov, while the second is based on a more recent approach suggested from the work of A. Cohen in \cite{cohen2025fractal}.

\medskip

These notes are largely, but not exclusively, inspired by the following classical books, which treat different aspects of the subject:

\begin{enumerate}
    \item[]{\cite{havinbook}} V. P. Havin and B. J\"oricke, \emph{The Uncertainty Principle in Harmonic Analysis},
    \item[]{\cite{katznelson2004introduction}} Y. Katznelson, \emph{An Introduction to Harmonic Analysis},
    \item[]{\cite{garnett}} J. B. Garnett, \emph{Bounded Analytic Functions}.
\end{enumerate}

We also mention further treatments of the subject in the work of P. Koosis \cite{koosis,koosis1998logarithmic,KoosisLogint2}, as well as in the book of J.-P. Kahane and R. Salem \cite{kahane1963ensembles}. Additional, more specialized references appear throughout the notes.

\subsection{Notations}

We begin by fixing some standard notation that will be used throughout.

\medskip

For $1 \leq p \leq \infty$, we denote by $L^p(\T)$ the usual Lebesgue spaces on $\T$, equipped with normalized arc-length measure $dm$. On the real line, we write $L^p(\R)$ for the Lebesgue spaces with respect to Lebesgue measure $dx$. In case we consider Lebesgue spaces wrt to a different positive measure $\mu$, we simple write $L^p(\mu)$.

Strictly speaking, elements of $L^p$ are equivalence classes of functions. In practice, we will often identify such classes with a representative, since functions that differ only on a set of measure zero are indistinguishable for our purposes.

\medskip

The Banach spaces of complex finite Borel measures on $\T$ and $\R$, equipped with the total variation norm, will be denoted by $M(\T)$ and $M(\R)$, respectively. If $E$ is a compact subset of either $\T$ or $\R$, we write $\mu \in M(E)$ to indicate that $\mu$ is supported in $E$. Here, the support of $\mu$ is the smallest closed set $E$ such that $|\mu|(I) = 0$ for every interval or arc $I$ disjoint from $E$. This notion agrees with the usual definition of support in distribution theory.

\medskip

For a compact set $E$, we denote by $C(E)$ the space of continuous functions on $E$. The spaces of $k$-times continuously differentiable functions are denoted by $C^k$, smooth functions by $C^\infty$, and compactly supported smooth functions by $C_0^\infty$.

\medskip

We will freely use standard duality results from functional analysis, such as the duality between $L^p$ and $L^q$, and the Riesz representation theorem identifying $C(E)^*$ with $M(E)$. The reader is assumed to be familiar with these facts.

\medskip

Finally, we adopt the notation $A \lesssim B$ to mean that $A \leq c B$ for some constant $c > 0$. If both $A \lesssim B$ and $B \lesssim A$ hold, we write $A \asymp B$. When carrying out estimates, the implicit constants involved may vary from line to line, but we shall find it convenient to simple denote by the same symbol $c$ or $C$. From time to time, we shall also make use of the standard big-O and little-o notation.

\medskip

Further notation will be introduced as needed throughout these notes.

\section{Uniqueness problems on Fourier series}
\subsection{Classical representation problems on Fourier series}

We shall here gather some classical results on Fourier series. A formal Fourier series of an integrable complex-valued integrable function $f\in L^1(\T)$ on the unit-circle $\T\cong [0,2\pi)$ is defined as 
\[
f(\zeta) \sim \sum_{n\in \mathbb{Z}} \widehat{f}(n) \zeta^n \qquad \zeta=e^{it},
\]
where the Fourier coefficients $\{\widehat{f}(n)\}_{n\in \mathbb{Z}}$ are complex numbers defined by
\[
\widehat{f}(n) := \int_{\T} f(\zeta) \zeta^{-n} dm(\zeta) , \qquad n\in \mathbb{Z}.
\]
Here $dm(e^{it})= \frac{dt}{2\pi}$ denotes the normalized Lebesgue measure. 

The concept of Fourier series originates from the influential work of J. Fourier in the early 1820's, motivated by his analytic theory of heat flow, studying the heat equation. His fundamental insight was that complicated functions (or "signals") could be decomposed into superpositions of simple oscillatory modes, and that the evolution of each mode could be analyzed independently. This principle of superposition lies at the heart of harmonic analysis, where we attempt to decompose complicated objects into simpler ones, study them individually, and the reassemble the pieces.

\medskip

In the theory of Fourier series, the following natural and fundamental problems arise:

\begin{itemize}
\item[] (\emph{Representation and convergence}) Which reasonable (periodic) functions $f$ can be represented as Fourier series, and in what sense does the Fourier series converge to $f$?
\item[] (\emph{Uniqueness}) Do the coefficients $\{\widehat{f}(n)\}_{n\in \mathbb{Z}}$ uniquely determine $f$?
\end{itemize}

A large part of Fourier analysis is devoted to understanding these questions in various capacities. Typically, one is interesting in this questions for different classes of functions $f$.

Before turning to convergence issues, let us record some basic algebraic properties of Fourier coefficients, which reflect how some simple operations on functions translate into transformations of their frequencies. 

For $f,g \in L^1(\T)$, the following relations hold:

\begin{itemize}
\item Linearity: $\widehat{\alpha f+\beta g}(n) = \alpha\widehat{f}(n) + \beta\widehat{g}(n)$, where $\alpha,\beta \in \C$ are numbers.

\item Translation: If $f_{t_0}(e^{it}) := f(e^{i(t-t_0)})$, then
\[
\widehat{f_{t_0}}(n) = e^{-int_0}\widehat{f}(n).
\]

\item Modulation: If $f(e^{it}) \mapsto e^{int_0}f(e^{it})$, then
\[
\widehat{\,e^{int_0}f\,}(n) = \widehat{f}(n-n_0).
\]

\item Convolution: If $f*g$ denotes convolution on $\T$, defined by
\[
(f*g)(e^{it}) := \int_{\T} f(e^{i(t-\tau)})g(e^{i\tau})\,dm(e^{i\tau}),
\]
then
\[
\widehat{f*g}(n) = \widehat{f}(n)\widehat{g}(n).
\]

\item Differentiation: If $f$ is differentiable, then
\[
\widehat{f'}(n) = in\,\widehat{f}(n).
\]
\end{itemize}

A convenient setting of consider Fourier series on is $L^2(\T)$, the Hilbert space of square-integrable functions on $\T$ wrt $dm$:
\[
\norm{f}^2_{L^2} = \int_{\T} \abs{f(\zeta)}^2 dm(\zeta) < \infty.
\]
One can easily verify that the family of monomials $\{\zeta^n\}_{n\in \mathbb{Z}}$ forms an orthonormal system:
\[
\int_{\T} \zeta^n \conj{\zeta^k} dm(\zeta) = \begin{cases}
    1 \qquad n=k, \\
    0 \qquad n\neq k.
\end{cases}
\]
Using Weierstrass approximation Theorem (see \thref{THM:WEITHM} below), one can also show the this system is complete in $L^2(\T)$, hence $\{\zeta^n\}_{n\in \mathbb{Z}}$ actually forms an orthonormal basis in $L^2(\T)$. We therefore obtain the beautiful Parseval identity, which intimately relates the mass of $f$ with total mass of all its frequencies:
\[
\int_{\T} |f(\zeta)|^2 dm(\zeta) = \sum_{n\in \mathbb{Z}} \abs{\widehat{f}(n)}^2.
\]
From this, it follows that Fourier partial sums 
\[
S_N(f)(\zeta) := \sum_{|n|\leq N} \widehat{f}(n) \zeta^n, \qquad \zeta \in \T, 
\]
converge to $f$ in the norm of $L^2(\T)$:
\[
\norm{f-S_N(f)}^2_{L^2}:= \sum_{|n|>N} \abs{\widehat{f}(n)}^2 \to 0, \qquad N\to \infty.
\]
The engine behind this is all attributed to orthogonality, and thus crucially relies on the underlying Hilbert space structure. In general, the theory becomes much more complicated once we leave this classical roam. Using Hilbert transforms, one can also show that
\[
\int_{\T} |S_N(f)-f|^p dm \to 0, \qquad N\to \infty,
\]
for any $p>1$. For instance, see \cite{katznelson2004introduction}.

We now consider the question of pointwise convergence of Fourier series. A simple and classical framework is the so-called Wiener algebra $A(\T)$ of continuous functions $f$ on $\T$ satisfying
\[
\norm{f}_{A} := \sum_{n\in \mathbb{Z}} \abs{\widehat{f}(n)} < \infty.
\]
This condition and the Weierstrass M-test ensures that the Fourier partials sum $S_N(f)$ converge to $f$ uniformly on $\T$. Following these discussions, one is naturally lead to the belief that the Fourier partial sums of a continuous function always converge. However, this is false.

\begin{thm}[du Bois--Reymond, 1870's] For any point $\zeta_0 \in \T$, there exists a continuous function $f$ on $\T$, such that
\[
\sup_{N\geq 1} \abs{S_N(f)(\zeta_0)} = + \infty.
\]
In particular, the Fourier series of a continuous function on $\T$ need not to converge at every point on $\T$.
\end{thm}
In initial proof was constructive, but in modern days, one can give a simple proof based on the Banach--Steinhaus principle of uniform boundedness in functional analysis.
\begin{proof}
Fix a point $\zeta_0 \in \T$, and assume that for any function $f\in C(\T)$ the Fourier partial sums
\[
S_N(f)(\zeta_0) := \sum_{|n|\leq N} \widehat{f}(n) \zeta_0^n \\ ,N=1,2,3,\dots
\]
converges. Regarding the family $(S_N)_N$ as bounded linear functionals on the Banach space $C(\T)$ of continuous functions on $\T$, we must then have 
\[
\sup_{N\geq 0} \abs{S_N(f)(\zeta_0)} \leq C(f),
\]
for all $f\in C(\T)$, hence by the Banach--Steinhaus principle of uniform boundedness:
\[
\sup_{N\geq 1} \norm{S_N}_{C(\T)^*} < \infty.
\]
Realizing the Fourier partial sums as  convolution operators 
\[
S_N(f)(e^{it}) = \int_{\T} f(e^{i\tau}) D_N(e^{i(t-\tau)}) d\tau, \qquad e^{it} \in \T, \qquad N=1,2,3,\dots
\]
where $D_N$ denotes the \emph{Dirichlet kernel}
\[
D_N(e^{it}) := \sum_{|n|\leq N} e^{int} = \frac{\sin((N+1/2)t)}{\sin(t/2)}, \qquad e^{it} \in \T, \qquad N=1,2,3,\dots
\]
By the Riesz representation theorem, each functional \(S_N(\cdot)(\zeta_0)\)
may be identified with integration against the signed measure whose density is
the translated Dirichlet kernel. In particular,
\[
\norm{S_N}_{C(\T)^*}
=
\norm{D_N(\zeta_0)\,dm}_{M(\T)}
=
\int_{\T}\abs{D_N}\,dm.
\]
Indeed, this follows by approximating $\text{sgn}D_N$ with functions of unit-norm in $C(\T)$. Now the desired contradiction arises from the fact that a straightforward computation gives
\[
\int_{\T} \abs{D_N} \, dm \asymp \log N.
\]
In conclusion, there must exist a function $f\in C(\T)$, whose Fourier partials sums are unbounded at $\zeta_0 \in \T$. 
\end{proof}

Despite this fact, there is fortunately a condition on $f$, which is just slightly stronger than mere continuity at a point, but on the upside, it is a local condition, which ensures pointwise convergence of the Fourier series. 

\begin{prop}(Dini's criterion) Let $f\in L^1(\T)$ with the property that for some $t_0\in [0,2\pi)$:
\[
\int_0^{2\pi} \frac{\abs{f(e^{it})-f(e^{it_0})}}{|t-t_0|} dt < \infty.
\]
Then 
\[
S_N(f)(e^{it_0}) \to f(e^{it_0}), \qquad N\to \infty.
\]
    
\end{prop}
\begin{proof}
Without loss of generality, one may take $t_0=0$. Using the explicit expression for the Dirichlet kernel, and the fact that $\int_{\T}D_N dm =1$, one can write
\[
S_N(f)(1) -f(1) = \int_{0}^{2\pi} \sin((N+1/2)t) \frac{f(e^{it})-f(1)}{\sin(t/2)} \frac{dt}{2\pi}.
\]
Now the assumption on $f$ ensures that the function
\[
\frac{f(e^{it})-f(1)}{\sin(t/2)}
\]
is integrable, hence the claim follows from the Riemann--Lebesgue \thref{LEM:RIELEB}, which is proved later. 
\end{proof}

In particular, it follows that the Fourier series of a H\"older continuous function $f$ on $\T$ converges to $f$ at every point in $\T$. 

Intrinsically, the convergence issues associated with the convergence of Fourier partials sums are mainly caused by the fact that the Dirichlet kernels fail be uniformly bounded in norm of $L^1(\T)$.

Now despite the counterexample of du Bois--Reymond, it was still believed for a long time that even if the Fourier series may diverge somewhere, these points must be rather exceptional, and the Fourier series of an integrable function only fails to converge at a set of Lebesgue measure zero. Then came A. Kolmogorov, and produced the wildest possible counter-example.

\begin{thm}[Kolmogorov, 1926] There exists a function $f\in L^1(\T)$ whose Fourier series diverges at every point in $\T$.
\end{thm}

Kolmogorov's sophisticated counter example shattered much of the remaining general belief about the convergence problem of Fourier series. However, N. Lusin still conjectured that convergence at $dm$-a.e point should hold for functions in $L^2(\T)$. Surprisingly, this was later confirmed in the celebrated work of  L. Carleson in \cite{carleson1966convergence}.

\begin{thm} [Carleson, 1965] For any  $f\in L^2(\T)$ the partial sums satisfy
\[
S_N(f)(\zeta) \to f(\zeta), \qquad N\to \infty,
\]
pointwise at $dm$-a.e $\zeta \in \T$.
\end{thm}
Not long after, Carleson's Theorem was extended to all $f\in L^p(\T)$ for $p>1$, by R. Hunt in 1967. A complementary result showing the sharpness of Carleson Theorem was given by J-P. Kahane and Y. Katznelson in \cite{KahaneKatznelson}

\begin{thm}[Kahane--Katznelson] For any compact set $E \subset \T$ of Lebesgue measure zero, there exists an $f\in C(\T)$ such that 
\[
\sup_{N\geq 1} \abs{S_N(f)(\zeta)} = +\infty, \qquad \zeta \in E.
\]
\end{thm}
A simpler proof of this result, and an $a.e$ version of Kolmogorov's Theorem is contained in \cite[Ch. II, section 3]{katznelson2004introduction}. Let us now include a short and elementary proof of Carleson's Theorem for functions of finite Dirichlet integral.

\begin{prop}[Mild Carleson] Let $f\in L^2(\T)$ with the additional property that
\begin{equation}\label{EQ:DIRSUM}
\sum_{n} (1+|n|)|\widehat{f}(n)|^2 < \infty.
\end{equation}
Then the Fourier partial sums $S_N(f)$ converge to $f$ at $dm$-a.e on $\T$.
\end{prop}
\begin{proof}
Let $(\varepsilon_n)_n$ be a decreasing sequence of positive numbers tending to zero, and consider the sets
\[
A_n := \left\{\zeta \in \T: |S_n(f)(\zeta)-f(\zeta)|> \varepsilon_n \right\}, \qquad n=1,2,3,\dots
\]
Note that by Chebychev's inequality, we have
\[
m(A_n) \leq \frac{\norm{S_n-f}^2_{L^2}}{\varepsilon_n^2} = \frac{1}{\varepsilon_n^{2}} \sum_{|k|>n} \abs{\widehat{f}(k)}^2, \qquad n=1,2,3,\dots
\]
Using this, and the monotonicity of $\varepsilon_n$, we get
\[
\sum_n m(A_n) \leq \sum_{k} \, \abs{\widehat{f}(k)}^2 \sum_{1\leq n<|k|} \frac{1}{\varepsilon_n^2} \leq \sum_{k} \, \abs{\widehat{f}(k)}^2 |k| \frac{1}{\varepsilon^2_k}.
\]
Note that the assumption on Fourier coefficients of $f$ enables us to choose $\varepsilon_k \downarrow 0$ slow enough so that the right side sum converges. Invoking the elementary Borel--Cantelli lemma, we conclude that the set of points belonging to infinitely many of the $A_n$'s has Lebesgue measure zero, hence for $dm$-a.e $\zeta \in \T$, there exists an integer $N=N(\zeta)>0$ such that
\[
\abs{S_N(f)(\zeta)-f(\zeta)} \leq \varepsilon_N.
\]
This shows that $S_N(f) \to f$ pointwise $dm$-a.e on $\T$.
\end{proof}

The label finite Dirichlet integral comes from the fact that the summability condition in \eqref{EQ:DIRSUM} is equivalent to
\[
\int_{|z|<1} \abs{\nabla P(f)}^2 dA< \infty,
\]
where $P$ denotes the (harmonic) Poisson extension of $f$ to the unit-disc. The intriguing thing here is that functions having finite Dirichlet integral may still be very wild on $\T$. We mention yet another milder version of Carleson's Theorem, which is also elementary (though not as simple as the one above), but requires the stronger assumption that the Fourier coefficients are $\ell^p$-summable: 
\[
\sum_{n\in \mathbb{Z}} \abs{\widehat{f}(n)}^p < \infty,
\]
for some $0<p<2$. This is contained in the B. Simon \cite[P. ? ]{BarrySimonHarmonic}.

As previously mentioned, most of the convergence issues arise due to the fact that the Dirichlet kernels fail to have an asymptotically bounded $L^1$-integral. To compensate for their lack of integrability, analyst have long attempted to rectify the kernels by considering averages of them. Recall that if a sequence of complex numbers $s_n \to L$, then so does its arithmetic average:
\[
\frac{1}{n}\sum_{k=1}^n s_k \to L, \qquad n\to \infty.
\]
This means that averages can only improve the state of convergence, and one is therefore naturally lead to consider arithmetic averages of the of the Fourier partial sums:
\[
\sigma_N(f)(\zeta) := \frac{1}{N} \sum_{n=1}^N S_n(f)(\zeta) = \sum_{|k|\leq N} \left( 1- \frac{|k|}{n} \right)\widehat{f}(n)\zeta^n \qquad \zeta \in \T.
\]
These are referred to as the Cesar\'o means and/or Fej\'er means of $f$. One may also express these means as convolution integrals
\[
\sigma_N(f)(e^{it}) = \int_{\T} f(e^{i\tau}) \mathcal{F}_N(e^{i(t-\tau)}) \frac{d\tau}{2\pi}, \qquad e^{it} \in \T,
\]
where $\mathcal{F}_N$ are the Fej\'er kernels:
\[
\mathcal{F}_N(e^{it}) = \sum_{|n|\leq N} \left( 1- \frac{|n|}{N} \right) e^{int} = \frac{1}{N} \frac{\sin^2(Nt/2)}{\sin^2(t/2)}, \qquad e^{it} \in \T.
\]
It turns out the family of Fej\'er kernels are much better behaved than the Dirichlet kernels, and satisfy the following conditions:
\begin{enumerate}
    \item[(i)] $\int_{\T}\mathcal{F}_N dm=1$,
    \item[(ii)] $\sup_{N\geq 1}\int_{\T} \abs{\mathcal{F}_N}dm< \infty$,

    \item[(iii)] For any $\delta>0$:
    \[
    \lim_{N\to \infty} \int_{\delta\leq |t|\leq \pi} \abs{\mathcal{F}_N(e^{it})} dt =0.
    \]
\end{enumerate}
These properties of the Fej\'er kernels ensure  that the sequence of measures $\mathcal{F}_N dm$ are approximates of the identity, that is, they converge in weak-star to the Dirac delta measure $d\delta_1$ in the space $M(\T)$. The problem with the Dirichlet kernel is that it violates $(ii)$, hence it leaves any ball in the norm of $M(\T)$. 

With these properties at hand, we can show that the Fej\'er means $\sigma_N(f)$ of a continuous function $f$ on $\T$ converges uniformly to $f$. Indeed, one may estimate as follows:
\begin{multline*}
\abs{\sigma_N(f)(e^{it})-f(e^{it})} \leq \int_{0}^{2\pi} \abs{\mathcal{F}_N(e^{i\tau}) }\abs{f(e^{i(\tau-t)}) -f(e^{i\tau})} \frac{d\tau}{2\pi} \\
\leq \sup_{|\tau|\leq \delta} \abs{f(e^{i(\tau-t)}) -f(e^{i\tau})} + 2\norm{f}_{L^\infty} \int_{\delta \leq|\tau|\leq \pi}\abs{\mathcal{F}_N(e^{i\tau}) } \frac{d\tau}{2\pi}.
\end{multline*}
It remains to take supremum over $e^{it} \in \T$, send $N\to \infty$ first, and then $\delta \to 0+$. As a consequence, we retrieve the following classical result on uniform approximation with trigonometric polynomials.
\begin{thm}[Weierstrass]\thlabel{THM:WEITHM} The trigonometric polynomials are dense in $C(\T)$.
\end{thm}
One can also show that the Fej\'er means converge in $L^p(\T)$ for all $p\geq 1$:
\[
\norm{ \sigma_N(f) - f }^p_{L^p} = \int_{\T} \abs{\sigma_N(f) - f }^p dm \to 0, \qquad N\to \infty.
\]
Indeed, this follows from a standard $3\varepsilon$-argument:
\[
\norm{\sigma_N(f)-f}_{L^p} \leq \norm{g-f}_{L^p} + \norm{g-\sigma_N(g)}_{L^\infty} + \norm{\sigma_N(f-g)}_{L^p}.
\]
Note that the last term can be estimated independent of $N$, using Minkowski's integral inequality and translation invariance of $L^p$-norms:
\[
\norm{\sigma_N(f-g)}_{L^p} \leq \sup_{N\geq1} \norm{\mathcal{F}_N}_{L^1} \norm{f-g}_{L^p}.
\]
Sending $N\to \infty$, we obtain 
\[
\norm{\sigma_N(f)-f}_{L^p} \leq C\norm{g-f}_{L^p}, 
\]
for any $g\in C(\T)$. Since $C(\T)$ is dense in $L^p(\T)$ by classical measure theory, we obtain the desired conclusion.

Another important summation method for Fourier series are the Abel--Poisson means, which rest on the following observation: if $c_n \to L$ then
\[
(1-r^2)\sum_{n=0}^\infty r^{2n} c_n \to L, \qquad |r|\to 1-.
\]
One is then naturally lead to study the following geometric averages:
\[
A_r(f)(e^{it}) = (1-r^2) \sum_{n=0}^\infty S_n(f)(e^{it}) r^{2n} = \sum_{n=0}^\infty \widehat{f}(n) r^{2|n|}e^{int}, \qquad |r|<1.
\]
Similar to the Fej\'er means, one may also express the Abel--Poisson means in integral form:
\[
A_r(f)(e^{it}) = \int_{\T}  \frac{1-r^2}{1-2r\cos(t-\tau)+r^2} f(e^{i\tau})\frac{d\tau}{2\pi}, \qquad |r|<1.
\]
One can also verify that the Abel--Poisson means satisfy the properties $(i)-(iii)$ similar to the Fej\'er means, thus it follows that the Abel means $A_r(f)$ converge to $f$ in $L^p(\T)$ for all $p\geq 1$, and if $f\in C(\T)$, then they also converge to $f$ uniformly on $\T$. 

In contrast with the Fej\'er means, which were trigonometric polynomials, the Abel means instead give rise to harmonic functions on the unit-disc $\D=\{z\in \C: |z|<1\}$:
\[
P(f)(re^{it}) := A_r(f)(e^{it}), \qquad re^{it} \in \D.
\]
This is most easily seen by the identity of the Poisson kernel
\[
P_{re^{it}}(e^{i\tau}) := \frac{1-r^2}{1-2r\cos(t-\tau)+r^2} = \frac{1-r^2}{|1-re^{i(t-\tau)}|^2} = \Re \frac{1+re^{i(t-\tau)}}{1-re^{i(t-\tau)}}, \qquad e^{it},e^{i\tau} \in \T.
\]
The harmonic function $P(f)$ is then referred to as the (harmonic) Poisson extension of $f$ to the unit-disc $\D$, which solves the Dirichlet problem
\[
\begin{cases}
    \Delta u = 0 \qquad \text{on} \, \, \D, \\
    u= f \, \, \, \, \qquad \text{on} \, \, \T.
\end{cases}
\]
Here, the equality $u=f$ needs to be interpreted appropriately, depending on the regularity of the given boundary data $f$. If $f \in L^1(\T, dm)$, then its Poisson integral $P(f)$ admits non-tangential boundary values almost everywhere on $\T$. More precisely, for $dm$-almost every $\zeta \in \T$, we have
\[
P(f)(z) \to f(\zeta)
\]
as $z \to \zeta$ within any fixed cone with vertex at $\zeta$ contained in $\D$. This result is one of the fundamental cornerstones of Hardy space theory. See \cite[Ch. II]{garnett}

An important transformation complementary to the Poisson extension is the harmonic conjugate transform, explicitly given by 
\[
Q(f)(re^{it}) := \int_{\T} \frac{2r \sin(t-\tau)}{|1-re^{i(t-\tau)}|^2} f(e^{i\tau}) \frac{d\tau}{2\pi}, \qquad re^{it} \in \D,
\]
where $f\in L^1(\T)$. Expanding the integral kernel into a geometric series, we can re-write it as the Fourier series
\[
Q(f)(re^{it}) = -i \sum_{n\neq 0} \text{sgn}(n) r^{2n} e^{int} \qquad re^{it} \in \D.
\]
Similarly, one can show that $Q(f)(z)$ has non-tangential boundary values $dm$-a.e on $\T$, and it converges to the periodic \emph{Hilbert transform} 
\[
\widetilde{f}(e^{it}) := \text{P.V} \int_{\T} \cot\left(\frac{t-\tau}{2} \right) f(e^{i\tau}) \frac{d\tau}{2\pi}, \qquad e^{it} \in \T.
\]
This is the periodic analogue of the classical Hilbert transform in $\R$
\[
H(f)(x) := \text{P.V} \int_{\R} \frac{f(t)}{t-x} \frac{dt}{\pi}, \qquad x \in \R,
\]
which is arguably one of the most important objects in classical harmonic analysis.

Together, the Poisson extension and the conjugate transform both give rise to an analytic extension of function $f\in L^1(\T)$ to the unit-disc $\D$ via the so-called Herglotz transform:
\[
\int_{\T} \frac{1+\conj{\zeta}z}{1-\conj{\zeta}z} f(\zeta) dm(\zeta) = P(f)(z) + i Q(f)(z), \qquad z\in \D.
\]
The reader may quickly realize that the above expression is essentially just a linear transformation of the Cauchy integral extension of $f$, given by
\[
\Ka(f)(z) := \int_{\T} \frac{f(\zeta)}{1- \conj{\zeta}z} dm(\zeta), \qquad z \in \D.
\]
These discussions naturally lead to the so-called Hardy spaces, and we defer some of their details to the appendix, and the reader can find a more complete treatment in \cite{garnett}. 

\subsection{Classical uniqueness problems of Fourier series}

In the previous subsection, we mainly discussed results on various representations involving Fourier series, and the problems of convergence associated with such representations. Here, attention to questions of uniqueness are consider, and guided by the following questions:

\emph{1.) Can we uniquely reconstruct a signal $f$ from its Fourier coefficients $\{\widehat{f}(n)\}_{n\in \mathbb{Z}}$?}

\emph{2.) If the Fourier series of $f$ exhibits a certain behavior on a subset of $\T$, does $f$ necessarily exhibit the same behavior there?}

It turns out that question $1.)$ admits a very satisfactory answer, which largely explains the central role of Fourier analysis.

\begin{cor}
Any $\mu \in M(\T)$ is completely determined by its Fourier coefficients. In other words, the map
\[
\mu \mapsto \{\widehat{\mu}(n)\}_{n\in \mathbb{Z}}
\]
is injective.
\end{cor}

\begin{proof}
By linearity, it suffices to consider $\mu \in M(\T)$ such that
\[
\widehat{\mu}(n)= \int_{\T} \zeta^{-n} \, d\mu(\zeta) = 0, \qquad n\in \mathbb{Z}.
\]
By the Weierstrass approximation theorem, trigonometric polynomials are dense in $C(\T)$. Hence $\mu$ annihilates all continuous functions on $\T$, and therefore $\mu \equiv 0$ by the duality $C(\T)^* \cong M(\T)$.
\end{proof}

In fact, the Weierstrass approximation theorem can be iterated to show that trigonometric polynomials are dense in the Banach space $C^k(\T)$ of $k$-times continuously differentiable functions on $\T$, equipped with the norm
\[
\norm{f}_{C^k} := \sum_{j=0}^k \sup_{\zeta \in \T} |f^{(j)}(\zeta)|.
\]
Indeed, once $f \in C(\T)$ can be approximated uniformly by trigonometric polynomials $T_n$, a primitive of $T_n$ (which is again a trigonometric polynomial) approximates a primitive of $f$ uniformly. Iterating this argument yields the desired density.

As a consequence, any distribution $S$ on $\T$ is uniquely determined by its Fourier coefficients
\[
\widehat{S}(n) := S\big(e^{it} \mapsto e^{-int}\big), \qquad n\in \mathbb{Z}.
\]

We now turn to question $2.)$. A first indication of the answer is provided by the following localization principle.

\begin{thm}[Riemann's localization principle]
If $f \in L^1(\T)$ vanishes in a neighborhood of a point $\zeta_0 \in \T$, then
\[
S_N(f)(\zeta_0) \to 0, \qquad N \to \infty.
\]
\end{thm}

The proof is an immediate consequence of Dini's criterion. An immediate consequence is that if two functions $f,g \in L^1(\T)$ agree on an arc, then their Fourier series have essentially the same behavior on any smaller subarc. 

We mention a classical uniqueness problem which has had a profound influence on the development of real analysis. The problem was to understand which compact sets $E \subset \T$ enjoy the following property: whenever $(a(n))_{n\in\mathbb{Z}}$ is a bounded sequence such that the associated Fourier series satisfies
\[
S_N(\zeta):= \sum_{|n|\leq N} a(n)\zeta^n \to 0, \qquad N\to \infty, \quad \zeta \in \T \setminus E,
\]
then necessarily $a(n)=0$ for all $n$. 

Such sets are called \emph{sets of uniqueness}. Note that the boundedness is necessary for the Fourier series to converge anywhere on $\T$.

This problem came about from the work of B. Riemann, and was considered by G. Cantor in the 1870s. Cantor proved that every closed countable set is a set of uniqueness. Upon developing the Lebesgue integral, H. Lebesgue established the following results, which are today regarded as fundamental in Fourier analysis.

\medskip

\noindent\textbf{Riemann--Lebesgue Lemma:} If $f\in L^1(\T)$, then $\widehat{f}(n)\to 0$ as $|n|\to \infty$.

\smallskip

\noindent\textbf{Cantor--Lebesgue Lemma:} If the partial sums $\sum_{|n|\leq N} a_n \zeta^n$ converge on a set of positive Lebesgue measure, then $a_n \to 0$.

\medskip

We shall prove the Riemann--Lebesgue lemma in the next subsection. The Cantor--Lebesgue lemma can be deduced from it and is left as an exercise.

Lebesgue's theory suggested that sets of measure zero should also be negligible for convergence of Fourier series, and hence candidates for sets of uniqueness. However, this intuition was disproved by D. Menshov, who constructed a compact set of Lebesgue measure zero which is not a set of uniqueness. Phrased in modern terms, his result can be formulated as follows:

\begin{thm}[Menshov]
Let $\mu\in M(\T)$ be non-trivial with $\widehat{\mu}(n) \to 0$ as $|n|\to \infty$. Then
\[
S_N(\mu)(\zeta) := \sum_{|n|\leq N} \widehat{\mu}(n)\zeta^n \to 0, \qquad N \to \infty, \quad \zeta \in \T \setminus \supp \mu.
\]
In particular, any set containing $\supp \mu$ cannot be a set of uniqueness.
\end{thm}

Here $\supp \mu$ denotes the smallest compact set $E \subset \T$ such that $|\mu|(I)=0$ for every open arc $I \subset \T \setminus E$. This definition also coincides with the usual notion of support for distributions.

Combining the Riemann--Lebesgue lemma with Menshov's theorem, we see that no set of positive Lebesgue measure can be a set of uniqueness.

This naturally leads to the study of sets that support singular measures $\mu$ whose Fourier coefficients vanish at infinity. Such measures were systematically investigated by N. Bari and A. Rajchman, who independently constructed examples supported on uncountable Cantor-type sets. Further contributions were made by Ivashev-Musatov, Kahane, Rajchman, Rudin, Piatetski-Shapiro, Salem, and Zygmund. A detailed historical account can be found in the survey of T. W. K\"orner \cite{korner1992sets}.

\subsection{Some classical problems on translates}

A recurring theme in harmonic analysis is the interplay between operators in the physical domain and their counterparts on the Fourier side. A simple but fundamental model is the relation between pointwise multiplication and translation of Fourier coefficients.

Given $f\in L^1(\T)$, recall that multiplication by the monomial $\zeta^n$ acts as a translation operator on the Fourier coefficients:
\[
\widehat{\,\zeta^n f\,}(m)=\widehat{f}(m-n), \qquad m,n\in\mathbb Z.
\]
Thus the multiplication operator $M_\zeta f(\zeta)=\zeta f(\zeta)$ induces translations on the Fourier side, and its integer powers satisfy
\[
M_\zeta^n f(\zeta)=\zeta^n f(\zeta), \qquad n\in\mathbb Z,
\]
where $M_\zeta^{-1}f(\zeta)=\overline{\zeta}\,f(\zeta)$. More generally, if $T$ is a trigonometric polynomial
\[
T(\zeta)=\sum_n \widehat{T}(n)\zeta^n,\qquad \zeta \in \T,
\]
then multiplication by $T$ can be expressed as
\[
T(\zeta)f(\zeta)=\sum_n \widehat{T}(n)\,M_\zeta^n f(\zeta), \qquad \zeta \in \T.
\]

In this subsection, we shall study subspaces of functions that are invariant under translations. To describe the general problem, let $X$ be a Banach space of functions in $L^1(\T)$ for which the multiplication operators $M_{\zeta}, M_{\conj{\zeta}}:X \to X$ act continuously. The problem is to describe the closed subspaces $\mathcal{V}$ of $X$ which are invariant under translations:
\[
M_{\zeta} \, \mathcal{V} \subseteq \mathcal{V}, \qquad M_{\conj{\zeta}} \,\mathcal{V} \subseteq \mathcal{V}.
\]
Phrased equivalently, we would like to classify the closed subspace $\mathcal{V}$ of $X$ satisfying
\[
M_{\zeta}\,\mathcal{V}= \mathcal{V}.
\]
In particular, we shall consider when a single function $f \in X$ generates a dense invariant subspace under this action, that is, when the linear span of the set
\[
\left\{\zeta^n f(\zeta) : \, \, n=0, \pm1, \pm2, \dots \right\} \qquad \text{is dense in} \, \, X.
\]
These problems become particularly interesting in spaces $X$ defined in terms of conditions involving Fourier coefficients, or by restrictions on their Fourier support.

We shall primarily be concerned with the problem in the most classical framework of the Hilbert space $L^2(\T)$, normed by
\[
\int_{\T} \abs{f(\zeta)}^2 \, dm(\zeta) = \sum_{n\in \mathbb{Z}} \, \abs{\widehat{f}(n)}^2 < \infty.
\]

Our first result is a classical theorem of N. Wiener, which provides a complete classification of closed subspaces of $L^2(\T)$ that are invariant under translation.

\begin{thm}[Wiener]\thlabel{THM:WIENER1} Let $\mathcal{V}$ be a translation invariant closed subspace of $L^2(\T)$. Then there exists a unique Lebesgue measurable subset $S \subset \T$ (modulo null-sets) such that
\[
\mathcal{V} = L^2(S) := \{f \in L^2(\T): f= 0 \, \text{dm-a.e off} \, \, S \}.
\]
\end{thm}

It is worth to mention that Wiener's Theorem also holds in the reflexive range of Lebesgue spaces $L^p(\T)$ with $1<p<\infty$, but the proof needs to be slightly modified.

\begin{proof}
Let $\mathcal{V}$ be a non-trivial closed translation invariant subspace of $L^2$, and $\Pi$ denote the orthogonal projection from $L^2$ onto $\mathcal{V}$. Set $f = \Pi(1)$ and note that $1-f \in \mathcal{V}^\perp$, hence 
\[
\int_{\T} \conj{(1-f(\zeta))}f(\zeta) \zeta^n dm(\zeta) = 0, \qquad n=0,\pm 1, \pm 2,\dots
\]
But since the polynomials are dense in $L^2$, we conclude that $f(1-\conj{f})=0$ $dm$-a.e on $\T$, hence $f=|f|^2$ and therefore $f=1_S$ for some Lebesgue measurable subset $S\subset \T$, which is uniquely determined up to sets of Lebesgue measure zero. Since $\mathcal{V}$ is closed, we clearly have $L^2(S) \subseteq \mathcal{V}$, hence it remains to prove that $L^2(S)$ is dense in $\mathcal{V}$. Let $g\in \mathcal{V}$ with the property that $g \perp L^2(S)$, which we may express as: 
\[
\int_{S} g(\zeta) \zeta^{-n} dm(\zeta) =0, \qquad n=0,\pm 1, \pm 2, \dots
\]
Then it follows that $1_S g =0$ $dm$-a.e. But since $1_{\T \setminus S} = 1- f \in \mathcal{V}^\perp$, we also have that 
\[
\int_{\T \setminus S} g(\zeta) \zeta^{-n} dm(\zeta) = 0, \qquad n=0,\pm 1, \pm 2, \dots
\]
hence $g=0$ $dm$-a.e on $\T \setminus S$, and we conclude that $g= 0$ $dm$-a.e on $\T$. This proves that $L^2(S)= \mathcal{V}$, hence completing the proof.
\end{proof}

We now record the following corollary, which provides a complete answer to our question in the classical setting of $L^2(\T)$.

\begin{cor} Let $f\in L^2(\T)$. Then the linear span of translates of $f$:
\[
\left\{f(\zeta) \zeta^n: \, \, n=0,\pm 1, \pm 2, \dots \right\}
\]
is dense in $L^2(\T)$ if and only if $f\neq 0$ $dm$-a.e on $\T$.
\end{cor}

Next, we consider a unilateral version of Wiener’s theorem on translates in $L^2(\T)$, namely the structure of closed subspaces $\mathcal{V}$ invariant under the multiplication operator $M_\zeta$:
\[
M_\zeta \, \mathcal{V} \subseteq \mathcal{V}.
\]
A fundamental example of such a subspace is given by the celebrated Hardy space
\[
H^2(\T):= \{ f\in L^2(\T): \widehat{f}(n)=0 \text{ for } n<0 \}.
\]

Since this space plays a crucial role, we will devote a few paragraphs explaining some of its features. First, one can identify elements in $H^2(\T)$ with analytic functions in the unit-disc $\D$ having square summable Taylor coefficients centered at the origin. To explain this connection, fix $f \in H^2(\T)$ and consider its Poisson extension
\[
P(f)(z) := \int_{\T} \frac{1-|z|^2}{|\zeta-z|^2}f(\zeta)\, dm(\zeta), \qquad z\in \D.
\]
Using the identity
\[
\frac{1-|z|^2}{|\zeta-z|^2} = \frac{1}{1-\overline{\zeta}z} + \frac{1}{1-\zeta\overline{z}} -1, \qquad \zeta \in \T, \qquad z\in \D,
\]
the assumption on $H^2(\T)$, and expanding into geometric series, one obtains
\[
P(f)(z) = \sum_{n\geq 0} \widehat{f}(n)z^n, \qquad z\in \D.
\]
This shows that $P(f)$ is analytic in $\D$, and by Parseval's identity its Taylor coefficients are square-summable. Conversely, every analytic function in $\D$ with square-summable Taylor coefficients arises as the Poisson extension of a unique element $f \in H^2(\T)$. Thus, $H^2(\T)$ may be identified isometrically with the Hilbert space of analytic functions on $\D$ with square-summable Taylor coefficients, commonly denoted $H^2(\D)$. 

In a similar way, one defines $H^\infty(\D)$ as the space of bounded analytic functions in $\D$. These are precisely the Poisson extensions of elements in
\[
H^\infty(\T) := \{f\in L^\infty(\T): \widehat{f}(n)=0, \, \, n<0 \}.
\]
In what follows, we allow the lax convention of not seriously distinguishing between these different realizations of the Hardy spaces, allowing the context to clarify the interpretation, and we often use the designation $H^2$, and $H^\infty$.

A key structural feature of $H^2$ is the inner--outer factorization: every nonzero $f \in H^2$ admits a unique (up to multiplication by unimodular constants) decomposition
\[
f(\zeta) = \Theta(\zeta) F(\zeta), \qquad \zeta \in \T
\]
where $\Theta \in H^\infty(\T)$ is inner, meaning $|\Theta(\zeta)|=1$ for $dm$-a.e $\zeta \in \T$, and $F\in H^2(\T)$ is outer, meaning that
\[
\log |\widehat{F}(0)| = \int_{\T} \log |F| dm >- \infty,
\]
and is therefore uniquely determined by its boundary modulus $|F|=|f|$. Roughly speaking, $\Theta$ is the argument of an $H^2$-function, which captures the zero set and phase of $f$, while $F$ encodes its magnitude: $|f|=|F|$ $dm$-a.e on $\T$. These factorization have fairly explicit expressions viewed as elements in $H^2(\D)$. These details are deferred to the appendix, and the reader is referred to \cite[Ch. II]{garnett} for a thorough treatment of Hardy space.
\medskip

The following celebrated structural Theorem by A. Beurling, gives a function theoretical characterization of all closed subspace of $H^2$ which are invariant under $M_\zeta$.

\begin{thm}[Beurling, 1949] Let $\mathcal{V}$ be closed $M_\zeta$-invariant subspace of $H^2(\T)$. Then there exists a unique inner function $\Theta$ (up to a unimodular constant), such that 
\[
\mathcal{V} = \Theta H^2 :=\{ \Theta f: f\in H^2 \}.
\]
    
\end{thm}

\begin{proof}
It follows that $M_\zeta \mathcal{V}$ is a closed subspace of $\mathcal{V}$, and hence we have the orthogonal decomposition
\[
\mathcal{V} = M_\zeta \mathcal{V} \oplus (M_\zeta \mathcal{V})^\perp.
\]
Suppose, for contradiction, that $(M_\zeta \mathcal{V})^\perp = \{0\}$, or equivalently that $M_\zeta \mathcal{V} = \mathcal{V}$. Then for any $f \in \mathcal{V}$ with $\|f\|_{H^2} = 1$ and any integer $n \ge 1$, there exists $g_n \in \mathcal{V}$ with $\|g_n\|_{H^2} = 1$ such that
\[
f(z) = z^n g_n(z), \qquad z \in \mathbb{D}.
\]
Iterating, this forces $f$ to vanish to arbitrarily high order at the origin, which is impossible unless $f \equiv 0$. This contradiction shows that $M_\zeta \mathcal{V} \subsetneq \mathcal{V}$. Consequently, there exists a nonzero element $f \in (M_\zeta \mathcal{V})^\perp$. By definition, it follows that
\[
\int_{\T} \lvert f(\zeta)\rvert^2 \zeta^{-n} dm(\zeta) = \int_{\T} f(\zeta) \conj{f(\zeta) \zeta^n} dm(\zeta) = 0, \qquad n=1,2, 3, \dots.
\]
Taking complex conjugates, we conclude that the non-zero Fourier coefficients of $|f|^2$ are zero, hence $|f|=\widehat{f}(0)\neq 0$ $dm$-a.e. Set $\Theta = f/\widehat{f}(0)$, which by definition is an inner function. Since $\mathcal{V}$ is closed it easily follows that $\Theta H^2 \subseteq \mathcal{V}$, hence it remains only to verify that $\Theta H^2$ is dense in $\mathcal{V}$. Let $g\in \mathcal{V}$ with the property that 
\[
\int_{\T} \Theta(\zeta)\zeta^{n} \conj{g(\zeta)} dm(\zeta)= 0, \qquad n=0,1,2,\dots
\]
Recall that since $\Theta \in (M_\zeta \mathcal{V})^\perp$ by definition, we also have that 
\[
\int_{\T} \Theta(\zeta) \conj{g(\zeta) \zeta^n} dm(\zeta) = 0, \qquad n=1,2,3, \dots.
\]
But these two conditions readily imply that $\Theta \conj{g} =0$, and since $\Theta$ is unimodular, we conclude that $g=0$ $dm$-a.e on $\T$. It follows that $\mathcal{V} = \Theta H^2$.
\end{proof}

We record the following consequence of Beurling's Theorem, which will appear useful at a later stage. 

\begin{cor}\thlabel{COR:CYCOUTER} For $f\in H^2$, the linear span of $\{f(\zeta)\zeta^n: \, \, n=0,1,2,\dots \}$ is dense in $H^2$, if and only if $f$ is outer.
\end{cor}
\begin{proof}
This is an immediate consequence of Beurling's Theorem and the uniqueness of the inner--outer factorization in $H^2$. If $f$ is outer, then smallest $M_\zeta$-invariant subspace containing $f$ cannot be of the form $\Theta H^2$, unless $\Theta$ is a unimodular constant.
\end{proof}

Using Beurling's Theorem as a stepping stone, we may now complete the characterization of all closed subspace of $L^2(\T)$, which are invariant under $M_\zeta$. 

\begin{cor}[Beurling--Wiener] Let $\mathcal{V}$ be a closed $M_{\zeta}$-invariant subspace of $L^2(\T)$. If $\mathcal{V}= M_\zeta \mathcal{V}$, then there exists a unique Lebesgue measure subset $S \subset \T$, such that 
\[
\mathcal{V}:= L^2(S).
\]
If $M_\zeta \mathcal{V} \subsetneq \mathcal{V}$, then there exists a unique unimodular function $U \in L^\infty(\T)$ (i.e $|U|=1$ $dm$-a.e on $\T$) such that 
\[
\mathcal{V} = U H^2:= \{ U\, f: f\in H^2(\T) \}.
\]
\end{cor}
\begin{proof}
The first case $\mathcal{V}= M_\zeta \mathcal{V}$ asserts that $\mathcal{V}$ is, in fact, translation invariant, and hence is just a re-statement of Wiener's Theorem. 

Moving forward, we thus assume that $M_\zeta \mathcal{V} \subsetneq \mathcal{V}$. Arguing as in the proof of Beurling's Theorem, we can pick a non-trivial element $f\in ( M_\zeta \mathcal{V})^\perp$, and deduce that $f= cU$, where $c\neq 0$ constant and $U$ is a unimodular function on $\T$. It is easy to see that $UH^2 \subseteq \mathcal{V}$, hence it only remains to show that $UH^2$ is dense in $\mathcal{V}$. However, that remaining part of the proof is again similar to Beurling's Theorem, we omit the details.
\end{proof}

We can now describe all functions in $L^2(\T)$ which are cyclic wrt to $M_\zeta$ in $L^2(\T)$. To this end, we make use of the following fact: if $f\in L^2(\T)$ has integrable logarithm:
\[
\int_{\T} \log |f| dm>- \infty,
\]
then there exists a function $F \in H^2$ such that $|F|=|f|$ on $\T$, and thus $f= UF$, for some unimodular $U \in L^\infty(\T)$.

\begin{thm} Let $f\in L^2(\T)$. Then the linear span of the set
\[
\{ \zeta^n f(\zeta): n=0,1,2,\dots \}
\]
is dense in $L^2(\T)$ if and only if $f\neq 0$ $dm$-a.e on $\T$ and $\int_{\T} \log |f|dm = -\infty$.
\end{thm}
\begin{proof} We denote by $[f]$ be the smallest closed $M_\zeta$-invariant subspace in $L^2(\T)$ containing $f$. We first argue that the conditions on $f$ are necessary. If $f=0$ on a subset $S$ of positive Lebesgue measure in $\T$, then it easily follows that $[f]\subseteq L^2(\T \setminus S)$, which is a proper subspace of $L^2(\T)$, hence $f$ cannot be cyclic. On the other hand, if 
\[
\int_{\T} \log |f| dm >- \infty,
\]
then an outer function $F \in H^2(\T)$ can be constructed so that $F=Uf$ for some unimodular function $U\in L^\infty(\T)$. But then $[f]=[F] \subseteq H^2(\T)$, which again is a proper subspace of $L^2(\T)$. 

Conversely, assume that $f \in L^2(\T)$ with $f\neq 0$ $dm$-a.e on $\T$, and $\log |f| \notin L^1(\T)$. Now the logarithmic divergence of $f$ implies that $[f]$ cannot take the form $UH^2(\T)$ for some unimodular function $U$, since non-trivial elements in $H^2(\T)$ have integrable logarithms. Furthermore, since $f$ vanishes almost nowhere on $\T$, $[f]$ can neither be of the form $L^2(S)$ for some measurable subset $S\subset \T$ which does not have full Lebesgue measure in $\T$. The Beurling--Wiener Theorem then ensures that $[f]= L^2(\T)$, which proves the claim.
\end{proof}

\subsection{Ideals in the Wiener algebra}
Moving forward, we shall now consider the problem of translations in the context of the Wiener algebra $A(\T)$, which consists of functions $f\in C(\T)$ with absolutely summable Fourier coefficients:
\[
\norm{f}_{A} := \sum_{n\in \mathbb{Z}} \abs{\widehat{f}(n)} < \infty. 
\]
In other words, the space $A(\T)$ is isometrically isomorphic with the sequence space $\ell^1$ via Fourier coefficients. The designation "algebra" attached to the space $A(\T)$ arises from the fact that $A(\T)$ is closed under pointwise multiplication, as demonstrated by the estimate
\[
\norm{f\cdot g}_{A} = \sum_{n} \Big \lvert \sum_k \widehat{f}(k)\widehat{g}(n-k) \Big \rvert \leq \norm{f}_{A} \norm{g}_{A}.
\]
In fact, the Wiener algebra is one of most prominent example in which Harmonic analysis meets the theory of (commutative) Banach algebras. In this more algebraic setting, closed translation invariant subspace of $A(\T)$ are most commonly referred to as closed ideals. Examples of closed ideals are given by
\begin{equation}\label{EQ:MAXIDEALSWIENER}
\mathcal{I}_E := \{f\in A(\T): f=0 \,\, \text{on} \, \, E \} \qquad E \subseteq \T \, \, \text{compact}.
\end{equation}

A complete characterization of all closed ideals in $A(\T)$, in a similar spirit to Beurling's Theorem, is unfortunately not available here. However, one can still obtain a result which is somewhat close to complete, as we shall see in \thref{}

Let us primarily attempt the simpler task of classifying the 
so-called \emph{maximal ideals} in $A(\T)$. Recall that an ideal $\mathcal{M}$ is said to be maximal, if whenever $\mathcal{I}$ is an ideal with 
\[
\mathcal{M} \subseteq \mathcal{I} \subseteq A(\T)
\]
then wither $\mathcal{I} =\mathcal{M}$ or $\mathcal{I} =A(\T)$. Note that maximal ideal must necessarily be closed. The principal result in this regime shows that all maximal ideals are of a fairly simple form.

\begin{thm}[Maximal ideals] \thlabel{THM:MAXIDWIENER} Every maximal ideal in $A(\T)$ is of the form 
\[
\mathcal{I}_{\zeta_0} = \left\{f\in A(\T): f(\zeta_0)=0 \right\} \qquad \zeta_0 \in \T.
\]
\end{thm}

In modern times, this result is typically proved using Banach algebra techniques, notably the so-called Gelfand transform. But we shall here remain within the realm of harmonic analysis. 

The proof of this result relies on classical result of Wiener, which typically goes under the name of Wiener's $1/f$-Theorem, and neatly characterizes all cyclic elements on $A(\T)$.

\begin{thm}[Wiener]\thlabel{THM:WIEN1/f} For a function $f\in A(\T)$, the following conditions are all equivalent:
\begin{enumerate}
    \item[(i)] $f\neq 0$ on $\T$. 
    \item[(ii)] $f$ is cyclic in $A(\T)$, that is, the closed linear span of
    \[
    \{f(\zeta) \zeta^n: n=0,\pm 1, \pm 2, \dots \}
    \]
    is dense in $A(\T)$.
    \item[(iii)] $f$ is invertible in $A(\T)$, that is, $1/f \in A(\T)$.
\end{enumerate}
\end{thm}
We remark that the principal part of this Theorem is the statement $(i) \implies (iii)$, and modern proofs of this result typically appears in the literature of commutative Banach algebras, which originates from the profound work of I. Gelfand. For instance, see \cite[Ch. VIII]{katznelson2004introduction}. However, we shall below outline a simple, short and most importantly, a self-contained proof by D. J. Newman in \cite{newman1975simple}. It essentially only requires the following simple lemma.

\begin{lemma}\thlabel{LEM:A1W1inf} Let $W^{1,\infty}(\T)$ denote the Sobolev space of essentially bounded functions $f$ on $\T$ whose distributional derivative is also essentially bounded on $\T$, equipped with the norm:
\[
\norm{f}_{W^{1,\infty}} := \norm{f}_{L^\infty} + \norm{f'}_{L^\infty}.
\]
Then there exists a numerical constant $C>0$, such that the following norm inequalities hold:
\[
\norm{f}_{L^\infty} \leq \norm{f}_{A} \leq C\norm{f}_{W^{1,\infty}} \qquad \forall f \in W^{1,\infty}(\T).
\]
\end{lemma}
\begin{proof}
The proof is simple and principally based on the following observation that follows from Cauchy--Schwarz inequality
\[
\sum_{n} \abs{\widehat{f}(n)} \leq \left( \sum_n (1+|n|)^2 \abs{\widehat{f}(n)}^2 \right)^{1/2} \cdot \left( \sum_n \frac{1}{(1+|n|)^2} \right)^{1/2} \leq C \left( \int_{\T} \abs{f}^2 + \abs{f'}^2 dm \right)^{1/2}.
\]
We omit the simple details.
\end{proof}

\begin{proof}[Proof of \thref{THM:WIEN1/f}]
\proofpart{1}{$(iii)$ implies $(ii)$:}
Assume $1/f \in A(\T)$ and approximate $1/f$ by trigonometric polynomials in the norm of $A(\T)$. Now since $A(\T)$ is a Banach algebra closed under pointwise multiplication, which corresponds to convolution on the Fourier side, multiplying by $f$ we see that polynomial multiples of $f$ approximate $1$ in the $A(\T)$--norm. Now since the polynomials are dense, we conclude that $f$ is indeed cyclic.
\proofpart{2}{$(ii) \implies (i)$} Arguing by contraposition, we note that if $f(\zeta_0)=0$, then the product $fT$ with $T$ being a trigonometric polynomial can never uniformly approximate a function $g$ with $g(\zeta_0)\neq 0$ on $\T$. 
\proofpart{3}{$(i)\implies (iii)$} Below, we shall outline a neat argument by D. J. Newman. Suppose $f\neq 0$ on $\T$ and note that by means of re-scaling $f$, we may assume that $\abs{f(\zeta)}\geq 1$ on $\T$. Fix a number $0<\varepsilon<1$ to be specified later, and pick a trigonometric polynomials $P_\varepsilon$ such that $\norm{f-P_\varepsilon}_{A}\leq \varepsilon$. Now using the following simple algebraic manipulation, it remains to show that the formal series expansion
\[
\frac{1}{f} = \frac{1}{P_{\varepsilon}} \frac{1}{1- \frac{P_\varepsilon - f}{P_\varepsilon}} = \frac{1}{P_{\varepsilon}} \sum_{n=0}^\infty \left(\frac{P_{\varepsilon} - f}{P_{\varepsilon}}\right)^{n},
\]
actually converges in the $A(\T)$-norm. First, we should at least verify that $1/P_{\varepsilon} \in C(\T)$, but this readily follows from the normalization of $f$:
\[
\varepsilon \geq \norm{f-P_\varepsilon}_{A} \geq \norm{f-P_\varepsilon}_{L^\infty}\geq \abs{f(\zeta)}- \abs{P_\varepsilon(\zeta)} \geq 1- \abs{P_\varepsilon(\zeta)}, \qquad \zeta \in \T,
\]
and thus $\norm{1/P_{\varepsilon}}_{L^\infty}\leq (1-\varepsilon)^{-1}$. Furthermore, if $M(\varepsilon):= \norm{P'_\varepsilon}_{L^\infty}$, then 
\[
\Big \lvert \left(\frac{1}{P^n_\varepsilon}\right)' \Big \rvert = n \frac{\abs{P'_\varepsilon}}{\abs{P_\varepsilon}^{n+1}} \leq M(\varepsilon) n (1-\varepsilon)^{-n-1}, 
\]
and hence we obtain the following Sobolev estimate of $1/P^{n}_\varepsilon$:
\[
\norm{1/P^n_{\varepsilon}}_{W^{1,\infty}} \leq (1-\varepsilon)^{-n} + M(\varepsilon) n (1-\varepsilon)^{-n-1} \leq 2M(\varepsilon) n (1-\varepsilon)^{-n-1}, \qquad n=0,1,2,\dots.
\]
With this estimate, using that $A(\T)$ is a Banach algebra, in conjunction with \thref{LEM:A1W1inf}, we obtain the following estimate of the terms in the formal series expression of $1/f$:
\begin{multline*}
\Big \lVert \left(\frac{P_{\varepsilon} - f}{P_{\varepsilon}}\right)^{n} \Big \rVert_{A} \leq \norm{(P_{\varepsilon} - f)^n}_{A} \cdot \norm{1/P^n_{\varepsilon}}_{A} \leq  \\ 
\norm{P_{\varepsilon} - f}^n_{A} \cdot \norm{1/P^n_{\varepsilon}}_{W^{1,\infty}} \leq 2M(\varepsilon)  \frac{n\varepsilon^n}{(1-\varepsilon)^{n+1}}, \qquad n=0,1,2,\dots
\end{multline*}
which is decays geometrically if $0<\varepsilon<1/2$. The proof is complete.
\end{proof}

We mention the following remarkable extension of Wieners $1/f$-Theorem, which was originally established by P. Levy.

\begin{thm}[Levy--Wiener] If $f\in A(\T)$ and $F$ is an analytic function in a neighborhood of the range of $f$, then the composition $F(f)$ belongs to $A(\T)$.
\end{thm}
\begin{proof}[Sketch of proof:]
One can imitate Newman's argument in the proof of Wiener's $1/f$-Theorem, now using the Taylor series expansion
\[
F(f)= F(P_\varepsilon + f-P_\varepsilon)  = \sum_{n=0}^\infty \frac{F^{(n)}(P_\varepsilon)}{n!} \left( f-P_\varepsilon\right)^n,
\]
where the trigonometric polynomial $P_\varepsilon$ are chosen so that $\norm{f-P_\varepsilon}_{A}$ is substantially smaller than the radius of convergence of $F$. We leave the details to the interested reader.
\end{proof}

With Wiener's $1/f$-Theorem at our disposal, we proceed to characterize all maximal ideals in the Wiener algebra.

\begin{proof}[Proof of \thref{THM:MAXIDWIENER}]
Let $\mathcal{M}$ be a maximal ideal, and consider the common zero set of $\mathcal{M}$ as the compact set
\[
E:= \bigcap_{f\in \mathcal{M}} f^{-1}(\{0\}).
\]
Observe that $\mathcal{M}\subseteq \mathcal{I}_E$. We now divide the proof in two different cases:

\proofpart{1}{$E$ is non-trivial:} Assume that the common zero set $E$ of $\mathcal{M}$ is non-trivial. Then $\mathcal{I}_E \subsetneq A(\T)$, hence by maximality of $\mathcal{M}$, we must have $\mathcal{M}= \mathcal{I}_E$. It therefore only remains to show that $E$ is a singleton. Note that for any $\zeta_0 \in E$, we have the inclusions
\[
\mathcal{I}_E \subseteq \mathcal{I}_{\zeta_0} \subsetneq A(\T).
\]
Once we show that $\mathcal{I}_{\zeta_0}$ is a maximal ideal, it readily follows that 
\[
\mathcal{M}= \mathcal{I}_E = \mathcal{I}_{\zeta_0},
\]
hence $E=\{\zeta_0\}$. To this end, assume that $\mathcal{J} \supsetneq \mathcal{I}_{\zeta_0}$ is a closed ideal, and pick a non-trivial $f\in \mathcal{J} \setminus \mathcal{I}_{\zeta_0}$, which necessarily satisfies $f(\zeta_0)\neq 0$. But since $f-f(\zeta_0) \in \mathcal{I}_{\zeta_0}\subset \mathcal{J}$, it follows that 
\[
1 = \frac{f}{f(\zeta_0)} + \frac{f(\zeta_0)-f}{f(\zeta_0)} \in \mathcal{J}.
\] 
Since the trigonometric polynomials are dense in $A(\T)$, we conclude that $\mathcal{J}= A(\T)$, hence $\mathcal{I}_{\zeta_0}$ is a maximal ideal.
\proofpart{2}{Lack of a common zero set:}

We now assume that $E= \emptyset$. This means that for any $\zeta \in \T$, we can find a function $f_\zeta \in A(\T)$ such that $f_\zeta(\zeta)\neq 0$, any by continuity of $f_\zeta$, we can find an open neighborhood $U_\varepsilon\subset \T$ of $\zeta$, such that $f_{\zeta}\neq 0$ on $U_\varepsilon$. Since $\{U_{\zeta}\}_{\zeta \in \T}$ forms an open cover of $\T$, we may invoke compactness in order to exhibit a finite sub-cover $U_{\zeta_1}, \dots, U_{\zeta_n}$. But then using the algebra assumption on $A(\T)$, we get that
\[
F(\zeta) = \sum_{j=1}^n f_{\zeta_j}(\zeta) \conj{f_{\zeta_j}(\zeta)} =\sum_{j=1}^n \abs{f_{\zeta_j}(\zeta)}^2 >0, \qquad \zeta \in \T
\]
and is a member of the maximal ideal $\mathcal{M}$. By Wiener's $1/f$--Theorem, $F$ must be cyclic in $A(\T)$, which at its turn forces $\mathcal{M}=A(\T)$.

\end{proof}

We now prove that the closed ideals in $A(\T)$ are essentially  characterized by zero sets. To clarify this point, we introduce the ideal
\[
\mathcal{J}_E := \text{Clos}\{f\in A(\T): f=0 \, \, \text{in some open nbh of} \, \, E \}, \qquad E \subseteq \T \, \, \text{compact}.
\]

\begin{thm}[Ideals in the Wiener algebra] \thlabel{THM:IDEALSWIENER} For any closed ideal $\mathcal{I}$ in $A(\T)$, there exists a unique compact set $E\subseteq \T$ such that
\[
\mathcal{J}_E \subseteq \mathcal{I} \subseteq \mathcal{I}_E.
\]
\end{thm}
A surprising and deep result of P. Malliavin is that there are sets $E$ for which $\mathcal{J}_E \neq \mathcal{I}_E$. On the other hand, a complementary result of W. Rudin says that any closed ideal $\mathcal{I}$ in $A(\T)$ gives rise to a unique common zero set:
\[
E(\mathcal{I})=\bigcap_{f \in \mathcal{I}} f^{-1}(\{0\}).
\]
In other words, the map $\mathcal{I} \mapsto E(\mathcal{I})$ is injective. For further details, we refer the reader to \cite[Ch. 7]{RudinFOURIERANALYSIS}.

\begin{proof}[Proof of \thref{THM:IDEALSWIENER}] 
As before, we denote the common zero set of $\mathcal{I}$ by 
\[
E:= \bigcap_{f\in \mathcal{I}} f^{-1}(\{0\}).
\]
The inclusion $\mathcal{I}\subseteq \mathcal{I}_E$ is obvious, hence we only need to show that $\mathcal{J}_E \subseteq \mathcal{I}$. To this end, it suffices to show that any arbitrary $f\in A(\T)$ which vanish in some open neighborhood of $E$, belongs to $\mathcal{I}$. 

Set $S:=\operatorname{supp}(f)$ and note that $S \subset \T \setminus E$. By assumption on $E$, we may for any $\zeta \in S$, find $f_\zeta \in \mathcal I$ such that $f_\zeta(\zeta)\neq 0$. By continuity, each $f_\zeta$ is non-vanishing on some neighborhood $U_\zeta$ of $\zeta$. Since the open sets $\{U_\zeta\}_{\zeta \in S}$ cover $S$, compactness allows us to extract a finitely subcover of $S$
\[
S \subseteq \bigcup_{j=1}^n U_{\zeta_j}.
\]
Set $f_j:=f_{\zeta_j} \in \mathcal I$, and define the function
\[
F:=\sum_{j=1}^n |f_j|^2 \in \mathcal I.
\]
By construction we have that $F>0$ on $S$. Pick a smooth positive function $\chi$ such that $\chi=1$ in a neighborhood of $S$, and set
\[
G:=F+(1-\chi).
\]
It is evident that $G \in A(\T)$, has no zeros on $\T$, and hence by Wiener's $1/f$-Theorem, it follows that $1/G \in A(\T)$. Furthermore, since $\chi=1$ on $S=\operatorname{supp}(f)$, we have $G=F$ on $S$, and thus 
\[
f=\frac{\chi f}{G}\,F, \qquad \T.
\]
Now the algebra assumption implies that $\frac{\chi f}{G} \in A(\T)$, and since $F \in \mathcal M$, we conclude that 
\[
f= F\frac{\chi f}{G} \in \mathcal M.
\]

\end{proof}

We conclude this section by presenting an alternative perspective on Wiener’s $1/f$-Theorem, viewing it as a Tauberian theorem. This interpretation is not immediately apparent from the classical formulation. Broadly speaking, Tauberian theorems aim to identify additional conditions under which a weaker mode of convergence can be strengthened to a stronger one.
Historically, A. Tauber showed that if a sequence of complex numbers ${c_n}$ converges to $0$ in the Cesàro sense and satisfies the Tauberian condition $c_n = o(1/n)$, then in fact $c_n \to 0$. Subsequently, J. E. Littlewood proved that this condition can be weakened to the corresponding big-O estimate.
A general form of Wiener’s Tauberian theorem, formulated in the setting of sequences, can be stated as follows.

\begin{thm}[Wiener Taburian Theorem] \thlabel{THM:WTTHM}Let $f \in A(\T)$ with $f\neq 0$ on $\T$, and assume that $(b_n)_n$ is a bounded sequence of complex numbers, for which the following limit exists:
\[
\lim_{n\to \infty} \sum_{k} \widehat{f}(n-k) b_k = L \sum_k \widehat{f}(k).
\]
Then the limit 
\[
\lim_{n \to \infty} \sum_{k} a_{n-k} b_k = L \sum_k a_k,
\]
exists for all complex-valued sequences $\{a_n\}_n \in \ell^1$.
\end{thm}
In other words, the Wiener Tauberian Theorem asserts that if a bounded sequence converges in certain mean with respect to a cyclic element in $A(\T)$, then it can be upgraded to converge in any such mean. For instance, one can take $f(\zeta) = (1-r\zeta)^{-1}$ with $0<r<1$,

\begin{proof}[Proof of \thref{THM:WTTHM}] 
Let $\mathcal{I}$ denote the set of all $g \in A(\T)$ with the property that the limit exists
\[
\lim_{n \to \infty} \sum_{k} \widehat{g}(n-k) b_k = L \sum_k \widehat{g}(k).
\]
Note that $\mathcal{I}$ is a linear space, invariant under translation. Furthermore, it also closed in $A(\T)$, hence it is a closed ideal, which by assumptions contains a function $f\in A(\T)$ with $f\neq 0$ on $\T$. According to Wiener's Theorem we conclude that $\mathcal{I}$ contains a cyclic element $f$, but is easily seen to imply that $\mathcal{I}=A(\T)$.
\end{proof}

\subsection{Lacunary Fourier series and Riesz products}\label{SEC:LACRIESZ}
General Fourier series can behave pretty badly and are somewhat too broad of a class to study in some meaningful way. Here, we restriction our attention to special class of Fourier series with very sparse support. We shall illustrate they exhibit more clear-cut dichotomies in regards to various properties.

Let $\Lambda \subset \mathbb{Z}$ is a subset of lacunary integers, that is, integers with are exponentially sparse:
\[
\kappa(\Lambda) := \inf \left\{ \frac{|n|}{|m|}: m,n \in \Lambda, |n|>|m| \right\}>1.
\]
We now consider an associated lacunary trigonometric series of the form
\[
F(\zeta) := \sum_{n \in \Lambda} \widehat{F}(n) \zeta^{n}, \qquad \zeta \in \T.
\]
In other words, Lacunary series are trigonometric series with exponentially sparse Fourier non-zero Fourier coefficients. A concrete example of an exponentially sparse sequence is $\{\pm 3^j\}_j$, which gives rise to a lacunary trigonometric series (sometimes also called Hadamard gap series). 

A key feature of lacunary series is that they exhibit a local-to-global phenomenon. This can be viewed as a manifestation of the uncertainty principle:
\begin{princ}
If the Fourier spectrum is highly compressed (in the sense of large lacunary spectral gaps), then the function's spatial properties, whether good (regularity) or bad (singularity), cannot remain localized. In fact, they must spread out maximally along the unit circle.
\end{princ}
The following classical results clarifies one aspect of this point.

\begin{thm}[Ostrowski--Hadamard] Let $\{N_j\}_j$ be lacunary sequence of positive integers and let $F$ be an analytic function in the unit-disc $\D$ defined by
\[
F(z) := \sum_{j} a_j z^{N_j}, \qquad z\in \D.
\]
If $F$ extends analytically across a single arc in $\T$, then it extends analytically across $\T$.

\end{thm}

\begin{thm}[Hadamard--Ostrowski] \thlabel{THM:HADAOSTROW}
Let $(N_j)$ be a lacunary sequence and
\[
F(z)=\sum_{j=1}^\infty a_j z^{N_j}
\]
analytic in $\D$. If $F$ extends analytically across a point of $\T$, then it extends analytically across all of $\T$.
\end{thm}
The proof presented below is a compressed version of Rudin's proof in \cite[Theorem 16.6]{RudinReal&complex}.

\begin{proof}
Without loss of generality, we may assume that $F$ extends analytically across a neighborhood of $\zeta=1$. Since $(N_j)$ is lacunary, we may choose an integer $\lambda\geq 1$ such that
\[
N_{j+1}>\Bigl(1+\frac{1}{\lambda}\Bigr)N_j,
\qquad j=1,2,3,\dots.
\]
Define the analytic polynomial
\[
\varphi(w):=\tfrac12\bigl(w^\lambda+w^{\lambda+1}\bigr), \qquad w \in \C,
\]
and set
\[
G(z):=F(\varphi(z)).
\]

We claim that $G$ extends analytically across the unit circle. To see this, we note that
\[
|\varphi(w)|<1 \quad \text{for } w \in \D, 
\qquad \varphi(1)=1,
\]
and
\[
\varphi'(1)=\tfrac12(2\lambda+1)\neq 0,
\]
it follows that $\varphi$ is conformal near $w=1$, and hence maps a neighborhood of $1$ onto a neighborhood of $1$. Since $F$ extends analytically across $1$, the composition $G=F\circ\varphi$ is therefore analytic in a neighborhood of the closed unit disc. In particular, there exists $\varepsilon>0$ such that
\[
G(z)=\sum_{m=0}^\infty b_m z^m, \qquad |z|<1+\varepsilon.
\]
On the other hand, we also have by definition that
\[
G(z)=\sum_j a_j \varphi(z)^{N_j}, \qquad z\in \D.
\]
For each $j$, a binomial expansion of the polynomial $\varphi(z)^{N_j}$ shows that it only contains monomial terms with exponents in
\[
\lambda N_j\leq m\leq (\lambda+1)N_j.
\]
Invoking the lacunary assumption $(\lambda+1)N_j<\lambda N_{j+1}$, we get that these intervals are pairwise disjoint, hence the following holds:
\[
\sum_{0\leq k\leq j} a_k \varphi(z)^{N_k}
=
\sum_{0\leq m\leq (\lambda+1)N_j} b_m z^m, \qquad j=1,2,3, \dots
\]
As the right-hand side converges uniformly for $|z|<1+\varepsilon$, so does the partial sums of the lacunary series. We conclude that $F$ extends analytically to some disc $|z|>1$.
\end{proof}

Moving forward, we now consider the behavior of lacunary trigonometric series. For simplicity, we shall solely restrict our attention to lacunary sequences $\Lambda$ with $\kappa(\Lambda)\geq 3$, but essentially all results presented here holds in general for $\kappa(\Lambda)\geq 1$, but are more technical. 



We shall now further illustrate that lacunary series are completely determined by their local behavior, where the notion of local depends on the rate of decay of the coefficients.

\begin{thm}[Zygmund] \thlabel{THM:ZYGL2}Let $\Lambda \subset \mathbb{Z}$ be a lacunary set of integers, and $f\in L^1(\T)$ with Fourier spectrum in $\Lambda$. Then there exists a constant $C(\Lambda)>0$ such that 
\[
\left( \sum_{n\in \Lambda} \abs{\widehat{f}(n)}^2 \right)^{1/2} \leq C(\Lambda) \int_{\T} |f| dm.
\]
Furthermore, for any Lebesgue measurable subset $E\subset \T$ of positive Lebesgue measure, there exists a constant $C(\Lambda, E)>0$, and a number $N(E)>0$ such that
\begin{equation}\label{EQ:LACL2Est}
 \sum_{\substack{n\in \Lambda \\ |n|>N(E)}} \abs{\widehat{f}(n)}^2  \leq C(\Lambda, E) \int_E \abs{f}^2 dm.
\end{equation}
\end{thm}

We now record the following important corollary.

\begin{cor} Let $\Lambda \subset \mathbb{Z}$ be lacunary integers and consider the lacunary series
\[
f(\zeta)= \sum_{n\in \Lambda} \widehat{f}(n)\zeta^n, \qquad \zeta \in \T.
\]
The the following statements hold:
\begin{enumerate}
    \item[(i)] If the series of $f$ converges pointwise on a set of positive Lebesgue measure, then $f\in L^2(\T)$.
    \item[(ii)] If $f\in L^2(\T)$ vanishes on a set of positive Lebesgue measure, then it vanishes identically in $\T$.
\end{enumerate}

\end{cor}
\begin{proof}
The proof of $(ii)$ is immediate from the local estimate in \eqref{EQ:LACL2Est}. To see $(i)$, assume $f\notin L^2(\T)$ but yet its Fourier partials sums $(S_n)$ converge pointwise on a set $E$ of positive Lebesgue measure in $\T$. By means of slightly sacrificing the size of $E$, we may upon invoking Egoroff's Theorem assume that the $(S_n)$ converge uniformly on $E$. But this is in violating with \eqref{EQ:LACL2Est}.
\end{proof}

Note that in view of Carleson's Theorem on convergence of Fourier series, the converse of $(i)$ is also true. However, one can give a much simpler proof for lacunary Fourier series, using Littlewood--Paley Theory. We shall not do it here, and refer the reader to Zygmund in \cite{zygmundtrigseries}.

Turning our attention to proving Zygmund's \thref{THM:ZYGL2}, we shall need to introduce another class of objects, which are somewhat the multiplicative counter part of lacunary series, and exhibit at least equally intriguing properties. Let $\{N_j\}_j$ be a lacunary sequence of positive integers with
\[
\kappa:=\inf_{j} \frac{N_{j+1}}{N_j}\geq 3.
\]
Let $(a_j)_j$ be real numbers in the interval $[-1,1]$, and define the following products:
\begin{equation*}\label{EQ:RieszReal}
    P_n(\zeta) := \prod_{j=1}^n \left( 1+ a_j \Re \zeta^{N_j} \right), \qquad \zeta \in \T, \qquad n=1,2,3,\dots
\end{equation*}
Note that the assumptions on $(a_j)_j$ ensure that $P_n$ are positive trigonometric polynomials with average 
\[
\widehat{P_n}(0)= \int_{\T} p_n(\zeta) dm(\zeta)=1, \qquad n=1,2,3,\dots
\]
Appealing to weak-star compactness of norm-bounded sets in the space $M(\T)$, we can extract a weak-star cluster point $\sigma \in M(\T)$, which is again a probability measure on $\T$, which refer to as the \emph{Riesz product} 
\begin{equation*}
d\sigma(\zeta):= \prod_{j=1}^\infty \left( 1+ a_j \Re \zeta^{N_j} \right)dm(\zeta) , \qquad \zeta \in \T \qquad \eqno{(R_1)}
\end{equation*}
associated with the parameters $(a_j)_j$ and $(N_j)_j$. 
A crucial property of these Riesz products is that the Fourier support of $\sigma$ is contained in the union of disjoint blocks of the form 
\begin{equation}\label{EQ:Lambdaj}
\Lambda_j := \left\{ (1-(\kappa-1)^{-1}) N_j \leq |n| \leq (1+(\kappa-1)^{-1}) N_j \right\}, \qquad  j=0,1,2,\dots,
\end{equation}
where $\Lambda_0 := \{1\}$ by default. In other words, the Fourier support of Riesz products live exponentially sparse blocks of integers. To verify this claim, one exploits the following identity in an inductive manner:
\[
P_{n+1}(\zeta) = P_n(\zeta) + a_{n+1} \Re(\zeta^{N_{n+1}}) P_n(\zeta), \qquad \zeta \in \T, 
\]
and in conjunction with the observation that the lacunary assumption on $(N_j)_j$ for $\kappa \geq 3$ ensures that the Fourier spectrum of two above two terms are mutually disjoint. 

Exploiting the property in \eqref{EQ:Lambdaj}, it is straightforward to show that the following Fourier identity on each Fourier block $\Lambda_j$ holds:
\[
\sum_{n\in \Lambda_{j}} \abs{\widehat{\sigma}(n)}^2 = a^2_{j} \prod_{k=1}^j (1+a^2_k), \qquad j=1,2,3,\dots.
\]
From this it follows that 
\[
\sum_j a^2_j < \infty \iff d\sigma = f dm \, \, \text{with} \, \, f\in L^2(\T, dm).
\]
Furthermore, if $\sum_k a^2_k$ diverges, then one can show that $\sigma$ is a probability measure which is singular wrt $dm$ on $\T$, but has full support, see \cite[Ch. V, Theorem 7.6]{zygmundtrigseries}. The curious reader may also attempt to deduce this claim from Zygmund's \thref{THM:ZYGL2}.

In summary, the Lebesgue decompositions of a Riesz product is either purely singular wrt $dm$, or purely absolutely continuous and with a density in $L^2(\T)$.

In a similar way one can also define Riesz products with purely imaginary non-zero frequencies:
\[
Q_n(\zeta) := \prod_{j=1}^n \left( 1+ ia_j \Re \zeta^{N_j} \right), \qquad \zeta \in \T, \qquad n=1,2,3,\dots
\]
where $-1<a_j<1$ for all $j$. This time, we have that 
\[
1\leq \abs{Q_n(\zeta)}^2 \leq \prod_{j=1}^\infty(1+a^2_j), \qquad \zeta \in \T, \qquad n=1,2,3,\dots.
\]
Now if $\sum_j a^2_j < \infty$, we may appeal to weak-star compactness of norm-bounded sets in $L^\infty(\T)$, in order to extract a subsequence which converges to the associated Riesz product:
\[
f(\zeta) := \prod_{j=1}^\infty \left( 1+ ia_j \Re \zeta^{N_j} \right), \qquad \zeta \in \T, \qquad \eqno{(R_2)}
\]
which is an essentially bounded function in $\T$.

We now turn to the proof of Zygmund's Theorem.
\begin{proof}[Proof of \thref{THM:ZYGL2}]
\proofpart{1}{The $L^2$-estimate:}
Let $\{\lambda_k\}_k$ be the enumeration of $\Lambda$ in increasing order. We shall first assume that $f\in L^1(\R)$ is real-valued. Consider the family of partial Riesz products of the form $(R_2)$:
\[
P_N(\zeta) = \prod_{k=1}^N \left(1 + i a_k \Re(\zeta^{\lambda_k} e^{i\theta_k} ) \right), \qquad \zeta \in \T,
\]
where $e^{-i\theta_k} \widehat{f}(\lambda_k) = \abs{f(\lambda_k)}$ and
\[
a_k := \frac{1}{2}\widehat{f}(\lambda_k) \left( \sum_{|k|\leq N}  \abs{\widehat{f}(\lambda_k)}^2 \right)^{-1/2}. 
\]
Note that 
\[
\abs{P_N(\zeta)}^2 \leq \prod_{k=1}^N (1+\abs{a_k}^2 ) \leq \exp \left( \sum_{k=1}^N \abs{a_k}^2 \right)\leq e^{1/4}, \qquad \zeta \in \T.
\]
With this at hand, we get
\begin{align*}
\frac{1}{4} \left(\sum_{|k|\leq N} \abs{\widehat{f}(\lambda_k)}^2\right)^{1/2} = \sum_{|k|\leq N} \widehat{f}(\lambda_k) a_k e^{-i\theta_k} \\
= \int_{\T} f(\zeta) \conj{P_N(\zeta)} dm(\zeta) \leq \norm{f}_{L^1} \norm{P_N}_{L^\infty} \leq e^{1/2} \norm{f}_{L^1}.
\end{align*}
If $f$ is complex-valued, then we simply apply the previous argument to both the real and imaginary parts of $f$.
\proofpart{2}{Local $L^2$-estimates:} Fix a subset $E\subset \T$ of positive Lebesgue measure, and we shall primarily assume that $f$ has Fourier coefficients supported in the lacunary set $\Lambda \cap \{|n|>N_0\}$, where $N_0>0$ is a large integer do be determined momentarily. We start by a straightforward expansion of the $L^2$-norm of $f$ restricted to $E$, which gives
\begin{multline*}
\int_{E} \abs{f}^2 dm = \sum_{k,j} \widehat{f}(\lambda_k) \conj{\widehat{f}(\lambda_j)} \int_{E} \zeta^{\lambda_k-\lambda_j} dm(\zeta) = \\
m(E) \sum_{k} \abs{\widehat{f}(\lambda_k)}^2 + 2 \sum_{(j,k):j<k} \widehat{f}(\lambda_k) \conj{\widehat{f}(\lambda_j)} \widehat{1_E}(\lambda_j-\lambda_k).
\end{multline*}
Applying Cauchy-Schwartz inequality to the last sum, we get 
\begin{multline*}
\big \lvert \sum_{(j,k):j<k} \widehat{f}(\lambda_k) \conj{\widehat{f}(\lambda_j)} \widehat{1_E}(\lambda_j-\lambda_k) \big \rvert \leq \left( \sum_{(j,k):j<k} \abs{\widehat{f}(\lambda_k)}^2 \abs{\widehat{f}(\lambda_j)}^2 \right)^{1/2} \left( \sum_{(j,k):j<k} \abs{\widehat{1_E}(\lambda_j-\lambda_k)}^2 \right)^{1/2} \\
\leq \norm{f}^2_{L^2} \left( \sum_{(j,k):j<k} \abs{\widehat{1_E}(\lambda_j-\lambda_k)}^2 \right)^{1/2} = \norm{f}^2_{L^2} \left( \sum_{|n|>N_0/3} \rho (n) \abs{\widehat{1_E}(n)}^2 \right)^{1/2},
\end{multline*}
where $\rho(n)$ denotes the number of integer pairs $(j,k)$ with $j<k$ and such that $n=\lambda_j-\lambda_k$. Now note that the lacunary assumption implies that $\lambda_k \geq 2(\lambda_1+\dots+ \lambda_{k-1})$, hence every integer $n$ has a unique representation of the form 
\[
n= \varepsilon_1 \lambda_1+ \varepsilon_2 \lambda_2 + \dots + \varepsilon_l \lambda_l,\qquad \varepsilon_s \in \{-1,0,1\},
\]
But this implies that $\rho(n)\leq 1$ for all $n$, and we therefore conclude that 
\[
\int_E \abs{f}^2 dm \geq  \sum_k \abs{\widehat{f}(\lambda_k)}^2 \left( m(E)  - 2\left( \sum_{|n|>N_0/3} \abs{\widehat{1_E}(n)}^2 \right)^{1/2} \right) \geq C(E,N_0) \sum_k \abs{\widehat{f}(\lambda_k)}^2,
\]
if $N_0>0$ chosen large enough. This completes the proof. 
\end{proof}

We mention on the passing that using the theory of Paley--Littlewood square functions, one can actually show that every $f\in L^1(\T)$ with lacunary Fourier support belongs to $L^p(\T)$. In fact, keeping a closed eye on the asymptotics of the embedding-constants in $L^p(\T)$ as $p\to \infty$, one can even show that containment into the Orlicz space $\exp(L^2)$:
\[
\int_{\T} \exp \left( c|f|^2 \right) dm < \infty,
\]
for some constant $c=c(f)>0$. This lies slightly outside the scope of our theme, and will therefore not be treated here. We refer the reader to E. Stein in \cite{stein1993harmonic}.

We shall now consider yet another result of A. Zygmund, which characterizes when lacunary series are essential bounded. It turns out that every essentially bounded function $f$ with Fourier coefficients supported on a lacunary set is automatically continuous on $\T$.

\begin{thm}[Zygmund] \thlabel{THM:ZYGL00}Let $\Lambda \subset \mathbb{Z}$ be a lacunary set of integers, and let $f\in L^\infty(\T)$ with Fourier coefficients spectrum in $\Lambda$. Then there exists a constant $C(\Lambda)>0$ such that 
\[
\sum_{n\in \Lambda} \abs{\widehat{f}(n)} \leq C(\Lambda) \norm{f}_{L^\infty}.
\]
Moreover, if $f\in L^2(\T)$ is essentially bounded on an arc $I\subset \T$, then there exists a constant $C(\Lambda,I)>0$ and a number $N(E)>0$, such that 
\[
\sum_{\substack{n\in \Lambda \\ |n|>N(E)}} \abs{\widehat{f}(n)} \leq C(\Lambda, I) \norm{f}_{L^\infty(I)}.
\]
\end{thm}
\begin{proof}
\proofpart{1}{$\ell^1$-summable coefficients:}
Let $(\lambda_k)_k:=\Lambda$ be an increasing enumeration of the lacunary sequence, and we shall for simplicity assume that $\kappa(\Lambda)\geq 3$. We shall first assume that $f$ is real-valued. In a similar vain as the proof of Zygmund's Theorem, we shall this time use Riesz products of type $(R_1)$:
\[
P_N(\zeta) := \prod_{k=1}^n \left( 1+\Re( e^{i\theta_j} \zeta^{\lambda_j}) \right), \qquad \zeta \in \T, \qquad N=1,2,3\dots
\]
where $e^{-i\theta_k} \widehat{f}(\lambda_k) = \abs{\widehat{f}(\lambda_k)}$. With this at hand, we have 
\[
\frac{1}{2}\sum_{|k|\leq N} \abs{\widehat{f}(\lambda_k)} = \sum_{|k|\leq N} \widehat{f}(\lambda_k)e^{i\theta_k} = \int_{\T}f(\zeta) P_N(\zeta) dm(\zeta) \leq \norm{f}_{L^\infty(\T)}, \qquad N=1,2,3,\dots
\]
Sending $N \to \infty$, we get the conclusion for real-valued $f$, and the extension to complex-valued is now obvious.

\proofpart{2}{Local estimate:}
Fix an arbitrary closed arc $I \subset \T$ and assume that $\supp{\widehat{f}} \subseteq \{|n|\geq n(I)\}$ where $n(I)>0$ large integer to be determined later. Fix a non-negative smooth bump-function $\psi_I$ which is supported on $I$ and satisfies $\widehat{\psi_I}(0)=1$, and note that a similar argument as in the previous step shows that
\[
\frac{1}{2}\sum_{|k|\leq N} \widehat{f}(\lambda_k) \widehat{P_N \cdot \psi_I}(-\lambda_k) =\int_{\T} f(\zeta) \psi_I(\zeta) P_N(\zeta) dm(\zeta) \leq \norm{f}_{L^\infty(I)} \norm{\psi_I}_{L^\infty}.
\]
We claim that it suffices to show that there exists a constant $0<c(I)<1/2$ and a large integer $n(I)>0$ such that 
\begin{equation}\label{EQ:PNPSI}
\abs{\widehat{P_N\cdot \psi_I}(n) - \widehat{P_N}(n)} \leq c(I), \qquad n(I)\leq |n| \leq \lambda_N.
\end{equation}
Indeed, once that task is accomplished, we use the fact that $P_N \psi_I$ is real-valued, and estimate as follows:
\begin{multline*}
\norm{f}_{L^\infty(I)} \norm{\psi_I}_{L^\infty} \geq \frac{1}{2} \abs{\sum_{n(I)\leq |k|\leq N} \widehat{f}(\lambda_k) \conj{\widehat{P_N}(\lambda_k)} } - \frac{c(I)}{2}\sum_{n(I)\leq |k|\leq N} \abs{\widehat{f}(\lambda_k)} \geq \\
\frac{1}{2} \abs{\sum_{n(I)\leq |k|\leq N} \widehat{f}(\lambda_k) e^{-i\theta_k} } - \frac{1}{4} \sum_{n(I)\leq |k|\leq N} \abs{\widehat{f}(\lambda_k)} = \frac{1}{4}\sum_{n(I)\leq |k|\leq N} \abs{\widehat{f}(\lambda_k)}.
\end{multline*}
This proves the desired estimate. Now in order to establish \eqref{EQ:PNPSI}, we use that $\widehat{\psi_I}(0)=1$, in conjunction with the fact that Fourier coefficients of $P_N$ are unimodular. We only give a rough sketch of the argument below:
\begin{multline*}
\abs{\widehat{P_N\cdot \psi_I}(n) - \widehat{P_N}(n)} \leq \sum_{j\neq n} \abs{\widehat{P_N}(j) \widehat{\psi_I}(n-j)} \\ \lesssim \norm{\psi''_I}_{L^\infty} \sum_{\substack{j: j\neq n \\ j \in \supp{\widehat{P_N}}}} \frac{1}{|n-j|^2} \lesssim  \norm{\psi''_I}_{L^\infty} \sum_{j>n(I)} \frac{1}{j^2} =: c(I).
\end{multline*}
This completes the proof.
\end{proof}

We shall now record yet another manifestation of the classical uncertainty principle on the simultaneous smallness of $f$ and its Fourier coefficients. The following corollary can be proved using a simple adaptation of the proof used to establish the second statement of \thref{THM:ZYGL00}, and whose details are left to the reader.
\begin{cor}[Zygmund] Let $\mu \in M(\T)$ with Fourier coefficients supported in a lacunary subset of integers. If $\abs{\mu}(I)=0$ for some arc $I \subset \T$, then $\mu \equiv 0$. 
\end{cor}
In other words, measures with lacunary Fourier spectrum must have full support, while measures with non-full support cannot have sparse Fourier spectrum.

We now close this subsection by clarifying how smoothness properties of lacunary series can also be encoded from their Fourier coefficients, thus we have yet another local-to-global phenomenon is this regime. 
 
\begin{thm}[Paley--Zygmund] Let $\{\lambda_k\}_k \subset \mathbb{Z}$ be a lacunary set, and $f \in L^\infty$ with Fourier coefficients supported in $\Lambda$. Then for any number $0<\alpha<1$, the following statements are all equivalent:
\begin{enumerate}
    \item[(i)] $f$ is $\alpha$-H\"older continuous on $\T$,
    \item[(ii)] $f$ is $\alpha$-H\"older continuous at a single point $\zeta_0 \in \T$,
    \item[(iii)] $\widehat{f}(\lambda_k) = \mathcal{O}(|\lambda_k|^{-\alpha})$, as $|k| \to \infty$.
\end{enumerate}
\end{thm}
\begin{proof}
Let $\{\lambda_k\}_k$ be an increasing enumeration of $\Lambda$. We only need to prove $(ii) \implies (iii)$ and $(iii) \implies (i)$. 

\proofpart{1}{$(ii) \implies (iii)$:} Without loss of generality, we may assume $\zeta_0 =1$. Since $f\in L^\infty(\T)$, Zygmund's \thref{THM:ZYGL00} implies that $f\in C(\T)$ with the property that there exists $C>0$ such that 
\[
\abs{f(\zeta)-f(1)} \leq C \abs{\zeta-1}^{\alpha}, \qquad \zeta \in \T.
\]
Now since $(\lambda_k)_k$ is lacunary, there exists a constant $c>0$ such that 
\[
\widehat{f}(n)=0, \qquad 1\leq \abs{n-\lambda_k}\leq c|\lambda_k|, \qquad k=1,2,3,\dots
\] 
For sufficiently large $\abs{\lambda_k}$, we can pick integers $c\abs{\lambda_k}/2 \leq N_k < c |\lambda_k|$ and note that one may express
\[
\widehat{f}(\lambda_k) = \sum_{n} \widehat{F_{N_k}}(n)\widehat{f}(\lambda_k - n) = \int_{\T} \left( f(\zeta) - f(1) \right) \mathcal{F}_{N_k}(\zeta) \zeta^{-\lambda_k} dm(\zeta),
\]
where $\mathcal{F}_N$ denotes the classical Fej\'er kernel given by the explicit formula
\[
\mathcal{F}_{N}(e^{it}) := \sum_{|k|< N}\left(1 - \frac{\abs{k}}{N} \right) e^{ikt} = \frac{1}{N}\frac{\sin^2(Nt/2)}{\sin^2 (t/2)}, \qquad t\in \R.
\]
Fix a number $0<\delta <1/2$ to be determined later and let $I(\delta)$ be the arc centered at $\zeta=1$ of length $\delta$. The H\"older estimate of $f$ at the point $\zeta_0=1$ implies
\[
\abs{\int_{I(\delta)} \left( f(\zeta) - f(1) \right) \mathcal{F}_{N_k}(\zeta) \zeta^{-\lambda_k} dm(\zeta) } \lesssim \int_{I(\delta)} \abs{\zeta-1}^{\alpha} \abs{\mathcal{F}_{N_k}(\zeta)}dm(\zeta) \leq \delta^{\alpha} \int_{\T} \abs{\mathcal{F}_{N_k}} dm \leq \delta^{\alpha}.
\]
Outside the interval $I(\delta)$, we use the rough estimate $\mathcal{F}_{N}(\zeta) \lesssim N^{-1} \abs{\zeta-1}^{-2}$, which implies
\[
\abs{\int_{\T \setminus I_(\delta)} \left( f(\zeta) - f(1) \right) \mathcal{F}_{N_k}(\zeta) \zeta^{-\lambda_k} dm(\zeta) } \lesssim N_k^{-1} \int_{\T\setminus I(\delta)} \abs{\zeta-1}^{\alpha-2} dm(\zeta) \lesssim N_k^{-1} \delta^{\alpha-1}.
\]
Now choosing $\delta \asymp \abs{\lambda_k}^{-1}$ and combining these estimates, we arrive at 
\[
\widehat{f}(\lambda_k) = \mathcal{O} \left( \abs{\lambda_k}^{-\alpha} \right), \qquad \abs{k} \to \infty.
\]
\proofpart{2}{$(iii) \implies (i)$:} Fix an arbitrary $\zeta, \xi \in \T$ and let $\lambda_n \in \Lambda$ is be the smallest integer with $\abs{\zeta-\xi} \geq |\lambda_n|^{-1}$. Using the lacunary assumption on $(\lambda_n)_n$, we get
\begin{multline*}
\abs{f(\zeta)-f(\xi)} \leq \sum_{k=1}^\infty \abs{\widehat{f}(\lambda_k)}\abs{\zeta^{\lambda_k}- \xi^{\lambda_k} }  \lesssim 
\abs{\zeta- \xi} \sum_{k=1}^n |\lambda_k|^{1-\alpha} + \sum_{k>n} |\lambda_k|^{-\alpha} \\ \lesssim \abs{\zeta-\xi} \abs{\lambda_n}^{\alpha-1}  + \abs{\lambda_n}^{-\alpha} \lesssim \abs{\zeta- \xi}^{\alpha}.
\end{multline*}
This proves that $f$ is $\alpha$-H\"older continuous on $\T$.

\end{proof}
We remark that if $f$ is a general $\alpha$-H\"older continuous function on $\T$, then we have the Fourier decay
\[
\widehat{f}(n) = \mathcal{O}(|n|^{-\alpha} ), \qquad |n| \to \infty.
\]
However, the converse is far from true in this generality. There exists positive measures $\mu$ in $\T$ supported in sets of Lebesgue measure zero, hence singular wrt $dm$, such that
\[
\widehat{\mu}(n) = \mathcal{O}(|n|^{-\alpha}), \qquad |n|\to \infty.
\]
In fact, the Ivashev--Musatov Theorem in Section 6 proves something much stronger.

\subsection{Fourier coefficients of measures} 
We now turn to a systematic study of the behavior of Fourier coefficients of elements in $M(\T)$, the space of finite complex Borel measures on $\T$. Within this framework, our goal is to explore the interplay between localization, as reflected in the smallness of the support, and regularity, as captured by the decay of the Fourier coefficients.

Recall that the classical Lebesgue decomposition from measure theory allows us to decompose any $\mu \in M(\T)$ of the form:
\[
d\mu = f dm + d\mu_c + d\mu_d
\]
where $f\in L^1(\T)$ denotes the Radon-Nikodym derivative of $\mu$ wrt the arc length measure $dm$, $\mu_c$ denotes the continuous singular part of $\mu$, while $\mu_d$ denotes the discrete part, written as a countable sum of atomic masses. As customary in the literature, $fdm$ is the \emph{absolutely continuous} part of $\mu$. The sum $\mu_s := \mu_c + \mu_d$ as the \emph{singular part} of $\mu$. Measures with trivial discrete part $\mu_d$, that is, of the form $d\mu =fdm + d\mu_c$ are typically said to be \emph{continuous measures}.

We shall start out with, perhaps, one of the most basic and classical observation in Fourier analysis. 

\begin{lemma}[Riemann--Lebesgue] \thlabel{LEM:RIELEB}For any $f\in L^1(\T)$ we have 
\[
\lim_{|n|\to \infty} \widehat{f}(n) = 0.
\]
\end{lemma}
To prove the lemma, we shall use the following simple and observation: if $f$ is $k$-times continuously differentiable function on $\T$, then repeated integration by parts gives:
\[
\widehat{f}(n) = (-in)^k \int_{\T} f^{(k)}(\zeta) \zeta^{-n} dm(\zeta) = \mathcal{O}(|n|^{-k}), \qquad |n|\to \infty.
\]

\begin{proof}
Approximate $f$ by a smooth function $f_\varepsilon$ in the norm of $L^1(\T)$, for instance by taking convolution of $f$ with an approximate of the identity. Integrating the Fourier coefficients of $f_\varepsilon$ by parts (once is enough) and using the triangle inequality, we get
\[
\abs{\widehat{f}(n)} \leq \abs{\widehat{f_\varepsilon}(n)} +\norm{f-f_\varepsilon}_{L^1} \leq C(\varepsilon)|n|^{-1} +\norm{f-f_\varepsilon}_{L^1}.
\]
The claim now follows by sending $|n| \to \infty$ first, and then $\varepsilon \to 0+$.
\end{proof}
Note that if $f\in L^2(\T)$ then its Fourier coefficients must tend to zero, in view of Parseval's identity:
\[
\sum_n \abs{\widehat{f}(n)}^2 =\int_{\T} \abs{f(\zeta)}^2 dm(\zeta)< \infty.
\]
This simple observation, in conjunction with $L^2(\T)$ being a dense subset of $L^1(\T)$, yields yet another (though principally similar) proof of the Riemann--Lebesgue Lemma. Now it is natural to ask whether the Riemann--Lebesgue lemma can be improved in some quantitative way. The following result refutes this in the strongest way possible.

\begin{thm}For any positive numbers $(w(n))_{n=0}^\infty$, there exists a positive function $f\in L^1(\T)$ such that
\[
\abs{\widehat{f}(n)} \geq w(n), \qquad n\neq 0.
\]
\end{thm}
Note that this result bears a closed resemble to \thref{THM:LEEUWKAHKATZ} by de Leeuw--Kahane--Katznelson, on behavior of Fourier coefficients of continuous functions. However, the proof below is much simpler, but it still requires some level of clear-headedness. 
\begin{proof}
 Note that by means of substituting the sequence $\{w(n)\}_{n\geq 0}$ by 
    \[
    \omega(n):=\sup_{k \geq n} w(k)
    \]
    we may in addition assume that $w(n) \downarrow 0$. We leave it as an exercise for the reader to construct a decreasing sequence of positive real numbers $(a(n))_{n\geq 0}$ such that
    \[
    \sum_{n} a(n) < \infty, \qquad \sum_{k> n} a(k) \geq w(n), \qquad n=0,1,2,\dots
    \]
With this at hand, consider claim that our candidate is given by the explicit formula
\[
f(\zeta) = \sum_{n=1}^\infty n(a(n-1)-a(n) ) \mathcal{F}_{n-1}(\zeta), \qquad \zeta \in \T,
\]
where $\mathcal{F}_n$ denotes the Fej\'er kernel of order $n$. Set 
\[
f_N(\zeta) =\sum_{k=1}^N k\left(a(k-1)-a(k) \right) \mathcal{F}_{k-1}(\zeta), \qquad N=1,2,3,\dots \qquad \zeta \in \T,
\]
and re-write the sum 
\[
\sum_{k=1}^n k \left( a(k-1)-a(k) \right) = \sum_{k=1}^{n-1} a(k) - na(n), \qquad n=1,2,3,\dots
\]
Using the monotonicity of $(a(n))_{n\geq 0}$, we have
\[
na(n) \leq 2 \sum_{n/2\leq k\leq n} a(k) \to 0, \qquad n\to \infty.
\]
Using these observations in conjunction with $\norm{\mathcal{F}_n}_{L^1}=1$ for all $n\geq 1$, we conclude that 
\[
\int_{\T}f_N dm = \sum_{k=1}^N k\left(a(k-1)-a(k) \right)  \leq \sum_{k=1}^\infty  a(k) < \infty, \qquad N=1,2,3,\dots
\]
By the monotone convergence theorem, we conclude that $f$ is a positive function in $L^1(\T)$. In fact, we also have that $f_N \to f$ in $L^1(\T)$. Furthermore, a straightforward computation of Fourier coefficients shows
\begin{multline*}
\widehat{f}(n) = \sum_{k=1}^\infty k \left( a(k-1)-a(k) \right) \widehat{\mathcal{F}_{k-1}}(n) \\ = \sum_{k\geq 1+|n|} k \left( a(k-1)-a(k) \right) \left( 1- \frac{|n|}{k} \right) = \sum_{k> |n|} a(k) \geq w(|n|).
\end{multline*}
This completes the proof.
\end{proof}

Next, we shall investigate the Fourier coefficients of measures with a non-trivial discrete component. The following lemma quantifies how a single point mass can be detected via Fourier series.  

\begin{lemma} \thlabel{PROP:WILIM}
Let $\mu \in M(\T)$. Then for any $\zeta_0 \in \T$:
\[
\mu ( \{\zeta_0 \}) = \lim_{N\to \infty} \frac{1}{2N+1} \sum_{|k|\leq N} \widehat{\mu}(k) \zeta_0^k
\]
\end{lemma}
\begin{proof}
Fix $\zeta_0 \in \T$ and consider the family of trigonometric polynomials
\[
Q_N(\zeta) = \frac{1}{2N+1} \sum_{|k|\leq N} \zeta^k \zeta_0^{-k}, \qquad \zeta \in \T, \qquad N=1,2,3,\dots
\]
In other words, $\{Q_N\}_N$ is a family of translated and re-normalized Dirichlet kernels, which satisfy the following properties:
\[
(i) \, \sup_{\T} \abs{Q_N} \leq 1, \qquad (ii) \, Q_N(\zeta_0)=1, \qquad (iii) \, \lim_{N\to \infty}\sup_{|\zeta-\zeta_0|\geq \delta} \abs{Q_N(\zeta)} = 0, \qquad \forall \delta >0.
\]
Exploiting these properties, using the dominated convergence theorem, and splitting the integral into a neighborhood of $\zeta_0$ and its complement, we obtain
\[
\frac{1}{2N+1} \sum_{|k|\leq N} \widehat{\mu}(k)\zeta_0^k =\int_{\T} Q_N(\zeta) d\mu(\zeta) \to \mu (\{\zeta_0\}).
\]

\end{proof}

We remark that an obvious modification of the $\{Q_N\}_n$ utilized in the proof actually shows that for any integers $(M_j)_j$ and $(N_j)_j$ with $N_j \uparrow \infty$, we have 
\[
\mu(\{\zeta\}) = \lim_{j\to \infty} \frac{1}{N_j+1} \sum_{k=M_j}^{M_j + N_j} \widehat{\mu}(k) \zeta^k, \qquad \zeta \in \T.
\]

As an immediate consequence, we obtain the following corollary, which characterizes measures with discrete support solely through their Fourier coefficients.

\begin{cor}[Wiener] \thlabel{COR:WIENERATOMICFOURIER}Let $\mu \in M(\T)$. Then 
\[
\sum_{\zeta \in A(\mu)} \abs{\mu(\{\zeta \})}^2 = \lim_{N\to \infty} \sum_{|n|\leq N} \abs{\widehat{\mu}(n)}^2,
\]
where $A(\mu)$ denotes the (countable) set of atoms of $\mu$. In particular, $\mu$ is a continuous measure if and only if 
\[
\lim_{N\to \infty} \frac{1}{2N+1} \sum_{|n|\leq N} \abs{\widehat{\mu}(n)}^2 =0.
\]
    
\end{cor}
Roughly speaking, Wiener's Theorem asserts that for any $\varepsilon>0$, the amplitudes $\abs{\widehat{\mu}(n)}\geq \varepsilon$ have zero density:
\[
\lim_{N\to \infty} \frac{\# \left\{ |n|\leq N: \abs{\widehat{\mu}(n) } > \varepsilon \right\} }{2N+1} =0.
\]
However, it turns out that the Fourier coefficients of a continuous measure do not have to be asymptotically small.

\begin{prop}[A continuous measure with non-vanishing amplitudes] \thlabel{PROP:CONTNORAJ} There are continuous measures $\mu \in M(\T)$ with $\widehat{\mu}(n) \nrightarrow 0$ as $|n|\to \infty$.
\end{prop}
\begin{proof}
We construct a general recipe using Riesz products. Let $(a_k)_k$ be positive numbers in $(0,1]$ consider an associated Riesz product of type $(R_1)$:
\[
d\mu(\zeta) := \lim_{N\to \infty} \prod_{j=1}^N \left( 1 + a_j \Re(\zeta^{3^j}) \right), \qquad \zeta \in \T,
\]
interpreted as the weak-star cluster point in $M(\T)$. Now if we impose the condition
\[
\sum_j a^2_j = \infty,
\]
then a classical result of A. Zygmund asserts that $\mu$ is singular wrt $dm$ on $\T$. For instance, see \cite{zygmundtrigseries}. Now to see that $\mu$ is continuous, we use the remark prior to \thref{PROP:WILIM}:
\[
3^{-k} \sum_{|n|\leq 3^{k}} \abs{\widehat{\mu}(n)}^2 \leq 3^{-k} \prod_{j=1}^k \left( 1+a^2_j \right) \leq 3^{-k} 2^{k} \to 0, \qquad k\to \infty,
\]
hence we conclude that $\mu$ is a continuous singular probability measure on $\T$. Now a successive expansion of the associated Riesz partial product $P_N$ imply that for each integer $N\geq 1$:
\[
P_N(\zeta) = 1+ \sum_{\varepsilon \neq 0}\left( 2^{-\sum_{j=1}^N \abs{\varepsilon_j} } \prod_{j:\varepsilon_j \neq 0} a_j \right) \zeta^{\left(3\varepsilon_1+3^2 \varepsilon_2+\dots+3^N \varepsilon_N\right)}
\]
where the sum is taken over all non-trivial tuples $\varepsilon:=(\varepsilon_1,\varepsilon_2,\dots, \varepsilon_N)\in \{-1,0,1\}^N$. Now if we impose that $a_j \to 0$, then we see that $\widehat{\mu}(n) \to 0$ as $|n|\to \infty$, while if we pick $a_j\equiv 1$, then $\limsup_{n\to \infty} \widehat{\mu}(n)>0$.

\end{proof}

Nevertheless, it turns out that there is a rigidity phenomenon for Fourier decay of measures. Roughly speaking, if the positive Fourier coefficients of a measure are asymptotically small, then the same must hold for the negative coefficients. This is made precise by the following result.

\begin{thm}[Rajchman, Milicer--Gruzewska] \thlabel{THM:RM-G}
If a measure $\mu \in M(\T)$ satisfies
\[
\lim_{n\to +\infty} \widehat{\mu}(n) = 0,
\]
then it automatically also satisfies
\[
\lim_{n\to -\infty} \widehat{\mu}(n) = 0.
\]
\end{thm}

Measure with the property $\widehat{\mu}(n) \to 0$ as $|n|\to \infty$ are called \emph{Rajchman measures}, and have historically played a tremendous role in studying uniqueness problems of Fourier series.

\begin{proof}[Proof of \thref{THM:RM-G}]
    Let $\mu \in M(\T)$ with the property that $\widehat{\mu}(n) \to 0$ as $n \to \infty$. Then for any trigonometric polynomial $T$, we have also have
    \[
   \sum_k \widehat{T}(k) \widehat{\mu}(n-k) =\int_{\T} T(\zeta) \zeta^{-n} d\mu(\zeta) \to 0, \qquad n\to \infty,
    \]
    hence measures of the form $d\mu_{T} := T d\mu$ also satisfy $\widehat{\mu_T}(n) \to \infty$ as $n\to \infty$. The principal idea is to show that the Fourier coefficients of the total-variation $\abs{\mu}$ of $\mu$ also tends to zero as $n \to \infty$, by means of approximating $\abs{\mu}$ by measures of the form $\mu_T$ in total variation-norm. Once that has been accomplished, we simply utilize the fact that since $\abs{\mu}$ is real-valued, thus  
    \[
    \widehat{\abs{\mu}}(n) = \widehat{\abs{\mu}}(-n) \to 0, \qquad n\to \infty,
    \]
    for which it easily follows that $\lim_{n\to -\infty} \widehat{\mu}(n)=0$. 
    
    To this end, recall that there exists a bounded Borel measurable function $\phi$ on $\T$ with $\abs{\phi}=1$ $\mu$-a.e and such that $d\abs{\mu}= \phi d\mu$. This is essentially the polar decomposition of a measure.
    
    Now, a classical Theorem of Lusin in measure theory asserts that for any $\varepsilon>0$, there exists $\psi_\varepsilon \in C(\T)$ such that $\sup_{\T} \abs{\psi_\varepsilon}\leq 1$ and 
    \[
    \abs{\mu} \left( \{\zeta\in \T: \phi\neq \psi_\varepsilon \} \right) \leq \varepsilon.
    \]
    Passing to appropriate Fej\'er means of $\psi_\varepsilon$, we can may exhibit trigonometric polynomials $T_\varepsilon$ such $\norm{\psi_\varepsilon-T_\varepsilon}_\infty \leq \varepsilon$. It follows that 
    \[
    \norm{\abs{\mu}- \mu_{T_\varepsilon} } \leq \int \abs{\phi-T_\varepsilon} d|\mu| \leq \int_{ \{\phi\neq \psi_\varepsilon \} } \abs{\phi-\psi_\varepsilon} d|\mu| + \norm{\mu} \varepsilon \leq \left(2+ \norm{\mu}\right)\varepsilon.
    \]
    With this estimate in mind, we obtain 
    \[
    \abs{\widehat{\abs{\mu}}(n)} \leq \norm{\abs{\mu}- \mu_{T_\varepsilon} } + \abs{\widehat{\mu_{T_\varepsilon}}(n)} \leq (2+\norm{\mu})\varepsilon + \abs{\widehat{\mu_{T_\varepsilon}}(n)}.
    \]
    Sending $n \to \infty$ first, and then $\varepsilon \to 0+$, the claim follows.
\end{proof}


Actually, a simple refinement of the above argument shows that the class of Rajchman measures are stable under absolute continuity. 

\begin{cor}[Rajchman, Milicer-Gruzewska] Let $\mu, \nu \in M(\T)$, with $\nu << \mu$. Then 
\[
\lim_{|n| \to \infty} \widehat{\mu}(n) =0 \qquad \implies \lim_{|n| \to \infty} \widehat{\nu}(n) =0.
\]
\end{cor}

For further quantitative refinements of the above results, we refer the reader to the work of de Leeuw and Katznelson in \cite{deleeuw1970two}.

On the contrary, it is natural to ask whether there corresponds to any sequence of complex numbers $\{c_n\}_n$ with $\lim_{|n|\to \infty} c_{n} = 0$, as measure $\mu\in M(\T)$ such that
\[
\widehat{\mu}(n)= c_n, \qquad n\in \mathbb{Z}.
\]
It turns out that not all sequences tending to zero arise as the amplitudes of a measure.

\begin{prop}\thlabel{PROP:Distc0} There exists a sequence of positive numbers $\{c_n\}_n$ with $\lim_{|n|\to \infty} c_{n} = 0$, which is not the Fourier coefficients of a complex finite Borel measure on $\T$.
\end{prop}
\begin{proof}

\proofpart{1}{Trigonometric polynomials with large mass and small amplitudes:}
The proof principally relies on the following construction: for any $\varepsilon>0$ there exists analytic polynomials $Q_\varepsilon$ satisfying
\[
\norm{Q_\varepsilon}_{L^1}\geq 1, \qquad \sup_{n} \abs{\widehat{Q_\varepsilon}(n)}\leq \varepsilon.
\]
Let $N(\varepsilon)$ be the smallest integer exceeding $1/\varepsilon^2$, and consider the lacunary analytic polynomial
\[
Q_\varepsilon(\zeta) := \varepsilon \sum_{n=1}^{N(\varepsilon)} \zeta^{3^n}, \qquad \zeta \in \T.
\]
By definition, we have $\sup_n \abs{\widehat{Q_\varepsilon}(n)} = \varepsilon$, while Parseval's identity implies that 
\[
\norm{Q_\varepsilon}^2_{L^2} \geq N(\varepsilon)\varepsilon^2 \geq 1.
\]
However, since $Q_\varepsilon$ is a lacunary polynomial, it follows from Zygmund's \thref{THM:ZYGL2} that $\norm{Q_\varepsilon}_{L^1} \gtrsim 1$, and the claim follows from a final re-scaling.

\proofpart{2}{Assembling the blocks:}
Let $\varepsilon_j \downarrow 0$, and $Q_j := Q_{\varepsilon_j}$ denote the trigonometric polynomials appearing in the statement above. Let $L_j$ denote the degree of $Q_j$, and choose positive integers $(N_j)_j$, large enough to satisfy
\[
N_{j+1}- N_j > 3L_{j+1}+ 3L_j.
\]
This ensures that the trigonometric polynomials $T_j := \zeta^{N_j}Q_j$ have mutually disjoint Fourier spectrum. Now consider the formal series 
\[
S(\zeta) := \sum_j \varepsilon^{-1/2}_j T_j(\zeta), \qquad \zeta \in \T.
\]
Now if $S_j$ denotes the Fourier spectrum of $T_j$, then we have 
\[
\abs{\widehat{S}(n)} \leq \varepsilon^{-1/2}_j \sup_{k} \abs{\widehat{Q_j}(k)} \leq \varepsilon^{1/2}_j, \qquad n\in S_j,
\]
hence $\lim_{|n|\to \infty} \widehat{S}(n)=0$. Consider the de la Vallée Poussin kernels defined as the trigonometric polynomials
\[
V_{N} := 2\mathcal{F}_{2N+1}-\mathcal{F}_{N}, \qquad N=1,2, 3 \dots
\]
where $\mathcal{F}_N$ denotes the Fej\'er polynomial of degree $N$. Since the Fourier coefficients satisfy $\widehat{V_N}(n)=1$ for $N<|n|\leq 2N$, vanish for $|n|>2N+1$, our choice of $N_j$ implies that
\[
S \ast \zeta^{L_j} V_{N_j}(\zeta) = \varepsilon^{-1/2}_j T_j(\zeta), \qquad j=1,2,3, \dots, \qquad \zeta \in \T.
\]
Now if we assume that $S \in M(\T)$, then 
\[
\varepsilon^{-1/2}_j \leq \varepsilon_j^{-1/2} \norm{T_j}_{L^1} = \norm{S \ast \zeta^{L_j} V_{N_j}}_{L^1} \leq 6\norm{S}_{M(\T)}, \qquad j=1,2,3,\dots
\]
which is incompatible for large enough $j$. This shows that $S$ cannot be a measure.
\end{proof}


It turns out that can much further and prove the following remarkable result, initially proved in \cite{piatetski1952problem}.

\begin{thm}[Piatetski--Shapiro] There exists a compact set $E\subset \T$ of Lebesgue measure zero, which satisfies the following properties:
\begin{enumerate}
    \item[(i)] $E$ support no non-trivial Rajchman measure,
    \item[(ii)] $E$ supports a non-trivial distribution $S$ with 
    \[
    \widehat{S}(n) \to 0, \qquad |n|\to \infty.
    \]
\end{enumerate}
    
\end{thm}

A neat and modern proof of this theorem was later given by J-P. Kahane, and a carefully outlined argument similar to that is contained in the PhD-Thesis of N. Lev, see \cite[Ch. 1]{lev2005piatetski}.

\subsection{Further results}
There is a general principle linking pointwise almost everywhere convergence to weak-type estimates for an associated maximal operator. Suppose $\{T_\alpha\}_\alpha$ is a family of linear operators acting on a Lebesgue space $L^p(X,\mu)$, where $(X,\mu)$ is a measure space. Define the associated maximal (sublinear) operator by
\[
T^* f(x) := \sup_{\alpha} |T_\alpha f(x)|, \qquad x \in X.
\]
If $T^*$ satisfies a weak-type estimate of the form
\[
\mu\big( \{x \in X : |T^* f(x)| > \lambda \} \big)
\leq C \frac{\|f\|_{L^p}^p}{\lambda^p}, \qquad \lambda > 0,
\]
then the set
\[
A := \left\{ f \in L^p(X,\mu) : \lim_{\alpha \to \alpha(0)} T_\alpha f(x) = f(x) \ \text{exists for $\mu$-a.e. } x \in X \right\}
\]
is closed in $L^p(X,\mu)$. The proof is simple, and can be found in \cite[P. 27]{duoandikoetxea2024fourier}.

If, in addition, one knows that $\lim_{\alpha \to \alpha(0)} T_\alpha f(x) = f(x)$ holds $\mu$-a.e. for all $f$ in a dense subset of $L^p(X,\mu)$, it follows that $A = L^p(X,\mu)$.

This principle underlies the proof of Carleson's theorem. In that setting, one studies the maximal partial sum operator
\[
S^* f(\zeta) := \sup_{N \geq 1} |S_N f(\zeta)|, \qquad \zeta \in \T, \quad f \in L^2(\T).
\]
Since pointwise convergence is immediate for smooth functions on $\T$, which form a dense subset of $L^2(\T)$, the central difficulty lies in establishing the corresponding weak-type estimate for the associated maximal operator $S^*$.

\medskip

We note that Beurling's theorem on $M_\zeta$-invariant subspaces extends to the Hardy spaces $H^p$ for all $p>0$, with the same description. In the case $p=\infty$, one must soften the topology on $H^\infty$ (to make it separable) and consider weak-star closed $M_\zeta$-invariant subspaces.

Beurling's theorem naturally raises the question of whether an analogous result holds in other spaces of analytic functions, most notably the Bergman space $L^2_a(\D)$ consisting of analytic functions $f$ on $\D$ with finite area integral
\[
\int_{\D} |f(z)|^2\, dA(z) < \infty.
\]
In this setting, the correspondence with boundary values available in the Hardy space $H^2(\D)$ is lost, and one no longer has access to a simple inner--outer factorization. Nevertheless, H.~Hedenmalm in \cite{HedenmalmFactBerg} and B.~Korenblum in \cite{korenblum1975extension}, showed that there are some forms of factorizations that still persist.

In analogy with Beurling's theorem, one studies closed $M_z$-invariant subspaces of $L^2_a(\D)$, where $M_z f(z) = zf(z)$. It was shown by C.~Apostol, H.~Bercovici, C.~Foiaș, and C.~Pearcy that there exist closed $M_z$-invariant subspaces $\mathcal{V}$ of $L^2_a(\D)$ for which the wandering subspace
\[
\mathcal{V} \ominus M_z \mathcal{V} := \mathcal{V} \cap (M_z \mathcal{V})^\perp
\]
is infinite dimensional. In particular, $M_z$-invariant subspaces of $L^2_a(\D)$ are far from being generated by a single function, in contrast with Hardy space.

Despite this, Aleman, Richter, and Sundberg proved that a meaningful analogue of Beurling's theorem still holds in $L^2_a(\D)$, see \cite{aleman1996beurling}. More precisely, every closed $M_z$-invariant subspace $\mathcal{V}$ is wandering in the sense of Halmos, that is, $\mathcal{V}$ is the smallest $M_z$-invariant subspace containing $\mathcal{V} \ominus M_z \mathcal{V}$.

We also mention the influential work of H.~Hedenmalm on singly generated and zero-based $M_z$-invariant subspaces \cite{hedenmalm1994factoring}, which laid much of the foundation for modern Bergman space theory. For further background on Bergman spaces, see \cite{hedenmalmbergmanspaces}.

Given a subset of integers $\Lambda$, we denote by $C_{\Lambda}(\T)$ the closed subspace of continuous functions $f$ on $\T$ whose Fourier coefficients are supported in the frequencies $\Lambda$:
\[
C_{\Lambda}(\T):= \{f\in C(\T): \widehat{f}(n)=0, \qquad n\neq \Lambda \}.
\]
We say that $\Lambda$ is a \emph{Sidon set} if there exists a constant $C(\Lambda) > 0$ such that
\[
\sum_{n} |\widehat{f}(n)| \leq C(\Lambda)\, \sup_{\zeta \in \T} |f(\zeta)|, 
\qquad \forall f \in C_{\Lambda}(\T).
\]
As we have seen, lacunary sequences $\Lambda = (N_k)_k$ satisfying $\inf_k N_{k+1}/N_k > 1$, are Sidon sets. A striking arithmetic characterization of Sidon sets was obtained by G.~Pisier in \cite{pisier1983arithmetic}.

For an integer $n$, let $R(n,\Lambda)$ denote the number of representations of $n$ as a finite sum
\[
n = \sum_j \varepsilon_j \lambda_j,
\]
where $\lambda_j \in \Lambda$ and $\varepsilon_j \in \{-1,0,1\}$. Pisier showed that $\Lambda$ is a Sidon set if and only if there exists $0 < \gamma < 1$ such that, for every finite subset $\Gamma \subset \Lambda$,
\[
\sup_{n \in \mathbb{Z}} R(n,\Gamma) \leq 3^{\gamma |\Gamma|}.
\]

In other words, additive relations among elements of $\Lambda$ must be exponentially sparse. J.~Bourgain also proved some intimate connections to associated Riesz products, in \cite{bourgain1985sidon}. See also Kronecker’s theorem on independent sets in Chapter~VI of \cite{katznelson2004introduction}.

\medskip

A very deep result which extends \thref{THM:HADAOSTROW} of Hadamard and Ostrowski on natural boundaries for lacunary Taylor series, is the following result:
\begin{thm}[Fabry's gap Theorem] Let $\{N_j\}_{j=1}^\infty$ be an increasing sequence of positive integers satisfying
\[
\lim_{j\to \infty}N_j /j = \infty.
\]
If $R>0$ denotes the radius of convergence of the analytic power series
\[
F(z):=\sum_{j=1}^\infty a_j z^{N_j}, \qquad |z|<R
\]
then $F$ does not extend analytically across any point in $|z|=R$.
\end{thm}
It turns out that the converse of Fabry's Theorem is also true, and is a Theorem of P\'olya. There are several proofs of Fabry's gap Theorem in \cite{havinbook}, where one is based on Turan's lemma, on estimating the supremum norm of polynomials on $\T$ in terms of the supremum norm on subarcs. There is also an analogue of Fabry's gap Theorem for Dirichlet series due to Carlson-Landau, see \cite[Ch. 4, §7]{havinbook}.

\medskip

In connection with the Poisson summation formula, a natural problem is to characterize discrete infinite measures $\mu$ on $\R$ whose Fourier transform $\widehat{\mu}$ is of a similar form, with both measures supported on uniformly discrete sets, that is,
\[
\inf_{\lambda \neq \lambda'} |\lambda - \lambda'| > 0.
\]
Such measures are necessarily of the form
\[
\mu = \sum_{\lambda \in \Lambda} a(\lambda)\,\delta_{\lambda}, 
\qquad 
\widehat{\mu} = \sum_{\xi \in \Xi} \widehat{\mu}(\xi)\,\delta_{\xi},
\]
where $\Lambda$ and $\Xi$ are uniformly discrete subsets of $\R$. Here we regard these objects as tempered distributions. Measures of this type are called \emph{crystalline measures}, and at first glance, one might expect the uncertainty principle to obstruct the existence of such objects. However, a basic example, arising from the Poisson summation formula, is given by the Dirac comb
\[
\mu = \sum_{n \in \mathbb{Z}} \delta_n.
\]
It is therefore of interest to understand crystalline measures and the structure of the sets that support them. These questions are closely related to the theory of quasicrystals, developed to a large extent by Y.~Meyer, whose mathematical framework was later instrumental in interpreting some experimental discoveries of the Nobel Laureate in Chemistry D.~Shechtman. We refer the reader to the survey of P.~Kurasov \cite{kurasov2025quasicrystal} and to Meyer’s work \cite{meyer1995quasicrystals} for further background.

\section{Logarithmic integrals on the unit-circle}
In this section, we shall explore the convergence of the logarithmic integral 
\[
\int_{\T} \log |f(\zeta)| dm(\zeta) >- \infty
\]
and determine its governing role in several approximation and uniqueness problems arising in harmonic analysis and complex function theory.

\subsection{The F. \& M. Riesz Theorem} \label{SUBSEC:RIESZTHM}
We shall begin this story with a theorem of the brothers F. and M. Riesz, on the characterization of measures whose Fourier coefficients are supported on positive frequencies.  

\begin{thm}[F. and M. Riesz] \thlabel{THM:FMRIESZ} Let $\mu$ be a non-trivial complex finite Borel measure on $\T$ with the property that 
\[
\int_{\T} \zeta^n d\mu(\zeta) = 0, \qquad n=0,1,2,\dots
\]
Then $\mu$ is absolutely continuous wrt to $dm$ on $\T$, and its Radon--Nikodym derivative $f:= d\mu /dm$ satisfies 
\[
\int_{\T} \log |f| dm >- \infty.
\]
\end{thm}
In the literature, the F. and M. Riesz Theorem is often regarded as a principal result within the theory of Hardy spaces, hence it is commonly phrased in complex analytic terms. To clarify this point, closed subspace of $L^1(\T)$:
\[
H^1(\T) := \{f\in L^1(\T): \widehat{f}(n)=0, \, \, n<0 \},
\] 
One may regard $H^1(\T)$ as the trace space of analytic functions $F$ in the analytic Hardy space $H^1(\D)$:
\[
\sup_{0<r<1} \int_{\T} \abs{F(r\zeta)} dm(\zeta) < \infty.
\]
More precise, every $F\in H^1(\D)$ has finite non-tangential limits equal to $f\in H^1(\T)$, while the Poisson extension to $\D$ of any $f\in H^1(\T)$ belongs to $H^1(\D)$. See appendix on Hardy spaces. In complex analytic terms, the F. and M. Riesz Theorem implies that any measure with vanishing negative Fourier coefficients is actually absolutely continuous wrt $dm$, with Radon-Nikodym derivative in $H^1(\T)$.

Clearly, the conclusion of the F. and M. Riesz Theorem holds if we only assume that the Fourier coefficients $\widehat{\mu}(n)=0$ for $n>M$ or $n<M$, for some $M\in \R$. In fact, the F. and. M. Riesz is yet another manifestation uncertainty principle in harmonic analysis.

\begin{princ} If the Fourier coefficients of a measure does not meet some neighborhood of either $\pm \infty$, then the measure is absolutely continuous with logarithmic integral Radon-Nikodym derivative on $\T$. 

\end{princ}

Neighborhoods at infinity are sets of the form $(-\infty, M)$ and/or $(M,\infty)$, where $M \in \mathbb{R}$. The proof presented below will be based on two key lemmas, which are both important in their own right. At then end, we shall also give the short proof bases on Poisson extensions and boundary behavior of Hardy space functions. Our first lemma is about uniform approximation by holomorphic polynomials on sets of Lebesgue measure zero, whose utility will re-surface on several occasions later.

\begin{lemma}[Uniform approximation on small sets]\thlabel{LEM:APPROX1} Let $E\subset \T$ be a compact set of Lebesgue measure zero, and let $g$ be a continuous function on $\T$ with $\sup_{\zeta \in \T} \abs{g(\zeta)}=1$. For any $\varepsilon>0$, there exists analytic polynomials $(Q_n)_n$ satisfying:
\[
\sup_{|z|\leq 1} \abs{Q_n(z)} \leq 1, \qquad \lim_n \sup_{E} \abs{Q_n-g} = 0, \qquad \lim_n \sup_{\T \setminus E_{\varepsilon}} \abs{Q_n} = 0,
\]
where $E_\varepsilon := \{\zeta \in \T: \dist{\zeta}{E}\leq \varepsilon \}$ denotes $\varepsilon$-neighborhood of $E$.
\end{lemma}
It is important to remark that the hypothesis $m(E)=0$ is crucial, otherwise taking $g\equiv 1$ we would be able to show that the Hardy space $H^1(\D)$ contains a non-trivial indicator function $1_E$, which violates the fact that such functions are logarithmically integrable on $\T$, see Appendix. The proof of this lemma will require a cocktail of several classical techniques from real analysis.
\begin{proof}
\proofpart{1}{Reduction to trigonometric polynomials:} Fix $0<\varepsilon<1$, and pick a function $\chi_\varepsilon \in C(\T)$ satisfying the properties:
\[
\supp{\chi_\varepsilon} \subseteq E_{2\varepsilon}, \qquad 0\leq \chi_k\leq 1, \qquad \chi_k = 1 \, \, \text{on} \, \, E_{\varepsilon}.
\]
This is basically the classical Urysohn lemma in the setting of $\T$. Now consider the continuous functions $g_\varepsilon := g \cdot \chi_\varepsilon$ which satisfy $\abs{g_\varepsilon} \leq \abs{g}\leq 1$ on $\T$ and $g_\varepsilon = g$ on $E_{\varepsilon}$ and $g_\varepsilon = 0$ off $E_{2\varepsilon}$. By means of passing to appropriate Fej\'er means of $g_\varepsilon$, we can find trigonometric polynomials $(T_k)_k$ which satisfy the following properties:
\[
(i) \, \sup_{\zeta \in \T} \lvert T_k(\zeta) \rvert \leq 1, \qquad (ii) \, \sup_{\zeta \in E} \lvert T_k(\zeta)-g(\zeta) \rvert \leq 1/k, \qquad (iii) \, \sup_{\zeta \in \T \setminus E_{3\varepsilon} } \lvert T_k(\zeta) \rvert \leq 1/k, 
\]
Thus, we have basically constructed trigonometric polynomials satisfying the required properties, hence it remains only to pass to analytic polynomials.
\proofpart{2}{Analytic polynomials:}
Let $N_k>0$ be a positive integer with the property that $\supp{\widehat{T_k}} \subseteq [-N_k, N_k]$. In this step we shall now make use of the crucial hypothesis $m(E)=0$. Indeed, $m(E)=0$ ensures that we can find a sequence of positive continuous functions $(\psi_k)_k$ on $\T$ and positive integers $N_k \uparrow+\infty$ such that 
\begin{equation}\label{EQ:psik}
\inf_{\zeta \in E} \psi_k(\zeta) \geq 2k, \qquad (2N_k +1) \int_{\T} \lvert \psi_k \rvert dm \leq 1, \qquad k=1,2,3, \dots
\end{equation}
Again, by means of passing to appropriate Fej\'er means, we may also assume that $\psi_k$ are real trigonometric polynomials of degree $> 2N_k$, say. This allows us to further define the non-trivial real-valued trigonometric polynomials
\[
p_k(\zeta) := \psi_k(\zeta)- \sum_{|n|\leq N_k} \widehat{\psi_k}(n) \zeta^n, \qquad k=1,2,3, \dots
\]
which satisfy $\supp{\widehat{p}_k} \cap [-N_k, N_k] = \emptyset$. Now consider the continuous functions $G_k = 1- \exp\left( -2P_+(p_k) \right)$ on $\T$, where
\[
P_+(p_k)(\zeta) = \sum_{n\geq0} \widehat{p_k}(n)\zeta^n,\qquad \zeta \in \T
\]
denotes the analytic part of the trigonometric polynomial $p_k$. Observe that the Fourier coefficients of each $P_+(p_k)$ are contained in $(N_k, \infty)$, hence it follows from the Taylor series expansion of the exponential function that also $\supp{\widehat{G_k}} \subseteq (N_k, \infty)$. This implies that $\supp{\widehat{T_k G_k}}\subseteq[0,\infty)$. Furthermore, since $p_k$ is real-valued, we also have
\[
\lvert T_k - T_k G_k \rvert = \lvert T_k \rvert \exp ( - 2\Re P_+(p_k) ) = \lvert T_k \rvert \exp ( - p_k ) \leq \lvert T_k \rvert \exp \left(- \psi_k + \sum_{|n|\leq N_k} |\widehat{\psi_k}(n)|  \right).
\]
Invoking the second estimate in \eqref{EQ:psik}, the last term in the exponent of the above equation as follows:
\[
\sum_{|n|\leq N_k} |\widehat{\psi_k}(n)| \leq  \sum_{|n|\leq N_k} \int_{\T} \abs{\psi_k} dm = (2N_k+1) \int_{\T}\abs{\psi_k} dm \leq 1.
\]
With this estimate at hand, we obtain the following uniform estimate on $E$:
\[
\sup_{\zeta \in E} \lvert T_k(\zeta) - T_k(\zeta) G_k(\zeta) \rvert \leq \exp(-2k + 1)
\]
Furthermore, outside each neighborhood $E_k$ of $E$, we also have
\[
\sup_{\T \setminus E_k} \lvert T_k G_k \rvert \leq \lvert T_k - T_k G_k \rvert + \sup_{\T \setminus E_k} \lvert T_k \rvert \leq (e+1) \sup_{\T \setminus E_k} \lvert T_k \rvert \leq (e+1)/k.
\]
Again, passing to appropriate Fej\'er means of $T_k G_k$, we now obtain holomorphic polynomials $(Q_k)_k$, which approximate $T_k$ uniformly on $E$ and arbitrary well off every neighborhood $E_k$ of $E$. In order to also ensure that $\abs{Q_n}\leq 1$, we just have to re-normalize these Fej\'er means. This completes the proof.

\end{proof}

As a corollary of \thref{LEM:APPROX1}, we also record the following result.

\begin{cor}\thlabel{COR:1Eapprox}
For any compact set $E \subset \T$ of Lebesgue measure zero, and any continuous function $g$ on $\T$ with $\sup_{\zeta \in \T} \abs{g(\zeta)}=1$, there exists analytic polynomials $(Q_n)_n$ with 
\[
\sup_{|z|\leq 1} \abs{Q_n(z)}\leq 1, \qquad \sup_E \abs{Q_n-g} \to 0, \qquad Q_n \to 0, \qquad \text{dm-a.e on} \, \, \T.
\]
\end{cor}
\begin{proof}
The proof is just a classical diagonalization-argument of \thref{LEM:APPROX1}. Indeed, pick a sequence $\varepsilon_k \to 0$ and consider for each $\varepsilon_k$-neighborhood of $E$ an analytic polynomial $Q_k$ satisfying 
\[
\sup_{|z|\leq 1} \abs{Q_k(z)} \leq 1, \qquad \sup_{E} \abs{Q_k-g} \leq \varepsilon_k, \qquad \sup_{\T \setminus E_{\varepsilon_k}} \abs{Q_k} \leq \varepsilon_k.
\]
\end{proof}

The next lemma is typically referred to as Jensen's lemma. The reader should not confuse it with Jensen's inequality for convex functions, although they are of similar flavor.

\begin{lemma}[Jensen's lemma] \thlabel{LEM:JENSEN} For any $f \in H^1(\T)$, we have 
\[
\log \lvert \widehat{f}(0) \rvert \leq \int_{\T} \log |f| dm
\]
\end{lemma}
We remark that in view of Jensen's inequality for convex functions, this is essentially an improvement of the triangle inequality:
\[
\lvert \widehat{f}(0) \rvert \leq \int_{\T} \lvert f \rvert dm,
\]
indicative of the fact that elements in the subspace $H^1(\T)$ exhibit much better behavior than all of $L^1(\T)$. Jensen's lemma is typically proved using Hardy space theory, for instance, see \cite{garnett}, but we shall below give a simple self-contained proof, as in \cite{havinbook}.

\begin{proof}[Proof of \thref{LEM:JENSEN}]
\proofpart{1}{Jensen's lemma for polynomials:} We shall first argue that Jensen's lemma holds for all analytic polynomials $Q$. In that case, $\log |Q|$ is a subharmonic function on $\C$, and since $\widehat{Q}(0) = Q(0)$ Jensen's lemma follows from the sub-mean inequality.

\proofpart{2}{Upgrade to $H^1$:} We shall now upgrade the estimate to all function in $H^1(\T)$. To this end, fix an arbitrary $f \in H^1(\R)$ and consider its Fej\'er means $f_j \in H^1(\T)$, which converge to $f$ in $L^1(\T)$, by Fej\'er's Theorem. Now for any $\varepsilon>0$, the step 1 of proof implies that
\[
\log |\widehat{f_j}(0) + \varepsilon| \leq \int_{\T}\log |f_j+ \varepsilon| dm.
\]
Now using the trivial estimate $\abs{\log |f_j +\varepsilon| - \log |f+\varepsilon|} \leq \frac{1}{\varepsilon}|f-f_j|$, we get 
\[
\log |\widehat{f_j}(0) + \varepsilon| \leq \int_{\T} \log |f+\varepsilon| dm + \frac{1}{\varepsilon} \int_{\T} \abs{f-f_j} dm.
\]
Sending $j\to \infty$, we arrive at 
\[
\log |\widehat{f}(0) + \varepsilon| \leq \int_{\T} \log |f+\varepsilon| dm \leq \int_{\T} \log \left(|f|+\varepsilon \right) dm.
\]
It now only remains to send $\varepsilon\to 0+$ and invoking the monotone convergence Theorem.

\end{proof}

We now prove the main result in this subsection.
\begin{proof}[Proof of \thref{THM:FMRIESZ}]
Let $\mu \in M(\T)$ with the property that
\[
\int_{\T} \zeta^n d\mu(\zeta) =0, \qquad n=0,1,2,\dots. 
\]
We first show that $\mu$ is absolutely continuous with respect to $dm$ on $\T$.

\medskip
\noindent
\emph{Step 1:  Reduction to a continuous measure.}
We primarily note that by the comment following \thref{PROP:WILIM}, used in the proof of Wiener's theorem, $\mu$ cannot have any atoms. Moving forward, we may therefore assume that $\mu$ is a continuous measure on $\T$.

\medskip
\noindent
\emph{Step 2: Polar decomposition and approximation.}
By the polar decomposition, there exists a unimodular function $\phi$ with $|\phi|=1$ $d|\mu|$-a.e.\ such that
\[
d|\mu| = \phi\, d\mu.
\]
By a classical Theorem of Lusin in measure theory, we may for each $\varepsilon>0$, find a continuous function $\psi_\varepsilon \in C(\T)$ with $\|\psi_\varepsilon\|_\infty \le 1$ such that
\[
|\mu|\big( \{ \zeta \in \T : \phi \neq \psi_\varepsilon \} \big) \le \varepsilon.
\]
Fix an arbitrary compact subset $E \subset \T$ with $m(E)=0$, and consider the $\varepsilon$-neighborhood of $E$ defined by
\[
E_\varepsilon := \{ \zeta \in \T : \dist{\zeta}{E} \leq \varepsilon \}.
\]
Invoking \thref{LEM:APPROX1}, we can select analytic polynomials $(Q_n)$ such that
\[
\sup_{\D} |Q_n| \le 1, \qquad
\sup_{\T \setminus E_\varepsilon} |Q_n| \to 0, \qquad
\sup_{E} |Q_n - \psi_\varepsilon| \to 0.
\]
Note that the assumption on $\mu$ implies that for all n:
\[
\int_{\T} Q_n \, d\mu = 0.
\]
Using this assumption and decomposing the above integral in the following way
\[
\int_{\T \setminus E_\varepsilon} Q_n \, d\mu + \int_E Q_n \, d\mu
= - \int_{E_\varepsilon \setminus E} Q_n \, d\mu,
\]
we obtain
\[
\Big| \int_{\T \setminus E_\varepsilon} Q_n \, d\mu + \int_E Q_n \, d\mu \Big|
\le |\mu|(E_\varepsilon \setminus E).
\]
Sending $n \to \infty$, we conclude
\[
\Big| \int_E \psi_\varepsilon \, d\mu \Big|
\le |\mu|(E_\varepsilon \setminus E).
\]

\medskip
\noindent
\emph{Step 3: Absolutely continuity:}
With the previous inequality at hand, we may now estimate as follows:
\begin{align*}
|\mu|(E) =\Big| \int_E \phi \, d\mu \Big| 
&\le \int_{\{\phi \neq \psi_\varepsilon\}} |\phi - \psi_\varepsilon| \, d|\mu|
   + \Big| \int_E \psi_\varepsilon \, d\mu \Big| \\
&\le 2\, |\mu|\big( \{\phi \neq \psi_\varepsilon\} \big)
   + |\mu|(E_\varepsilon \setminus E) \\
&\le 2\varepsilon + |\mu|(E_\varepsilon \setminus E).
\end{align*}
Letting $\varepsilon \to 0$ and using continuity of $\mu$, we conclude that $|\mu|(E)=0$.

By inner regularity of finite Borel measures on $\T$, this implies that $|\mu|(B)=0$ for every Borel set $B$ with $m(B)=0$, and thus we conclude that
\[
d\mu = f\, dm, \qquad f \in L^1(\T).
\]

\medskip
\noindent
\emph{Step 4: Hardy space and logarithmic integrability.}
Since $\widehat{f}(n)=\widehat{\mu}(n)$ for all $n$, we have $f \in H^1(\T)$. It remains to show that $\int_{\T} \log |f|\,dm > -\infty$ whenever $f \in H^1(\T)$ is nontrivial. Suppose for the sake of obtaining a contradiction that
\[
\int_{\T} \log |f|\,dm = -\infty.
\]
Let $n \ge 0$ be the smallest integer such that $\widehat{f}(n)\neq 0$, and define
\[
f_n(\zeta) := \zeta^{-n} f(\zeta).
\]
Then $f_n \in H^1(\T)$ and $\widehat{f_n}(0)=\widehat{f}(n)\neq 0$.

By Jensen's inequality,
\[
\log |\widehat{f}(n)|
= \log \Big| \int_{\T} f_n\, dm \Big|
\le \int_{\T} \log |f_n|\, dm
= \int_{\T} \log |f|\, dm
= -\infty,
\]
a contradiction. Hence $\int_{\T} \log |f|\,dm > -\infty$.
\end{proof}

We also give an alternative proof, which involves Poisson integrals, but requires prior knowledge on boundary behavior of Hardy space functions. 

\begin{proof}[Alternative proof of the F. and M. Riesz Theorem]
Suppose $\mu$ is a complex finite Borel measure on $\T$ with $\widehat{\mu}(n)=0$ for all $n\leq 0$, and consider Poisson extension of $\mu$:
\[
P(\mu)(z) := \int_{\T} \frac{1-|z|^2}{\abs{\zeta-z}^2} d\mu(\zeta), \qquad z \in \D.
\]
Now using the identity
\[
\frac{1-|z|^2}{\abs{\zeta-z}^2} = \Re \frac{2}{1-\conj{\zeta}z} -1 = \frac{2}{1-\conj{\zeta}z} + 2\sum_{n=0}^\infty \conj{z}^n \zeta^n -1, \qquad z\in \D, \qquad \zeta \in \T,
\]
we actually see that
\[
P(\mu)(z) = 2 \int_{\T} \frac{d\mu(\zeta)}{1-\conj{\zeta}z}, \qquad z\in \D,
\]
which defines an analytic function in $\D$. Furthermore, it follows from Fubini's Theorem that
\[
\int_{\T} \abs{P(\mu)(r\xi)}dm(\xi) \leq \int_{\T} \int_{\T} \frac{(1-r^2)}{\abs{\zeta-r\xi}^2} d\abs{\mu}(\zeta) dm(\xi) \leq \norm{\mu}_{M(\T)}, \qquad 0<r<1.
\]
which implies that $P(\mu)\in H^1(\D)$. It follows from the appendix on Hardy spaces that $P(\mu)$ has radial limit $dm$-a.e on $\T$ equals to an $L^1(\T)$-function $f$ with $\widehat{f}(n)=0$ for $n<0$, i.e. $f\in H^1(\T)$, and furthermore $P(\mu)(z) = P(f)(z)$. It follows by uniqueness of Poisson integrals that $d\mu = fdm$ with $f\in H^1(\T)$.

\end{proof}

Let us conclude this subsection with a simple corollary of the F. and M. Riesz theorem.

\begin{cor}
For any compact set $E \subsetneq \T$, the analytic polynomials $\Po$ are dense in $C(E)$.
\end{cor}

\begin{proof}
Suppose that $\Po$ is not dense in $C(E)$. Then, by duality (Hahn--Banach and Riesz representation Theorem), there exists a non-trivial measure $\mu \in M(E)$ such that
\[
\int_{\T} \zeta^n \, d\mu(\zeta) = 0, \qquad n=0,1,2,\dots.
\]
By the F. and M. Riesz theorem, there exists $f\in L^1(\T)$ such that $d\mu = f\,dm$, and
\[
\int_{\T} \log |f|\,dm > -\infty.
\]
Since $\supp \mu \subseteq E \subsetneq \T$, the function $f$ vanishes on a set of positive Lebesgue measure. This contradicts the logarithmic integrability of $f$, and the result follows.
\end{proof}

\subsection{The Rudin--Carleson Theorem}
Here, we consider a result attributed to the independent efforts of W. Rudin and L. Carleson, which somewhat highlights the sharpness of the brothers Riesz Theorem. 
To this end, we shall need a classical construction which goes back to the work of Fatou, asserting that any compact set of Lebesgue measure zero in $\T$ is an admissible zero set for a function $f\in H^1(\T)$ which has continuous extension to $\T$.

\begin{thm}[Fatou]\thlabel{THM:R-C} For any compact set $E \subset \T$ of Lebesgue measure zero, there exists a non-trivial function $f$ which is continuous on $\T$ with 
\[
\widehat{f}(n)=0, \qquad n<0
\]
such that $f=0$ on $E$.
\end{thm}

The following proof is essentially a modification of an argument due to B. Korenblum, outlined in \cite{hedenmalmbergmanspaces}.  For a slightly different proof, we refer the reader to P. Koosis in \cite{koosis}.

\begin{proof}[Proof of \thref{THM:R-C}] Let $\{I_n\}_n$ denote the connected components of the open set $\T \setminus E$. To each arc $I_n$, we perform a so-called Whitney decomposition $\{I_{n,j}\}_j$, which has the following property 
\begin{equation}\label{EQ:WHITNEY}
\dist{I_{n,j}}{E} = |I_{n,j}| = \frac{1}{3 \cdot 2^{|j|} } |I_n|.
\end{equation}
In other words, their length are comparable to their distance to $E$. We now re-label the joint collection of all Whitney decompositions of $\{I_k\}_k$ by $\{J_n\}_n$. Let $\xi_n$ denote the center of each Whitney arc $J_n$, set $r_n = |J_n| +1$. Now consider the functions
\begin{equation}\label{EQ:hn}
h_n(z) = \frac{\xi_n |J_n|}{r_n\xi_n - z}, \qquad z\in \cD.
\end{equation}
We now record the following properties of \eqref{EQ:hn}, whose verification are left to the reader:
\begin{enumerate}
    \item[(i)] each $h_n$ is analytic in a neighborhood of $\cD$,
    \item[(ii)] $\Re h_n \geq 0$, 
    \item[(iii)] there exists an numerical constant $c>0$, such that $\Re h_n(\zeta) \geq c$ for $\zeta \in J_n$,
    
\end{enumerate}
Now pick a increasing sequence of positive numbers $(\lambda_n)_n$ with $\lambda_n \uparrow +\infty$, such that 
\[
\sum_n \lambda_n \lvert J_n \rvert < \infty,
\]
and consider the (outer) function
\[
f(z) = \exp \left( - \sum_n \lambda_n h_n(z) \right), \qquad z\in \D.
\]
It follows from the property $(ii)$ of $h_n$ that $f: \D \to \D$ is analytic, and since the $\{\xi_n\}$ only accumulate to the compact set $E$, we see that $f$ extends analytically across each compact subarc of $\T \subset E$. For instance, this can easily be verified using Morera's Theorem on a small disc centered at an arbitrary point in $\T \setminus E$. We shall now verify the continuity property of $f$ and its vanishing on the set $E$. To this end, fix a point $\zeta \in \T \setminus E$, and let $J_n$ be the Whitney subarc that contains $\zeta$. It follows from property $(iii)$ of $h_n$ that 
\[
\lvert f(\zeta) \rvert \leq \exp \left( - \sum_n \lambda_n \Re h_n(\zeta) \right) \leq \exp (-c\lambda_n )
\]
Now a key observation is that in order for $\dist{\zeta}{E}\to 0$, it must pass through infinitely many Whitney arcs $\{J_n\}_n$, since their lengths are always comparable to their distance to $E$. Now this observation in conjunction with the fact that $\lambda_n \uparrow \infty$, implies that $f(\zeta) \to 0$ as $\dist{\zeta}{E}\to 0$. We conclude that $f$ is continuous on $\T$, vanishes precisely on $E$, and satisfies $\widehat{f}(n)=0$ for $n<0$, since it is a bounded analytic function in $\D$.

\end{proof}

With this lemma at hand, we shall now turn to the proof of the Rudin-Carleson Theorem, phrased as a problem on boundary interpolation with analytic functions in $\D$.

\begin{thm}[Rudin-Carleson]\thlabel{COR:PEAKINT} Let $E \subset \T$ be a compact set of Lebesgue measure zero. Then for any $g\in C(\T)$ and any $\varepsilon>0$, there exists $\varphi_\varepsilon \in H^\infty(\D) \cap C(\T)$ such that $\varphi_\varepsilon=g$ on $E$, and 
\[
\sup_{\zeta \in \T} \lvert \varphi_\varepsilon (\zeta) \rvert \leq (1+\varepsilon)\sup_{\zeta \in E} \lvert g(\zeta) \rvert .
\]
\end{thm}
Here, we denote by $H^\infty(\D)$ the Banach space of bounded analytic functions in $\D$, equipped with the supremum $\sup_{|z|<1} \abs{f(z)}$. The word "peak" comes from the fact that the interpolating function $\varphi_\varepsilon$ on $E$, roughly attains its "peak value" on that set, that is, its supremum norm is essentially smaller elsewhere on $\T$. We shall also see in the proof how the notion of peak naturally occurs. The proof below is constructive, and essentially due to L. Carleson, which is outlined on P.129 in \cite{garnett}. There are simpler proofs by means of duality, which are far from constructive, see \cite{garnett} and references therein.

\begin{proof}[Proof of \thref{COR:PEAKINT}] 
\proofpart{1}{Constructing a peak function:}
Fix an arbitrary compact set $E\subset \T$ of Lebesgue measure zero and our goal is first to exhibit a so-called peak function for $E$. To this end, consider the holomorphic function $h = \log f$ on $\D$, where $f$ is as in \thref{THM:R-C}. Note that $\Re h$ is a strictly negative on $\T$, with $\Re h= -\infty$ precisely on $E$ and $h \in C^\infty(\cD \setminus E)$. Now consider the following conformal transformation of $h$;
\[
H(z) := \frac{h(z)}{h(z)-1}, \qquad z\in \D,
\]
which intrinsically gives rise to a holomorphic self-map $H: \D \to \D$, which extends continuously to $\T$, satisfies $H =1$ on $E$, while $\lvert H \rvert <1$ on $\cD \setminus E$. This gives our desired peak function on $E$.

\proofpart{2}{Approximation scheme:} Now fix an arbitrary continuous function $g$ on $\T$, and choose a large integer $n_0>0$, so that the continuous function $g_0 := H^{n_0} \cdot g$ attains its maximum on $E$:
\[
\sup_{\zeta \in \T} \abs{g_0(\zeta)}= \sup_{\zeta \in E} \abs{g(\zeta)}.
\]
Invoking \thref{LEM:APPROX1}, we can find holomorphic polynomials $(p_j)_j$, and open neighborhoods $(E_j)_j$ of $E$ such that $p_j \to g_0$ uniformly on $E$ and
\[
\sup_{E_j} \lvert p_{j+1}-p_j \rvert \leq 2^{-j}, \qquad \sup_{\T} \lvert p_j \rvert \leq \sup_{\zeta \in E} \abs{g(\zeta)}.
\]
Next, we pick positive integers $(N_j)_j$ so large that 
\[
\sup_{\zeta \in \T \setminus E_j} \lvert H^{N_j}(\zeta) \rvert \lvert p_{j+1}(\zeta)-p_j(\zeta) \rvert \leq 2^{-j}, \qquad \sup_{\zeta \in \T \setminus E_j} \lvert H^{N_j}(\zeta)p_j(\zeta) \rvert \leq 2^{-j}
\]
This can indeed be arranged since $\sup_{\T \setminus E_j}|H|<1$, for all fixed $j$. Now consider the functions 
\[
\varphi_j(z) := H^{N_j}(z)p_{j}(z) + \sum_{k\geq j} H^{N_k}(z)\left( p_{k+1}(z)-p_k(z)\right), \qquad z\in \cD, \qquad j=1,2,\dots
\]
The Weierstrass $M$-test readily implies that each $\varphi_j$ extends continuously to $\T$. Since $H=1$ on $E$, and $p_k \to g$ uniformly on $E$, we have $\varphi_j =g$ on $E$ for each $j$. Furthermore,
\[
\lvert \varphi_j \rvert \leq \lvert H^{N_j} p_j \rvert + \sum_{k\geq j} 2^{-k} \leq \sup_{\zeta \in E} \abs{g(\zeta)} + 2^{1-j}.
\]
The proof is complete.
\end{proof}
Invoking a more careful argument, one can actually show that the norm estimate also holds with $\varepsilon=0$, but this is not important for our further developments. 

We complete this section by mentioning that the initial proof of the brothers Riesz Theorem was essentially based on the above construction of peak sets, and the argument is carefully outlined in the book of P. Koosis \cite{koosis}. 

\subsection{Mean polynomial approximation on the unit-circle}

Here we consider the problem of mean polynomial approximation on $\T$. Given a positive finite Borel measure $\mu$ on $\T$, we ask when are the analytic polynomials dense in the space $L^2(\mu)$. Fortunately, this result has a very satisfactory answer which goes back to the work of G. Szeg\"o on orthogonal polynomials.

\begin{thm}[Szeg\"o] \thlabel{THM:SZEGÖ} Let $\mu$ be a positive finite Borel measure on $\T$ with Lebesgue decomposition $d\mu = w dm + d\mu_s$. Then 
\[
\inf_{q \in \mathcal{P}} \int_{\T} \lvert q(\zeta) -\conj{\zeta} \rvert^2 d\mu(\zeta) = \exp \left( \int_{\T} \log w dm \right)
\]
where $\Po$ denotes the set of analytic polynomials.
\end{thm}

As an immediate corollary of Szeg\"o's Theorem, we obtain the following.
\begin{cor} \thlabel{COR:POLL2mu} The analytic polynomials are dense in $L^2(d\mu)$ if and only if 
\[
\log w \notin L^1(\T).
\]
    
\end{cor}

Note that the singular part of $\mu$ is irrelevant. However, we shall below list a short proof of \thref{COR:POLL2mu}, under the assumption that $\mu$ is absolutely continuous wrt $dm$, which does not require the Szeg\"o's beautiful distance formula, and only relies on a duality argument and the F. and M. Riesz Theorem.

\begin{proof}[Proof of \thref{COR:POLL2mu}] If $\log w \in L^1(\T)$, then we can define a an outer function $W \in H^1(\T)$ such that $|W|=w$ a.e on $\T$, and hence $W=Uw$ for some unimodular function $U\in L^\infty(\T)$. But then $U_0(\zeta)= U(\zeta)\conj{\zeta}$ is a non-trivial element in the Hilbert space $L^2(wdm)$ with the property
\[
\int_{\T} \zeta^n \conj{U_0(\zeta)} w(\zeta) dm(\zeta)= \int_{\T} \zeta^{n+1} \conj{W(\zeta)} dm(\zeta) = \conj{\widehat{W}(-n-1)}=0, \qquad n=0,1,2,\dots
\]
In other words, $U_0$ annihilates the analytic polynomials in $L^2(wdm)$, hence they cannot be dense in $L^2(wdm)$ by the Hahn--Banach Theorem. Conversely, assume $\log w \notin L^1(\T)$ and let $g\in L^2(wdm)$ with 
\[
\int_{\T} \zeta^n \conj{g(\zeta)} w(\zeta)dm(\zeta)=0, \qquad n=0,1,2, \dots
\]
Invoking the F. and M. Riesz Theorem, we must have $gw\in H^1(\T)$ and unless $g\equiv 0$, we must also have
\[
\int_{\T} \log |gw| dm >- \infty.
\]
However, this contradicts the assumption $\log w \notin L^1(\T)$, hence $g\equiv 0$ and we conclude that the analytic polynomials are dense in $L^2(wdm)$.
\end{proof}

We now re-direct our attention to the proof of \thref{THM:SZEGÖ} below, which will be based on an Beurling's Theorem, and an approximation lemma on small sets, which allows us to dispense with the singular part of $\mu$. The approximation lemma in question is roughly a slight extension of \thref{LEM:APPROX1}, with some added real analysis flavor. 

\begin{lemma}\thlabel{LEM:APPROXE} Let $\mu$ be a positive finite Borel measure on $\T$ which is singular wrt $dm$. Then there exists analytic polynomials $(p_n)_n$ such that 
\[
(i) \, \sup_{z\in \D} \lvert p_n(z) \rvert \leq 1, \qquad (ii) \, p_n \to \conj{\zeta} \, d\mu-a.e, \qquad (iii) \, p_n \to 0, \, dm-a.e.
\]
    
\end{lemma}
\begin{proof}

Since finite Borel measures are Radon measure, we can find compacts subsets $\{E_n\}_n$ of $\T$ with $E_n \subseteq E_{n+1}$, each of Lebesgue measure zero, such that 
\[
\mu(\T) \leq  \mu  (E_n) + 1/n.
\]
Now applying \thref{COR:1Eapprox}, we can find analytic polynomials $(Q_{n,j})_j$, such that for each $n$:
\[
\sup_{z\in \D} \lvert Q_{n,j}(z) \rvert \leq 1, \qquad \sup_{\zeta\in E_n} \abs{Q_{n,j}(\zeta)-\conj{\zeta}} \to 0, \qquad Q_{n,j} \to 0, \qquad \text{dm-a.e on} \, \, \T,
\]
as $j\to \infty$. Now using the uniform convergence on $E_n$ and the estimates 
\[
\mu \left( \{ |Q_{n,j}- \conj{\zeta} | > \varepsilon \} \right) \leq \mu(\T \setminus E_n) + \mu \left( E_n \cap \{ |Q_{n,j}- \conj{\zeta} | > \varepsilon/2 \} \right),
\]
we may for each $n$ find an integer $j_n\geq 1$, such that $Q_{n,j_n} \to \conj{\zeta}$ in measure $\mu$. By means of modifying $Q_{n, j_n}$ further, if necessary, we may using a similar argument also assume that $Q_{n,j_n} \to 0$ in measure $dm$. Now letting $(p_n)_n$ be an appropriate subsequence of $(Q_{n,j_n})_n$, we make $p_n \to \conj{\zeta}$ $d\mu$-a.e and $p_n \to 0$ $dm$-a.e, which proves the claim.
    
\end{proof}

We are now ready to prove Szeg\"o's Theorem.

\begin{proof}[Proof of \thref{THM:SZEGÖ}]
\proofpart{1}{Removing the singular part:} Let $(p_n)_n$ be the analytic polynomials in \thref{LEM:APPROXE} with $p_n(0)=0$, $\sup_{\T} \lvert p_n \rvert \leq 1$ and $p_n \to 1$ $d\mu_s$-a.e $p_n \to 0$ $dm$-a.e on $\T$. Now since $q_n:= q + p_n(\conj{\zeta}-q)$ are analytic polynomials, whenever $q$ is, we have
\[
\inf_{q \in \mathcal{P}} \int_{\T} \lvert q(\zeta)- \conj{\zeta} \rvert^2 d\mu(\zeta) \leq \int_{\T} \lvert q_n(\zeta) - \conj{\zeta}  \rvert^2 d\mu(\zeta) =   \int_{\T} \lvert 1 - p_n(\zeta) \rvert^2 \lvert q(\zeta) - \conj{\zeta} \rvert^2 d\mu(\zeta).
\]
Letting $n\to \infty$ and applying the dominated convergence theorem, we conclude that 
\[
\inf_{q\in \Po} \int_{\T} \lvert q(\zeta) - \conj{\zeta}\rvert^2 d\mu(\zeta) \leq \inf_{q\in \Po} \int_{\T} \lvert q(\zeta) - \conj{\zeta}\rvert^2 w(\zeta) dm(\zeta).
\]
Since the reverse inequality trivially holds, we conclude that the singular part $\mu_s$ of $\mu$ plays no role in the minimization problem, hence we may wlog assume that $d\mu = w dm$.
\proofpart{2}{log-integrable:} Suppose that $\int_{\T} \log w dm >- \infty$, and consider the outer function $W$ on $\D$ with $\lvert W \rvert^2 = w$ $dm$-a.e on $\T$, which means that 
\[
W(0)^2 = \exp \left( \int_{\T} \log w dm \right).
\]
It follows from Cauchy-Schwartz inequality that for any analytic polynomial $q$, we have
\begin{multline*}
\int_{\T} \lvert q(\zeta) - \conj{\zeta}\rvert^2 w(\zeta) dm(\zeta) = \int_{\T} \lvert \underbrace{W(\zeta)\zeta q(\zeta) -W(\zeta)}_{:=G} \rvert^2 dm(\zeta) \\
\geq \lvert G(0) \rvert^2 = \lvert W(0) \rvert^2 = \exp \left( \int_{\T} \log w dm \right).
\end{multline*}
To prove the reverse estimate, we recall since $W$ is outer, \thref{COR:CYCOUTER} ensures that we can find analytic polynomials $(q_n)_n$ such that $Wq_n \to W(0)$ in the Hardy space $H^2$. Passing to to another subsequence if necessary, we may also assume that $q_n(0)\to 1$. Now considering the analytic polynomials $p_n(\zeta) := \conj{\zeta}(1-q_n(\zeta) /q_n(0))$, we obtain
\[
\inf_{q\in \Po} \int_{\T} \lvert p_n(\zeta) - \conj{\zeta} \rvert^2 w(\zeta) dm(\zeta) = \lvert q_n(0) \rvert^{-2} \int_{\T} \lvert q_n W \rvert^2 dm \to W(0)^2 = \exp \left( \int_{\T} \log w dm \right).
\]
This proves the claim when $\int_{\T} \log w dm >- \infty$.
\proofpart{3}{log-divergence:} Now if $\int_{\T} \log w dm = -\infty$, then we consider the perturbed family of weights $w_{\varepsilon}= w+ \varepsilon$, with $\varepsilon>0$, which are clearly log-integrable for each fixed $\varepsilon>0$. Now applying the previous step, we get 
\[
 \inf_{q\in \Po} \int_{\T} \lvert q(\zeta) - \conj{\zeta}\rvert^2 w(\zeta) dm(\zeta) \leq \inf_{q\in \Po} \int_{\T} \lvert q(\zeta) - \conj{\zeta}\rvert^2 w_{\varepsilon} (\zeta) dm(\zeta) = \exp \left( \int_{\T} \log w_\varepsilon dm \right), \qquad \forall \varepsilon>0.
\]
Sending $\varepsilon \to 0+$ and invoking the monotone convergence theorem, proves the claim in this case. The proof of Szeg\"o's Theorem is complete.
\end{proof}

We now record the following uniqueness result which arises from Szeg\"o's Theorem.

\begin{cor}\thlabel{COR:SZUP} Let $w \in L^1(\T)$ non-negative with $\int_{\T} \log w dm = - \infty$. Then whenever $f\in L^1(\T)$ with 
\[
\int_{\T} \frac{\abs{f(\zeta)}^2}{w(\zeta)} dm(\zeta) < \infty, \qquad \widehat{f}(n) = 0 \qquad n=0,1,2, \dots 
\]
then $f=0$ on $\T$.
\end{cor}
\begin{proof}
Let $f\in L^1(\T)$ be a function satisfying the above properties. Note that we can identify the dual space of $L^2(wdm)$ with the weight space $L^2(w^{-1}dm)$ considered in the densely defined classical $L^2$-pairing:
\[
\int_{\T} f(\zeta) \conj{g(\zeta)} dm(\zeta).
\]
The assumptions on $f$ imply that $f\in L^2(w^{-1}dm)$ with the properties that $f$ annihilates the analytic polynomials in $L^2(wdm)$. Now since Szeg\"o's Theorem ensures that they are dense, we conclude that $f\equiv 0$.
\end{proof}

We shall briefly mention an interesting consequence of Szeg\"o's Theorem to the theory of orthogonal polynomials. To this end, let $\mu$ be a probability measure on $\T$ and construct orthogonal (analytic) polynomials $\{\varphi_n\}_{n=0}^\infty$ in $L^2(\mu)$ using the Gram-Schmidt procedure of $\{1,\zeta, \zeta^2, \dots \}$ such that
\[
\int_{\T} \varphi_n(\zeta) \conj{\varphi_m(\zeta)} d\mu(\zeta) = \delta_{m,n}
\]
where $\delta_{m,n}$ are the usual Kronecker deltas. It is not difficult to see that the $\varphi_n$'s can be chosen monic, and more surprisingly, these orthogonal polynomials satisfy the following so-called Szeg\"o recursion:
\[
\varphi_{n+1}(z) = z\varphi_n(z) - \alpha_n z^n\conj{\varphi_n(1/\conj{z})}, \qquad  z\in \C
\]
where $\{\alpha_n\}_n$ is the so-called \emph{Geronimus-Verblunsky sequence} of $\mu$. Now a key structural result in the theory is that every probability measure $\mu$ gives rise to a unique Geronimus-Verblunsky sequence $\{\alpha_n\}_n$ with $|\alpha_n|<1$. Conversely, every sequence of complex numbers $\{\alpha_n\}_n$ with $|\alpha_n|<1$ gives rise to a unique probability measure, for which it is a Verblunsky off. Now if $d\mu = w dm + d\mu_s$, then one can show using Szeg\"o's Theorem that 
\[
\prod_{n=0}^\infty \left( 1-\abs{\alpha_n} \right) = \exp \left( \int_{\T} \log w dm \right),
\]
and the later sum is easily seen to converge if and only if $\{\alpha_n\}_n \in \ell^2$. For great brief survey on these matters, and more, we refer the reader to B. Simon in \cite{simon2005opuc}.

We complete this subsection by highlighting a classical result from 1941 due to A. Kolmogorov, which is essentially a non-analytic version of Szeg\"o's Theorem, involving approximation with real-valued trigonometric polynomials.

\begin{thm}[Kolmogorov] Let $d\mu= w dm + d\mu_s$ be a positive finite Borel measure on $\T$. Then 
\[
\inf_p \int_{\T} \lvert 1 - \Re p(\zeta) \rvert^2 d\mu(\zeta) = \left( \int_{\T} \frac{dm}{w} \right)^{-1},
\]
where the infimum is taken over all analytic polynomials $p$ with $p(0)=0$.
\end{thm}
Below, we shall give a simpler proof than the original one, which essentially follows the similar route as Szeg\"o's Theorem.
\begin{proof}[Sketch of proof]
\proofpart{1}{Removing the singular part:}
Consider real-valued polynomials $(T_n)_n$ with $\widehat{T_n}(0)=0$ and such that $\sup_{\T} \lvert T_n\rvert \leq 1$, and $T_n \to 1$ $d\mu_s$-a.e, while $T_n \to 0$ $dm$-a.e. Indeed, just take $T_n = \Re Q_n$, where $Q_n$ as in the first step of Szeg\"o's Theorem. A similar argument as before shows that the distance does not depend on $\mu_s$, hence we may without loss of generality assume that $d\mu = w dm$.

\proofpart{2}{$1/w$ integrable:}
We first assume that $\int_{\T} \frac{dm}{w} < +\infty$. Now for any analytic polynomial $p$ with $p(0)=0$, Cauchy-Schwartz gives
\[
1 = \big \lvert \int_{\T}(1-\Re p ) dm \big \rvert^2 \leq \int_{\T} \frac{dm}{w} \cdot \int_{\T}\lvert 1- \Re p \rvert^2 w dm.
\]
Hence we have obtain a lower bound for the distance. To show that this lower bound is actually also an upper bound, consider the smallest translation invariant subspace $\mathcal{V}$ of $L^2(\T)$, which contains $v:=w^{1/2} $. Now since 
\[
\int_{\T} \frac{dm}{w} < +\infty,
\]
$v$ is non-zero $dm$-a.e on $\T$. By Wiener's Theorem, we must have that $\mathcal{V}= L^2(\T)$, hence we can find real-valued trigonometric polynomials $(T_n)_n$ such that
\[
v(\zeta)T_n(\zeta)\to c\,\frac{1}{v(\zeta)}
\qquad\text{in }L^2(\T),
\]
where 
\[
c= \left(\int_{\T} \frac{dm}{w} \right)^{-1}.
\]
Now since
\[
\widehat{T_n}(0)
=\int_{\T} vT_n \cdot \frac{1}{v}\,dm
\to c\int_{\T}\frac{dm}{w}=1,
\]
we may upon normalizing assume that $\widehat{T_n}(0)=1$ for all $n$. Now consider the analytic polynomials
\[
p_n:=1-(T_n+i\widetilde{T_n}),
\]
which satisfy $p_n(0)=0$ and $1-\Re p_n=T_n$. It follows that
\[
\int_{\T}|1-\Re p_n|^2\,w\,dm
=\int_{\T}|T_n|^2 w\,dm
=\int_{\T}|vT_n|^2\,dm
\to \int_{\T}\left|c\,\frac{1}{v}\right|^2dm
=\left(\int_{\T}\frac{dm}{w}\right)^{-1}.
\]
This proves the matching upper bound.

\proofpart{3}{non-integrable $\frac{1}{w}$:} Now if $\int_{\T} \frac{dm}{w} = +\infty$, then we consider the family of perturbed weights $w_\varepsilon := w + \varepsilon$ and repeat the final previous step. The remaining details are left to the reader. 

\end{proof}

\subsection{Simultaneous polynomial approximation in the disc}

In this final part of the section, we shall delve into certain aspects of the problem on polynomial approximation the the closed disc $\cD$. A natural extension of Szeg\"o's problem in this context goes as follows: 

\emph{Let $\mu$ be a  positive finite Borel measures $\mu$ on $\cD$. When are the analytic polynomials $\Po$ dense in $L^2(\mu)$? }

It turns out that this problem is immensely rich and connects to various other interesting problems in complex analysis, harmonic analysis, and operator theory, and several known figures, such as J. Brennan, S. Khrushchev, T. Kriete, A. Volberg, and others have contributed a fair share to this subject. In broad generality, J. Thomson in \cite{thomson1991approximation} found a characterization expressed in terms of analyticity. If $\Po^2(\mu)$ denotes the closure of $\Po$ in $L^2(\mu)$ then it was proved that $\Po^2(\mu) = L^2(\mu)$ if and only if no point-evaluation defines a bounded linear functional on $\Po^2(\mu)$. Now one direction is rather simple, thus if $\lambda \in \cD$ is a bounded point-evaluation functional on $\Po^2(\mu)$, then this means that there exists a constant $C(\lambda)>0$ such that
\[
\abs{p(\lambda)}^2 \leq C(\lambda) \int_{\cD} \, \abs{p(z)}^2 d\mu(z). \qquad p\in \Po
\]
But this and a classical normal family argument implies that all functions in $\Po^2(\mu)$ extend analytically in a neighborhood of $\lambda$, hence $\Po^2(\mu) \subsetneq L^2(\mu)$.

Even in the case when $\Po^2(\mu)$ has a non-trivial bounded point-evaluation functional, it is still of interest to understand the space $\Po^2(\mu)$. A particularly interesting case is when every point in $\D$ is assume to be a bounded point-evaluation functional on $\Po^2(\mu)$. J. Thomson in \cite{thomson1991approximation} proved the following remarkable decomposition: there exists a Borel set $E\subseteq \T \cap \supp{\mu}$ of 
\[
\Po^2(\mu) \cong \Po^2(\mu \lvert_{ \supp{\mu}\setminus E}) \oplus L^2(\mu \lvert_{E}),
\]
interpreted in the sense isometric isomorphism of Hilbert spaces. Here the space 
\[
\Po^2(\mu \lvert_{ \supp{\mu}\setminus E})
\]
is \emph{irreducible}in the sense that it contains no non-trivial characteristic function, which roughly means that it is a genuine space of analytic functions on the part of $\mu$ restricted to $\supp{\mu}\setminus E$, while you can approximate any $L^2$-function on the set $\supp{\mu} \cap E$. This result highlights the occurrence of a certain \emph{simultaneous approximation} phenomenon, which we shall get more aquatinted with. To clarify this point, Thomson's result implies that there exists analytic polynomials $(p_n)_n$ with the properties
\[
\int_{\cD \setminus E} \abs{p_n}^2 d\mu \to 0, \qquad \int_E \abs{p_n-1}^2 d\mu \to 0.
\]
With this perspective in mind, we continue the story from an uncertainty principle perspective, by asking a slightly broader question: Suppose we have a Banach space $X$ of analytic functions in $\D$ which contains $\Po$ as a dense subset. Can we classify the compact subsets $E \subset \T$, for which there exists analytic polynomials $\{Q_n\}_n$ so that the following simultaneous approximation phenomenon occurs:
\[
Q_n \to 0 \, \, \text{in} \, \, X, \qquad Q_n \to 1  \, \, \text{uniformly on} \, \, E?
\]

The following classical results gives essentially puts this question into two different categories.

\begin{thm}[Khinchin-Ostrowski] \thlabel{THM:KO}Let $E \subset \T$ be a set of positive Lebesgue measure and let $(Q_n)_n$ be analytic polynomials with the property that there exists a constant $C>0$ such that
\[
\sup_{0<r<1} \int_{\T} \log^+ \abs{Q_n(r\zeta)} dm(\zeta) \leq C, \qquad Q_n \to 0 \, \, \text{a.e on} \, \, E.
\]
Then $(Q_n)_n$ converges uniformly to $0$ on compacts subsets of $\D$.
\end{thm}
The Khinchin--Ostrowksi Theorem roughly asserts that if functions in $X$ have finite Nevanlinna characteristic 
\[
\sup_{0<r<1} \int_{\T} \log^+ \abs{f(r\zeta)} dm(\zeta) < \infty,
\]
then simultaneous approximation can only occur on subsets $E \subset \T$ of Lebesgue measure zero. As we saw earlier from the peak construction of the Rudin-Carleson Theorem, for any compact set $E\subset \T$ of zero Lebesgue measure, there exists analytic functions $(f_n)_n$ which converge uniformly to $1$ on $E$ and uniformly to zero in $\D$. It is not difficult from here to pass to analytic polynomials and show that any compact set of Lebesgue measure zero gives rise to a simultaneous approximation phenomenon on spaces $X$ whose topology is weaker than uniform convergence in $\D$.

\begin{proof}[Proof of \thref{THM:KO}]
By means of sacrificing arbitrary small but positive Lebesgue measure of $E$, we may by means of invoking Egoroff's Theorem assume that $\lim_n Q_n = 0$ uniformly on $E$. Note that by the conformal Jensen lemma (see \thref{COR:logPf} in appendix), we have for any $z\in \D$:
\begin{multline*}
\log \abs{Q_n(z)} \leq \int_{\T} \log \abs{Q_n(\zeta)} P_z(\zeta) dm(\zeta) \\ 
\leq \int_{E} \log |Q_n| P_z dm + \frac{1+|z|}{1-|z|} \int_{\T} \log^+ |Q_n| dm.
\end{multline*}
Since $Q_n(r\zeta) \to Q_n(\zeta)$ as $r\to 1-$, an immediate consequence of Fatou's Lemma yields 
\[
\log \abs{Q_n(z)} \leq \int_{E} \log |Q_n| P_z dm + \frac{1+|z|}{1-|z|}C, \qquad z \in \D.
\]
Sending $n\to \infty$ and using the assumption that $Q_n \to 0$ uniformly on $E$, we conclude that $Q_n \to 0$ uniformly on compact subsets of $\D$.
\end{proof}

Note that the assumption on $E$ having positive Lebesgue measure was crucial. For a slightly stronger statement, we refer to \cite[P. 280]{havinbook}.

We shall now move our attention to class of analytic functions $X$, which do not have finite Nevanlinna characteristic, and exhibit simultaneous approximation on sets of positive Lebesgue measure. In order to maintain our initial perspective of mean polynomial approximation, a convenient setting to phase this problem within is the class of Bergman spaces $L_a^2(dA_\alpha)$ which consists of analytic functions $f$ in $\D$ satisfying
\[
\int_{\D} \abs{f(z)}^2  dA_{\alpha}(z)< \infty, 
\]
where $dA_{\alpha}(z)=(1-|z|)^{\alpha-1}dA(z)$ denotes a weighted Lebesgue area measure on $\D$, wrt to the parameter range $\alpha>0$ to ensure that it is a finite measure. 

The following remarkable result will be the principal focus of this section, which not only provides a geometric characterization of simultaneous approximation in $L^2_a(dA_\alpha)$, but in addition connects it to a unilateral Fourier uniqueness problem.



\begin{thm}[Khrushchev] \thlabel{THM:KHRUSH}
Let $K \subset \T$ be a compact set of positive Lebesgue measure. Then the following statements are equivalent:
\begin{enumerate}
    \item[(i)] \textbf{Unilateral uniqueness.}
    $K$ supports no non-trivial $\mu \in M(\T) $ with
    \[
    \sum_{n\geq 0} (1+n)^{\alpha} |\widehat{\mu}(n)|^2 < \infty,
    \]
    for some/any $\alpha>0$.

    \item[(ii)] \textbf{Simultaneous approximation.}
    For some/any $\alpha>0$, there exist analytic polynomials $(Q_n)_n$ such that 
    \[
    \int_{\D} |Q_n-1|^2 dA_\alpha \to 0, \qquad \sup_{K} |Q_n| \to 0.
    \]

    \item[(iii)] \textbf{Entropic content.}
    $K$ contains no compact subset of positive Lebesgue measure which has finite Beurling–Carleson entropy:
    \[
    \sum |I_n| \log \frac{1}{|I_n|} < \infty,
    \]
    where $\{I_n\}_n$ are the connected components of $\T \setminus K$.
\end{enumerate}
\end{thm}
Note that the characterization in $(iii)$ gives a geometric characterization of such compact sets $E$. Notably, it is independent of $\alpha>0$, hence clarifying that if the statements in $(i)$ and $(ii)$ hold for some $\alpha>0$, then also hold for any $\alpha>0$.

We note that the condition in $(i)$ is unilateral, that is, we only require Fourier decay of positive Fourier frequencies. To phrase it differently, it is intrinsically a uniqueness phenomenon for the Cauchy integral in the analytic Dirichlet-Sobolev space $\ell^{2,\alpha}_a$. To see this, note that if $\mu \in M(\T)$, then 
\[
\Ka (\mu)(z) := \int_{\T} \frac{d\mu(\zeta)}{1-\conj{\zeta}z} = \sum_{n=0}^\infty \widehat{\mu}(n) z^n, \qquad z\in \D.
\]
Therefore, the condition $(i)$ can be reformulated as $E$ supports a measure complex finite Borel measure $\mu$ for which its Cauchy integral $\Ka (\mu)$ defines a non-trivial element in the fractional Dirichlet-Sobolev space
\[
 \int_{\D} \abs{\Ka(\mu)'(z)}^2 (1-|z|)^{\alpha-1} dA(z) \asymp \sum_{n\geq 0} (1+n)^\alpha \abs{\widehat{\mu}(n)}^2 < \infty. 
\]
In contrast, the bilateral uniqueness problem on which sets $E$ supports a complex finite Borel measure with $\{\widehat{\mu}(n)\}_{n\in \mathbb{Z}} \in \ell^{2,\alpha}$, the problem in the range $\alpha>1$ has a simple solution due to the classical Sobolev embeddings (Murray's Theorem). Such compact sets must have non-empty interior, whereas Khrushchev's Theorem implies that there are totally disconnected sets $E$, which can support a measure whose Cauchy integral is $C^\alpha(\T)$, for any desirable $\alpha>0$. In the range $\alpha= 1$, the bilateral problem is related to removable sets for analytic Sobolev functions $W^{1,2}(\C)$, which initially appear in the work of  L. Ahlfors and A. Beurling in \cite{ahlfors1950conformal}. 
On the other hand, the simultaneous approximation phenomenon in $(ii)$ can formally be translated to assertion that the following Thomson decomposition holds:
\[
\Po^2(dA+dm\lvert_K) \cong \Po^2(dA) \oplus L^2(dm\lvert_K).
\]
 Below, we show how using a simple Cantor-constructions, one can exhibit sets of finite Beurling--Carleson entropy.

\begin{example}[Fractal sets which lack entropic content] \label{EX:CANTORBC} Let $(\alpha_n)_{n=0}^\infty$ be positive real numbers in the interval $(0,1)$, to be specified in the proceed. Let $E_0$ be the compact subset of $[0,1)$ formed by removing from $[0,1)$ the middle open arc of length $\alpha_0$. Next, we define $E_1$ to be the compact subset of $E_0$, by removing from the $2$ connected components of $E_0$ the middle open arcs of length $\alpha_1/2$. Continuing in this manner, we get at the $n$-stage a compact subset $E_n$ with $2^n$ connected component, and thus we define $E_{n+1}$ be removing from each connected component of $E_n$ the middle open arcs of length $\alpha_{n}/2^n$. This gives a compact subsets $(E_n)_{n\geq 0}$ of $[0,1)$ and we set 
\[
E := \bigcap_{n=0}^\infty \left\{e^{2\pi i t}: t \in E_n \right\}.
\]
Now if we impose the condition
\[
a:=\sum_{n\geq 1} \alpha_n <1,
\]
then $E$ is a compact set of Lebesgue measure $1-a$. Note that the connected components $\{J\}$ of $\T \setminus E$ are the union the removed arcs at each stage. Now if $(I_j(n))_{j=1}^{2^n}$ denotes the arcs removed from $E_n$ at stage of defining $E_{n+1}$, and keeping in mind that $|I_j(n)|=\alpha_n 2^{-n}$ for all $j$, we get 
\[
\sum_J |J| \log \frac{1}{|J|} = \sum_{n=0}^\infty \sum_{j=1}^{2^n} \abs{I_j(n)} \log \frac{1}{|I_j(n)|} \asymp \sum_{n=0}^\infty \alpha_n \left(n + \log \frac{1}{\alpha_n} \right).
\]
Therefore, choosing $\alpha_n$ appropriately, we can ensure that $E$ has finite or infinite Beurling--Carleson entropy.

\end{example}

However, in the statement $(iii)$ of Khrushchev's Theorem, the claim is that $K$ should contain no compact subset of finite Beurling--Carleson entropy, which is substantially stronger than the set $K$ having infinite Beurling--Carleson entropy. However, it turns out that the self-similarity of the above Cantor-constructions actually ensure that $E$ contains no subset of Beurling--Carleson, whenever $E$ itself has infinte have infinite Beurling--Carleson entropy. See \cite[Theorem 5.1]{khrushchev1978problem}.

The proof of Khrushchev's Theorem will be divided into several steps. First, we shall clarify that the one-sided uniqueness problem for Cauchy integral is a dual reformulation in terms of simultaneous polynomials approximation. Observe that a simple expansion into Taylor series and orthogonality yields:
\[
\int_{\D} \abs{f(z)}^2 dA_{\alpha}(z) \asymp \sum_{n\geq 0} \abs{\widehat{f}(n)}^2 (1+n)^{\alpha}.
\]
With this observation in mind, we can identify the Hilbert spaces
\[
L^2_a(dA_{\alpha}) \cong \ell^{2,-\alpha}_a,
\]
to be isometric with equivalent norms. From this point of view, it is not difficult to see that the dual space of $\ell^{2,\alpha}_a$ is $\ell^{2,-\alpha}_a$ and vice versa, considered in the \emph{Cauchy-pairing} borrowed form $\ell^2$:
\[
\sum_{n\geq 0} \widehat{f}(n) \conj{\widehat{g}(n)}, \qquad f\in \ell^{2,\alpha}_a, \qquad g\in \ell^{2,-\alpha}_a.
\]
In addition, whenever $fg \in L^1(\T)$, then we can also manifest the dual pairing as an integral:
\[
\sum_{n\geq 0} \widehat{f}(n) \conj{\widehat{g}(n)} = \int_{\T} f(\zeta) \conj{g(\zeta)} dm(\zeta).
\]

With this observations in mind, we are now ready to settle the equivalence between the $(i)$ and $(ii)$ in Khrushchev's Theorem.
\begin{lemma}[Duality]\thlabel{PROP:SA} Let $\alpha>0$ and $E \subsetneq \T$ be a compact subset of positive Lebesgue measure. Then there exists no non-trivial $\mu \in M(E)$ with
\[
\sum_{n>0} \abs{\widehat{\mu}(n)}^2(1+n)^{\alpha}< \infty,
\]
if and only if there exists polynomials $(Q_n)_n$ with the following properties:
\[
\int_{\D} \abs{Q_n-1}^2 dA_\alpha \to 0, \qquad \sup_{E} \, \abs{Q_n} \to 0.
\]
\end{lemma}
\begin{proof}
Fix an arbitrary $\mu \in M(E)$ with 
\[
\sum_{n\geq 0} \, \abs{\widehat{\mu}(n)}^2 (1+n)^{\alpha} < \infty,
\]
and suppose there exists for each integer $N\geq 0$, analytic polynomials $(Q_n)_n$ with $Q_n \to \zeta^{N}$ uniformly on $E$ and $Q_n \to 0$ in $L^2_a(dA_\alpha)$. Then by duality we have
\[
\widehat{\mu}(N) = \int_E \zeta^{-N} d\mu(\zeta) = \lim_n \int_{E} \conj{Q_n(\zeta)} \, d\mu(\zeta) = \lim_n \sum_{k=0}^\infty  \conj{\widehat{Q}_n(k)} \, \widehat{\mu}(k)=0.
\]
Hence $\widehat{\mu}(N)=0$ for all $N\geq0$, and by the F. and M. Riesz Theorem $d\mu = h dm$ for some $\conj{h} \in H^1(\T)$. But since $\supp{\mu}\subseteq E$ and $\T \setminus E$ has positive Lebesgue measure, we get that $h=0$, hence $\mu \equiv 0$. We conclude that if simultaneous approximation holds on $E$, then $E$ supports no non-trivial $\mu \in M(E)$ with 
\[
\sum_{n \geq 0} \, \abs{\widehat{\mu}(n)}^2 (1+n)^{\alpha}< \infty.
\]
Conversely, if the simultaneous approximation does not hold, then $\Po \times \Po$ cannot be dense in the Banach space $L^2(dA_{\alpha}) \times C(E)$, equipped with the inherited norm
\[
\norm{f}_{L^2(dA_\alpha)} + \sup_{\zeta \in E} \abs{h(\zeta)}.
\]
Using duality and the Hahn-Banach Theorem, we can find $(g,\mu) \in \ell^{2,\alpha} \times M(E)$ such that 
\[
\int_{E} \zeta^{-n} d\mu(\zeta) + \int_{\T} g(\zeta) \zeta^{-n} dm(\zeta) =0, \qquad n=0,1,2, \dots.
\]
But this readily implies that $\Ka (d\mu) = - g \in \ell^{2, \alpha}$ with $\supp{\mu}\subseteq E$, hence $E$ supports a non-trivial $\mu$ with 
\[
\sum_{n\geq 0} \abs{\widehat{\mu}(n)}^2 (1+n)^\alpha< \infty.
\]

\end{proof}
Having established the equivalence between $(i)$ and $(ii)$ in Khrushchev's Theorem, we shall now turn to proving the necessity and sufficiency of the entropic content in $(iii)$. Khrushchev's original proof of the necessity goes via $(i) \implies (ii)$, and involves a conformal invariant version of classical Khinchin-Ostrowski Theorem, on special Stolz-Privalov domains associated with $E$. Since this approach slightly deviates from the theme in this notes, we shall give a constructive and direct proof of $(i) \implies (iii)$. Our proof is principally based on construction, dating back to the work of L. Carleson in \cite{carlesonuniqueness}, and recently appeared in \cite{limani2021constructions}.

\begin{lemma}\thlabel{LEM:CARLOUTER} Let $E \subset \T$ be a set of finite Beurling--Carleson entropy of positive Lebesgue measure. There exists an analytic function $f: \D \to \D$ such that 
\[
F(e^{it}) := 1_{\T \setminus E}(e^{it}) f(e^{it}), \qquad e^{it} \in \T
\]
defines a $C^\infty$-function on $\T$. 
\end{lemma}
\begin{proof} 
Let $\{I(n)\}_n$ be the connected components of $\T \setminus E$. For each open arc $I(n)n$, we carry out \textit{Whitney decomposition} of $I(n)$ in order to obtain subarcs $\{I_j(n)\}_j$ with disjoint interiors satisfying the properties:
\[
|I_j(n)| \asymp \dist{I_j(n)}{E} \asymp |I_n|2^{-|j|}.
\]
Now since $E$ has finite Beurling-Carleson entropy it follows that 
\begin{multline*}
\sum_{j,n} \abs{I_j(n)} \log \frac{1}{\abs{I_j(n)}} \lesssim \sum_{n} |I(n)| \log \frac{1}{|I(n)|}  \sum_j 2^{-|j|} (1+|j|) \\ \lesssim  \sum_{n} |I(n)| \log \frac{1}{|I(n)|}< \infty.
\end{multline*}
Let $\{J_k\}_k$ be a re-labeling of the joint collection of arcs $\{I_{j}(n)\}_{n,j}$ and pick positive numbers \{$\lambda_k\}_k$ tending to infinity such that 
\[ 
\sum_k \lambda_k |J_j| \log \frac{1}{|J_k|} < \infty. 
\] 
Let $\xi_k$ denote the center of $J_k$, set $r_k =1 + |J_k|$, and consider the analytic function 

\begin{equation} \label{DEF:CARLOUTER} 
f(z) =  \exp \left( -\sum_{k} h_k(z) \right), \qquad z\in \D,
\end{equation} 
where
\[
h_k(z) := \lambda_k |J_k| \log \frac{1}{|J_k|} \frac{\xi_k}{r_k \xi_k - z}, \qquad z \in \D. 
\]
It is not hard to see that the real part of each $h_k$ is positive
\begin{gather*} \Re h_k(z) = |J_k| \log \frac{1}{|J_k|}\frac{\Re (r_k - \conj{z}\xi_k)}{|r_k \xi_k - z|^2} > 0, \qquad z\in \D
\end{gather*} 
keeping in mind that $r_k > 1$ and $|\conj{z}\xi_k| < 1$. This implies that $f: \D \to \D$ is a bounded analytic function. Since the points $\{r_k \xi_k\}_k$ only cluster at the set $E$, it is not difficult to check that $f$ extends analytically across each fixed arc $J_k$, hence $f$ extends analytically to $\T \setminus E$. Furthermore, we have for  $z \in J_k$, that the comparability of the quantities 
\[
|r_k \xi_k - z| \asymp \Re(r_k - \conj{z}\xi_k) \asymp |J_k| 
\]
and thus it follows that
\[ 
|f(z)| \leq \exp( -\Re h_k(z)) \leq \exp( - c \lambda_k \log \frac{1}{|J_k|})  = |J_k|^{c\lambda_k}, \qquad z\in J_k,
\] 
for some numerical constant $c>0$. Recalling that $(J_k)_k$ where Whitney arcs, that is, they satisfy $|J_k| \asymp \dist{J_k}{E}$, we conclude that 
\[ 
|f(z)| \leq C \dist{z}{E}^{c\lambda_k} , \qquad z \in J_k,
\] 
for some positive constant $C > 0$ independent of $k$. Note that as $z$ tends to $E$ along the complement $\T \setminus E$, it needs to pass through infinitely many intervals $J_k$. Since $\lambda_k$ tends to infinity, we obtain that 
\begin{equation} \label{gest}
\abs{F(z)} = |f(z)| = o(\dist{z}{E}^N), 
\end{equation} 
as $z \to E$ along $\T \setminus E$, for any desirable integer $N>0$. Now for $e^{it} \in \T \setminus E$, the derivatives $F^{(n)}(e^{it})$ have the form $G(e^{it})F(e^{it})$, where $G$ is a linear combination of products of derivatives of $\log f(e^{it})$ with respect to $t$. But it is clear that such products cannot blow-up faster than powers of $\dist{e^{it}}{E}^{-L}$ for $e^{it} \in \T \setminus E$, for some positive integer $L=L(n)$ depending only on the number of derivatives $n$ taken. Arguing in this manner readily shows that $F$ is $C^\infty(\T)$, we omit the details.

\end{proof}

\thref{LEM:CARLOUTER} will enough to deduce the implication $(i)\implies (iii)$, hence it remains to gather the preparations for sufficiency of the condition $(iii)$, which entails in proving $(iii) \implies (ii)$. 

To this end, we shall need rephrase the property in $(iii)$ involving a quantitative metric. Given a continuous increasing function $\kappa:[0,1]\to [0,1]$ with $\kappa(0)=0$, we define the $\kappa$-Hausdorff content of a Borel set $K\subset \T$ as 
\[
\hil^\kappa (K) := \inf \left\{ \sum_{j} \kappa(|I_j|): \, \{I_j\}_j \, \, \text{cover of open arcs of} \, \, K \, \, \right\}.
\]
It is straightforward to verify that $\kappa$-Hausdorff content satisfies the following natural properties listed below:
\begin{align*}
\textbf{Monotonicity:}\quad
& \hil^\kappa(K_1) \le \hil^\kappa(K_2)
&& \text{if } K_1 \subseteq K_2, \\
\textbf{Content of arcs:}\quad
& \hil^\kappa(I) = \kappa(|I|)
&& \text{for any arc } I \subseteq \T, \\
\textbf{Sub-additivity:}\quad
& \hil^\kappa\!\left(\bigcup_j K_j\right)
\le \sum_j \hil^\kappa(K_j).
\end{align*}

The following structural theorem goes back to the work of N. G. Makarov in \cite{makarov1991class}, and will play a crucial role the proof of sufficiency of $(iii)$.

\begin{lemma}[Makarov's lemma] Let $E \subsetneq \T$ be a compact set of positive Lebesgue measure and set
\[
\kappa(t):= t \log \frac{e}{t}, \qquad 0<t\leq 1.
\]
Then $E$ contains no subset of finite Beurling-Carleson entropy of positive Lebesgue measure if and only if
\[
\hil^{\kappa} (I \setminus E) = \hil^{\kappa}(I),
\]
for any arc $I \subset \T$. If $E$ satisfies the above condition, then there exists for any arc $I$ a positive finite Borel measure $\nu_I$ supported in $I \setminus E$, such that
\[
\norm{\nu_I} = \frac{1}{2} \kappa \left( \abs{I\cap E} \right), \qquad \nu_I(J) \leq C \kappa (\abs{J}), \qquad J \subset \T,
\]
where $C>0$ is a numerical constant independent of $I$.
\end{lemma}

Note that the second statement is essentially a localized quantitative version of Frostman's mass distribution principle.
\begin{proof}
\proofpart{1}{Necessity:}
Note that $\hil^{\kappa} (I)= \kappa(|I|)$ for any arc $I\subset \T$. Suppose that $\hil^{\kappa}(I \setminus E) < \kappa(|I|)$, for some closed arc $I$. This implies that we can pick a cover $(I_k)_k$ of open arcs of $I\setminus E$, such that 
\[
\sum_k \kappa (|I_k|) < \kappa (|I|).
\]
Now since $\kappa$ is increasing and sub-additive, we have for any positive numbers $\{\ell_k\}_k$ that
\[
\kappa (\sum_k \ell_k) \leq \sum_k \kappa(\ell_k).
\]
This implies that $K:=I \setminus \cup_k I_k \subseteq E$ is a compact set of positive Lebesgue measure, which has finite Beurling-Carleson entropy. 
\proofpart{2}{Sufficiency:}
Conversely, assume $\hil^{\kappa}(I \setminus E) = \hil^{\kappa}(I)$ for any arc $I \subset \T$. 
Fix an arbitrary compact subset $K \subseteq E$ of positive Lebesgue measure, and let $\{J_k\}_{k\geq 0}$ be the connected components of $\T \setminus K$. We need to show that $\sum_k \kappa(|J_k|) = \infty$, which would imply that no compact subset of $E$ having positive Lebesgue measure is a Beurling-Carleson set. Suppose not, then there exists a number $k_0>0$, such that 
\[
\sum_{k> k_0} \kappa(|J_k|) \leq \kappa(|K|)/2.
\]
Now the complement of the open set $V = \cup_{0\leq k\leq k_0} J_k$ is a compact set consisting of a finite union of closed arcs $I_1, I_2, \dots, I_l$. Using the assumption we get
\[
\kappa(|I_j|)= \hil^{\kappa}(I_j \setminus E) \leq \hil^{\kappa}(I_j \setminus K) \leq \sum_{\substack{k>k_0 \\
J_k \subseteq I_j }} \kappa(|J_k|).
\]
Since $K \subseteq \T \setminus V$, we may sum up in $j$ and use sub-additivity of $\kappa$ in order to obtain
\[
\kappa(|K|) \leq \sum_j \kappa(|I_j|) \leq \sum_j \sum_{\substack{k>k_0 \\
J_k \subseteq J_j }} \kappa(|J_k|) \leq \sum_{k>k_0} \kappa(|J_k|) \leq \kappa(|K|)/2,
\]
which contradicts that $|K|>0$.
\proofpart{3}{Frostman measure:} Recall that Frostman's Lemma (for instance, see Theorem ? in \cite{carleson1967selected}) asserts that if a compact subset $K\subset \T$ satisfies $\hil^{\kappa}(K)>0$, then there exists a probability measure $\mu \in M(K)$ such that $\mu(I) \leq C\kappa(|I|)$ for any arc $I \subset \T$. It now just remains to choose such a \emph{Frostman measure} $\mu_I$ supported in a compact subset of $I \subset E$ and simple re-scale it to $\norm{\nu_I} = \frac{1}{2} \kappa \left( \abs{I\cap E} \right)$, while using the boundedness of $\kappa$ to ensure that we still maintain $\mu(I) \leq C\kappa(|I|)$ for any arc $I \subset \T$, with possibly a slightly larger numerical constant.
\end{proof}

As reader may have noticed, the above argument works equally well for a general increasing gauge-function $\kappa:[0,1] \to [0,1]$, which vanishes at zero. 

We now complete the proof of the remaining part of Khrushchev's Theorem.
\begin{proof}[The remaining proof of \thref{THM:KHRUSH}]
\proofpart{1}{$(i) \implies (iii):$} Arguing by contraposition, assume that $K$ contains a Beurling--Carleson set $E$ of positive Lebesgue measure. The following trick was shown to the author by B. Malman. Appealing to \thref{LEM:CARLOUTER}, we can find a bounded analytic function $f: \D \to \D$ such that 
\[
F(\zeta) := 1_{\T \setminus E}(\zeta) f(\zeta), \qquad \zeta \in \T
\]
defines a smooth $C^\infty$-function in $\T$. With this function at our disposal, we now note that by analyticity of $f$ in $\D$, we have 
\begin{equation}\label{EQ:CAUCHYCONJ}
\Ka(\conj{f})(z) = \conj{f(0)} \qquad \iff \int_{\T} \frac{1_E(\zeta) \conj{f(\zeta)}}{1-\conj{\zeta}z} dm(\zeta) = \Ka(\conj{F})(z) + \conj{f(0)}.
\end{equation}
Now recalling that $F \in C^\infty(\T)$ and in view of the expansion
\[
\Ka(\conj{F})(z) = \sum_{n\geq 0} \widehat{\overline{F}}(n)\, z^n, \qquad z\in \D
\]
also ensures that the right hand side in \eqref{EQ:CAUCHYCONJ} is $C^\infty(\T)$, hence also the left hand side. But this readily means that the measure
\[
d\mu(\zeta) = 1_E(\zeta) \conj{F(\zeta)}dm(\zeta)
\]
satisfies the property that $\Ka(d\mu) \in C^\infty(\T)$, that is, 
\[
\sum_{n\geq 0} \abs{\widehat{\mu}(n)}^2(1+n)^{\alpha} < \infty
\]
for any $\alpha >0$. This establishes $(i) \implies (iii)$, and we now turn to the proof of $(iii) \implies (ii)$, which will be divided into two steps.


\proofpart{2}{Sufficiency-reduction to a real-analysis problem:} Suppose that $E$ contains no Beurling--Carleson subset of positive Lebesgue measure. Following the main ideas of S. Khrushchev, we shall construct finite real-valued measures $\{\mu_n\}_n$ on $\T$, satisfying the following properties:
\begin{enumerate}
    \item[(i)] $\mu_n (\T)=0$,
    \item[(ii)] $\abs{\mu_n(I)} \leq C \kappa(|I|)$, for any arc $I \subset \T$,
    \item[(iii)] $\frac{d\mu_n}{dm}(\zeta) \leq -n$  for $dm$-a.e $\zeta \in E$.
\end{enumerate}
By means of convolving with a smooth approximate of the identity, we may in addition assume that $\mu_n$'s are smooth on $\T$. We now form the sequence of outer functions $(f_n)_n$ by
\[
f_n(z) = \exp \left( \frac{1}{\sqrt{n}} \int_{\T} \frac{\zeta+z}{\zeta-z} d\mu_n(\zeta) \right), \qquad z \in \D.
\]
Fix a number $\alpha>0$ and observe that the corresponding properties $(i)-(iii)$ of $\mu_n$ translate into the following properties on $(f_n)_{n}$, phrased in that same order:
\[
(i) \, f_n(0)=1, \qquad (ii) \, \sup_{z\in \D} \, (1-|z|)^{\alpha/3} \abs{f_n(z)} \leq C(\alpha), \qquad (iii) \, \sup_E \abs{f_n} \leq e^{-\sqrt{n}},
\]
for any $\alpha>0$, where $C(\alpha)>0$ is a constant only depending on $\alpha>0$. 

It readily follows by definition that $f_n \to 1$ uniformly on compact subsets of $\D$. Note that $(ii)$ also ensures that $\{f_n\}_n$ are uniformly bounded in $L^2_a(dA_{\alpha})$, we may appeal to the Banach saks Theorem, in order to pass to an appropriate convex combination of the $f_n$'s (we shall not bother re-labeling), so that
\[
\int_{\D} \abs{f_n-1}^2 dA_{\alpha} \to 0.
\]
Note also that the uniform control in $(ii)$ still ensures that $f_n \to 0$ uniformly on $E$. Passing further to appropriate Cesar\'o means of $f_n$, it is not difficult to exhibit analytic polynomials $(p_n)_n$ such that 
\[
\int_{\D} \abs{p_n(z)-1}^2 dA_{\alpha}(z) \to 0, \qquad \sup_E \abs{p_n} \to 0.
\]
This shows that $E$ satisfies the simultaneous approximation phenomenon in $(ii)$.
\proofpart{3}{Sufficiency-construction of measures:} 
It only remains to construct real-valued measures on $\T$ satisfying the properties $(i)-(iii)$, listed in the previous step. To this end, fix $n\geq 1$, let $N=N_n\geq 1$ be an integer to be determined later, and consider the collection of dyadic arcs $\Dy_N(\T)$ of length $2^{-N}$. Applying to Makarov's Lemma, we can for any $I \in \Dy_N$ find a positive finite Borel measures $\nu_I$ supported in $I \setminus E$, such that
\[
\norm{\nu_I} = \frac{1}{2} \kappa \left( \abs{J\cap E} \right), \qquad \nu_I(J) \leq C \abs{J} \log \frac{1}{|J|}, \qquad J \subseteq \T,
\]
where $C>0$ is a numerical constant independent of $I$. Now for each $I \in \Dy_N$, we set
\[
d\mu_I = \sum_{I \in \Dy_N} \left( d\nu_I - \frac{\norm{\nu_I}}{|I\cap E|} 1_{I\cap E} dm \right).
\]
Note that $\mu_I(I)=0$, and for each arc $J\subset \T$, we have 
\[
\abs{\mu_I(J)} \leq \nu_I(J) + \abs{J \cap I \cap E} \log \frac{1}{\abs{I\cap E}} \leq (C+1)\abs{J}\log \frac{1}{|J|}.
\]
Furthermore, we also have 
\[
\frac{d\mu_n}{dm}(\zeta) \leq -\frac{1}{2} \log \frac{1}{|I\cap E|} \leq - \frac{1}{2} N_n \leq -n, \qquad \zeta \in I \cap E,
\]
provided $N_n \geq 2n$. Now take 
\[
\mu_n = \sum_{I\in \Dy_N} \mu_I
\]
and note that for arc $J\subset \T$, there exists at most two arcs $I_1, I_2 \in \Dy_N$, which contain the end-points on $J$, and such that
\[
\abs{\mu_n(J)} \leq \abs{\mu_{I_1}(J)} + \abs{\mu_{I_2} (J)} \leq 2(C+1) \abs{J} \log \frac{1}{|J|}.
\]
This shows that $(\mu_n)_n$ satisfies $(ii)$, while the properties $(i)$ and $(iii)$ readily follows from the previous observation of $\mu_I$. The proof is now complete.
    
\end{proof}

\subsection{Further results}

We saw in subsection \ref{SUBSEC:RIESZTHM} that Jensen's inequality implies that if $f\in H^1(\T)$ then 
\[
\int_{\T} \log |f|dm >-\infty.
\]
Conversely, it is also natural to ask which $f\in L^1(\T)$ satisfy have convergent logarithmic integral. This problem was solved by J. Bourgain, who proved the following beautiful Theorem.

\begin{thm}[J. Bourgain, \cite{bourgain1986problem}] A function $f \in L^1(\T)$ has finite logarithmic integral on $\T$ if and only if there exists $g,h \in H^1(\T)$, such that $f= g \conj{h}$ $dm$-a.e on $\T$.
\end{thm}
The proof is based on a deep result on the factorization of unimodular functions on $\T$, expressed in terms of quotient of Blaschke products.

We note that there are smooth version of Fatou's Theorem, which essentially dates back to the work of L. Carleson in \cite{carleson1967selected} in the setting of $C^k(\T)$ for all finite $k>0$, and which was later extended to $C^\infty(\T)$ by B. Taylor and D. Williams \cite{taylor1970ideals}.

\begin{thm}[Carleson, 1967, Taylor-Williams, 1970]\thlabel{THM:TAYLORWILL} Let $E\subset \T$ be a compact set of Lebesgue measure zero, which has finite Beurling--Carleson entropy. There exists an analytic function $f:\D \to \D$ which extends to $C^\infty(\T)$ and satisfies
\[
\bigcap_{n=0}^\infty \left\{\zeta \in \T: f^{(n)}(\zeta) =0 \right\} \subseteq E.
\]
\end{thm}

The proof of \thref{THM:TAYLORWILL} is essentially outlined in our proof \thref{LEM:CARLOUTER}, but differs from original one in \cite{taylor1970ideals}. It is worth to remark that the condition of finite Beurling--Carleson entropy is also necessary. Indeed, if $f$ is an analytic function in $\D$ which extends to function of class $C^1(\T)$ (or even of H\"older class), and vanishes on $E$, then 
\[
\abs{f(\zeta)} \leq \sup_{\T} \abs{f'} \cdot \dist{\zeta}{E}, \qquad \zeta \in \T.
\]
Combining this with the fact that $\log |f|$ is integrable on $\T$, one concludes that
\[
\int_{\T} \log \frac{1}{\dist{\zeta}{E}} dm(\zeta) \lesssim \int_{\T} \abs{\log |f|} dm +1 < \infty.
\]
Now the integral on the left hand side is easily shown to be comparable to the Beurling--Carleson entropy of $E$:
\[
\sum_{k} \abs{I_k} \log \frac{1}{\abs{I_k}} \asymp \int_{\T} \log \frac{1}{\dist{\zeta}{E}} dm(\zeta)  
\]
where $\{I_k\}_k$ are the connected components of $\T \setminus E$. With these discussion, we conclude that the boundary zero sets of analytic functions in $\D$, with smooth extensions to $\T$ are characterized as the compact sets of Lebesgue measure zero having finite Beurling--Carleson entropy. These sets were also recently shown to play crucial role in the theory of model spaces, see \cite{dyakonov2006smooth}, \cite{limani2022model}, and references therein.

Another refinement of Szeg\"o's problem on mean polynomial approximation in $\T$, is to consider the angle between the analytic polynomials $\Po_0$ of average zero, and the non-analytic polynomials $\conj{\Po}$ (analytic polynomials in the variables $\conj{\zeta}$). In the framework of stochastic process, this problem has the amusing interpretation of measuring angle between "the past and the future". More precisely, the result goes as follows.

\begin{thm}[Helson-Szeg\"o] Let $d\mu = w dm + d\mu_s$ be a positive finite Borel measure on $\T$ and consider 
\[
\cos (\theta) := \inf_{\substack{g \in \Po_0, f \in \conj{\Po} \\ \norm{g}_{L^2(\mu)} = \norm{f}_{L^2(\mu)}=1 }} \Big\lvert  \int_{\T} f \conj{g} d\mu \Big\rvert .
\]
Then there exists real-valued functions $u,v \in L^\infty(\T)$ with $\norm{v}_{L^\infty(\T)}< \pi/2$, such that 
\[
\log w = u + H(v),
\]
if and only if $\theta= \pi/2$.
\end{thm}
The above so-called Helson-Szeg\"o decomposition of the weight $\log w$ is much stronger merely requiring $\log w$ to be integrable on $\T$. In fact, it turns out that it is also equivalent to the following so-called Muckenhoupt $A_2$-condition:
\[
\sup_{I \subset \T} \left(\frac{1}{m(I)} \int_{I} w dm\right) \left( \frac{1}{m(I)} \int_{I} w^{-1} dm \right) < \infty,
\]
where the supremum is taken over all arcs $I\subseteq \T$. The Muckenhoupt $A_2$-condition plays a crucial role in harmonic analysis, as it characterizes the boundedness on  $L^2(wdm)$ of the Hilbert transform 
\[
H(f)(e^{i\theta}) := \lim_{\varepsilon \to 0+} \int_{|t-\theta|>\varepsilon } \frac{f(e^{it})}{e^{it}-e^{i\theta}} dt, \qquad e^{i\theta} \in \T
\]
and other classical singular integral operators of Calder\'on-Zygmund type. Another curious fact is that there exists no known direct proof showing that either the boundedness of $H$ on $L^2(wdm)$ or the Muckenhoupt $A_2$-condition implies the Helson-Szeg\"o decomposition of $\log w$. Indeed, all known proofs go through via the angle between subspaces. Partial progress and some neat insights to this problems were announced in the short paper  \cite{coifman1983constructive}. For further details on these matters, we refer the reader to the excellent book of J. Garnett in \cite {garnett}.
 
A natural way to generalize Khrushchev's problem is to ask for which positive integrable functions $w$ on $\T$, can we find polynomials $(Q_n)_n$ for which 
\begin{equation}\label{SA:l2w}
\int_{\D}\abs{Q_n -1}^2 dA \to 0, \qquad \int_{\T} \abs{Q_n}^2 w dm \to 0.
\end{equation}
Phrased differently, we are asking for which weights $w$ on $\T$ does the following splitting phenomenon in the Thomson decomposition occur:
\[
\Po^2(dA+ wdm) \cong \Po^2(dA) \oplus L^2(wdm).
\]
A. Volberg showed in \cite[Ch. 12, p. 82] {havin2006linear} that there exists a weight with full support, which is small enough so that the above holds. It was known among experts that if there exists a set $E\subset \T$ (of positive Lebesgue measure) of finite Beurling--Carleson entropy for which
\[
\int_E \log w dm >- \infty,
\]
then it is impossible to exhibit such polynomials. In fact, it was conjectured by T. Kriete and B. MacCluer in \cite{kriete1990mean} that simultaneous approximation phenomenon in \eqref{SA:l2w} occurs if and only if
\[
\int_E \log w dm = - \infty
\]
for all sets of positive Lebesgue measure which have finite Beurling--Carleson entropy. This conjecture was indeed recently confirmed by L. Bergqvist and B. Malman in \cite{bergqvist2024distributing}, and may essentially be regarded as a local version of Szeg\"o's Theorem.

\section{Fourier uniqueness problems in $L^2(\R)$}
Here, we shall visit several manifestations of the uncertainty principle in the most classical framework of $L^2(\R)$. We start with a brief crash-course on the Fourier transform, and immediately move on to some of its most classical results, and continue further specialized Theorems.
\subsection{The Fourier transform on the real line}
The Fourier transform of an integrable function $f \in L^1(\R)$ is defined as
\[
\widehat{f}(\xi) := \int_{\R} f(x) e^{-ix\xi} \, dx, \qquad \xi \in \R.
\]
Some authors adopt different normalizations, but this will not play any essential role for our purposes. Before turning to convergence issues, let us record some basic structural properties, which hold whenever both sides are well-defined:

\begin{itemize}
\item Linearity: $\widehat{af+bg} = a\widehat{f} + b\widehat{g}$.
\item Translation: $\widehat{f(\cdot - x_0)}(\xi) = e^{-ix_0\xi}\widehat{f}(\xi)$.
\item Modulation: $\widehat{e^{ix\xi_0}f(x)}(\xi) = \widehat{f}(\xi-\xi_0)$.
\item Convolution: $\widehat{f*g} = \widehat{f}\,\widehat{g}$.
\item Differentiation: $\widehat{f'}(\xi) = i\xi\,\widehat{f}(\xi)$.
\end{itemize}

In contrast with Fourier coefficients, the Fourier transform is defined on an unbounded set, and thus requires integrability at infinity in order to be a well-defined function.

A particularly convenient setting for studying the Fourier transform is the Schwartz class $\mathcal{S}(\R)$, consisting of all smooth functions $f$ on $\R$ such that
\[
\sup_{x\in \R} (1+|x|)^M |f^{(N)}(x)| < \infty,
\]
for all integers $M,N \geq 0$. In other words, functions in $\mathcal{S}(\R)$, together with all their derivatives, decay faster than any polynomial at infinity. 

The Schwartz class enjoys a number of remarkable stability properties: it is closed under differentiation and multiplication, and the Fourier transform preserves it. More precisely, one can show that
\[
f \in \mathcal{S}(\R) \quad \Longleftrightarrow \quad \widehat{f} \in \mathcal{S}(\R),
\]
so that the Fourier transform defines a bijective linear map on $\mathcal{S}(\R)$. In this setting, the Fourier inversion formula takes the form
\[
f(x) = \frac{1}{2\pi} \int_{\R} \widehat{f}(\xi) e^{ix\xi} \, d\xi, \qquad x \in \R.
\]

Moreover, for $f,g \in \mathcal{S}(\R)$, one readily verifies the polarized Plancherel identity
\[
\int_{\R} f(x)\,\overline{g(x)} \, dx
=
\frac{1}{2\pi} \int_{\R} \widehat{f}(\xi)\,\overline{\widehat{g}(\xi)} \, d\xi.
\]
In particular, taking $f=g$ yields the analogue of Parseval's identity:
\[
\int_{\R} |f(x)|^2 \, dx
=
\frac{1}{2\pi} \int_{\R} |\widehat{f}(\xi)|^2 \, d\xi, \qquad f \in \mathcal{S}(\R).
\]

Since $\mathcal{S}(\R)$ is dense in $L^2(\R)$, Plancherel's identity allows us to extend the Fourier transform uniquely to a continuous linear operator on $L^2(\R)$. The above identities continue to hold for all $f \in L^2(\R)$, with convergence understood in the $L^2$-sense. Details of this extension can be found in standard references such as \cite{katznelson2004introduction, KornerFourierAnalysis, RudinReal&complex}.

\medskip

We now turn to a remarkable identity that connects a function with its Fourier transform through discrete sampling.

\begin{thm}[Poisson summation formula]
For any $f \in \mathcal{S}(\R)$, we have
\[
2\pi \sum_{n\in \mathbb{Z}} f(2\pi n)
=
\sum_{n\in \mathbb{Z}} \widehat{f}(n).
\]
\end{thm}

\begin{proof}
Consider the $2\pi$-periodic function
\[
\phi(t) := 2\pi \sum_{k\in \mathbb{Z}} f(t+2\pi k), \qquad t \in \R.
\]
We may regard $\phi$ as a function on $\T$. Since $f \in \mathcal{S}(\R)$, the series defining $\phi$ converges absolutely and uniformly, and $\phi$ is continuous. Moreover,
\[
\int_{\T} |\phi(\zeta)| \, dm(\zeta)
\leq \int_{\R} |f(x)| \, dx,
\]
so $\phi \in L^1(\T)$.

We now compute its Fourier coefficients:
\[
\widehat{\phi}(n)
=
\int_0^{2\pi} \phi(t) e^{-int} \, dt
=
2\pi \sum_{k\in \mathbb{Z}} \int_0^{2\pi} f(t+2\pi k) e^{-int} \, dt.
\]
By a change of variables,
\[
\widehat{\phi}(n)
=
2\pi \int_{\R} f(t) e^{-int} \, dt
=
\widehat{f}(n).
\]
Thus the Fourier coefficients of $\phi$ coincide with the values of $\widehat{f}$ at the integers.

Since $\phi$ is smooth and periodic, its Fourier series converges uniformly, and therefore
\[
\phi(0)
=
\sum_{n\in \mathbb{Z}} \widehat{\phi}(n)
=
\sum_{n\in \mathbb{Z}} \widehat{f}(n).
\]
On the other hand,
\[
\phi(0)
=
2\pi \sum_{n\in \mathbb{Z}} f(2\pi n),
\]
which completes the proof.
\end{proof}

We note that the Poisson summation formula continues to hold under weaker assumptions, provided the expressions involved are interpreted in a suitable sense.

\subsection{The Heisenberg-Uncertainty principle}

We begin with perhaps the most classical result in the theory, attributed to the physicist Heisenberg and forming one of the cornerstones of quantum mechanics. 

Below, we state the result for functions in the Schwartz class. Although it suffices to assume $f,f' \in L^2(\R)$, we work in $\mathcal{S}(\R)$ to avoid technicalities. In this framework, the uncertainty principle takes the following form.

\begin{thm}[Heisenberg's UP] For any $x_0, \xi_0 \in \R$, we have 
\[
\norm{f}^2_{L^2} \leq 2\left( \int_{\R}(x-x_0)^2 \abs{f(x)}^2 dx \right)^{1/2} \left( \int_{\R}(\xi-\xi_0)^2 \abs{\widehat{f}(\xi)}^2 d\xi \right)^{1/2}, \qquad f\in \mathcal{S}(\R).
\]
\end{thm}

\begin{proof}
By means of translating $f$, we may assume that $x_0=0$. Likewise, we may also re-scale $f$ and assume that $\xi_0=0$. According to Plancharel's Theorem, the Heisenberg inequality can be rephrased as the following estimate:
\[
\norm{f}^2_{L^2} \leq 2 \norm{Mf}_{L^2} \norm{Df}_{L^2}
\]
where $Mf(x) = xf$ and $Df(x) = -i f'(x)$, regarded as linear operators. Consider the commutator regarded as the linear map defined by $[M,D ] = MD - DM$. Note that a straightforward calculation gives 
\[
[M,D]f(x) =-ixf'(x) +i (xf(x) )' = -ixf'(x) + if(x) +ixf'(x) = if(x).
\]
Furthermore, viewed as a densely defined linear operators on $L^2$, one can easily show that it is self-adjoint. It follows from these observations that
\begin{multline*}
\norm{f}^2_{L^2} = \frac{1}{i}\langle [M,D] f, f \rangle_{L^2} = \frac{1}{i} \left( \langle M D f, f \rangle_{L^2} - \langle D M f, f \rangle_{L^2} \right) = \\
\frac{1}{i} \left( \langle Df , Mf \rangle_{L^2} - \langle Mf , Df \rangle_{L^2}\right) = 2 \Im \langle Df, Mf \rangle_{L^2} \leq 2 \norm{Mf}_{L^2} \norm{Df}_{L^2}.
\end{multline*}
\end{proof}

More generally, the same argument proves the following claim, which typically appears in quantum physics. 

\begin{cor}[UP for commutators] Let $\hil$ be a complex Hilbert space and let $A, B$ be self-adjoint densely defined linear operators on $\hil$. If $C=[A,B]$ denotes the commutator, then 
\[
\norm{Cu}^2_{\hil} \leq 2 \norm{Au}_{\hil} \norm{Bu}_{\hil}, \qquad u\in D_A \cap D_B.
\]
Here $D_A$ and $D_B$ denotes the domains of $A$ and $B$, respectively.
\end{cor}
Of course, this estimate becomes more interesting if the commutator $C$ is related to the identity operator, in some natural way.

\subsection{Functions with compactly supported Fourier transform}
Next, we shall look a yet another classical manifestation of the uncertainty principle, which stems from the Paley--Wiener Theorem on the relation between compactly supported functions and their Fourier transform.

An entire analytic function $f$ is said to be of exponential-type $a$ with $a\geq 0$, if for any $\varepsilon>0$, there exists a constant $C(\varepsilon)>0$ such that
\[
\abs{f(z)} \leq C(\varepsilon) \exp \left( (a+\varepsilon)|z| \right), \qquad z \in \C.
\]
We denote this set of functions by $\mathcal{E}_a$, and declare that an entire analytic function is of (finite) exponential-type, if $f\in \cup_{a\geq 0} \mathcal{E}_a$. The following classical result characterizes a large class of entire functions of exponential-type in terms of their Fourier transform.

\begin{thm}[Paley-Wiener] Let $a>0$ and $f\in L^2(\R)$. Then $\supp{\widehat{f}}\subseteq [-a,a]$ if and only if there exists $F\in \mathcal{E}_a$ such that $F=f$ a.e on $\R$. In particular, $f$ extends to a function in $\mathcal{E}_a$, whose restriction to the real line belongs to $L^2(\R)$.
\end{thm}

Note that the statement is symmetric in the sense that we can swap the roles of $f, \widehat{f}$.

In order to prove the Paley-Wiener Theorem, we shall need generalization of the maximum principle to unbounded domain. Results of this type go under the name of the Phragm\'en-Lindel\"of principle, but we shall below phrase half-space version of it, suitable for our purposes. 

\begin{thm}[Phragm\'en-Lindel\"of, \, halfplane] \thlabel{THM:PL-half} Let $f$ be a bounded continuous function in $\conj{\C}_+$ which is analytic in $\C_+$. Suppose that there exists a number $0<\alpha<1$ and a constant $C>0$ such that
\[
\abs{f(z)} \leq C\exp \left( \, |z|^{\alpha} \, \right) \qquad z\in \C_+.
\]
Then $f$ belongs to $H^\infty(\C_+)$ with 
\[
\sup_{z\in \C_+} \, \abs{f(z)} = \sup_{x\in \R} \, \abs{f(x)}.
\]
\end{thm}
We note that the growth restriction on $f$ in $\C_+$ cannot be relaxed to include $\alpha=1$, as seen from the example $f(z)= e^{-iz}$. In similar way, the same conclusion also holds for the lower half-plane $\C_-$.

\begin{proof} Set $M:= \sup_{x\in \R} \abs{f(x)}$, and without loss of generality we may assume that $M=1$. Fix $0<\alpha<\beta<1$. For $\varepsilon>0$, consider the perturb analytic functions
\[
f_\varepsilon (z):= \exp \left(\, -\varepsilon (-iz)^\beta \, \right) f(z), \qquad z\in \C_+.
\]
Let $R>0$ number to be specified in a moment, and consider the semi-circle in $\C_+$ defined by
\[
C_R:=\{z= Re^{it}: \, \, 0\leq t \leq \pi \}.
\]
Note that
\[
\Re \left(\varepsilon (-iRe^{it})^\beta \right) = \varepsilon R^\beta \cos(\beta(t-\frac{\pi}{2})) \geq c(\beta) \varepsilon R^\beta, \qquad 0\leq t\leq \pi,
\]
for some constant $c(\beta)>0$. With this at hand, we obtain the estimate
\[
\abs{f_\varepsilon(Re^{it})} \leq  \abs{f(Re^{it})} \exp \left( - c(\beta)\varepsilon R^\beta \right) \leq C \exp \left( R^\alpha - c(\beta)\varepsilon R^\beta \right) \leq  1,
\]
whenever $R>0$ sufficiently large. If $D_R$ denotes the semi-disc in $\C_+$ with boundary $C_R \cup \R$, then it follows that 
\[
\sup_{z \in \C_+ \setminus D_R} \abs{f_\varepsilon(z)} \leq 1.
\]
Now since
\[
\abs{f_\varepsilon(x)}= \abs{f(x)} \leq 1, \qquad x \in \R,
\]
the maximum principle also ensures that 
\[
\sup_{z \in \conj{D}_R} \, \abs{f_\varepsilon(z)}\leq 1.
\]
Therefore, we conclude that
\[
\sup_{z \in \conj{\C}_+} \, \abs{f_\varepsilon(z)}\leq 1.
\]
Since the bound is independent on $\varepsilon>0$, we can let $\varepsilon\to 0+$ in order to conclude that 
\[
\sup_{z \in \conj{\C}_+} \abs{f(z)} \leq 1.
\]
This proves the claim.
\end{proof}

\begin{cor}[Phragm\'en-Lindel\"of, Quadrants] \thlabel{THM:PLQ} Consider a rotated quadrant
\[
\mathcal{Q}_\theta:=\{z \in \C:  |\arg(z)-\theta| < \frac{\pi}{4} \}, \qquad \theta \in \R.
\]
Let $f$ be an analytic function in $\mathcal{Q}_\theta$ which extends continuously to $\conj{ \mathcal{Q}_\theta}$, and satisfies the growth estimate:
\[
\abs{f(z)} \leq C\exp\left( \, |z|^\alpha \, \right), \qquad z\in \mathcal{Q}_\theta,
\]
for some $0<\alpha<2$, and constant $C>0$. Then 
\[
\sup_{z \in \partial \mathcal{Q}_\theta} \abs{f(z)} = \sup_{ z\in \mathcal{Q}_\theta} \abs{f(z)}.
\]

\end{cor}

\begin{proof}
    The proof is an immediate consequence of the Phragm\'en--Lindel\"of Theorem for half-spaces, by means of mapping $\C_+$ conformally onto $\mathcal{Q}_\theta$ via the map 
    \[
    \varphi(z) := e^{i\theta_0} \sqrt{z} \qquad z\in \C_+
    \]
    with $\theta_0= \theta - \frac{\pi}{4}$. Now let $f:\mathcal{Q}_\theta \to \C$ analytic with continuous extension to $\conj{\mathcal{Q}}_\theta$ and satisfies the growth condition
    \[
    \abs{f(z)} \leq C \exp \left( \, |z|^\alpha \, \right), \qquad z\in \mathcal{Q}_\theta,
    \]
    for some $0<\alpha<2$. Then the composition $g= f \circ \phi$ defined an analytic function in $\C_+$, which extends to continuously up to the boundary $\R$ and satisfies:
    \[
    \abs{g(z)} \leq C \exp \left( \,|z|^{\alpha/2} \, \right) \qquad z\in \C_+. 
    \]
    It follows from \thref{THM:PL-half} that 
    \[
    \sup_{z\in \partial \mathcal{Q}_\theta} \abs{ f(z)} = \sup_{z\in \partial \C_+} \abs{g(z)} = \sup_{z\in \C_+} \abs{g(z)} = \sup_{z\in \mathcal{Q}_\theta} \abs{ f(z)}.
    \]
\end{proof}

We now turn to the proof the Paley--Wiener Theorem, which is fortunately not too long.
\begin{proof}[Proof of the Paley-Wiener Theorem]
    \proofpart{1}{Necessity:} First, we note that if $f\in L^2(\R)$ with $\supp{\widehat{f}}\subseteq[-a,a]$ implies that $\widehat{f}\in L^1(\R)$. Using the inverse formula 
    \[
    f(x) = c\int_{-a}^{a} \widehat{f}(\xi) e^{ix\xi} d\xi
    \]
    for some irrelevant normalizing constant $c\neq 0$, we may invoke Morera's Theorem to conclude that $f$ extends to an entire analytic function. The claim on $f$ being of exponential-type $a$ follows from the Cauchy-Schwarz inequality:
    \[
    \abs{f(z)} \leq \int_{-a}^a \abs{\widehat{f}(x)} e^{x\Im(z)} dx \leq \norm{\widehat{f}}_{L^1} e^{a|\Im z|}, \qquad z\in \C.
    \]
    \proofpart{2}{Sufficiency for bounded $f$:} We shall primarily assume that $f\in L^\infty(\R)$ and continuous on $\R$ with the property that there exists $F\in \mathcal{E}_a$ with $F=f$ a.e on $\R$. For any $b>a$, consider the functions $G_b(z)=F(z)e^{ibz}$, which are bounded and continuous on $\R$. Furthermore, on the imaginary axis, we have
    \[
    \abs{G_b(iy)}= \abs{F(iy)}e^{-by}\leq C(\varepsilon) e^{-y(b-a-\varepsilon)}, \qquad y>0,
    \]
    whenever $0<\varepsilon<b-a$. Furthermore, $G_b$ also satisfies the following estimate on $\C_+$:
    \[
    \abs{G_b(z)}= \abs{F(z)}e^{-by} \leq C(\varepsilon)\exp \left( (\alpha+\varepsilon)|z| \right).
    \]
    Applying the Phragmen--Lindelöf principle for rotated quadrants in \thref{THM:PLQ} on both the first and the second quadrants, we conclude that $G_b\in H^\infty(\C_+)$. Therefore, it follows that $\supp{\widehat{G_b}}\subseteq[0,\infty)$, and hence $\supp{\widehat{f}} \subseteq [-b,\infty)$. A similar argument applied to $G_{-b}$ on the lower-half plane $\C_-$ also implies that $G_{-b} \in H^\infty(\C_-)$, hence $\supp{\widehat{f}} \subseteq (-\infty,b]$. We therefore conclude that $\supp{\widehat{f}}\subseteq[-b,b]$ for all $b>a$. Taking intersects we conclude that the Fourier spectrum of $f$ is contained in $[-a,a]$.
    \proofpart{3}{General case:}
    Let $f\in L^2(\R)$ and assume there exists a function $F \in \mathcal{E}_a$ with $F=f$ a.e on $\R$. Pick a smooth family of compactly supported approximates of identity $\varphi_\varepsilon$ on $\R$, and consider the function 
    \[
    F_\varepsilon(z) := \int_{\R} F(z-x) \varphi_\varepsilon(x) dx, \qquad z\in \C.
    \]
    If $f_\varepsilon$ denotes the restriction of $F_\varepsilon$ to the real line, then it is evident that $f_\varepsilon = f\ast \varphi_\varepsilon \in C_b(\R)$. Using a Phragmen-Lindel\"of argument similar as in step 2, we also deduce that $F_\varepsilon \in \mathcal{E}_a$. According to the second step, it therefore follows that 
    \[
    \supp{\widehat{f_\varepsilon}}\subseteq [-a,a],
    \]
    for all $\varepsilon>0$. But since $f_\varepsilon \to f \in L^2(\R)$, it easily follows from Plancherel's Theorem that $\supp{\widehat{f}}\subseteq [-a,a]$. This completes the proof.
\end{proof}

We shall now record two interesting consequences of the Paley--Wiener Theorem with a uncertainty flavor.

\begin{cor}[UP via Paley--Wiener] \thlabel{COR:UPP-W} Let $f\in L^2(\R)$ non-trivial with compact Fourier spectrum $\supp{\widehat{f}}$. Then $f$ cannot vanish on a set accumulation points in $\C$, and it satisfies the growth restriction 
\[
\log |f(z)| \leq A|z| \qquad |z|>1,
\]
for some $A>0$.
\end{cor}

In other words, elements in $L^2(\R)$ with compactly supported Fourier spectrum can neither vanish too much, nor grow too fast at infinity. 

Our next observation is a more qualitative and discrete refinement of the above observation, and is attributed to father of information theory, B. Shannon. 

\begin{thm}[Shannon sampling formula] Let $f$ be a smooth function in $\R$ with $\supp{f}\subseteq [-a,a]$ for some $a>0$. Then the following sampling formula holds
\[
\widehat{f}(\xi) = \sum_{n \in \mathbb{Z}} \widehat{f}\left(\frac{\pi n}{a} \right) \frac{\sin(ax-\pi n)}{(ax-\pi n)}, \qquad \xi \in \R.
\]
\end{thm}

\begin{proof}
By means of re-scaling and considering the function 
\[
t \mapsto \frac{a}{\pi}f\left(\frac{at}{\pi}\right)
\]
we may assume that $\supp{f}\subseteq[-\pi,\pi]$. Then we have
\[
\widehat{f}(\xi)=\int_{-\pi}^{\pi} f(x)e^{-ix\xi}\,dx, \qquad \xi \in \R.
\]
Now regarding $f$ as a smooth $2\pi$-periodic function on $\R$, we see that the Fourier coefficients satisfy the identity:
\[
c_n:=\frac{1}{2\pi}\int_{-\pi}^{\pi} f(x)e^{-inx}\,dx = \frac{1}{2\pi}\widehat{f}(n), \qquad n\in \mathbb{Z}.
\]
Since $f$ is smooth, its Fourier series converges absolutely to $f$ on $[-\pi, \pi]$:
\[
f(x)=\frac{1}{2\pi}\sum_{n\in\mathbb{Z}} \widehat{f}(n)e^{inx}, \qquad x\in[-\pi,\pi].
\]
Substituting this into the Fourier integral an changing the order of integration, we obtain
\[
\widehat{f}(\xi)
=
\int_{-\pi}^{\pi} \frac{1}{2\pi}\sum_{n\in\mathbb{Z}} \widehat{f}(n)e^{inx}e^{-ix\xi}\,dx
=
\sum_{n\in\mathbb{Z}} \widehat{f}(n)\frac{1}{2\pi}\int_{-\pi}^{\pi}e^{i(n-\xi)x}\,dx = \sum_{n\in\mathbb{Z}} \widehat{f}(n)\frac{\sin \pi(\xi-n)}{\pi(\xi-n)}.
\]
with the usual interpretation $1$ at $\xi=n$. The claim now readily follows by re-scaling back, we omit the straightforward calculations.
\end{proof}

Note that we do not actually need $f$ to be smooth, it suffices only to require that its Fourier series with period $a$ converges uniformly, in order to justify the convergence of the sampling formula at every point $\xi \in \R$. In fact, the same conclusion holds for $f \in L^2(\R)$ if we settle for the weaker notion of convergence in $L^2(\R)$.

A consequence of Shannon's sampling formula is that the Fourier transform of a compactly supported $f\in L^2(\R)$ cannot vanishes on a dense enough lattice, unless it vanishes identically.

\begin{cor} If $0<a\leq b$ and $f \in L^2(\R)$ with 
\[
\supp{f}\subseteq[-a,a], \qquad \widehat{f}\left( \frac{\pi n}{b} \right)=0, \qquad  n\in \mathbb{Z},
\]
then $f\equiv 0$. 
\end{cor}
Note that the density condition on the lattice $\pi n/b$ cannot be improved to any $b<a$. For instance, take the function
\[
g(\xi) = \frac{\sin(\varepsilon \xi)\sin(b\xi)}{\xi}, \qquad \xi \in \R
\]
for $0<\varepsilon< a-b$, which is easily seen to the Fourier transform of a bounded function $f$ supported in $[-a,a]$.

\subsection{Measures with spectral gaps} 

Here, we shall consider investigate some properties of measures with spectral gaps. Let $M(\R)$ denote the space of finite complex Borel measures in $\R$. We say that $\mu \in M(\R)$ has a spectral gap, if 
\[
\supp{\widehat{\mu}} \cap I = \emptyset,
\]
for some finite interval $I\subset \R$. The following simple observation illustrates a clear manifestation of the uncertainty principle. 

\begin{cor}[UP via spectral gaps] If $\mu$ is a compactly supported complex Borel measure in $\R$, whose Fourier $\widehat{\mu}$ has a spectral gap, that is,
\[
\supp{\widehat{\mu}} \cap I = \emptyset,
\]
for some finite interval $I\subset \R$, then $\mu \equiv 0$.
\end{cor}
\begin{proof} If $\supp{\mu}\subseteq[-a,a]$ for some $a>0$, then the Fourier transform
\[
\widehat{\mu}(\xi) := \int_{-a}^{a} e^{-i\xi x} d\mu(x) \qquad \xi\in \C,
\]
extends to an entire analytic function in $\C$. For instance, this can be proved using Morera's Theorem. It follows that $\widehat{\mu}$ cannot vanish in an arc in $\R$, unless it vanishes identically.
   
\end{proof}

Moving forward, we shall investigate further properties of measures with spectral gaps. Given a number $a>0$, we denote 
\[
M_a(\R)= \left\{\mu \in M(\R): \, \, \supp{\widehat{\mu}} \cap [-a,a] = \emptyset \right\},
\]
which is easily shown to define a closed subspace of $M(\R)$, equipped with the usual total variation norm on $\R$. Define the Bernstein space 
\[
B_a(\R) := \left\{f\in L^\infty(\R): \, \, f \in \mathcal{E}_a \right\}, \qquad a>0
\]
equipped with the supremum norm in $\R$. Note that $B_a(\R)$ consists of bounded continuous functions in $\R$ with extend to entire analytic functions of exponential type $a$.

Our main concern in this subsection is to unravel the following result, and a few of its far-reaching consequences.

\begin{thm}[Pollard] Let $a>0$ and $\mu \in M(\R)$, and consider define the associated Pollard function:
\[
P_{\mu, a}(z) := \sup \left\{\abs{f(z)}: \, \, f \in B_a(\R), \, \, \int_{\R} \abs{f} d\abs{\mu} \leq 1 \right\}, \qquad z\in \C.
\]
Then whenever $\mu \in M_a(\R)$, we have for any $\sigma\in \R$ that
\[
\int_{\R} \frac{\log^+ P_{\mu,a}(x+i\sigma)}{1+x^2}dx < \infty.
\]

\end{thm}

Roughly speaking, the Pollard function $P_{\mu,a}$ measures, in a rather indirect way, how large functions in $B_a(\R)$ can become at a given point $z\in \C$, under a normalization imposed by the measure $\mu$. There is no reason to expect that the supremum is attained by a single function $f \in B_a(\R)$, and in general no such maximizer exists. The Pollard function should therefore be viewed as a nonlinear envelope capturing the collective behavior of all admissible functions.

Pollard’s Theorem reveals a subtle phenomenon, namely that the presence of a spectral gap imposes a strong global growth restriction on functions in $B_a(\R)$ along horizontal lines. 

In order to prove Pollard's Theorem, we shall need two lemmas. The first lemma links the Bernstein spaces to measures with spectral gaps.

\begin{lemma}\thlabel{LEM:MaBaperp} For any $a>0$ we have the equality of sets
\[
M_a(\R) = B_a(\R)^\perp :=\left\{\mu \in M(\R): \int_{\R} f(x) \conj{d\mu}(x)=0, \qquad \forall f\in B_a(\R) \right\}.
\]
\end{lemma}
\begin{proof}
Using Paley--Wiener Theorem, we shall primarily establish the identification
\begin{equation}\label{EQ:BaIDENT}
B_a(\R) = \left\{f\in L^\infty(\R): \, \, \supp{\widehat{f}}\subseteq[-a,a] \right\}.
\end{equation}
Note that if $f\in L^\infty(\R)$ with $\supp{\widehat{f}}\subseteq[-a,a]$, then Plancherel's Theorem ensures that $f \in L^2(\R)$. Invoking the Paley--Wiener Theorem, we get that $f\in \mathcal{E}_a$, and hence $f\in B_a(\R)$. To prove the reverse containment, fix an arbitrary $f\in B_a(\R)$ and consider the family of perturbed functions
\[
f_\varepsilon(z) := f(z) \frac{\sin(\varepsilon z)}{\varepsilon z}, \qquad z \in \C,
\]
for $\varepsilon>0$. One readily checks that $f_\varepsilon \in \mathcal{E}_{a+\varepsilon}$ with $f_\varepsilon \in L^2(\R)$, hence invoking the Paley--Wiener Theorem again implies that $\supp{\widehat{f_\varepsilon}}\subseteq[-(a+\varepsilon),(a+\varepsilon)]$ for all $\varepsilon>0$. Now one can easily verify that 
\begin{equation}\label{EQ:TEMPDIST}
\widehat{f_\varepsilon}(\xi) = \widehat{f}\ast \frac{1}{2\varepsilon}1_{[-\varepsilon, \varepsilon]}(\xi) \to \widehat{f}(\xi), \qquad \xi \in \R
\end{equation}
in the sense tempered distributions (that is, continuous linear functions on the Schwartz space $\mathcal{S}(\R)$), and thus deduce that $\supp{\widehat{f}}\subseteq[-a,a]$.

    Having that settled, we see that the inclusion $M_a(\R) \subseteq B_a(\R)^\perp$ follows immediately from the polarized version of Plancherel's identity that
    \[
    \int_{\R} f(x) \conj{d\mu}(x) = \int_{-a}^{a} \widehat{f}(\xi) \conj{\widehat{\mu}(\xi)} d\xi =0.
    \]
    For the reverse containment, observe that the Bernstein space $B_a(\R)$ is invariant under translations along the real line: $x\mapsto f(t+x)$ for $t\in \R$. Furthermore, the function
    \[
    f_a(x) := \frac{\sin(ax)}{ax}, \qquad x\in \R,
    \]
    is a member of $B_a(\R)$. Applying the polarized Plancherel identity, we get for any $\mu \in B_a(\R)$:
    \[
    0=\int_{\R} f_a(t+x) \conj{d\mu}(x) = \int_{\R} e^{-it\xi} \widehat{f_a}(\xi)\conj{\widehat{\mu}(\xi)} d\xi= \int_{-a}^{a} e^{-it\xi} \conj{\widehat{\mu}(\xi)} d\xi, \qquad t\in \R.
    \]
    But this implies that the Fourier transform of the $L^2(\R)$-function $1_{[-a,a]}\widehat{\mu}$ is zero. Since $\widehat{\mu}$ is a continuous function, it follows that 
    \[
    \widehat{\mu}(\xi)=0, \qquad \xi\in [-a,a],
    \]
    hence $\mu \in M_a(\R)$. This establishes that reverse containment $B_a(\R)^\perp \subseteq M_a(\R)$, and thus completes the proof of the lemma.
\end{proof}
We are now ready to prove Pollard's Theorem.
\begin{proof}[Proof of Pollard's Theorem]
The proof will be divided in 3 steps. Fix $a>0$ and let $\mu \in M_a(\R)$, which we may without loss of generality assume is non-trivial. For simplicity, we only prove the claim for $\sigma=0$, as the general case is readily achieved by a simple adaption of the argument below.

\proofpart{1}{The annihilating trick:} We observe that whenever $f\in B_a(\R)$, then so is the quotient
\[
F_z(\lambda):= \begin{cases}
\frac{f(\lambda)-f(z)}{\lambda-z} \qquad \qquad \lambda\neq z \\
f'(z) \qquad \qquad \qquad  \lambda=z,
    
\end{cases}
\]
for any fixed $z \in \C$. Indeed, using the Cauchy integral formula, we see that 
\[
\abs{f'(\lambda)} \leq \int_{|\lambda-\zeta|=1} \abs{f(\zeta)} \, \frac{\abs{d\zeta}}{2\pi} \leq C(\varepsilon) \exp\left( (a+\varepsilon)|\lambda| \right), \qquad \lambda\in \C.
\]
Now if $|z-\lambda|\leq 1$, then it readily follows from this estimate that
\[
\abs{F_z(\lambda)}=\frac{\abs{f(\lambda)-f(z)}}{\abs{\lambda-z}} \leq \sup_{\zeta \in \C: |\lambda-\zeta|\leq 10} \abs{f'(\zeta)} \leq C(\varepsilon) \exp\left( (a+\varepsilon)|\lambda| \right).
\]
The case $\{\lambda \in \C: |\lambda-x|> 1\}$ is even simpler, hence $F_z(\lambda)\in \mathcal{E}_a$. Furthermore, since $f$ is entire and belongs to $f\in L^\infty(\R)$, it also follows that $F_z \in L^\infty(\R)$. 
Now using this and invoking \thref{LEM:MaBaperp}, we obtain
\[
\int_{\R} F_z(x) \conj{d\mu}(x) =0 \iff \int_{\R} \frac{f(x) \conj{d\mu}(x)}{x-z} = f(z) \int_{\R} \frac{\conj{d\mu}(x)}{x-z}, \qquad z\in \C
\]
for all $f \in B_a(\R)$, which expressed in terms of Cauchy integrals reads as:
\[
\Ka(f\conj{d\mu})(z) = f(z) \Ka(\conj{d\mu})(z), \qquad z\in \C, \qquad f\in B_a(\R).
\]
\proofpart{2}{Estimating the Pollard function with $\sigma=1$:}
Moving forward, we may normalize $\mu \in M_a(\R)$ so that its total-variation satisfies
\[
\norm{\mu}_{M(\R)}:=\int_{\R} \abs{d\mu}\leq 1.
\]
But this implies that the constant function $1$, which is an element of $B_a(\R)$ satisfies $\int_{\R} 1 \abs{d\mu}\leq 1$, and hence 
\[
P_{\mu,a}(z) \geq 1, \qquad z\in \C.
\]
 Since $\mu \neq 0$, we have that the Cauchy integral $\Ka(\conj{d\mu})$ is non-trivial analytic function in either the upper half plane $\C_+$ or in the lower half plane $\C_-$, so say $\C_+$. Now using the annihilating identity in step 1, we have for any $f\in B_a(\R)$ with $\int_\R |f| d|\mu|\leq 1$:
\[
\abs{f(z)}= \frac{\abs{\Ka(f\conj{d\mu})(z)}}{\abs{\Ka(\conj{d\mu})(z)}} \leq \frac{1}{\abs{\Ka(\conj{d\mu})(z)}} \int_{\R} \frac{\abs{f(x)}}{|x-z|} d\abs{\mu}(z) \leq \frac{1}{\abs{\Ka(\conj{d\mu})(z)}}, \qquad \Im(z)\geq 1,
\]
at the points where $\Ka(\conj{d\mu})(z)\neq 0$. Consequently, we conclude that 
\begin{equation}\label{EQ:ESTPOLL1}
\log^+ P_{\mu,a}(z) = \log P_{\mu,a}(z) \leq \log \frac{1}{\abs{\Ka(\conj{d\mu})(z)}}, \qquad \Im(z)\geq 1.
\end{equation}
Now since $z\mapsto \Ka(\conj{d\mu})(z+i)$ is a bounded analytic in $\C_+$, it follows from Jensen's Theorem in the upper-half plane that
\[
\int_{\R} \frac{\log^+ P_{\mu,a}(x+i)}{1+x^2}dx \leq\int_{\R}\log \frac{1}{\abs{\Ka(\conj{d\mu})(x+i)}} \frac{dx}{1+x^2} < \infty.
\]
This proves the claim for $\sigma =1$, but it is evident that the same argument as above also works for any $\sigma>0$.

\proofpart{3}{Estimating the Pollard function with $\sigma=0$:}
In order to prove it for $\sigma=0$, we primarily note that whenever $f \in B_a(\R)$ then $e^{iaz}f(z)$ belongs to $H^\infty(\C_+)$, and hence by Jensen's Theorem in $\C_+$ we have 
\[
\log |f(z)| \leq a \abs{y} + \int_{\R} \log |f(t)| \frac{y}{y^2+(t-x)^2} \frac{dt}{\pi}, \qquad z=x+iy\in \C_+,
\]
for all $f\in B_a(\R)$. Since the same holds true for the translation $z\mapsto f(z+i\alpha)$ with $\alpha\neq 0$, we obtain
\[
\log |f(x)| \leq a + \int_{\R} \log |f(t+i)| \frac{1}{1+(t-x)^2} \frac{dt}{\pi}, \qquad x \in \R.
\]
Taking supremum over $f\in B_a(\R)$ with $\int_{\R} |f| d|\mu|\leq 1$, we conclude 
\[
\log P_{\mu, a}(x) \leq a + \int_{\R} \log P_{\mu,a}(t+i) \frac{1}{1+(t-x)^2} \frac{dt}{\pi}, \qquad x \in \R.
\]
Integrating thus inequality wrt $(1+x^2)^{-1}dx$ and using \eqref{EQ:ESTPOLL1}, we arrive at
\begin{multline*}
\int_{\R} \frac{\log P_{\mu, a}(x)}{1+x^2} dx \leq a + \int_{\R} \log P_{\mu,a}(t+i) \int_{\R} \frac{1}{1+x^2} \frac{1}{1+(t-x)^2} \frac{dx}{\pi} \frac{dt}{\pi} \\
= a +\int_{\R} \frac{\log P_{\mu,a}(t+i) }{1+t^2}\frac{dt}{\pi} < \infty.
\end{multline*}
In the last step, we utilized that $(1+x^2)= - \Im( (i+x)^{-1} )$ and the reproducing property for the Poisson extension of bounded harmonic functions in $\C_+$. This completes the proof.
\end{proof}

With Pollard's Theorem at hand, we shall now record deep result of Beurling, on the incompatibility between measure having a spectral gap and thin support.

\begin{thm}[Beurling] \thlabel{THM:BEURLING} If $\mu \in M(\R)$ supported in a closed set $S\subset \R$ with
\begin{equation}\label{EQ:distPOISSON}
\int_{\R} \frac{\dist{x}{S}}{1+x^2} dx = + \infty.
\end{equation}
If $\mu$ has a spectral gap, then $\mu \equiv 0$.
\end{thm}

The condition \eqref{EQ:distPOISSON} roughly measures how thin the closed set $S$ is 
at neighborhoods of infinity, and the Beurling's Theorem asserts that measures supported in such thin sets $S$ must be full Fourier spectrum in $\R$. For instance, it is not difficult to show that $\{J_k\}_k$ is a disjoint union of open finite Whitney intervals in $\R$, that is, they satisfy
\[
|J_k| \asymp \dist{J_k}{S}
\]
where $S= \R \setminus \cup_k J_k$ denotes the closed complement. It easy to see that $S$ satisfies \eqref{EQ:distPOISSON} if and only if
\[
\sum_k \frac{|J_k|^2}{1+\dist{J_k}{0}^2} =+\infty.
\]
In other words, it is the lack off square summability of relative lengths 
\[
\sum_{k: \dist{J_k}{0}\geq 1} \frac{|J_k|^2}{\dist{J_k}{0}^2} = +\infty,
\]
that determines this phenomenon. 
\begin{proof}[Proof of Beurling's Theorem] By means of multiplying $\mu$ with a modulating factor $e^{iax}$, we may assume that $\mu \in M_a(\R)$ for some $a>0$. Furthermore, upon normalizing $\mu$, we may also assume that 
\[
\int_{\R} d|\mu|\leq 1.
\]
As in the proof of Pollard's Theorem, this ensures that the associated Pollard function $P_{\mu,a}\geq 1$. We shall now divide the proof into  cases, depending on whether $\R \setminus S$ has an unbounded connected component, or not. 

\proofpart{1}{An unbounded connected components:} This case follows from classical Hardy space theory. Indeed, suppose $S\cap (c,\infty)= \emptyset$ for some $c\in \R$. By means of translating $\mu$, we may assume that $\supp{\mu}\subseteq(-\infty,0]$. But then its Fourier transform 
\[
\widehat{\mu}(z) = \int_{-\infty}^0 e^{-izx} d\mu(x), \qquad z\in \C_+
\]
extends to a bounded analytic function in $\C_+$. But then $\mu$ cannot have a spectral gap, that is, $\widehat{\mu}$ cannot vanishes on arc, unless $\mu$ vanishes identically.

\proofpart{2}{Bounded connected components:} 
We now assume that all the connected components of $\R \setminus S$ have finite length. To this end, we consider a Whitney decomposition of $\R \setminus S$ which consists of arcs $\{I_k\}_k$ with 
\[
|I_k| \asymp \dist{I_k}{S}.
\]
Note that the condition in \eqref{EQ:distPOISSON} is equivalent to 
\[
\sum_k \frac{|I_k|^2}{1+ \dist{0}{I_k}^2} =+ \infty.
\]
Fix an arbitrary Whitney arc $I=I_k$, and note that it suffices to show that there exists $f_I \in B_a(\R)$ with $|f_I|\leq 1$ on $S$ such that 
\begin{equation}\label{EQ:F_JPOLL}
\log \abs{f_I(x)} \gtrsim |I| - c, \qquad x \in I,
\end{equation}
for some absolute constant $c>0$, independent of $I$. Indeed, such an estimate implies that the associated Pollard function satisfies the estimate:
\[
\int_{I} \frac{\log P_{\mu,a}(x) }{1+x^2} \, dx \gtrsim \int_I \frac{|I|-c}{1+x^2} dx \gtrsim \frac{|I|^2}{1+\dist{0}{I}^2} - |I|.
\]
Summing up over all (pairwise disjoint) Whitney arcs $(I_k)_k$, we get the condition \eqref{EQ:distPOISSON} is violated, hence Pollard's Theorem implies that $\mu \equiv 0$.

\proofpart{3}{Estimating the Pollard function:} We now construct the required functions in \eqref{EQ:F_JPOLL}. Set $I=(c-r/2,c+r/2)$ for $r>0$ and consider 
\[
f_I(z):=\cos\!\Big(a\sqrt{(z-c)^2-r^2}\Big), \qquad z\in \C.
\]
Using the Taylor expansion of cosine function, we readily see that $f_I$ is an entire analytic function. Furthermore, it has exponential type $a>0$, is uniformly bounded by $1$ in all of $\R$, and thus $f_I \in B_a(\R)$ with 
\[
\int_{\R} \abs{f_I} d|\mu| \leq 1.
\]
To prove the estimate from below, we primarily note that for $x\in I$, the argument inside the square root is negative, hence 
\[
\psi_I(x)=\cosh\!\Big(a\sqrt{r^2-(x-c)^2}\Big), \qquad x\in I.
\]
Now using the simple estimate 
\[
\sqrt{r^2-(x-c)^2} \geq r/\sqrt{2}, \qquad x \in I
\]
in conjunction with $\cosh(t)\geq e^{|t|}/2$, we conclude that 
\[
\psi_I(x) \geq\frac{1}{2} \exp\left( ar/\sqrt{2}\right) =\frac{1}{2} \exp\left( a|J|/\sqrt{2}\right), \qquad x\in I.
\]
This proves the desired claim, hence the proof is complete.

\end{proof}

The last result of this subsection illustrates that measures which are integrable wrt to regular weights that grows too rapidly at $\pm \infty$, must have full Fourier support, unless they vanish identically.

\begin{thm}\thlabel{THM:UPWEIGHEDAPPROX} Let $W$ be a continuous function on $\R$ with $\inf_{\R}W\geq 1$, and such that $\log W$ is Lipschitz continuous in $\R$. Assume $\mu \in M(\R)$ with 
\[
\int_{\R} W d|\mu| <\infty.
\] 
If 
\[
\int_{\R} \frac{\log W(x)}{1+x^2} dx = +\infty,
\]
then $\mu$ cannot have a spectral gap, unless $\mu \equiv 0$.
\end{thm}
\begin{proof}
\proofpart{1}{Reductions and strategy:}
By means of multiplying $\mu$ with a modulating factor $e^{iax}$, we may assume that $\mu \in M_a(\R)$ for some $a>0$. Furthermore, we may upon normalizing $\mu$ also assume that 
\[
\int_{\R} d|\mu|\leq 1.
\]
This ensures that the associated Pollard function satisfies $P_{\mu,a}\geq 1$. The idea is to use Pollard's Theorem. To this end, it suffices to construct, for each $x_0\in\mathbb R$, a function $\psi_{x_0}\in B_a(\mathbb R)$ such that
\[
|\psi_{x_0}(x)|\le W(x)\qquad (x\in\mathbb R),
\]
and
\[
|\psi_{x_0}(x_0)|\ge c\,W(x_0)^\rho
\]
for some constants $c,\rho>0$ independent of $x_0$. Indeed, then
\[
\int_{\mathbb R} |\psi_{x_0}|\,d|\mu|
\le
\int_{\mathbb R} W\,d|\mu|=: C < \infty.
\]
Normalizing by $C>0$ it follows that the associated Pollard function satisfies:
\[
P_{\mu,a}(x_0)\ge \frac{|\psi_{x_0}(x_0)|}{C}\ge \frac{c}{C}\,W(x_0)^\rho, \qquad x_0 \in \R.
\]
Since $x_0 \in \R$ was arbitrary, it follows that
\[
\int_{\mathbb R}\frac{ \log P_{\mu,a}(x) }{1+x^2}\,dx \geq \int_{\R} \frac{\rho\log W(x)+\log \frac{c}{C}}{1+x^2}dx =+\infty,
\]
hence invoking Pollard's theorem, we conclude that $\mu \equiv 0$.
\proofpart{2}{Estimating the Pollard function:}
Fix an arbitrary point $x_0\in\mathbb R$, set $\Omega=\log W$, which is a positive real-valued function, let $L\ge 0$ denote a Lipschitz constant for $\Omega$. Now set
\[
R:=\frac{\Omega(x_0)}{\sqrt{a^2+L^2}} >0
\]
and define the associated
\[
\psi_{x_0}(z):=\cos\!\Big(a\sqrt{(z-x_0)^2-R^2}\Big).
\]
As in the proof of Beurling's Theorem, it follows that $\psi_{x_0}\in B_a(\mathbb R)$. Furthermore, we also have 
\[
|\psi_{x_0}(x_0)|
=
\cosh(aR)
\ge \frac12 e^{aR}
=
\frac12 \exp\!\Big(\frac{a}{\sqrt{a^2+c^2}}\,\Omega(x_0)\Big),
\]
which readily gives the required lower bound with $\rho=\frac{a}{\sqrt{a^2+c^2}}$. It only remains to verify that 
\[
|\psi_{x_0}(x)|\le W(x), \qquad x\in \R.
\]
To this end, we primarily note that if $|x-x_0|\geq R$, then the argument inside the square root is positive, and hence
\[
\psi_{x_0}(x) \leq 1 \leq W(x), \qquad |x-x_0|\geq R.
\]
It therefore only remains to prove the estimate for $|x-x_0|<R$, and it is in this step that the Lipschitz assumption on $\Omega= \log W$ enters the picture. It is straightforward to see that 
\[
\abs{\psi_{x_0}(x)} = \cosh \left(a \sqrt{R^2-(x-x_0)^2} \right) \leq \exp \left( a \sqrt{R^2-(x-x_0)^2} \right), \qquad |x-x_0|<R.
\]
Now it is not difficult to see by means of expanding squares that
\[
a \sqrt{R^2-(x-x_0)^2} \leq \Omega(x_0) - L|x-x_0| \leq \Omega(x), \qquad |x-x_0|<R
\]
where the second inequality actually holds for all $x$ due to the Lipschitz assumption on $\Omega$. With these estimates at hand, we arrive at
\[
\abs{\psi_{x_0}(x)}  \leq \exp \left( a \sqrt{R^2-(x-x_0)^2} \right) \leq \exp \left( \Omega(x) \right) = W(x), \qquad |x-x_0|< R.
\]
The proof is complete.
\end{proof}

One can actually show that \thref{THM:UPWEIGHEDAPPROX} still holds under the weaker assumption that $\log W$ is merely uniformly continuous, but the required modifications of $\psi_{x_0}$ become slightly more tedious, at the costing of making the main ideas behind the proof less transparent. The curious readers may attempt to strengthen \thref{THM:UPWEIGHEDAPPROX} on their own, or consult \cite[P. 215]{havinbook}.

\subsection{Problems related to uniqueness pairs in $L^2(\R)$}
We continue the study of compatibility conditions between the support of a measure on $\mathbb R$ and its Fourier spectrum. In this subsection, we restrict our attention to functions in $L^2(\R)$ and address the problem of when one can exhibit non-trivial functions with prescribed support and Fourier spectrum. 

A function $f\in L^1_{loc}(\R)$ is said to be essentially supported in a Lebesgue measurable subset $E$ of $\R$, if $f=0$ a.e on $\R \setminus E$. We designate this property by $\text{ess supp}(f) \subseteq E$. Given two sets $E,F \subset \R$, the pair $(E,F)$ is a declared to be a \emph{uniqueness pair} if whenever $f \in L^2(\R)$ with 
\[
\text{ess supp}(f) \subseteq E, \qquad \text{ess supp}(\widehat{f}) \subseteq F \implies f\equiv 0.
\]
Note the Plancherel's Theorem ensures that the notion of uniqueness pairs is symmetric in the sense that $(E,F)$ is a uniqueness pair if and only if $(F,E)$ is. Recall that the Paley--Wiener \thref{COR:UPP-W} implies that a pair of closed sets $(E,F)$ is a uniqueness pair if whenever $E\subsetneq \R$ and $F$ is a bounded set, and vice versa. Below, we shall investigate further results of this type. The first principle result goes as follows.

\begin{thm}[Amrein--Berthier] \thlabel{THM:A-B} Let $f\in L^2(\R)$ and $E,F \subset \R$ be sets of finite Lebesgue measure. Then there exists a constant $C(E,F)>0$, such that 
\begin{equation}\label{EQ:A-B}
\norm{f}^2_{L^2} \leq C(E,F) \left( \norm{1_{E^c}f}^2_{L^2} + \norm{1_{F^c} \widehat{f}}^2_{L^2} \right).
\end{equation}
In particular, if both $f$ and $\widehat{f}$ are essentially supported on sets of finite Lebesgue measure, then $f=0$.
\end{thm}
 A pair of Lebesgue measurable subsets $(E,F)$ is declared to be a \emph{strong uniqueness pair} if \eqref{EQ:A-B} holds, hence every strong uniqueness pair is clearly a uniqueness pair. The Amrein--Berthier Theorem asserts that $(E,F)$ is a strong uniqueness pair if they both have finite Lebesgue measure. We remark that M. Benedicks proved a weaker result in \cite{benedicks1985fourier}, namely that $(E,F)$ is a uniqueness pair if $E,F$ both have finite Lebesgue measure. His proof only involves the Poisson summation formula and does require the estimate in \eqref{EQ:A-B}. 

\begin{proof}[Proof of \thref{THM:A-B}] The proof consists of a nice blend of operator theoretical arguments, conveniently divided into 3 steps.
\proofpart{1}{Norm of the localization operators:} Consider the localization operator $Tf= 1_E \mathcal{F}^{-1}(1_F \mathcal{F}f)$ which is a bounded linear operator from $L^2(\R) \to L^2(\R)$ with $\norm{T}_{L^2 \to L^2} \leq 1$. We claim that if $\alpha:=\norm{T}_{L^2\to L^2}<1$, then the proof follows. Indeed, set $P_F (f) = \mathcal{F}^{-1}(1_F \mathcal{F}f)$ and note that $f= P_F(f) + P_{F^c}(f)$, and $T(f)= 1_E P_F(f)$, which implies that
\begin{multline*}
\norm{f}_{L^2} = \norm{1_E f}_{L^2} + \norm{1_{E^c}f}_{L^2} \leq \norm{T(f)}_{L^2(E)} + \norm{1_E P_{F^c}(f)}_{L^2} + \norm{1_{E^c}f}_{L^2} \\ \leq \alpha \norm{f}_{L^2} + \norm{\widehat{f}}_{L^2(F^c)} + \norm{f}_{L^2(E^c)}.
\end{multline*}
Since $\alpha\in (0,1)$, the result follows with $C= (1-\alpha)^{-1}$. 
\proofpart{2}{Compactness:} We primarily note that $L=T \circ \mathcal{F}$ is actually Hilbert-Schmidt operator,
\[
L(f)(x) = 1_{E}(x)\int_{F} f(y) e^{ixy} dy, \qquad x\in \R,
\]
with Hilbert-Schmidt kernel $K(x,y) := 1_E(x) 1_F(y) e^{ixy}$ for $x,y\in \R$, hence $L$ is compact, and thus $T$ is compact on $\T$. Now if $\norm{T}_{L^2 \to L^2}=1$, then by weak compactness of the unit-ball of $L^2$ and compactness of $T$, there exists a non-trivial element $f\in L^2(\R)$ with $\norm{f}_{L^2\to L^2}\leq 1$, such that 
\[
\norm{T(f)}_{L^2}=1.
\]
In fact, since $\norm{T}_{L^2\to L^2} \leq 1$, we actually also assume that $\norm{f}_{L^2}=1$, hence by Plancherel's Theorem
\[
\int_{E} \abs{\mathcal{F}^{-1}(1_F \widehat{f})(x) }^2 dx = \int_{\R} \abs{\widehat{f}(\xi)}^2 d\xi.
\]
It follows that $\text{ess supp}(\widehat{f}) \subseteq F$ and $\text{ess supp}(f) \subseteq E$, which is equivalent to the equation $T(f)=f$, that is, $1$ is an eigenvalue of $T$ with eigenvector $f$. It is well-known that each eigenspace of $T$ is finite dimensional, hence we will reach the desired contradiction once we can produce an infinitude of linearly independent elements in the eigenspace $\mathcal{E}$ corresponding to the eigenvalue $1$.
\proofpart{3}{Shifting measurable sets:}
Here we shall prove that for any $E_0 \subset E$ with $m(E_0)>0$ and any $0<\varepsilon < m(E_0)$, there exists a number $a\in \R$ such that $m(E)+\varepsilon/2 \leq m(E \cup (E_0+a) ) \leq m(E)+ \varepsilon$. To this end, consider the positive function
\[
T(a):= m(E \cup (E_0+a) )= m(E) + m(E_0) - \int_{E_0} 1_{E}(t+a) dt, \qquad a\in \R,
\]
which is easily seen to be continuous on $\R$ since the translations of $L^1$-functions are continuous in the $L^1$-norm. Furthermore, we have $T(0)=m(E)$ while $T(a) \to m(E) + m(E_0)$ as $|a|\to \infty$, hence the claim readily follows from the intermediate value theorem. With this lemma at hand, we shall now produce an eigenspace of a localization operator with infinitely linearly independent vectors. Set $E_1 = \text{ess supp}(f) \subset E$, let $k_0\geq 0$ with $2^{-k_0}< m(E_1)$, and pick $a_1 \in \R$ such that $m(E) +2^{-k_0-1} \leq m(E \cup (E_1 + a_1)) \leq m(E) + 2^{-k_0}$. Recursively and by means of induction, we can find real numbers $(a_n)_{n=0}^\infty$ such that for all $n=1,2,\dots$:
\[
S_n := E_0 + \sum_{j=1}^n a_j, \qquad E_{n+1}:= E_n \cup S_n, \qquad m(E_n) + 2^{-n-k_0-1} \leq m(E_{n+1}) \leq m(E_n) + 2^{-n-k_0}. 
\]
Now set $E^* = \cup_n E_n$ which has finite Lebesgue measure since $E$ has. Now set $f_n(x) = f(z-A_n)$ where $A_n= \sum_{j=1}^n a_j$ and observe that 
\[
\text{ess supp}(\widehat{f_n}) = \text{ess supp}(\widehat{f})\subseteq F, \qquad \text{ess supp}(f_n) = S_n \subseteq E_n  \subseteq E^*, \qquad n=1,2,\dots
\]
It therefore follows that $(f_n)_n$ all satisfy $T'(f_n)=f_n$, where $T'(f)(x) = 1_{E^*} \mathcal{F}^{-1}(1_F \widehat{f})$ again is a compact operator on $L^2(\R)$. Furthermore, the $(f_n)_n$ are linearly independent since $S_{n+1}\setminus S_n$ have positive Lebesgue measure by construction. But this contradicts the fact that eigenspaces of compact operators are finite dimensional. We therefore conclude that $\norm{T}_{L^2 \to L^2}<1$, and proof of the Amrein-Berthier Theorem follows.

\end{proof}

The above proof of the Amrein--Berthier Theorem was almost purely operator theoretical, and this is no coincidence. In fact, it is equivalent to the following statements on projections. 

\begin{prop} \thlabel{PROP:AB} Let $(E,F)$ be a pair of Lebesgue measurable subsets of $\R$ and consider the associated orthogonal projections on $L^2(\R)$ defined by 
\[
P_E (f) = 1_E f, \qquad  \widehat{P_F}(f) = \mathcal{F}^{-1}(1_F \mathcal{F}(f)), \qquad f\in L^2(\R).
\]
Then the following statements are all equivalent:
\begin{enumerate}
    \item[(i)] $(E,F)$ is a strong uniqueness pair.
    \item[(ii)] $\norm{P_E \widehat{P_F}}_{L^2 \to L^2} <1$.
    \item[(iii)] $I-P_E \widehat{P_F}$ is invertible on $L^2(\R)$.
    \item[(iv)] There exists a constant $0<c(E,F)< 1$, such that 
    \[
    c(E,F) \int_{\R} \abs{f(x)}^2 dx \leq \int_{E^c} \abs{f(x)}^2 dx,
    \]
    for all $f \in L^2(\R)$ with $\text{ess supp}(\widehat{f}) \subseteq F$.
\end{enumerate}
    
\end{prop}
\begin{proof} Recall that the localization operator $T$ on $L^2(\R)$ appearing in the proof of the Amrein--Berthier Theorem is simply the composition $T= P_E \widehat{P_F}$. 
\proofpart{1}{$(i)\implies(ii)$:}
Now if the pair $(E,F)$ is a strong uniqueness pair, then the argument in step 2 of \thref{THM:A-B} implies that $\norm{T}_{L^2 \to L^2}<1$, otherwise there exists a non-trivial function $f\in L^2(\R)$ with $\text{ess supp}(f) \subseteq E$ and $\text{ess supp}(\widehat{f}) \subseteq F$, which is in contradiction with \eqref{EQ:A-B}. 
\proofpart{2}{$(ii) \implies (iii)$:} This is simple, if $\norm{T}_{L^2 \to L^2}<1$, then $I-T$ is invertible on $L^2(\R)$ with inverse given by the Von Neumann series $\sum_{k\geq 0} T^k$. 

\proofpart{3}{$(iii) \implies (iv)$:} If $I-T$ is invertible on $L^2(\R)$ then the argument in step 1 above shows that $\norm{T}_{L^2 \to L^2}<1$. Now recall the following simple identity
\begin{equation}\label{EQ:PID}
\norm{\widehat{P_F}f}^2_{L^2} = \norm{P_{E^c}\widehat{P_F}f}^2_{L^2} + \norm{P_{E}\widehat{P_F}f}^2_{L^2}, \qquad f\in L^2(\R).
\end{equation}
Now using the fact that $\widehat{P_F}$ is a projection, we get 
\[
\norm{P_{E}\widehat{P_F}f}^2_{L^2} = \norm{P_{E}\widehat{P_F} \widehat{P_F}f}^2_{L^2} \leq \norm{T}_{L^2\to L^2} \norm{\widehat{P_F}f}^2_{L^2}, \qquad f\in L^2(\R)
\]
hence returning back to \eqref{EQ:PID} we obtain 
\[
(1-\norm{T}_{L^2 \to L^2} ) \norm{g}^2_{L^2} \leq \norm{P_{E^c}g}^2_{L^2}, \qquad g\in \Im(\widehat{P_F}),
\]
which gives the desired claim.
\proofpart{4}{$(iv) \implies (i)$:}
At last, we basically need to reverse the argument in the previous step, with the initial observation that the statement in $(iv)$ can be rephrased as 
\[
c(E,F) \norm{\widehat{P_F}f}^2_{L^2} \leq \norm{P_{E^c}\widehat{P_F}f}^2_{L^2}, \qquad f\in L^2(\R).
\]
Utilizing this estimate in conjunction with the identity \eqref{EQ:PID}, we obtain
\[
\norm{P_{E}\widehat{P_F}f}^2_{L^2} \leq \left( 1- c(E,F) \right) \norm{\widehat{P_F}f}^2_{L^2}\leq \left( 1- c(E,F) \right) \norm{f}^2_{L^2}, \qquad f\in L^2(\R).
\]
This proves that $\norm{T}_{L^2\to L^2} \leq 1-c(E,F)<1$, and it follows from step 1 of the proof of the Amrein--Berthier Theorem that $(E,F)$ is a strong uniqueness pair. 

\end{proof}

The following quantitative manifestation of the uncertainty principle is worthy of recording.

\begin{cor}[UP-Uniqueness Pairs] If there exists a non-trivial function $f\in L^2(\R)$ with $\text{ess supp}(f) \subseteq E$ and $\text{ess supp}(\widehat{f}) \subseteq F$, then $\abs{E}\abs{F}\geq 1$.
\end{cor}
\begin{proof}
    Let $T= P_E \widehat{P_F}$ be the localization operator and note that a straightforward calculation gives the norm estimate
    \[
    \norm{T}_{L^2 \to L^2} \leq \abs{E}\abs{F}.
    \]
    Now if $\abs{E}\abs{F}<1$, then $I-T$ is invertible with the inverse given by a Von Neumann series as demonstrated in \thref{PROP:AB}, and thus we conclude from the same Proposition that $(E,F)$ is a strong uniqueness pair. This implies that there cannot exists a non-trivial function $f\in L^2(\R)$ with $\text{ess supp}(f) \subseteq E$ and $\text{ess supp}(\widehat{f}) \subseteq F$. 
\end{proof}

It turns out that strong uniqueness pairs $(E,F)$ give rise to the following independence phenomenon: the behavior of $f$ on $E$ has no interference with the behavior of its Fourier transform $\widehat{f}$ on $F$, and vice versa. To make this formal, we record the following corollary of the Amrein--Berthier Theorem, which allows one to solve the following system of equations.

\begin{cor} Let $(E,F)$ be a strong uniqueness pair. Then for any pair of $g,h \in L^2(\R)$, there exists $f\in L^2(\R)$ such that 
\[
f= g \, \, \text{a.e on} \, \, E, \qquad \widehat{f}= h \, \, \text{a.e on} \, \, F.
\]
\end{cor}

\begin{proof} Set $T= P_E \widehat{P_F}$, and note that the adjoint operator is $T^* = \widehat{P_F}P_E$, and we have the simple identities $I-P_E = P_{E^c}$ and $I-\widehat{P_F}= \widehat{P_{F^c}}$. If $(E,F)$ is a strong uniqueness pair the \thref{PROP:AB} ensures that $I-T$ is invertible on $L^2(\R)$, hence it remains only to verify that the $L^2$-function
\[
f:= (I-\widehat{P_{F}})(I-T)^{-1}g + (I-P_{E})(I-T^*)^{-1} h
\]
satisfies the properties $P_E f =g$ and $\widehat{P_F}f= h$. However, this is essentially pure algebra, as demonstrated below:
\[
P_E (f) = P_E (I-\widehat{P_{F}}) (I-T)^{-1}g = (P_E-T)(I-T)^{-1}g = -P_{E^c}(1-T)^{-1} g + g = g.
\]
In the last step, we use that $P_{E^c} (1-T)^{-1}=0$, which readily follows from the Von Neumann series expansion of $T$. Similarly $\widehat{P_F}f= h$, we omit the details.
    
\end{proof}

One can also establish periodic versions of the Amrein--Berthier Theorem, which involving Fourier coefficients. Theorems of these type are attributed to Mikheyev, and turn our to have some interesting connections to lacunary Fourier series. However, this is beyond the scope of these notes, and we refer the reader to \cite[P. 102]{havinbook} for further details on these matters.

As a consequence of the Paley-Wiener Theorem, and as noted in the beginning of the subsection, any compact set $E \subset \R$ and closed proper set $F\subsetneq \R$ gives rise to a uniqueness pair $(E,F)$. The following example shows that  strong uniqueness pairs with large Fourier spectrum can also arise.

\begin{prop} There exists a strong uniqueness pair $(E,F)$, where $E$ has finite Lebesgue measure, while $F$ has infinite Lebesgue measure.

\end{prop}
\begin{proof}

Let $S \subset \R$ of finite Lebesgue measure and let $\Lambda \subset \mathbb{Z}$ be a lacunary series with $\kappa(\Lambda)\geq 3$ (see Chapter 2.3), say, and consider $F:= \cup_{\lambda \in \Lambda} [\lambda-a, \lambda+a]$ for a fixed $a>0$. Clearly, $F$ as infinite Lebesgue measure $\abs{F}= \infty$. Fix a large number $N>0$ to be determined later, and observe that the lacunary assumption on $\lambda$ ensures that the intervals $I_\lambda :=[\lambda-a,\lambda+a]$ eventually become disjoint, for each fixed $a>0$. Therefore, choosing $N>0$ large enough, we can ensure that the intervals $I_\lambda$ with $\lambda \in \Lambda_N := \Lambda \cap \{|\lambda|>N\}$ are all disjoint. Now, set $F_N := \cup_{\lambda \in \Lambda_N} I_\lambda$ and pick $f\in L^2(\R)$ with $\text{ess supp}(\widehat{f})\subseteq F_N$ and write $f= \sum_{\lambda \in \Lambda_N} f_\lambda$ with $\text{ess supp}(\widehat{f_\lambda})\subseteq I_\lambda$. According to part $(iv)$ of \thref{PROP:AB}, it suffices to verify that for sufficiently large $N>1$, there exists a constant $C=C(S,F_N)>0$
\[
\int_{\R} \abs{f(x)}^2 dx \leq C\int_{S^c} \abs{f(x)}^2 dx.
\]
A straightforward expansion and using Plancherel's Theorem implies
\begin{multline*}
\int_S \abs{f}^2 dx = \sum_{\lambda \in \Lambda_N} \int_{S} \abs{f_\lambda}^2 dx + \sum_{\substack{\lambda \neq \nu \\ \lambda, \nu \in \Lambda_N}} \int_{S} \int_{S} f_\lambda(x) \conj{f_{\nu}(y)} dx dy  \\ = \sum_{\lambda \in \Lambda_N} \norm{P_S \widehat{P_{I_\lambda}}f_\lambda}^2_{L^2} +  \sum_{\substack{\lambda \neq \nu \\ \lambda, \nu \in \Lambda_N}} \int_{I_\lambda} \int_{I_\nu} \widehat{f_\lambda}(\xi) \conj{\widehat{f_\nu}(\eta)} \conj{\widehat{1_{S}}(\xi-\eta)} d\xi d\eta =: I + II.
\end{multline*}
By the Amrein--Berthier Theorem, we have that $\rho:=\norm{P_S \widehat{P_{I_\lambda}}}_{L^2 \to L^2}<1$, hence 
\[
I= \sum_{\lambda\in \Lambda_N} \norm{P_S \widehat{P_{I_\lambda}}f_\lambda}^2_{L^2} \leq \rho^2 \sum_{\lambda\in \Lambda_N} \norm{f_\lambda}^2_{L^2} = \rho^2 \norm{f}^2_{L^2}.
\]
To estimate the quantity $II$, we apply triangle-inequality and the Cauchy-Schwarz on the terms:
\[
\abs{II} \leq \sum_{\substack{\lambda \neq \nu \\ \lambda, \nu \in \Lambda_N}} \abs{\int_{I_\lambda} \int_{I_\nu} \widehat{f_\lambda}(\xi) \conj{\widehat{f_\nu}(\eta)} \conj{\widehat{1_{S}}(\xi-\eta)} d\xi d\eta} \leq \norm{f}^2_{L^2} \left( \sum_{\substack{\lambda \neq \nu \\ \lambda, \nu \in \Lambda_N}} \int_{I_\lambda} \int_{I_\nu} \abs{\widehat{1_S}(\xi-\eta)}^2 d\xi d\eta \right)^{1/2}.
\]
We shall now make use of the lacunary assumption of $\Lambda$, in a similar manner as in the proof of the second statement of Zygmund's \thref{THM:ZYGL2}. To this end, we may note that a change of variable gives
\[
\sum_{\substack{\lambda \neq \nu \\ \lambda, \nu \in \Lambda_N}} \int_{I_\lambda} \int_{I_\nu} \abs{\widehat{1_S}(\xi-\eta)}^2 d\xi d\eta = \int_{\R} \abs{\widehat{1_S}(\xi)}^2 \sum_{\substack{\lambda \neq \nu \\ \lambda, \nu \in \Lambda_N}} \abs{I_{\lambda} \cap \left( I_\nu + \xi \right)} d\xi = \int_{\R} \abs{\widehat{1_S}(\xi)}^2 R(\xi) d\xi.
\]
Now since the intervals $(I_\lambda)_{\lambda \in \Lambda_N}$ are disjoint and of length $2a>0$, we note that for each fixed $\xi \in \R$, the function $R(\xi)$ counts the numbers of non-diagonal pairs of $(\lambda, \nu)\in \Lambda_N \times \Lambda_N$ such that $\abs{\lambda - \nu + \xi}\leq 2a$. For any desirable number $M>0$, we can by means of choosing $N=N(M)>1$ large enough, ensure that the quantity $|\lambda-\nu|$ is so large that no $\abs{\xi}\leq M$ gives any contribution to $R(\xi)$. Furthermore, the lacunary assumption also implies that there is a constant $C=C(\kappa(\Lambda))>0$, for which the uniform upper bound $R(\xi)\leq C$, holds for all $\xi\in \R$. Combining these observations, we get 
\[
\sum_{\substack{\lambda \neq \nu \\ \lambda, \nu \in \Lambda_N}} \int_{I_\lambda} \int_{I_\nu} \abs{\widehat{1_S}(\xi-\eta)}^2 d\xi d\eta = \int_{|\xi|>M} \abs{\widehat{1_S}(\xi)}^2 R(\xi) d\xi \leq C\int_{|\xi|>M} \abs{\widehat{1_S}(\xi)}^2 d\xi \leq 1-\rho^2 -\varepsilon,
\]
provided that $0<\varepsilon<1-\rho^2$ and $M>0$ chosen large enough. Combining these estimates of $(I)$ and $(II)$, we obtain
\[
\int_{S} \abs{f}^2 dm \leq \left( \rho^2 + C\int_{|\xi|>M} \abs{\widehat{1_S}(\xi)}^2 d\xi \right)\norm{f}^2_{L^2} \leq (1-\varepsilon) \int_{\R} \abs{f}^2 dx.
\]
Adding $\int_{S^c} |f|^2 dm$ to both sides and subtracting terms, we arrive at the desired estimate
\[
\int_{\R} \abs{f}^2 dx \leq \frac{1}{\varepsilon} \int_{S^c} \abs{f}^2 dx,
\]
which proves that $(S,F_N)$ is a strong uniqueness pair for sufficiently large $N$.
\end{proof}

Restricting our attention to case of bounded spectrum, our last principal result in this section gives a metric characterization of measurable sets $E \subset \R$, which give rise to strong uniqueness pairs $(E,F)$, for any bounded subset $F$. 

\begin{thm}[Logvinenko-Sereda] Let $E$ be a measurable subset of $\R$. Then $(E,F)$ is a strong uniqueness pair for any bounded measurable set $F \subset \R$ if and only if, there exists a number $r>0$ such that 
\[
\delta_r(E^c) := \inf_{|I|= r} \frac{|I\setminus E|}{|I|}>0,
\]
where the infimum is taken of over all intervals $I \subset \R$ of length $r$.

\end{thm}
In the literature, the above condition on $E^c$ is commonly referred to as $E^c$ being relatively $r$-dense in $\R$. The proof presented here will principally rely on reformulating the notion of $r$-density in terms of harmonic measure, in order to deduce a "two-constant" Theorem.

\begin{lemma}[Two-constant Lemma] \thlabel{LEM:2const} A measurable set $S \subset \R$ is relatively $r$-dense for some $r>0$ if and only if the harmonic measure on $\C_+ := \{\Im z>0 \}$ restricted to the horizontal line $y=1$:
\[
\omega_{x+i}(S) := \int_{S} \frac{1}{1+(t-x)^2}\frac{dt}{\pi},\qquad x\in \R
\]
is bounded below: $\gamma := \inf_{x\in \R} \omega_{x+i}(S) >0$. In that case, the following two-constant type estimate holds:
\[
\int_{\R} \abs{P(f)(x+i)}^2 dx \leq 2 \left( \int_{E^c} \abs{f}^2 dx\right)^{\gamma} \norm{f}^{2(1-\gamma)}_{L^2}, \qquad f \in H^2(\R).
\]

\end{lemma} 

\begin{proof}
The first statement is just a straightforward estimate involving the Poisson kernel, whose verification is left to the reader, and we therefore only verify the second claim. Suppose $\gamma >0$, and fix $f \in H^2(\R)$, and set $\alpha = \omega_{x+i}(S)\geq \gamma$, which leads to $1-\alpha = \omega_{x+i}(S^c)$. Using subharmonicity and Jensen's inequality, we get
\begin{multline*}
2\log |P f(x+i)| \leq \int_{S} \log |f|^2 d\omega_{x+i} + \int_{S^c} \log |f|^2 d\omega_{x+i} \leq \\ \alpha \log \int_S |f|^2 \frac{d\omega_{x+i}}{\alpha} + (1-\alpha)\log \int_{S^c} |f|^2 \frac{d\omega_{x+i}}{1-\alpha} \\ 
= \alpha \log \frac{1}{\alpha} + (1-\alpha) \log \frac{1}{1-\alpha} + \alpha \log \int_S \abs{f}^2 d\omega_{x+i} + (1-\alpha) \log \int_{S^c} \abs{f}^2 d\omega_{x+i} \\
\leq \log 2 + \gamma \log \int_S \abs{f}^2 d\omega_{x+i} + (1- \gamma) \log \int_{\R} |f|^2 d\omega_{x+i}.
\end{multline*}
Cleaning up these expressions, we arrive at the pointwise estimate 
\[
\abs{P(f)(x+i)}^2 \leq 2 \left( \int_S \abs{f}^2 d\omega_{x+i} \right)^{\gamma} \cdot \left( \int_{\R} \abs{f}^2 d\omega_{x+i} \right)^{1-\gamma}, \qquad x\in \R.
\]
Integrating this inequality, applying H\"older with exponents $1/\gamma$ and $1/(1-\gamma)$, and using the fact that $d\omega_{x+i}$ is a probability measure for each $x\in \R$, we obtain the desired estimate. 
\end{proof}



\begin{proof}[Proof of the Logvinenko-Sereda Theorem] 
\proofpart{1}{Necessity of the density condition:} We shall first establish the easier part, which is the necessity of $E$ being relatively $r$-dense. Fix an arbitrary bounded set $F \subset \R$, and assume that $(E,F)$ is a strong uniqueness pair, which according to part $(iv)$ of \thref{PROP:AB} implies that there exists a constant $c=c(E,F)\in (0,1)$ such that 
\[
c \int_{\R} \abs{f(x)}^2dx \leq \int_{E^c} \abs{f(x)}^2 dx,
\]
for all $f \in \Im \widehat{P_F} := \{f\in L^2(\R): \text{ess supp}(\widehat{f}) \subseteq F \}$. Now fix an arbitrary $f\in \Im \widehat{P_F}$ with $\norm{f}_{L^2}=1$, and choose an interval $I\subset \R$ large enough, so that 
\[
\int_{\R \setminus I} \abs{f(x)}^2 dx \leq c/2.
\]
Recall that $\Im \widehat{P_F}$ is a closed subspace of $L^2(\R)$, which is invariant under translation $f(x) \mapsto f_h(x)=f(x-h)$ for $h\in \R$. Set $I_h := \{x+h: x\in I\}$ denote the translations of $I$ by $h$ and observe that
\begin{multline*}
c = c\norm{f_h}^2_{L^2} \leq \int_{E^c} \abs{f_h(x)}^2 dx \leq \int_{E^c \cap I_h} \abs{f_h(x)}^2 dx + \int_{\R \setminus I} \abs{f(x)}^2 dx \\ \leq c/2 + \int_{E^c \cap I_h} \abs{f_h(x)}^2 dx \leq c/2 + \omega_f (|E^c \cap I_h|),
\end{multline*}
where $\omega_f(t) := \sup_{|S|\leq t} \int_{S} \abs{f}^2 dx$. Now since $\omega_f(t) \downarrow 0$ as $t\downarrow 0$, and the above estimate is independent of $h\in \R$, we conclude that 
\[
\inf_{h\in \R} \frac{\abs{E^c \cap I_h}}{\abs{I_h}} = \delta_{|I|}(E^c) >0.
\]
\proofpart{2}{Sufficiency:} 

Now suppose $E^c \subset \R$ is relatively $r$-dense for some $r>0$ and observe that according to the "two-constant" Theorem, we have that 
\[
\int_{\R} \abs{P(f)(x+i)}^2 dx \leq 2 \left( \int_{E^c} \abs{f}^2 dx\right)^{\gamma} \norm{f}^{2(1-\gamma)}_{L^2}, \qquad f \in H^2(\R),
\]
where $\gamma := \inf_{x\in \R} \omega_{x+i}(E^c) \in (0,1)$. For the sake of abbreviation, we simply set $P(f)(x):=P(f)(x+i)$ and note that a standard computation involving the Fourier transform shows that 
\[
\widehat{P(f)}(\xi) = \frac{1}{2\pi}e^{-|\xi|}\widehat{f}(\xi), \qquad \xi \in \R,
\]
Fix $a>0$ and let $f\in L^2(\R)$ with $\supp{\widehat{f}}\subseteq [-a,a]$, and note that $f_a(x) := f(x) e^{-iax}$ belongs to $H^2(\R)$. With this at hand, we now argue as follows:
\begin{multline*}
\norm{f}^2_{L^2} = \norm{f_a}^2_{L^2} = \int_{0}^{2a} \abs{\widehat{f_a}(\xi)}^2 d\xi \leq 2\pi e^{2a} \int_{0}^{2a} e^{-2|\xi|} \abs{\widehat{f_a}(\xi)}^2 \frac{d \xi}{2\pi} \\ 
\leq 2\pi e^{2a} \int_{\R} \abs{P(f_a)(x)}^2 dx \leq 4\pi e^{2a} \left( \int_{E^c} \abs{f}^2 dx\right)^{\gamma} \norm{f}^{2(1-\gamma)}_{L^2}.
\end{multline*}
Canceling the exponents we arrive at 
\[
\int_{\R} \abs{f}^2 dx \leq \left(4e^{2a}\right)^{1/\gamma} \int_{E^c} \abs{f}^2 dx,
\]
for all $f\in L^2(\R)$ with $\supp{\widehat{f}}\subseteq[-a,a]$. The proof is now complete.
\end{proof}
There are also different proofs of the Logvinenko-Sereda Theorem which essentially only involves real-analysis, based on Bernstein-type estimates, and where essentially sharp constants are obtain. For instance, see the recent work of O. Kovrijkine in \cite{kovrijkine2001some}. Furthermore, there is also an $L^p$-version of the "two-constant" Theorem, which naturally allows one to generalize the Logvinenko-Sereda to this setting, as well. See P. 115 in \cite{havinbook}.

\subsection{Further results}
Let $f$ be a measurable function on $\R$ with the properties that
\[
f(x) = \mathcal{O}( \exp (-ax^2) ), \qquad \widehat{f}(\xi) = \mathcal{O}( \exp (-b\xi^2 ) )
\]
for some positive numbers $a,b>0$. Now if one considers the dilatation $f_{\delta}(x)= f(\delta x)$ with $\delta>0$, then we have 
\[
f_\delta(x) = \mathcal{O}( \exp (-a\delta^2 x^2) ), \qquad \widehat{f}(\xi) = \mathcal{O}( \exp (-b \xi^2 /\delta^2 ), 
\]
Hence such an attempt of improving the decay of $f$ at infinity, worsens the decay of $\widehat{f}$, and vice versa. This suggests the manifestation of an underlying uncertainty principle, which can formally be stated as follows.
 
\begin{thm}[Hardy's uncertainty principle] Let $f$ measurable function on $\R$ with the property that there exists positive numbers $\alpha, \beta>0$, and an integer $N\geq 0$, such that 
\[
\abs{f(x)} \leq C (1+|x|)^N e^{- \alpha \pi x^2}, \qquad \abs{\widehat{f}(\xi)} \leq C (1+|\xi|)^N e^{-\beta \pi \xi^2}
\]
for some constant $C>0$. Then $\alpha \beta =1$, unless $f\equiv 0$. Furthermore, there exists a polynomials $P$ of degree at most $N$, such that $f(x)=P(x)e^{-\alpha\pi x^2}$.
\end{thm}
The proof of Hardy's uncertainty principal mainly relies on clever applications of the celebrated Phragm\'en-Lindelöf Theorem in complex analysis. For instance, a sketch of proof is outlined by T. Tao in his blog, see \cite{}. In a similar fashion, we also mention the following result of A. Beurling, which is less precise, but whose simplicity makes it rather attractive.

\begin{thm}[Beurling] If $f\in L^1(\R)$ with 
\[
\int_{\R} \int_{\R} \abs{f(x)} \abs{\widehat{f}(\xi)} e^{-|x\cdot \xi|} d\xi dx < \infty,
\]
then $f\equiv 0$.
\end{thm}
On a historical note, Beurling's result appeared in his collected works in \cite{beurling1989collected}, but the proof was unfortunately not preserved.  In 1990, L. H\"ormander successfully recovered the proof from old notes taken during his conversations with A. Beurling between 1964-1968, where they spent time together at the Institute for Advanced studies. See \cite{hormander1991uniqueness}. We also refer to the more recent work of H. Hedenmalm in \cite{hedenmalm2012heisenberg}, which gives a different take on the matter.

We shall now visit two extensions of the classical Heisenberg uncertainty principle. The first result is due to M. Cowling and J. F. Price, and is essentially an $L^p$-version.

\begin{thm}[Cowling-Price] Let $1\leq p,q\leq \infty$, $0<\alpha<1$, $\theta, \phi >0$ which satisfy the relationship 
\[
\alpha \left(\theta+ \frac{1}{p}- \frac{1}{2}\right) = (1-\alpha) \left( \phi + \frac{1}{q}- \frac{1}{2} \right).
\]
Then there exists a constant $C>0$, possibly depending on all these parameters, such that for any Schwartz function $f$ on $\R$:
\[
\norm{f}_{L^2}\leq C \left( \int_{\R} |x|^{\theta p} \abs{f(x)}^p dx \right)^{\alpha/p}\left( \int_{\R} |\xi|^{\phi q} \abs{\widehat{f}(\xi)}^q d\xi \right)^{(1-\alpha)/q}.
\]
\end{thm}
For further details and other results of similar flavor, we refer the reader to \cite{cowling1984bandwidth}. Next, we mention an entropic version of the Heisenberg uncertainty principle. The entropic uncertainty principle was initially conjectured in 1957 by I. Hirschman in \cite{hirschman1957note} and later proved by W. Beckner in 1975, see \cite{beckner1975inequalities}. The result goes as follows.

\begin{thm}[Entropic UP] \thlabel{THM:ENTUP} For any $f\in L^2(\R)$ the following inequality holds:
\[
\int_{\R} \abs{f(x)}^2 \log \frac{1}{|f(x)|^2} dx+ \int_{\R} \abs{\widehat{f}(\xi)}^2 \log \frac{1}{|\widehat{f}(\xi)|^2} d\xi \geq \log \frac{e}{2}.
\]
\end{thm}
This theorem is rather deep, and hinges on the sharp Hausdorff-Young inequality, regarding the operatornorm of the Fourier transform $\mathcal{F}: L^p(\R) \to L^q(\R)$, where $1<p<2$ and $q=p/(p-1)$. See \cite{beckner1975inequalities}.

Theorems of two-constant type are classical in analytic function theory, and may somewhat be regarded as integrable version of Hadamards three circle/line Theorem, quantitative versions of the Phragm\'en--Lindel\"of principle, exploiting convexity. We highlight the following classical two-constant Theorem, whose proof can readily be extracted from the argument displayed the two-constant lemma.

\begin{thm}[Two-constant Theorem] \thlabel{THM:TWOCONST} Let $\Omega \subset \C$ be Jordan domain. Suppose there exists a closed set $E\subset \partial \Omega$ and closed curve $\Gamma \subset \Omega$ the property that the harmonic measure on $\Omega$, evaluated at $E$ wrt reference points in $\Gamma$ satisfies:
\[
\alpha:= \inf_{z\in \Gamma} \omega(z,E,\Omega) \in (0,1).
\]
Then for any $f \in C(\partial \Omega)$, the following uniform estimate holds for the harmonic extension $u_f$ to $\Omega$:
\[
\sup_{\lambda \in \Gamma} \abs{u_f(\lambda)} \leq C \sup_E \abs{f}^\alpha \sup_{\partial \Omega \setminus E} \abs{f}^{1-\alpha}
\]
where $C>0$ is a numerical constant.
\end{thm}

More refined and technical versions of the Logvinenko-Sereda Theorem were established by C. Shubin, R. Vakilian and T. Wolff in \cite[Theorem 2.1]{shubin1998some}. In the related field of time-frequency analysis, there is also a version of the Amrein--Berthier Theorem for the generator of a Gabor system. For instance, see \cite{nitzan2013quantitative} and references therein.

We complete this chapter by including some remarks in relation recent works on annihilating pairs. In \cite{nazarov1993local}, F. Nazarov obtained a quantitative estimate of the Amrein--Berthier constant 
\[
C(E,F) \leq C e^{C|E| \cdot |F|}
\]
for any sets $E,F$ of finite Lebesgue measure and where $C>0$ is an absolute constant. A neat proof is contained in \cite[Ch. 4, § 11]{havinbook} More recently, A. Kulikov, F. Nazarov, M. Sodin in \cite{kulikov2025fourier} considered uniqueness pairs where both sets $(\Lambda, \Sigma)$ where discrete subsets of $\R$. To give a flavor of their result, we only state a mild version of it.

Let $\Lambda=\{\lambda_j\}_j$ and $\Sigma= \{\sigma_j\}_j$ be real numbers arranged in increasing order and with $\lim_{j\to \pm \infty} \lambda_j =\lim_{j\to \pm \infty} \sigma_j= \pm \infty$. We say that a pair $(\Lambda, \Sigma)$ is \emph{super critical} if 
\[
\limsup_{|j|\to \infty} |\lambda_j|(\lambda_{j+1}-\lambda_j)< \frac{1}{2}, \qquad \limsup_{|j|\to \infty} |\sigma_j|(\sigma_{j+1}-\sigma_j)< \frac{1}{2},
\]
an we call the pair $(\Lambda,\Sigma)$ \emph{subcritical} if 
\[
\limsup_{|j|\to \infty} |\lambda_j|(\lambda_{j+1}-\lambda_j) > \frac{1}{2}, \qquad \limsup_{|j|\to \infty} |\sigma_j|(\sigma_{j+1}-\sigma_j)> \frac{1}{2},
\]
A version of their result goes as follows.
\begin{thm}[Kulikov--Nazarov--Sodin] 
Any supercritical pair $(\Lambda, \Sigma)$ is a uniqueness pair for the Schwartz class $\mathcal{S}(\R)$:
\[
f\in \mathcal{S}(\R) \qquad f\lvert_{\Lambda}=\widehat{f}\lvert_{\Sigma}=0 \qquad \implies f\equiv 0.
\]
Any subcritical pair $(\Lambda, \Sigma)$ is a non-uniqueness pair for $\mathcal{S}(\R)$.
\end{thm}

The interesting point regarding their work is that their condition of supercritical pairs naturally arises from the a sharp version of the classical Wirtinger's inequality, that is, the one-dimensional analogue of the Poincar\'e-inequality. Prior, Fourier uniqueness pairs where consider by M. Viazovska and D. Radchenko, and special applications to the sphere-packing problem, see \cite{radchenko2019fourier, viazovska2017sphere}. However, these works were based on modular forms, and a great survey on the connection between Fourier interpolation formulas and optimal tilings of spheres in a Euclidean space can be found in \cite{cohn2024sphere}.

On a related note, we also mention the work of H. Hedenmalm, A. Montes-Rodriguez, A. Bakan, on Heisenberg uniqueness problems in the setting of hyperbolic Fourier series, see \cite{hedenmalm2011heisenberg1,hedenmalm2020klein}.

\section{Logarithmic integrals on the real line}

Passing to the setting of the real line, we now highlight the role of the global size condition
\[
\int_{\R} \frac{\log w(x)}{1+x^2}\,dx = -\infty
\]
in a range of problems. For convenience, we denote by $dP(x) = (1+x^2)^{-1}dx$ the Poisson measure on $\R$, associated with the half-planes
\[
\C_+ := \{ z \in \C : \Im z > 0 \}, \qquad \C_- := \{ z \in \C : \Im z < 0 \},
\]
and denote by $L^1(dP)$ the corresponding space of integrable functions.

\subsection{Some classical results from Hardy spaces}

We begin with the version of the F. and M. Riesz theorem on the real line, as it will be utilized repeatedly.

\begin{thm}[F. and M. Riesz on $\R$]
Let $\mu$ be a finite complex Borel measure on $\R$ such that
\[
\widehat{\mu}(\xi) = 0, \qquad \xi < 0.
\]
Then $d\mu = f\,dx$ with $f \in L^1(\R)$, and moreover
\[
\int_{\R} \frac{\log |f(x)|}{1+x^2}\,dx > -\infty.
\]
\end{thm}

Equivalently, the density $f = \frac{d\mu}{dx}$ belongs to the Hardy space
\[
H^1(\R) := \{ f \in L^1(\R) : \widehat{f}(\xi)=0 \ \text{for } \xi<0 \}.
\]
More generally, for $1 \le p \le \infty$, we set
\[
H^p(\C_+) := \left\{ F \text{ analytic in } \C_+ : \sup_{y>0} \int_{\R} |F(x+iy)|^p\,dx < \infty \right\}.
\]
Each $f \in H^p(\R)$ arises as the boundary limit of a unique $F \in H^p(\C_+)$,
\[
f(x) = \lim_{y \to 0^+} F(x+iy) \quad \text{in } L^p(\R),
\]
and conversely every $F \in H^p(\C_+)$ is the Poisson extension
\[
F(z) = \int_{\R} \frac{y}{y^2 + (x-t)^2} f(t)\, \frac{dt}{\pi} \qquad z = x+iy \in \C_+
\]
of some $f \in H^p(\R)$. Similarly, one also defines the Banach space $H^\infty(\C_+)$ of bounded analytic functions in $\C_+$, equipped with the norm
\[
\sup_{z \in \C_+} |F(z)| < \infty.
\]
Functions in $H^\infty(\C_+)$ are Poisson extensions of elements in
\[
H^\infty(\R) := \{ f \in L^\infty(\R) : \widehat{f}(\xi)=0 \ \text{for } \xi<0 \},
\]
and each $F \in H^\infty(\C_+)$ admits boundary limits almost everywhere on $\R$, which gives rise to an element $f \in H^\infty(\R)$.

Moving forward, we shall take these results for granted, and refer the reader to \cite[Ch. II]{garnett} for details and proofs.
\begin{proof}[Sketch of proof]
We have already given two proofs in the unit-disc, we roughly demonstrate how we can transport them to $\R$.
Consider the conformal map $\varphi$ from $\C_+$ onto the unit disc $\D$ defined as
\[
\varphi(z)=\frac{z-i}{z+i}, \qquad z\in \C_+
\]
which extends to a homeomorphism from $\R$ onto $\T \setminus \{1\}$. Now define $\nu$ to be the finite Borel measure on $\T$, defined as the push-forward of $\mu$ wrt to $\varphi$, and note that
\begin{equation}\label{EQ:PUSHFORWARD}
\int_{\mathbb T}g(\zeta)\,d\nu(\zeta)
=
\int_{\mathbb R}g(\phi(x))\,d\mu(x), \qquad g \in C(\T)
\end{equation}
Now if we can show that 
\begin{equation} \label{EQ:Phiann}
\int_{\mathbb R}\phi(x)^n\,d\mu(x)=0, \qquad n=0,1,2,\dots
\end{equation}
then it follows from \eqref{EQ:PUSHFORWARD} and the F. and M. Riesz Theorem on $\T$ that $d\nu = hdm$ with $h$ and $\log |h| \in L^1(\T)$. Then $f= h \circ \varphi^{-1} \in L^1(\R)$ and a change of variable shows
\[
\int_{\R} \log |f(x)| \frac{dx}{\pi(1+x^2)} = \int_{\T} \log |h(\zeta)| dm(\zeta) >-\infty.
\]

To prove \eqref{EQ:Phiann}, we primarily note that since $\widehat{\mu}$ is continuous and vanishes on $(-\infty,0)$, we must have $\mu(\R)= \widehat{\mu}(0)=0$. Fix an arbitrary integer $n\geq 1$ and note that the identity 
\[
\phi(x)=\frac{x-i}{x+i}=1-\frac{2i}{x+i},
\]
allows us to rewrite
\[
\int_{\mathbb R}\phi(x)^n\,d\mu(x)
=
\sum_{k=1}^n \binom{n}{k}(-2i)^k
\int_{\mathbb R}\frac{d\mu(x)}{(x+i)^k},
\]
hence it remains only to show that each term vanishes. To this end, the following identity is easily verified for each integer $k \geq 1$:
\[
\frac{1}{(x+i)^k}
=
\frac{(-i)^k}{(k-1)!}\int_0^\infty t^{k-1}e^{-t}e^{itx}\,dt, \qquad x\in \R.
\]
Since $\mu$ is finite, Fubini's theorem yields 
\begin{align*}
\int_{\mathbb R}\frac{d\mu(x)}{(x+i)^k}
=
\frac{(-i)^k}{(k-1)!}\int_0^\infty t^{k-1}e^{-t}
\left(\int_{\mathbb R} e^{itx}\,d\mu(x)\right)\,dt \\
= \frac{(-i)^k}{(k-1)!}\int_0^\infty t^{k-1}e^{-t} \, \widehat{\mu}(-t)\,  dt=0.
\end{align*}
This proves the desired claim.

\end{proof}

The following result will be used frequently, is

\begin{thm}[Outer functions] \thlabel{THM:OUTERFUNC}
Let $1\leq p< \infty$ and $f\in L^p(\R)$ with 
\[
\int_{\R} \frac{\log|f(x)|}{1+x^2} dx >- \infty.
\]
Then there exists an (outer) function $F \in H^p(\C_+)$ such that 
\[
|F(x)|=|f(x)| \qquad \text{a.e} \, \, x \in \R.
\]
\end{thm}
\begin{proof}
    The proof is simple and follows from the Schwartz integral formula
    \[
    F(z):= \exp \left( i\int_{\R}\left( \frac{1}{z-t}-\frac{t}{t^2 +1} \right)\log |f(t)| \frac{dt}{\pi} \right), \qquad z\in \C_+.
    \]
    Clearly $F$ is analytic in $\C_+$, and we also have the identity:
    \begin{equation}\label{EQ:OUTERCHAR}
    \log \abs{F(z)} = \int_{\R} \frac{y}{y^2+(x-t)^2} \log|f(t)| \frac{dt}{\pi}, \qquad z=x+iy \in \C_+.
    \end{equation}
    This identity and standard properties of the Poisson kernel imply that $|F(x)|=|f(x)|$ for a.e $x\in \R$. Now if $f\in L^p(\R)$ for $1\leq p< \infty$, then applying Jensen's inequality to the identity in \eqref{EQ:OUTERCHAR} yields
    \[
    \int_{\R} |F(x+iy)|^p dx \leq \int_{\R} \int_{\R} \frac{y}{y^2+(x-t)^2} \abs{f(t)}^p dt \frac{dx}{\pi} = \int_{\R} \int_{\R} \abs{f(t)}^p dt,
    \]
    hence $F \in H^p(\C_+)$. The case $f\in L^\infty(\T)$ is even simpler.
\end{proof}
We mention that the identity in \eqref{EQ:OUTERCHAR} actually characterizes outer functions $F$ (up to multiplication with a unimodular constant), and asserts that $\log |F|$ is the Poisson extension of an integrable function in $L^1(dP)$.

\subsection{Bernstein's problem on weighted approximation}
Most of the content in this section is based and intimately related to Bernstein's problem on weighted approximation on the real line. Recall that Weierstrass Theorem ensures that any continuous function on a finite interval of the real line can be approximated by polynomials, and this cannot naively be extend to yield an uniform approximation in all of $\R$.

However, a natural way to modify the problem is to introduce a (non-negative) weight-function $W$ on $\R$, which is large at infinity, and measures the proximity by 
\[
\lvert f(x)-p(x) \rvert W(x) \leq \varepsilon,
\]
where $p$ are polynomials. Given a weight $W$ on $\R$, we denote by $C_0(W)$ the space of continuous functions $f$ on $\R$ with $\lim_{|x| \to \infty} W(x)f(x) =0$ and normed by 
\[
\norm{ f }_{C_0(W)} := \sup_{x\in \R} W(x)\lvert f(x)\rvert.
\]
Now for $C_0(W)$ to contain the set of all polynomials $\Po$, we need to require that $W$ \emph{decays rapidly at infinity}: 
\[
\lim_{|x| \to \infty} W(x) |x|^n=0, \qquad n=0,1,2,\dots
\]
\textit{Polynomials approximation:}\\
The celebrated Bernstein problem asks to characterize the weights $W$ such that the polynomials $\Po$ are dense in $C_0(W)$:
\[
\inf_{p\in \Po} \norm{f-p}_{C_0(W)}=0, \qquad \forall f\in C_0(W).
\]
It is useful to dualize the approximation problem using functional analysis. To this end, we denote by $M(W)$ the space of locally finite complex Borel measures on $\R$, which are summable wrt $W^{-1}$. That is, we consider the Banach space $M(W)$ equipped with the norm 
\[
\lvert \lvert \mu \rvert \rvert_{M(W)} := \int_{\R} W^{-1}(x) d\lvert \mu \rvert(x)< \infty.
\]
Now it is not difficult to see that $C_0(W)^* \cong M(W)$ in the usual $L^2$-pairing in $\R$:
\[
\Big \lvert \int_{\R} f(x) d\mu(x) \Big \rvert \leq \norm{f}_{C_0(W)} \norm{\mu}_{M(W)}.
\]
\textit{Moment problem:}\\
Now given a rapidly decaying weight $W$ at infinity, we have by the Hahn-Banach theorem that the polynomials are not dense in $C_0(W)$ if and only if there exists a non-trivial $\mu \in M(W)$ such that its moments satisfy 
\[
m_n(\mu) := \int_{\R} x^n d\mu(x) =0, \qquad n=0,1,2, \dots
\]
In other words, the dual space $M(W)$ lacks the property that its elements are uniquely determined by their moment sequences $\{m_n(\mu)\}_{n\geq 0}$. Consider the Fourier transform
\[
\widehat{\mu}(\xi) := \int_{\R} e^{-ix\xi} \, d\mu(x), \qquad \xi \in \R
\]
and note that $\mu \in M(W)$ is uniquely determined to its moment sequence if and only if 
\[
\widehat{\mu}^{(n)} (0) = 0 ,\qquad n=0,1,2,\dots \implies \mu \equiv 0.
\]
\textit{Quasi-analyticity:} \\
In other words, the polynomials are dense in $C_0(W)$ if and only if the class of functions 
\[
\mathcal{F}M(W) :=\{\widehat{\mu}: \mu \in M(W) \}
\]
is quasi-analytic in $\R$, that is, if $f\in \mathcal{F}M(W)$ with $f^{(n)}(x_0)=0$ for all $n=0,1,2,\dots$ and some point $x_0 \in \R$, then $f\equiv 0$.

We have thus rephrased Bernstein's problem in three different ways, and it is desirable to identify reasonably explicit conditions on $W$, under which these phenomena occur. The main goal of this subsection is to resolve Bernstein’s problem for a sufficiently rich class of weights. A complete solution lies beyond the scope of these notes, and further comments are deferred to the subsection on Further results. The first theorem applies to a broad yet concrete class of weights. Let $\Po$ denote the set of polynomials.

\begin{thm}\thlabel{THM:BERN} Let $W$ be an even weight which decays rapidly at $\infty$, and is of the form 
    \[
    W(x) = \exp \left( -\int_c^{|x|} \frac{w(t)}{t}dt \right), \qquad |x| \geq c
    \]
    where $c>0$ is a constant and $w$ is a positive non-decreasing function on $[c,\infty)$ with $\lim_{x\to +\infty} w(x) = +\infty$. Then $\Po$ is dense in $C_0(W)$ if 
    \[
    \int_\R \frac{\log W(x)}{1+x^2} dx = - \infty.
    \]
\end{thm}
We remark that if $W$ is assumed to be differentiable, then the above specific form of $W$ is equivalent to the condition that
\[
-\frac{xW'(x)}{W(x)} \uparrow +\infty, \qquad x\to \infty.
\]
Examples of such eligible weights are, for instance, $W(x)= e^{-k|x|}$ with $k>0$, but the theorem becomes simple for such weights as the Fourier transform of elements in $M(W)$ extends analytically in a neighborhood of the origin. A more interesting class of weights consists of those decaying at the rate
\[
W(x) \asymp \exp \left( -\frac{k|x|}{\log |x| \log \log |x| \dots  \log \log \dots \log |x|} \right), \qquad |x| \to \infty,
\]
for any finite iterations of logarithms. It is natural to wonder whether the logarithmic integral divergence is sharp. The following result asserts that it is indeed also necessary.

\begin{thm}\thlabel{THM:NECBERN} Let $W \in L^1(\R)$ be a weight with $\log W \in L^1(dP)$. Then there exists a non-trivial measure $\mu \in M(W)$ with identically vanishing moment sequence.  
\end{thm}

The take-away from \thref{THM:BERN} in conjunction with \thref{THM:NECBERN} is that for sufficiently regular weights $W$, the density of $\Po$ in $C_0(W)$ is equivalent to the lack of (global) logarithmic integrability:
\[
\int_\R \frac{\log W(x)}{1+x^2} dx = - \infty.
\]
The proof of \thref{THM:NECBERN} is short and self-contained, hence we give it first.

\begin{proof}[Proof of \thref{THM:NECBERN}]
Define the weight
\[
w(x):=W(x)e^{-|x|^{1/2}}, \qquad x\in \R.
\]
Since $W\in L^1(\R)$ and $\log W\in L^1(dP)$, we have $w\in L^1(\R)$ and $\log w\in L^1(dP)$. By \eqref{THM:OUTERFUNC}, there exists a non-trivial outer function $g\in H^1(\R)$ such that
\[
|g|=w
\qquad \text{a.e. on } \R.
\]
Consider
\[
\widehat{g}(\xi):= \int_{\R} g(x)e^{-ix\xi}\,dx, \qquad \xi\in \R,
\]
and recall that $g\in H^1(\R)$, implies that its Fourier transform $\widehat{g}$ vanishes $dm$-a.e on $(0,\infty)$. Moreover, because $|g(\xi)|\leq W(\xi)e^{-|\xi|^{1/2}}$, all moments of $g$ are finite, so $\widehat{g}\in C^\infty(\R)$. But since $\widehat{g}$ is smooth and vanishes on $(0,\infty)$, it follows that we must have
\[
\widehat{g}^{(n)}(0)=0,\qquad n=0,1,2,\dots.
\]
Now define the measure on $\R$ by $d\mu(\xi):=g(\xi)\,d\xi$, which is seen to be an element in $M(W)$:
\[
\int_{\R}\frac{1}{W(\xi)}\,d|\mu|(\xi)
=
\int_{\R}\frac{|g(\xi)|}{W(\xi)}\,d\xi
=
\int_{\R}e^{-|\xi|^{1/2}}\,d\xi
<\infty.
\]
Furthermore, we also have
\[
\int_{\R}\xi^n\,d\mu(\xi)
=
\int_{\R}\xi^n g(\xi)\,d\xi
=
(-i)^n \,\widehat{g}^{(n)}(0)
=
0, \qquad n=0,1,2,\dots
\]
In conclusion, $\mu \in M(W)$ is a non-trivial measure with an identically vanishing moment sequence.
\end{proof}

In order to prove \thref{THM:BERN}, we shall need a lemma, which extracts the precise way in which the additional regularity assumption of $W$ will be utilized.

\begin{lemma}\thlabel{LEM:REGW} Let $W$ be a rapidly decaying weight as in the statement of \thref{THM:BERN}. Then there exists a constant $A=A(W)>0$, such that
    \[
    \max_{x\geq 0} x^{N(\xi)} W(x) \leq  A\xi^{N(\xi)} W(\xi), \qquad \xi \geq \max(c,1)
    \]
    where $N(\xi)$ denotes the smallest integer which $N(\xi) \geq w(\xi)$.
\end{lemma}
\begin{proof}
There is no harm in assuming that $c\geq 1$. Fix $\xi \geq c$, let $N=N(\xi)\geq w(\xi)$ be the smallest such integer, and write
\[
\sup_{x\geq 0} x^N W(x) \leq c^N \sup_{0\leq x\leq c}W(x) + \sup_{x\geq c} x^N W(x).
\]
Note that since $w$ is non-decreasing on $[c,\infty)$, we have
\[
-\log W(\xi) = \int_c^\xi \frac{w(t)}{t} dt \leq w(\xi) \log \frac{\xi}{c} \leq N \log \frac{\xi}{c},
\]
hence it follows that 
\[
c^N = \xi^N (c/ \xi)^N \leq \xi^N W(\xi).
\]
This takes care of first term, hence it remains only to estimate the second term. We use the definition of $W$ as follows:
\[
\log W(x) x^{w(\xi)} = - \int_c^x \frac{w(t)-w(\xi)}{t}dt +  w(\xi)\log c.
\]
Since $w$ is increasing and the integrand changes sign at $t=\xi$, we see that the maximum is attained at $\xi$, therefore
\[
\sup_{x\geq c} x^{w(\xi)} W(x) = \xi^{w(\xi)} W(\xi) \leq \xi^{N(\xi)}W(\xi), \qquad \xi \geq c.
\]
Combining the two estimates completes the proof.
\end{proof}

We now turn to the proof of sufficiency. 

\begin{proof}[Proof of \thref{THM:BERN}]
We argue by duality, so let $\mu \in M(W)$ with the property that $m_n(\mu)=0$ for $n=0,1,2,\dots$. Note that since $W$ is bounded on $\R$, we have 
\[
\int_{\R} d\abs{\mu}(x) \leq \sup_{x\in \R} W(x) \cdot \norm{\mu}_{M(W)},
\]
hence $\mu$ has finite total variation in $\R$, and hence its Fourier transform $\widehat{\mu}$ defines a bounded continuous function on $\R$.  
\proofpart{1}{The Cauchy transform and moments:}
Consider the Cauchy transform of $\mu$, that is, 
\[
\Ka (\mu) (z) := \int_{\R} \frac{d\mu(x)}{x-z}, \qquad z\in \C \setminus \R
\]
and note that a simple expansion into a geometric sum shows that for all integers $n\geq 0$:
\[
\Ka (\mu)(z) = z^{-(n+1)} \Ka( x^{n+1}d\mu) (z) -\sum_{k=0}^n m_n(\mu) z^{-(k+1)}, \qquad z\in \C \setminus \R.
\]
Hence if all the moments $(m_n(\mu))_n$ vanish, the following identity holds:
\begin{equation}\label{EQ:CAUCHYMOM}
\Ka (\mu)(z) = z^{-n} \Ka (x^n \mu)(z), \qquad z\in \C \setminus \R, \qquad n=1,2,3,\dots
\end{equation}
Note that $F(z):=\Ka(\mu)(z+i)$ is bounded and analytic in the upper-half plane $\C_+ := \{z\in \C: \Im(z)>0 \}$. Indeed, this follows from the assumption that $\mu\in M(w)$:
\[
\lvert F(z) \rvert \leq \int_{\R} \frac{d\lvert\mu \rvert (x) }{|x-z-i|} \leq \int_{\R} d\lvert\mu \rvert (x)  \leq \sup_{x\in \R} W(x) \cdot \norm{\mu}_{M(W)}, \qquad z\in \C_+.
\]
\proofpart{2}{Reduction to logarithmic integral divergence:} 
We claim that the principal aim is to prove that
\begin{equation}\label{EQ:LOGDIVF}
\int_{\R} \frac{\log |F(x)|}{1+x^2} dx = -\infty,
\end{equation}
which would imply that $F \equiv 0$, since functions in $H^\infty(\C_+)$ are logarithmically integrable wrt $dP$ on $\R$. But this would at its turn imply that $\Ka (\mu)=0$, and using the identity
\[
\Ka (\mu)(z) = \int_0^\infty \widehat{\mu}(\xi) e^{i\xi z} d\xi, \qquad z\in \C_+,
\]
we would conclude that $\widehat{\mu}=0$ a.e on $[0,\infty)$, which by the F. and M. Riesz Theorem forces $d\mu = \conj{h} dx$ a.e on $\R$ with $h\in H^1(\R)$. Now since $m_n(\mu) = m_n(\conj{\mu})$, we may then repeat the entire argument to the complex conjugate $\conj{\mu}$, and arrive to the conclusion that $d\mu= g dx$ a.e on $\R$ for some $g \in H^1(\R)$. Together, these condition force $\mu \equiv 0$, hence also the desired conclusion that $\Po$ is dense in $C_0(W)$. 
\proofpart{3}{Moment estimates:}
In order to prove \eqref{EQ:LOGDIVF}, we primarily note that using \eqref{EQ:CAUCHYMOM}, we get for any $n=0,1,2, \dots$
\[
\lvert F(\xi) \rvert = \abs{\Ka(\mu)(\xi+i)} \leq |\xi|^{-n} \int_{\R} \frac{|x|^n}{|\xi+i-x|} d\abs{\mu}(x) \leq 2|\xi|^{-n} \norm{\mu}_{M(W)} \sup_{x\geq 0} |x|^n W(x).
\]
It is at this step that the regularity properties of $W$ will enter the picture. Choosing $\xi\geq c$ and $n=N(\xi)\geq w(\xi)$ and invoking \thref{LEM:REGW}, we conclude that 
\[
\abs{ F(\xi) } \leq A(w) \norm{\mu}_{M(W)} W(\xi), \qquad \xi \geq c.
\]
This estimate implies that that $\log |F|$ is not logarithmically integrable wrt $dP$ on $\R$, from the assumption that $\log W \notin L^1(dP)$. This completes the proof.

\end{proof}

We complete this section by recording the following result, which summarizes our discussion in this subsection.

\begin{cor}\thlabel{COR:2SIDMAJ}
Let $W$ be a weight satisfying the hypothesis of \thref{THM:BERN}. Then there exists a non-trivial measure $\mu \in M(W)$ with $m_n(\mu)=0$ for $n=0,1,2,\dots$ if and only if $\log W \in L^1(dP)$.
\end{cor}

We remark this corollary is essentially a real-line analogue of \thref{COR:SZUP}, which was derived from Szeg\"o's Theorem in the unit-circle.

\subsection{Problems on determining majorants}
A slight modification of Bernstein's problem on weighted polynomial approximation, rephrased as a problem on quasi-analyticicity, goes as follows. Describe the weights $W$ on $\R$, for which there exists a non-trivial function $f \in L^1(\R)$ with the following properties:
\begin{enumerate}
    \item[(1.)] $\abs{f(x)}  \leq  W(x)$, for all $|x|\geq R$ sufficiently large,
    \item[(2.)] The Fourier transform $\widehat{f}$ has a zero of infinite order at the origin:
\[
\widehat{f}(\xi) = \mathcal{O}(|\xi|^n), \qquad \xi \to 0.
\]
for all integers $n\geq 0$.
\end{enumerate}

We declare $W$ to be a $d$-majorant (determining majorant) if there exists no such non-trivial function $f\in L^1(\R)$. In other words, we are interested in exhibiting a non-trivial absolutely continuous measure $fdm \in M(W)$. The main result in this subsection goes as follows.

\begin{thm} Let $W$ be a weight in $L^1(\R)$. For $W$ to be a $d$-majorant it is necessary that $\log W \notin L^1(dP)$. If $W$ additionally satisfies the hypothesis of \thref{THM:BERN}, then the condition $\log W \notin L^1(dP)$ is also sufficient.
\end{thm}

\begin{proof}
\proofpart{1}{Necessity:} The proof is very similar to that of \thref{THM:NECBERN}, we only sketch it. Arguing by contraposition, assume that $\log W \in L^1(dP)$. Set $w(x) := W(x) \exp(-|x|^{1/2})$ and note that $\log w \in L^1(dP)$ and $w \in L^1(\R)$. We can therefore form the outer function $f \in H^1(\R)$ such that $|f|=w$ a.e on $\R$. Now since $x^n w(x)$ is integrable for all integers $n\geq 0$, we conclude that $\widehat{f} \in C^\infty(\R)$, and since $\supp{\widehat{f}}\subseteq[0,\infty)$, we must have that $\widehat{f}^{(n)}(0)=0$ for all integers $n=0,1,2,\dots$. Since $w\leq W$, we conclude that $W$ is not a $d$-majorant.
\proofpart{2}{Sufficiency:} Now suppose that $W$ satisfies the hypothesis of \thref{THM:BERN} and $\log W \notin L^1(dP)$. Let $f\in L^1(\R)$ be an arbitrary element with 
\[
(i) \, \abs{f(x)} \leq W(x),  \qquad |x|>R \qquad (ii) \, \widehat{f}(\xi) = \mathcal{O}(|\xi|^n), \qquad \xi \to 0, \qquad n\geq 0.
\]
Since $\widehat{f}$ is continuous on $\R$ we have $\widehat{f}(0)=0$, and inductively, we can deduce that $\widehat{f}$ is infinitely differentiable at $\xi=0$ with $\widehat{f}^{(n)}(0)=0$ for $n=0,1,2,\dots$. Now since $W$ decreases rapidly at infinity, we can differentiate inside the integral to conclude that $\widehat{f}\in C^\infty(\R)$, hence the moments of $f$: $m_n(fdm)=0$ for all $n=0,1,2,\dots$. Now consider $F(z):= \Ka(f)(z+i)$ which defines bounded analytic function on $\C_+$. Furthermore, the vanishing moments imply that we can use \eqref{EQ:CAUCHYMOM}, which gives 
\[
\abs{F(\xi)} = |\xi|^{-n} \abs{ \Ka(x^n f)(\xi +i)}  \leq |\xi|^{-n} \int_{\R} \frac{|x|^n W(x)}{|\xi+i-x|}dx\leq 2|\xi|^{-n} \int_{0}^\infty x^n W(x)dx,
\]
for all $\xi\neq 0$ and any $n=0,1,2,\dots$. As in the proof of \thref{THM:BERN}, once we show that $\log F \not\in L^1(dP)$ we can conclude the $f=0$ in a similar way. To achieve this, we shall prove an estimate $F$ on for large enough $\xi$ in terms of $W$. In view of the previous display, we shall estimate the moments of $W$ by utilizing the regularity properties of $W$ as in the proof of \thref{LEM:REGW}. To this end, fix $\xi\geq c$ and set $n=n(\xi)\geq w(\xi)$ be the smallest such integer and note that:
\begin{multline*}
\int_0^\infty x^n W(x) dx \leq \frac{c^{n+1}}{n+1} \sup_{0\leq x\leq c} W(x) + \int_{c}^\infty \frac{dx}{x^2} \sup_{x\geq c}x^{w(\xi)} W(x) \\
\leq \frac{c^{n+1}}{n+1} \sup_{0\leq x\leq c} W(x) + \frac{1}{c} \sup_{x\geq c}x^{n} W(x) \leq A(W)\xi^{n+1} W(\xi),
\end{multline*}
where $A(W)>0$ is constant independent of $\xi$. Consequently, it follows that 
\[
\abs{F(\xi)} \leq 2A(w)\xi W(\xi), \qquad |\xi|>c,
\]
hence $\log |F| \notin L^1(dP)$. This completes the proof.

\end{proof}

We shall now consider a unilateral version of the problem on $d$-majorants. Suppose we have a weight $W$ and a non-trivial function $f\in L^1(\R)$ with the property that
\begin{equation}\label{EQ:1EST}
\abs{f(x)} \leq  W(x), \qquad x >c. 
\end{equation}
Let us first examine the simplest scenario when $W\equiv 0$. In that case, the right translation of $f_c$ of $f$ by $c$ is supported in $(-\infty, 0]$, hence by the F. and M. Riesz Theorem we have
\[
\int_{\R} \frac{\log |\widehat{f}(\xi)|}{1+x^2}dx >- \infty.
\]
This implies that $\widehat{f}$ cannot vanish on a set of positive measure in $\R$. On the contrary, the complex conjugate of any $H^1(\R)$-function satisfies such an estimate, hence one can generally not find any better restriction on the zero set of $\widehat{f}$. However, one can easily construct a non-trivial $f\in H^1(\R)$, which is smooth on $\R$ and \emph{flat} at $0\in \R$: 
\[
f^{(k)}(0)=0, \qquad  k=0,1,2,\dots
\]
This implies that the unilateral condition is much weaker than the two-sided. However, it turns out that $\int_{\R} \log W dP=- \infty$ still imposes a restriction on how small $\widehat{f}$ can be.

\begin{thm}[Unilateral $d$-majorant]\thlabel{THM:1SdMAJ}
Let $W$ be an even weight which satisfies $\int_{\R} \log WdP =-\infty$. Then whenever $f\in L^1(\R)$ is a non-trivial function satisfying the unilateral
\[
\abs{f(x)} \leq W(x), \qquad x>c
\]
then its Fourier transform $\widehat{f}$ cannot vanish on a set of positive Lebesgue measure on $\R$.
\end{thm}

We shall consider a slightly different statement, which will be proved in the similar way as \thref{THM:1SdMAJ}, but has the additional upside that it can be generalized further, and does not require any additional regularity conditions. The result in question may be viewed as a generalization of the F. and M. Riesz Theorem, is due to A. Beurling.
Given a measure $\mu \in M(\R)$, we denote by 
\[
V_\mu (t) := \abs{\mu}\left([t,+\infty) \right) \qquad t>0
\]
the total-variation of $\mu$, measured at neighborhoods at $+\infty$. The main result in this subsection goes as follows.

\begin{thm}[Beurling] \thlabel{THM:BEURLING1} Let $\mu \in M(\R)$ with the property that 
\[
\int_0^\infty \frac{ \log V_\mu (t) }{1+t^2}dt =- \infty.
\]
Then $\widehat{\mu}$ cannot vanish on a set of positive Lebesgue measure, unless $\mu$ is identically zero.
\end{thm}

Let us first deduce \thref{THM:1SdMAJ} from Beurling's theorem.

\begin{proof}[Proof of \thref{THM:BEURLING1}]
Set $d\mu(x)=f(x)\,dx$. For $t>c$ large enough, we have
\[
V_\mu(t)=\int_t^\infty |f(x)|\,dx \leq \int_t^\infty W(x)\,dx=:w(t).
\]
According to Beurling's \thref{THM:BEURLING1}, it suffices to show that
\[
\int_0^\infty \frac{\log w(t)}{1+t^2}\,dt=-\infty.
\]

Since $W$ is even and $\int_{\R}\log W\,dP=-\infty$, we have
\[
\int_c^\infty \frac{\log W(x)}{1+x^2}\,dx=-\infty.
\]
Moreover, for all $t>c$ we have
\[
w(t)\geq \int_t^{t+1} W(x)\,dx,
\]
and hence by Jensen's inequality,
\[
\log w(t)\geq \int_t^{t+1} \log W(x)\,dx, \qquad t>c.
\]
Integrating against the Poisson measure $(1+t^2)^{-1}dt$ and applying Fubini's theorem, we obtain
\[
\int_c^\infty \frac{\log w(t)}{1+t^2}\,dt
\gtrsim
\int_c^\infty \log W(x)
\left( \int_{x-1}^{x} \frac{dt}{1+t^2} \right) dx.
\]
Since $(1+t^2)^{-1} \asymp (1+x^2)^{-1}$ for $t \in [x-1,x]$, we get
\[
\int_c^\infty \frac{\log w(t)}{1+t^2}\,dt
\gtrsim
\int_c^\infty \frac{\log W(x)}{1+x^2}\,dx
=
-\infty.
\]

\end{proof}

Now the proof of \thref{THM:BEURLING1} will principally be based on the following proposition, which is essentially a variant of the classical Hadamard two-constant theorem.

\begin{prop}\thlabel{PROP:TCTHM} Let $h$ be a continuous function in $\conj{\C}_+$ and analytic in $\C_+$, and $0<\varepsilon<1$, $A>0$ be constants, and $E\subset \R$ such that $0\in E$ is a Lebesgue point of density. Consider the Laplace-transform of $h$:
\[
H(z) := \int_0^\infty h(t) e^{itz} dx, \qquad z\in \C_+.
\]
If $h$ satisfies the hypothesis
\[ 
(a.) \, \abs{h(t)}\leq \varepsilon \, \, \text{on} \, \, E, \qquad (b.) \, \abs{h(x+iy)} \leq e^{Ay}, \, \, x\in \R, \, \, y>0,
\]
then 
\[
\abs{H(\xi + i)} \leq e^{(A-\xi)} + \varepsilon^{\alpha} \frac{1-e^{(A-\xi)}}{\xi -A}, \qquad \xi >0,
\]
where $\alpha>0$ constant only depending on $E$.
\end{prop}
\begin{proof}
\proofpart{1}{Lebesgue point of density:} The assumption $0 \in E$ being a Lebesgue point of density readily implies that there exists numbers $0<\delta_0<1$, such that 
\[
\abs{E \cap (-\delta, \delta)}\geq \delta, \qquad 0<\delta < \delta_0.
\]
Let $\omega(E,z)$ denote the harmonic measure of $E\subset \R$ wrt to the domain $\C_+$ and with reference point $z\in \C_+$, explicitly defined by the Poisson formula
\[
\omega(E,z) := \int_E \frac{y}{(x-t)^2 + y^2} \frac{dt}{\pi}, \qquad z= x+iy, \qquad y>0.
\]
The Lebesgue density assumption on $0\in E$ implies that
\[
\omega_{i\delta} (E) \geq \int_{E\cap(-\delta,\delta)} \frac{\delta}{t^2+\delta^2} \frac{dt}{\pi} \geq \frac{|E \cap (-\delta, \delta)|}{2\pi \delta} \geq \frac{1}{2\pi}
\]
whenever $0<\delta \leq \delta_0$. Now since $\delta \mapsto \omega_{i\delta}(E)$ is a positive continuous function on $[\delta_0,1]$, we conclude that $\inf_{0<\delta\leq 1} \omega_{i\delta}(E)>0$. 
\proofpart{2}{Cauchy's Theorem:} Now fix a point $\zeta= \xi +i$ with $\xi \in \R$. Using Cauchy's Theorem, we can transform the Laplace transform $H$ into a sum of integrals of $t \mapsto h(x)e^{it\zeta}$ 
over the line segments $\gamma_1 := \{it: 0\leq t \leq 1\}$, $\gamma_2:= \{t+i: 0\leq t \leq R\}$, $\gamma_3 := \{R+it: 0\leq t \leq 1\}$. Here $R>0$ is a large number, which will eventually be sent to $+\infty$. Now the integral over the right vertical segment is simple to estimate:
\[
\abs{\int_{\gamma_3} h(t) e^{it\zeta} dt} \leq e^{-R} \int_{0}^1 \abs{h(R+it)} e^{-t \xi} dt  \leq e^{-R} \int_0^1 e^{t(A-\xi)}dt  \to 0, \qquad R\to +\infty,
\]
Meanwhile the integral over the horizontal segment $\gamma_2$ can be estimated as follows:
\[
\abs{\int_{\gamma_2} h(t) e^{it\zeta} dt} \leq e^A \int_{0}^R \abs{e^{it\zeta}} dt \leq e^{A-\xi} \int_0^\infty e^{-t} dt =  e^{A-\xi}.
\]
Therefore, it only remains to estimate the integral over $\gamma_1$, and this is where the estimate of the harmonic measure will be utilized.
\proofpart{3}{A two-constant estimate:} Note that $G_A(z):= h(z)e^{izA}$ belongs to $H^\infty(\C_+)$ with $\abs{G_A(z)} \leq e^{Ay} e^{-Ay} \leq 1$ for $z\in \C_+$. Set $\alpha:= \inf_{0<t<1} \omega_{it}(E) \in (0,1)$, and note that $\abs{G_A} = \abs{h}\leq \varepsilon$ on $E$. We claim that another two-constant estimate, as previously encountered in \thref{LEM:2const}, (see also \thref{THM:TWOCONST}) gives
\begin{equation*}
\abs{h(it) e^{-At} }= \abs{G_A(it)} \leq \varepsilon^{\alpha}, \qquad  0\leq t\leq 1.
\end{equation*}
We leave these details for the reader. With this at hand, we may now estimate the integral over the vertical segment $\gamma_1$ on the imaginary axis as:
\[
\abs{\int_{\gamma_1}h(t) e^{it\zeta} dt} \leq \int_0^1 \abs{h(it)} e^{-t\xi} dt \leq \varepsilon^\alpha \int_0^1 e^{t(A-\xi)} dt =  \varepsilon^\alpha \frac{1-e^{(A-\xi)}}{\xi-A}.
\]
The proof is now complete.
    
\end{proof}

We now turn to the proof of Beurling's Theorem, which requires a bit more work.

\begin{proof}[Proof of \thref{THM:BEURLING1}] We may without loss of generality assume that $\norm{\mu}_{M(\R)}\leq 1$. Consider the bounded analytic function $F(z):= \Ka (\mu)(z+i)$ with $z\in \C_+$ and we shall aim to prove that 
\[
\int_1^\infty \frac{\log |F(\xi)|}{1+\xi^2} d\xi = - \infty.
\]
This will again imply that $\Ka(\mu)=0$ on $\C_+$. Applying a similar argument to $G(z):= \Ka(\mu)(-z-i)$ bounded analytic in $\C_+$, we conclude that $\Ka(\mu)=0$ on $\C\setminus \R$. Then, for instance, a Fatou--Plemelj Jump argument involving the classical formula 
\[
\mu([a,b]) = \lim_{\varepsilon \to 0+} \frac{1}{2\pi i} \int_a^b \left( \Ka(d\mu)(x+i\varepsilon)- \Ka(d\mu)(x-i\varepsilon) \right) dx, \qquad a<b
\]
implies that $\mu\equiv 0$. See \cite{cauchytransform}.
To this end, pick $A>0$ and consider the decomposition $\mu= \mu_A + \nu_A$, where $\mu_A(B):= \mu(B \cap (-\infty, A) )$ for all Borel sets $B\subset \R$. Writing $\Ka(\mu)=\Ka(\mu_A) + \Ka(\nu_A)$, we have 
\[
\abs{\Ka (\nu_A)(\xi + i)} \leq \int_{A}^\infty \frac{1}{\abs{\xi+i-t}} d|\mu|(t) \leq V_{\mu}(A).
\]
Now set $H:=\Ka (\mu_A)$ defined by the formula
\[
H(z) := \int_{0}^\infty \widehat{\mu_A}(\xi) e^{i\xi z} d\xi, \qquad z\in \C_+.
\]
Clearly, $\widehat{\mu_A}$ is continuous on $\R$ and extends holomorphically to $\C_+$, and
\[
\abs{\widehat{\mu_A}(x+iy)} \leq \int_{-\infty}^{A} e^{ty} d\abs{\mu}(t) \leq e^{Ay}\norm{\mu}_{M(\R)}\leq  e^{Ay}, \qquad y>0.
\]
Furthermore, on the set $E(\mu):= \{\xi \in \R: \widehat{\mu}(\xi)=0\}$, we have 
\[
\abs{\widehat{\mu_A}(\xi)}= \abs{\widehat{\nu_A}(\xi)} \leq V_{\mu}(A), \qquad \xi \in E.
\]
Applying \thref{PROP:TCTHM}, we get 
\[
\abs{\Ka(\mu_A)(\xi+i)} \leq e^{A-\xi} + V_{\mu}(A)^{\alpha} \frac{1-e^{A-\xi}}{\xi-A}, \qquad \xi>0.
\]
Now since this is true for all $A>0$, we may pick $A= \xi/2$, and combining with the estimate of $\Ka(\nu_A)$, we get 
\begin{multline*}
\abs{F(\xi)} \leq e^{-\xi/2} + V_{\mu}(\xi/2) + 2 V_{\mu}(\xi/2)^{\alpha} \frac{1-e^{-\xi/2}}{\xi} \leq \\  e^{-\xi/2} + 2V_{\mu}(\xi/2)^{\alpha} \leq 2 \max (e^{-\xi/2}, V_{\mu}(\xi/2) )^{\alpha} , \qquad \xi \geq 2,
\end{multline*}
where we used the fact that $\alpha = \inf_{0<\delta<1} \omega_{i\delta}(E) \in (0,1)$ and that $0\leq V_{\mu}(\xi) \leq 1$ for all $\xi$. Consequently, it follows that 
\[
\int_2^{\infty} \frac{\log \abs{F(\xi)}}{1+\xi^2} d\xi \lesssim 1 + \alpha \int_{2}^\infty \frac{\max \left( - \xi/2, \log V_{\mu}(\xi/2) \right)}{1+\xi^2} d\xi.
\]
Using monotonicity of $V_\mu$ it is not difficult to show that the integral on the right-hand-side is divergent by the assumption initial assumption on $V_{\mu}$. This completes the proof. 
\end{proof}

\subsection{On measures exhibiting one-sided rapid Fourier decay}

We gather here some interesting consequences Beurling's Theorem, primarily to the periodic setting of measure on the unit-circle. The first result gives a criterion for when Wiener algebra functions behave like their analytic counterpart. 

\begin{thm}[Beurling] \thlabel{THM:BEU1} Let $f\in A(\T)$ with the property that 
\[
\sum_{n=1}^\infty \frac{\log \rho_n}{n^2}= - \infty,
\]
where $\rho_n := \sum_{k\geq n} \abs{\widehat{f}(k)}$. Then $f$ cannot vanish on a set of positive Lebesgue measure in $\T$, unless $f\equiv 0$.
\end{thm}
We remark that the same conclusion also holds with $\rho_n := \sum_{k\leq -n} \abs{\widehat{f}(k)}$.
\begin{proof}
Consider the $\mu\in M(\R)$ defined by $\mu = \sum_{k \in \mathbb{Z}} \widehat{f}(k) \delta_k$, where $\delta_k$ denotes the Dirac delta mass at the integer $k$. It readily follows that 
\[
\int_1^\infty \frac{\log V_{\mu}(t)}{t^2} dt = \sum_{n=1}^\infty \int_{n}^{n+1} \frac{\log V_{\mu}(t)}{t^2} dt \geq \sum_{n=1}^\infty \frac{\log V_{\mu}(n)}{n(n+1)} \geq \frac{1}{2}\sum_{n>1} \frac{\log \rho_n}{n^2} = - \infty.
\]
Since $\widehat{\mu}(t) = f(e^{-it})$, the claim is an immediately consequence of Beurling's Theorem.
\end{proof}
Another remarkable consequence of Beurling's Theorem is a yet another classical result by M. Cartwright and N. Levinson, which actually pre-dates it.

\begin{thm}[Cartwright-Levinson] \thlabel{THM:LEVCART} Let $\{w(n)\}_{n=1}^\infty$ be a non-decreasing sequence of positive real numbers satisfying 
\[
\sum_{n=1}^\infty \frac{w(n)}{n^2} = +\infty.
\]
If the Fourier coefficients of a measure $\mu \in M(\T)$ satisfies a rapid unilateral decay of the form:
\[
\widehat{\mu}(n) = \mathcal{O} \left( e^{-w(|n|)} \right), \qquad n\to -\infty \qquad (\text{or} \, \, n\to \infty),
\]
then $\supp{\mu}=\T$, unless $\mu \equiv 0$.
\end{thm}

\begin{proof} Suppose for the sake of obtaining a contradiction that there exists $\mu \in M(\T)$ such that $\supp{\mu} \cap I = \emptyset$ for some arc $I \subset \T$, yet its negative Fourier coefficients satisfy the aforementioned unilateral decay. Pick a smooth function $\phi$ which is supported on an arc centered at $\zeta =1$, and whose length is no larger than $|I|/10$, and consider the continuous function $f= \phi \ast \mu$ on $\T$, which vanishes identically on a smaller subarc $J\subset I$. Furthermore, we have by monotonicity of $(w(n))_n$ that:
\[
\rho_n := \sum_{k\leq -n} \abs{\widehat{f}(n)}  = \sum_{k\leq -n}\abs{\widehat{\phi}(k)}  \abs{\widehat{\mu}(k)} \leq  C e^{-w(|n|)} \sum_{k\leq -n} \abs{\widehat{\phi}(k)} \leq C'e^{-w(|n|)}.
\]
It follows from \thref{THM:BEU1} that $f= \phi \ast \mu \equiv 0$ on $\T$. But since this conclusion is true for any smooth approximate of the identity $\phi$ with sufficiently small support, we conclude that $\mu =0$.
    
\end{proof}
We now record the following interesting consequence for finite measures in $\R$ with spectral gaps: $\supp{\widehat{\mu}} \cap I = \emptyset$, for some interval $I\subset \mathbb{R}$.

\begin{cor}Let $\nu \in M(\R)$ with the property that $\text{diam}(\supp{\nu}) < 2a$ for some $a>0$. Assume further that the Fourier transform of $\nu$ satisfies the unilateral decay
\[
\widehat{\nu}(\pi n/a) = \mathcal{O} \left( e^{- w(|n|)} \right), \qquad n\to - \infty, 
\]
for some increasing sequence of positive numbers $\{w(n)\}_{n=1}^\infty$ with 
\[
\sum_{n \geq 1} \frac{w(n)}{n^2}= + \infty.
\]
Then $\nu \equiv 0$.
\end{cor}
\begin{proof}
Pick a number $0<b<a$ with $\text{diam}(\supp{\nu}) < 2b$. By means translating $\nu$, we may assume that $\supp{\nu}\subseteq [-b,b]$. Consider the homeomorphism $h: [-a,a] \to \T$ defined by $h(x)=e^{-i\pi x/a}$, and define the finite Borel measure $\mu$ on $\T$, via the pull-back
\[
\int_{\T} f(\zeta) d\mu(\zeta) = \int_{-a}^a f(h(x)) d\nu(x), \qquad f\in C(\T).
\]
Note that since $\widehat{\mu}(n)= \widehat{\nu}(\pi n/a)$, and $\mu$ vanishes on the arc $\{e^{ix}: a<x<b\}$, the claim immediately follows from the Cartwright-Levinson Theorem.
\end{proof}

\subsection{A problem with many faces - Completeness, Moments and Zeros} 

We close this chapter by considering a more classical moment-type problem, with the principal intent of illustrating how it relates to several others classical problems. 

Let $\Lambda$ be a sequence of positive real numbers. We ask for which $\Lambda$ does the following moment-type uniqueness property for measures supported in a finite compact interval $I\subset[0,\infty)$: any complex finite Borel measure $\mu$ supported in $I$ satisfies the uniqueness property:
\[
\int_{I} t^{\lambda} d\mu(t) =0, \qquad \forall \lambda \in \Lambda \implies \mu \equiv 0.
\]
Denote by $C(I)$ the Banach space of continuous functions on $I$ equipped with the usual supremum norm. Using that the fact that the dual space of $C(I)$ consists of complex finite Borel measures $\mu$ on $I$, we see that our moment-type uniqueness problem is equivalent to the completeness problem of asking for which sequences $\Lambda$ the linear manifold
\[
\mathcal{P}(\Lambda):= \text{Span}\{t^{\lambda}: \lambda \in \Lambda \} \qquad \text{dense in} \qquad C(I).
\]
For simplicity, and without loss of generality, we may assume that $I=[0,1]$. These problems dates back to the work of H. M\"untz and O. Sz\'asz, and have the following satisfactory answer.

\begin{thm}[M\"untz-Sz\'asz]\thlabel{THM:MUNTZSASZ} $\mathcal{P}(\Lambda)$ is dense in $C(I)$ if and only if $1\in \Lambda$ and
\begin{equation}\label{EQ:1/lambdadiv}
\sum_{\lambda \in \Lambda} \frac{1}{\lambda} = +\infty.
\end{equation}
Furthermore, if 
\[
\sum_{\lambda \in \Lambda} \frac{1}{\lambda} <\infty,
\]
then $t^{\kappa}$ with $\kappa \notin \Lambda$ and $\kappa>0$ does not belong to closure of $\mathcal{P}(\Lambda)$.
\end{thm}
The surprising feature of this result is that in contrast with complete orthonormal systems for $L^2(I)$, which are extremely sensitive to removal of a single, the system of monomials $t^\lambda$ with $\lambda \in \Lambda$ is rather insensible. As we shall see from the proof below, the completeness result also holds in the framework of Lebesgue spaces $L^p(I)$ for $1\leq p<\infty$. In the proof below, we shall illustrate that the condition \eqref{EQ:1/lambdadiv} intrinsically arises from the characterization of zero sets of functions in the Hardy spaces.

\begin{proof}[Proof of \thref{THM:MUNTZSASZ}] 
Suppose $\Lambda$ is a sequence of positive real numbers satisfying $1 \in \Lambda$ and the condition \eqref{EQ:1/lambdadiv}. Arguing by duality, assume $\mu$ is a complex finite Borel measure on $I$ with 
\[
\int_{I} t^{\lambda} d\mu(t) =0, \qquad \lambda \in \Lambda.
\]
Consider the function 
\[
F(z) := \int_I t^{-iz} d\mu(t), \qquad \Im (z)>0,
\]
which by means of invoking Morera's Theorem is easily seen to be a bounded analytic function in $\Im(z)>0$, hence an element in the Hardy space $H^\infty(\C_+)$. Furthermore, the annihilating assumption of $\mu$ implies that $F(i\lambda)=0$ for all $\lambda \in \Lambda$. Now since the zeros  $\{\lambda_n\}_n \subset \C_+$ of any non-trivial function $f$ in $H^\infty(\C_+)$ must satisfy the Blaschke condition 
\[
\sum_n \frac{\Im \lambda_n}{1+|\lambda_n|^2} < \infty,
\]
contrary to the assumption in \eqref{EQ:1/lambdadiv}, we conclude that $F \equiv 0$, hence $\mu \equiv 0$. This shows that $\{1\} \cap \{t^{\lambda}: \lambda \in \Lambda\}$ has a dense linear space in $C(I)$. \\

Conversely, we shall show that if
\[
\sum_{\lambda \in \Lambda} \frac{1}{\lambda} < \infty,
\]
and $\kappa \notin \Lambda$, then we can construct a non-trivial complex finite Borel measure on $I=[0,1]$ with
\begin{equation}\label{EQ:ANNILMEASURE}
\int_{I} t^{\lambda} d\mu(t) =0, \qquad \lambda \in \Lambda, \qquad \int_{I} t^{\kappa} d\mu(t)\neq 0.
\end{equation}

Indeed, a routine duality argument involving the Hahn--Banach separation theorem ensures that $t^{\kappa}$ cannot belong to the closure of $\mathcal{P}(\Lambda)$ in $C(I)$.

We consider the infinite Blaschke product in the upper-half plane, defined by
\[
B_{\Lambda}(z) = \frac{z-i}{z+i} \prod_{\lambda \in \Lambda} \frac{i(1+\lambda)-z}{i(1+\lambda)+z}, \qquad \Im(z)>0.
\]
Indeed, the condition
\[
\sum_{\lambda\in\Lambda} \frac{1}{\lambda+1} < \infty
\]
is equivalent to imposing the Blaschke condition ensuring that infinite product defining $B_\Lambda$ is a non-trivial bounded analytic function in $\C_+$, which vanishes only along the imaginary axis
\[
\left\{i(1+\lambda): \lambda \in \Lambda \cup \{0\} \right\}.
\]
Now consider the bounded analytic function
\[
F(z):= \frac{B_{\Lambda}(z)}{(z+i)^{10}}, \qquad \Im(z)>0
\]
whose zeros are precisely those of $B_{\Lambda}$ in $\C_+$, and $F$ extends to an element in $H^1(\R)\cap H^2(\R)$. By the Paley--Wiener theorem for the half-plane, there exists a function $f \in L^2(0,\infty)$ such that
\[
F(z) = \int_0^{\infty} f(x) e^{izx} , dx, \qquad \Im z>0.
\]
Now set $z= i(1+\alpha)$ with $\alpha>0$ and note that the change of variable $-\log t = x$ shows that 
\[
F(i(1+\alpha)) = \int_0^\infty f(x) e^{-x(1+\alpha)}dx = \int_0^1 t^{\alpha} f(-\log t) dt.
\]
Defining the complex finite Borel measure $\mu$ on $[0,1]$ by
\[
d\mu(t) := f(-\log t) , dt, \qquad 0<t\le 1,
\]
we conclude that \eqref{EQ:ANNILMEASURE} holds. This completes the proof.

\end{proof}

At last, we end this section by demonstrating that the aforementioned problems also may be reformulated as a certain quasi-analyticity problem of the Fourier transform. This time, we shall restrict our attention of positive integers $\Lambda$. We ask what conditions on $\Lambda \subseteq \mathbb{Z}_+$ ensures that for any $\mu \in M(I)$ we have
\begin{equation}\label{EQ:DEEPZERO}
\widehat{\mu}^{(\lambda)}(0)=0 \qquad \forall \lambda \in \Lambda \implies \mu \equiv 0.
\end{equation}
In some sense, the question boils down to describing how sparse the Taylor coefficients of the Fourier transform of compactly supported measures on $\R$ can be. Now the simple observation that 
\[
\widehat{\mu}^{(n)}(0) = i^n \int_{I} t^n d\mu(t),
\]
in conjunction with the M\"untz-Sz\'asz Theorem shows that \eqref{EQ:DEEPZERO} holds if and only if 
\[
\sum_{\lambda \in \Lambda} \frac{1}{\lambda} = +\infty.
\]
\subsection{Further results}

There is an improvement of the Beurling-Cartwright-Levinson Theorems by A. Volberg, which goes as follows.

\begin{thm}[Volberg] Let $M:[1,\infty)\to [0,\infty)$ be a concave function with $M(t)/\sqrt{t}$ increasing in a neighborhood at infinity, and such that
\[
\int_1^\infty \frac{M(t)}{t^2} dt = + \infty.
\]
Then any non-trivial $f\in L^1(\T)$ with $\widehat{f}(n) = \mathcal{O}(e^{-M(|n|)})$, as $n \to -\infty$, must satisfy
\[
\int_{\T} \log |f| dm >-\infty.
\]
\end{thm}
We remark that proof involves a mixture of very intricate potential theory and E. M. Dyn'kin's work on pseudo-analytic continuations of functions. A detailed proof can be found in \cite{havinbook} and also P. Koosis \cite{koosis1998logarithmic}. Further refinements and continuations of such results were also carried out by A. Borichev, which also be located in \cite{havinbook}. Now there are local versions of the above results, for $M$ which decay slower than that of Volberg. The following result is recent and due to B. Malman, and shows that when the spectral decay is slightly slower, then one retrieves a local version of Volberg's Theorem.

\begin{thm}[Malman, 2024] Let $\mu \in M(\T)$ whose Fourier coefficients satisfy the following rapid one-sided decay:
\[
\widehat{\mu}(n) = \exp (-c n^{1/2} ), \qquad n\to \infty.
\]
Then $d\mu = f dm$ with $f\in L^1(\T)$, and the $\text{ess supp}(f)$ can be chosen to be an open set $U$, with the property that for any $\zeta \in U$, there exists an arc $I(\zeta)$ such that 
\[
\int_{I(\zeta)} \log |f| dm>- \infty.
\]
\end{thm}
We note that the above result may also be regarded as a generalizing the brothers Riesz Theorem. The proof of Malman's Theorem in \cite{malman2023shift} involves an interesting problem on polynomial approximation in the unit-disc and extends upon the former work of S. Khrushchev in \cite{khrushchev1978problem}.


A more detailed treatment of these problems can be found in \cite{koosis1998logarithmic} by P. Koosis, with many interesting historical notes. It is worth mentioning that L. de Branges introduced a different perspective to the Bernstein problem, which later led to several improvements. we only mention a few great series of recent works in this direction: the work of A. Borichev and M. Sodin in \cite{borichev2001krein}, and also by A. Poltoratski in \cite{poltoratski2014bernstein}.

\section{Small support and amplitudes}
\subsection{Fourier majorants}

Recall that the classical Parseval Theorem implies that if $S$ is a non-trivial distribution on the unit-circle $\T$ with Fourier coefficients:
\[
\widehat{S}(n) := S(\zeta \mapsto \zeta^{-n}), \qquad n=0, \pm 1, \pm 2, \dots
\]
satisfying 
\[
\sum_n \abs{\widehat{S}(n)}^2 < \infty,
\]
then $S$ must be an $L^2$-function on $\T$, and thus be supported on a set of positive Lebesgue measure. From this point of view, one encounters yet another manifestation of the uncertainty principle, asserting that if we can to construct a distribution or measure with small support in the sense of Lebesgue measure, then their Fourier coefficients are bounded to have large $\ell^2$-mass. It is thus natural to further investigate a qualitative threshold for this phenomenon. For instance, we may ask if any positive sequence $(w(n))_{n=0}^\infty$ with $\sum_n w(n)^2 := \infty$ an \emph{admissible Fourier majorant} for a distribution $S$ with support on a set of zero Lebesgue measure in $\T$? In other words, does there exists a distribution $S$ on $\T$, supported on a set of Lebesgue measure zero, such that 
\[
\abs{\widehat{S}(n)} \leq w(|n|), \qquad n= \pm 1, \pm 2, \dots
\]
The following remarkable result originates for the profound work of Ivashev-Musatov \cite{ivavsev1958fourier}, which was later refined and improved by T. W. K\"orner in \cite{korner1977theorem1,korner1977theorem2,korner1986theorem3}.

\begin{thm}[Ivashev-Musatov, K\"orner] \thlabel{THM:IM} Let $(w(n))_{n=0}^\infty$ be a sequence of positive numbers which satisfy the following properties:
\begin{enumerate}
    \item[(i)] $ \sum_n w(n)^2 = + \infty$,
    \item[(ii)] there exists a constant $C>1$ such that for any integer $n\geq 1$:
    \[
    C^{-1} w(n) \leq w(k) \leq C w(n), \qquad n\leq k \leq 2n.
    \]
\end{enumerate}
There there exists a probability measure $\mu$ supported on a set of Lebesgue measure zero on $\T$, such that 
\[
\abs{\widehat{\mu}(n)} \leq w(|n|), \qquad n=\pm1, \pm 2, \dots.
\]
\end{thm}
Roughly speaking, the theorem asserts that square divergent sequences whose oscillation between $n$ and $2n$ is controlled are admissible Fourier majorants for measure with small support. The initial proof of Ivashev--Musatov required more restrictive regularity conditions on $(w(n))_n$, but several of these we successively removed by K\"orner. 

Now it is now natural to ask whether the additional regularity condition in $(ii)$ is necessary. Although, it is well-known that the condition $(ii)$ may likely be relaxed, the following result demonstrates that it cannot simply be removed. 

\begin{thm}[K\"orner] \thlabel{THM:KÖRsuff} There exists decreasing sequence of positive numbers $(w(n))_n$ with the property  $\sum_n w^2(n) = \infty$, but any non-trivial distribution $S$ on $\T$ with 
\[
\abs{\widehat{S}(n)} \leq C w(|n|), \qquad n=0,\pm 1, \pm 2, \dots
\]
for some constant $C>0$, must have full support in $\T$.
\end{thm}

Note that \thref{THM:KÖRsuff} has the additional property that no compact subset $E\subsetneq \T$ can support a non-trivial distribution $S$ with 
\[
\abs{\widehat{S}(n)} \leq Cw(|n|), \qquad n=0,\pm 1, \pm2, \dots
\]
In the proceeding sections, we shall also prove a one-sided analogue of the above result. It will be deduced from the Cartwright--Levinson Theorem. With these perspectives in mind, the following complementary result was recently obtained by the author in \cite{limani2024fourier}.

\begin{thm}[Limani]\thlabel{THM:LIMIMCOMP} Let $(w(n))_{n=0}^\infty$ be a sequence of positive real numbers with $w(n) \downarrow 0$, and $0<\delta<1$ be number. Then there exists a compact set $E=E(\delta)\subset \T$ with $m(E)>1-\delta$ such that any non-trivial $\mu \in M(E)$ satisfies with 
\[
\sup_{n\geq 0} \frac{\abs{\widehat{\mu}(n)}}{w(n)} = +\infty.
\]
\end{thm}
Note that $w$ cannot even be the unilateral Fourier majorant of a measure $\mu \in M(E)$.

\subsection{The Ivashev-Musatov Theorem}

This section presents the principal lemma underlying the proof of the Ivashev-Musatov Theorem. The lemma involves the construction of smooth functions that are both highly localized and possess uniformly small amplitudes, which may be of independent interest.

\begin{thm}[Principal IM-Lemma]\thlabel{THM:PRINCIPAL} Let $\{w(n)\}_{n=0}^\infty$ positive numbers satisfying the hypothesis $(i)-(ii)$ of the IM-Theorem. Then for any $0<\varepsilon<1$, there exists $\psi_\varepsilon \in C^\infty(\T)$, such that 
\begin{enumerate}
    \item[(I)] $\psi_\varepsilon \geq 0$ on $\T$,
    \item[(II)] $\int_{\T} \psi_\varepsilon dm =1$,
    \item[(III)] $m(\supp{\psi_\varepsilon dm} )\leq \varepsilon$,
    \item[(IV)] $\abs{\widehat{\psi_\varepsilon}(n)} \leq w(|n|)$, for all $n\neq 0$.
\end{enumerate}
\end{thm}
Note that the regularity condition $(ii)$ of $\{w(n)\}_n$ implies that there exists a number $M=M(w)>1$ such that $w(1) n^{-M} \leq w(n)$ for all $n\geq 1$. This observation implies that any $f\in C^\infty(\T)$ automatically satisfies the condition $(IV)$ for large enough $n$, and therefore the lemma is intrinsically a statement about the existence of smooth functions with uniformly small amplitudes of "lower" frequencies. With the Principal lemma at hand, one can easily obtain the Ivashev-Musatov Theorem as an immediate corollary, as we demonstrate below.

\begin{proof}[Proof of the Ivashev-Musatov Theorem] 
Consider the family of probability measures $d\mu_\varepsilon := \psi_\varepsilon dm$ appearing in the prinicap IM-Lemma, and apply the Helly selection Theorem in order to extract a weak-star cluster point $\mu$, which again is a probability $\mu$ on $\T$ by $(I)$ and $(II)$, and according to $(III)$ must be supported on a set of Lebesgue measure zero. At last, the uniform bound in $(IV)$ in conjunction with the weak-star convergence readily gives

\[
\abs{\widehat{\mu}(n)} \leq \limsup_{\varepsilon \to 0+} \abs{\widehat{\psi_\varepsilon}(n)} \leq w(|n|), \qquad  \forall n\neq 0.
\]

\end{proof}

We shall now comment on the conventional approach in proving the Ivashev-Musatov Theorem, which involves oscillatory integrals. First, observe that if $f\in C^\infty(\T)$ and note that $f_N(\zeta)= f(\zeta^N)$ where $N>0$ is a large integer, then $\supp{f_{N}}= \supp{f}$ remains small, since composition with $\zeta^N$ constitutes to spreading it out the support of $\supp{f}$ among $N$ uniformly spread arcs. Meanwhile, the frequencies of $\widehat{f_N}$ are now only supported on integer multiples of $N$, with 
\[
\widehat{f_{N}}(n) = 
\begin{cases} 
\widehat{f}(k) \qquad ,n=Nk \\
0 \,\, \, \, \, \, \, \qquad ,n\notin N \mathbb{Z}.
\end{cases}
\]
In other words, this procedure annihilates frequencies at a fixed arithmetic scale $N$, and shifts forward the amplitudes of $f$. However, this procedure does not shrink the lower order Fourier coefficients of $f$, thus a natural idea is to substitute $\zeta^N$ by a suitable unimodular function $U(\zeta)$, with the intent further decreasing the amplitudes of small frequencies. More specifically, let $U$ be a unimodular function: $\abs{U}=1$ on $\T$, and consider the composition $f_{U}(\zeta) := f(U(\zeta))$. If we further impose that $U$ has a holomorphic extension to $\{|z|<1\}$ i.e it is an inner function, the $\supp{\widehat{U}}\subseteq [0,\infty)$, hence means of multiplying $U$ by $\zeta$, we may wlog assume that $\int_{\T} U(\zeta) dm(\zeta) =0$. It follows that $f_U$ non-negative if $f$ is, $\widehat{f_U}(0)=\widehat{f}(0)$, and  then L\"owner's lemma on harmonic measures ensures that $m(\supp{f_U})= m(\supp{f})$. In case $U$ is not inner, then one needs to make sure that pre-images of $U$ do not distort the measure too much. Now, assuming such arrangements have been made, we note that since $f$ is smooth on $\T$, the Fourier coefficients of $f_U$ can be computed via
\[
\widehat{f_U}(n) = \sum_k \widehat{f}(k) \int_{\T} U^k(\zeta) \zeta^{-n} dm(\zeta)
\]
which indicated how the framework of oscillatory integrals and stationary phase arguments enter the picture. Now, to cook-up a unimodular function $U$ which "cancels" out the large amplitudes of the lower frequencies $f$, is arguably where the crucial difficulty lies, and cannot comprehensibly be outlined as the previous steps. This is roughly the initial approach taken by Ivashev-Musatov in \cite{ivavsev1958fourier}, and was further simplified by T. W. K\"orner in \cite{korner1977theorem1,korner1977theorem2}, but required stronger additional regularity assumptions on the Fourier majorant $\{w(n)\}_n$. Below, we shall follow a different route suggested by T. W. K\"orner in his further work in \cite{korner1986theorem3}, which involves carefully construction of smooth functions with huge oscillations designed to cancel out lower order frequencies.

The following lemma summarizes how the hypothesis $(ii)$ of \thref{THM:IM} will be used throughout. By means of substituting $w(n)$ with $\min(w(n), 1)$, we may from now on, and onward and without loss of generality, always assume that $w(n) \leq 1$ for all $n$.

\begin{lemma} \thlabel{LEM:DOUBLE}
Let $\{w(n)\}_{n\geq 0}$ be positive real numbers satisfying the hypothesis $(ii)$ of the IM-Theorem. Then there exists a number $M=M(w)>1$ such that:
\[
w(m) \leq \min \left(\frac{m}{n}, \frac{n}{m}\right)^{M} w(n), \qquad m,n\geq 0.
\]

\end{lemma} 
In particular, the above lemma implies that $\{w(n)\}_n$ has at most polynomial decay/growth.

\begin{proof}
Let $C(w)>1$ be the constant appearing in the hypothesis $(ii)$ of the IM-Theorem. Repeatedly using the assumption $(ii)$ on $\{w(n)\}_n$, we get
\[
w(m)\leq C(w)^{\log_2(n/m)} w(n) = 2^{\log_2(n/m)\log_2(C(w))} w(n) \leq \left(\frac{n}{m}\right)^M w(n), \qquad m, n \geq 1.
\]
with $M :=\log_2 C(w)$. The lemma now follows by switching the roles of $m,n\geq 1$ in the above argument.
\end{proof}

Our principal tool for proving the principal lemma will be based on this fine building-block by T. W. K\"orner.

\begin{lemma}[T. W. K\"orner, Lemma 2.2 in \cite{korner1986theorem3}]\thlabel{LEM:KORNER22} For any integer $M>1$, there exists a constant $C(M)>0$, such that the following statement holds. For any closed arc $I\subseteq \T$ and for any integer $N\geq 1$, there exists functions $f_{I,N} \in C^\infty(\T)$ satisfying the following properties:
\begin{enumerate}
    \item[(a.)] $\supp{f_{I,N}}\subseteq I$,
    \item[(b.)] $-10\leq f_{I,N}\leq 10$ on $\T$,
    \item[(c.)] $\int_{I} \abs{f_{I,N}} dm = m(I)$,
    \item[(d.)] the set $\{\zeta \in I: f_{I,N}(\zeta)=- 10 \}$ is a finite union of arcs, which occupy a substantial proportion of $I$: $m\left( \{\zeta \in I: f_{I,N}(\zeta)= - 10 \} \right) \geq (50)^{-1}m(I)$.
    \item[(e.)] The following Fourier estimate holds:
    \[
    \abs{\widehat{f_{I,N}}(n)} \leq C(M) m(I) N^{-1/2} \min \left((|n|m(I)/N)^M, (|n|m(I)/N)^{-M} \right), \qquad n\neq 0.
    \]
\end{enumerate}
\end{lemma}
This lemma is quite sophisticated and basically allowed T. W. K\"orner to circumvent the previous rather cumbersome approach involving oscillatory integrals. Its proof is not difficult, albeit technical and deviates from the general theme presented here. We do not include it. 

Our first step towards the principal lemma will require the following stepping-stone, which crucially makes use of the assumptions $(i)-(ii)$ of the IM-Theorem.


\begin{lemma}\thlabel{LEM:KORSTEP1}
Let $(w(n))_{n=0}^\infty$ be positive real numbers satisfying the hypothesis $(i)-(ii)$ in the IM-Theorem. Then for any $\varepsilon>0$ and any $0<\delta<1$, we can find functions $\psi_{\varepsilon,\delta} \in C^\infty(\T)$ satisfying the following properties:
\begin{enumerate}
    \item[(i)] $\psi_{\varepsilon,\delta} \geq 0$ on $\T$,
    \item[(ii)] $\int_{\T} \psi_{\varepsilon,\delta} dm =1$,
    \item[(iii)] for any arc $J\subset \T$ with $m(J)\geq \delta$, we have 
    \[
    m\left( \{\zeta \in J: \psi_{\varepsilon,\delta}(\zeta)=0 \}\right) \geq 100^{-1} m(J),
    \]
     and the set $\{\zeta \in J: \psi_{\varepsilon,\delta}(\zeta)=0 \}$ is a finite union of arcs,
    \item[(iv)] the Fourier estimate holds:
    \[
    \abs{\widehat{\psi_{\varepsilon,\delta}}(n)} \leq \varepsilon w(|n|), \qquad n\neq 0.
    \]
\end{enumerate}
\end{lemma}
\begin{proof}
\proofpart{1}{Use of hypothesis:}
We shall make crucial use of the hypothesis of $(w(n))_n$ in a very clever, but non-canonical way. According to \thref{LEM:DOUBLE}, we can pick a number $M=M(w)>1$, such that 
\[
w(2^m) \leq 2^{-M|m-l|} w(2^l), \qquad m,l\geq 0.
\]
Note that the hypothesis $(i)$ and $(ii)$ ensures that divergence of the series
\[
\sum_{\substack{n \geq 0: \\ 2^{-n} < 2^n w(2^n)^2 }} 2^n w(2^n)^2 =+ \infty.
\]
Therefore, we can for any fixed $0<\delta<1$, find positive integers $(d_k)_k$ and a partition $(I_k)_k$ of $\T$, consisting of closed arcs with mutually disjoint interiors, for which the following statement holds:
\[
\sup_{k} m(I_k) \leq \delta/4, \qquad 2^{-d_k} \leq m(I_k) \leq 2^{d_k} w(2^{d_k})^2.
\]
Now the lower bound on $m(I_k)$ allows us to produce positive integers $(N_k)_k$ with the properties
\[
2^{d_k} m(I_k) \leq N_k \leq 2^{d_k +1}m(I_k).
\]
Our specific choice of parameters will be clarified in the next step, when we shall verify the property $(iv)$.
\proofpart{2}{Constructing the function:}
Applying \thref{LEM:KORNER22}, we can exhibit corresponding functions $f_k := f_{I_k, N_k}$, with parameters $M+1$ and $(I_k, N_k)$ as described in the previous step, and form the function
\[
F(\zeta)= \sum_k f_{k}(\zeta), \qquad \zeta \in \T.
\]
Since $(I_k)_k$ have disjoint interior, we readily see from $(a.)$ of \thref{LEM:KORNER22} that $-10\leq F \leq 10$, while $(c.)$ implies that 
\[
\int_{\T} \abs{F} dm = \sum_k \int_{I_k} \abs{f_k} dm = \sum_k m(I_k) =1.
\]
We now claim that the desired functions are of the form:
\[
\psi_{\varepsilon}(\zeta) := \frac{10+ F(\zeta)}{\int_{\T}(10+F)dm}, \qquad \zeta \in \T.
\]
Indeed, the properties $(i)$ and $(ii)$ readily follow by construction. To verify $(iii)$, we fix an arbitrary arc $J\subset \T$ with $m(J)\geq \delta$. Since $m(I_k)\leq \delta/4$ for all $k$, we can find a finite sub-collection $(I'_j)_j$ of $(I_k)_k$ so that the union $\cup_j I'_j$ is contained in $J$, and $\sum_j m(I'_j) \geq m(J)/2$. Recall that the property $(c.)$ of \thref{LEM:KORNER22} implies that $\{ \psi_{\varepsilon}=0\} = \{ F=-10 \}$, is again a finite union of closed arcs, hence we get 
\begin{multline*}
m(\{\zeta \in \T: \psi_{\varepsilon}(\zeta)=0\}) \geq m(\{ \zeta \in \cup_{j}I'_j : F(\zeta)=-10 \}) \\
= \sum_j m(\{\zeta \in I'_j: f_j(\zeta)=-10 \}) \geq \frac{1}{50}\sum_j m(I'_j) \geq \frac{1}{100} m(J).
\end{multline*}
This proves the claim in $(iii)$. 
\proofpart{3}{The Fourier estimate:}
It only remains to prove $(iv)$. To this end, we primarily note that since $F$ is real-valued and the estimate
\[
\int_{\T}(10+F)dm \geq 10 - \int_{\T}\abs{F} dm \geq 9,
\]
holds, we see that $\abs{\widehat{\psi_{\varepsilon}}(n)} \leq \abs{\widehat{F}(n)}$ for $n\neq 0$, hence it suffices to prove the estimate for the positive Fourier coefficients of $F$. First, we note that the carefully chosen parameters $(N_j, I_j)$ correspond to the following Fourier estimate of each $f_k$: 
\[
\abs{\widehat{f_k}(n)} \leq  C(M) w(2^{d_k}) \min\left( (n2^{-d_k})^{-(M+1)}, (n2^{-d_k})^{(M+1)} \right), \qquad n\neq 0.
\]
Now if $2^{l}\leq n\leq 2^{l+1}$ some integer $l \geq 0$, then using \thref{LEM:DOUBLE} and our initial choice of $M>1$, we get 
\[
\abs{\widehat{f_k}(n)} \leq  C(M) w(2^l) 2^{-|l-d_k|} \leq C(M) w(n) 2^{-|l-d_k|}
\]
Using this and the fact that each $f_k$ have disjoint support, we obtain
\[
\abs{\widehat{F}(n)} \leq \sum_k \abs{\widehat{f_k}(n)} \leq C(M) w(n) \sum_k 2^{-|l-d_k|} \leq C(M) w(n)\sum_k 2^{-|k|} \leq C(M) w(|n|), \qquad n>0.
\]

\proofpart{3}{Re-scaling the majorant:}
Note that we have only managed to derive a big-O version of the desired estimate in $(iv)$, but there is one magical last step, which allows us to obtain the little-o statement. To this end, we simply repeat the same argument with the re-scaled majorant $(\varepsilon \cdot w(n))_n$, for $\varepsilon>0$ small enough. It is easy to see that that all other estimates remain intact. This completes the proof of the lemma.
\end{proof}

With this lemma at hand, we now turn to our main task.

\begin{proof}[Proof of the Principal IM-Lemma]
Arguing by means of induction, our goal is to construct a sequence of smooth functions $\{f_n\}_{n=0}^\infty$ satisfying the following properties:
\begin{enumerate}
    \item[$(1)_n$] $f_n\geq 0$ on $\T$,
    \item[$(2)_n$] $m\left( \supp{f_n dm} \right) \leq (1-100^{-2})^n$, and $\supp{f_{n+1}dm}\subseteq \supp{f_n dm}$,
    \item[$(3)_n$] $1+2^{-n} \leq \widehat{f_n}(0)$,
    \item[$(4)_n$] $\abs{\widehat{f_n}(k)}\leq (1-2^{-n})w(|k|)$, for all $k\neq 0$.
\end{enumerate}
Before carrying out the inductive procedure, we shall highlight a few simple remarks that will be utilized in this process. \\

Take $f_0=2$, and suppose $f_1, \dots, f_n$ have already constructed so that each $f_j$ satisfies $(1)_j-(4)_j$ for $1\leq j \leq n$. Fix $\varepsilon,\delta>0$ to be determined momentarily and set $f_{n+1} := f_n \cdot \psi_{\varepsilon, \delta}$, where $\psi_{\varepsilon, \delta}$ as in \thref{LEM:KORSTEP1}. Now clearly $(1)_{n+1}$ holds, and $\supp{f_{n+1}dm}\subseteq \supp{f_n dm}$. In order to verify $(2)_n$, let set $E_n := \supp{f_n dm}$ and let $(J_k)_k$ be a finite collection of open arcs covering $E_n$:
\[
E_n \subseteq \cup_k J_k, \qquad \frac{101}{100}m(E_n) \geq \sum_j m(J_k).
\]
Now pick $0< \delta(n) < \min_k m(J_k)$ and note that if follows from property $(iii)$ of \thref{LEM:KORSTEP1} that 
\[
m(E_{n+1}) \leq \sum_k m(\left\{ J_k:\psi_{\varepsilon, \delta}>0 \right\}) \leq \left(1-\frac{1}{100}\right) \sum_k m(J_k)  \leq (1-100^{-2}) m(E_n)
\]
where $E_{n+1} := \supp{f_{n+1}dm}$. It remains only to verify $(3)_n$ and $(4)_n$, and it is here we shall make our specific choice of $\varepsilon= \varepsilon(n)>0$. Note that since all functions involved are real-valued, we only need to consider non-negative Fourier coefficients. First, we note that since $\widehat{\psi_{\varepsilon,\delta}}(0)=1$, we have for any integer $k\geq 0$:
\[
\widehat{f_{n+1}}(k) - \widehat{f_n}(k) = \sum_{j\neq k} \widehat{f_n}(j) \widehat{\psi_{\varepsilon,\delta}}(k-j)= \left( \sum_{|j|\leq k/2} + \sum_{\substack{|j|>k/2 \\ j\neq k} }\right)\widehat{f_n}(j) \widehat{\psi_{\varepsilon,\delta}}(k-j).
\]
Now in order to estimate the first sum, we simply argue as
\[
\abs{\sum_{|j|\leq k/2} \widehat{f_n}(j) \widehat{\psi_{\varepsilon,\eta}}(k-j)} \leq \varepsilon \sum_{|j|\leq k/2} \abs{\widehat{f_n}(j)} w(|k-j|)\leq C(w)\varepsilon w(|k|)\sum_j \abs{\widehat{f_n}(j)}.
\]
The second sum is simpler, and we instead estimate as follows:
\begin{multline}\label{EQ:SIGMA2} 
\abs{\sum_{\substack{|j|>k/2 \\ j\neq k} }\widehat{f_n}(j) \widehat{\psi_{\varepsilon,\delta}}(k-j)} \leq \varepsilon \sum_{|j|>k/2} \abs{\widehat{f_n}(k)}  \leq \varepsilon C(n,M) \sum_{|j|>k/2} |j|^{-M-1} \\ \leq  \varepsilon C(n,M) k^{-M} \leq \varepsilon w(|k|) C(n,M).
\end{multline}
Combining these estimates, we get 
\[
\abs{\widehat{f_{n+1}}(k) - \widehat{f_n}(k)} \leq \varepsilon w(|k|) \left( C(w) \norm{\widehat{f}}_{\ell^1} + c(n,N) \right).
\]
Using the induction hypothesis $(4)_n$, we get
\[
\abs{\widehat{f_{n+1}}(k)} \leq (1-2^{-n}) w(|k|) + \varepsilon w(|k|) \left( C(w) \norm{\widehat{f}}_{\ell^1} + c(n,N) \right) \leq (1-2^{-n-1}) w(|k|), \qquad k\neq 0,
\]
provided $\varepsilon=\varepsilon(n)>0$ is chosen sufficiently small. In a similar way, we can also ensure that $(3)_{n+1}$ also holds, hence $f_{n+1}$ satisfies the hypothesis $(1)_{n+1}-(4)_{n+1}$, and the induction step is complete.

\end{proof}


\subsection{The indispensable regularity condition}

Here prove that the condition $(ii)$ in the IM-Theorem cannot simply be removed. To this end, we shall need the following manifestation of the uncertainty principle, which asserts that if a (pseudo)measure $\mu$ attains a substantial portion of its mass from Fourier coefficients up to order $N$, and its coefficients tremendously drop in value, the the support of $\mu$ is bound to spread around $\T$.

\begin{lemma}[K\"orner-Meyer] For any $\gamma, \delta, N>0$, there exists $\varepsilon= \varepsilon(\gamma, \delta, N)>0$, such that the following statement holds: whenever $S$ is a distribution on $\T$ with 
\[
(i)\, \sum_{|n|\leq N} \abs{\widehat{S}(n)}^2 \geq \gamma, \qquad (ii) \, \sup_{|n|>N} \abs{\widehat{S}(n)} \leq \varepsilon,
\]
then $\supp{S}$ meets every $\delta$-neighborhood of $\T$, i.e, it is $\delta$-dense.
    
\end{lemma}
\begin{proof} The proof is neat argument by contradiction, which K\"orner attributes to Y. Meyer. Suppose the conclusion fails for some triple $(\gamma,\delta, N)$, hence we can find distributions $(S_k)_k$ on $\T$ and numbers $\varepsilon_k \downarrow 0$, where each $S_k$ satisfies $(i)-(ii)$ with $\varepsilon_k$, but $\sup_{\zeta \in \T} \dist{\supp{S_k}}{\zeta}\geq \delta$. Considering a finite partition $\{I_j\}_j$ of $\T$ into arcs of length $\delta/10$, say, we see by the pigeon hole principle that there exists at least one arc $I_0$, for which 
\[
\supp{S_{k}} \cap I_0 = \emptyset,
\]
for infinitely many $k$'s. Therefore, by means of passing to such a subsequence, we may assume that 
\[ 
\sum_{|n|\leq N} \abs{\widehat{S_k}(n)}^2 \geq \gamma, \qquad \sup_{|n|>N} \abs{\widehat{S_{k}}(n)} \leq \varepsilon_{k}, \qquad \supp{S_k} \cap I_0 = \emptyset. 
\]
Now since $(\widehat{S}_(n))_n$ with $k=1,2,$ defines a uniformly bounded sequence in $\ell^\infty$, we can apply Helly's selection Theorem in order to extract a weak-star cluster point in $\ell^\infty$, which gives rise to a distributions $S$ on $\T$ with these pre-scribed Fourier coefficients. Since $\varepsilon_k \to 0$, and $N>0$ is fixed, we obtain 
\[
\sum_{|n|\leq N} \abs{\widehat{S}(n)}^2 \geq \gamma, \qquad \widehat{S}(n)= 0, \qquad |n|>N, \qquad \supp{S} \cap I_0 = \emptyset.
\]
But this forces $S$ to be a trigonometric polynomials which vanishes on an arc $I_0$. Recall the simple observation that if $S$ is a trigonometric polynomial of order $N$, then $Q(z) := z^N S(z)$ extends to an analytic polynomial on $z \in \C$, which vanishes at the same points as $S$ does. We therefore conclude that $S$ vanishes identically, which immediately contradicts that a portion of its $\ell^2$-mass exceeded $\gamma>0$.
\end{proof}

We are now ready to complete the proof of K\"orner's Theorem. 

\begin{proof}[Proof of \thref{THM:KÖRsuff}]
We shall construct the intended sequence $\{w(n)\}_n$ by means of induction. First, we set $N_1=\varepsilon_1=1$, and assume that positive integers $N_1 < N_2 < \dots < N_k$, and positive numbers $\varepsilon_1>\varepsilon_2 > \dots > \varepsilon_k>0$ have already been constructed. Now pick $N_{k+1}>N_k$, such that 
\[
\left(N_{k+1}-N_k \right) \varepsilon^2_{k} 2^{-2k} \geq 1.
\]
Invoking the K\"orner-Meyer Lemma, we can find a number $0<\varepsilon_{k+1}< \varepsilon_k$ such that the following statement holds: whenever $S$ is a distribution on $\T$ with 
\[
(i)_k \, \sum_{|n|\leq N_{k+1}} \abs{\widehat{S}(n)}^2 \geq 2^{-k}, \qquad (ii)_k \, \sup_{|n|>N_{k+1}} \abs{\widehat{S}(n)} \leq \varepsilon_{k+1},
\]
then $\sup_{\zeta \in \T} \dist{\zeta}{\supp{S}} \leq 2^{-k}$. We are now ready to define a sequence of positive numbers $\{w(n)\}_n$ as follows: For each $k=0,1,2,\dots$ set $w(N_k) = \varepsilon_k 2^{-k}$, and interpolate the intermediate values of $w(n)$ linear-wise:
\[
w(n) := \frac{\varepsilon_k}{2^{k}} \frac{N_{k+1}-n}{N_{k+1}-N_k}  + \frac{\varepsilon_{k+1}}{2^{k+1}}   \frac{n-N_{k}}{N_{k+1}-N_k}, \qquad N_k \leq n < N_{k+1}, \qquad k=0,1,2, \dots.
\]
Clearly, $\{w(n)\}_n$ is decreasing with $w(n) \downarrow 0$. Furthermore, it follows from the construction of $(N_k)_k$ that 
\[
\sum_{N_k\leq n \leq (N_{k+1}+N_k)/2} w(n)^2 \geq \frac{(N_{k+1}-N_k)}{2}\varepsilon^2_k 2^{-2k} \geq \frac{1}{2}, \qquad k=0,1,2 \dots,
\]
and hence $\sum_n w(n)^2 = + \infty$. Now let $S$ be an arbitrary non-zero distribution on $\T$ with 
\[
\abs{\widehat{S}(n)} \leq Cw(|n|), \qquad n\neq 0.
\]
We claim that both $(i)_k$ and $(ii)_k$ must hold for all sufficiently large $k$. Now the first statement is trivial, while the second property follows from the monotonicity of $\{w(n)\}_n$:
\[
\sup_{|n|>N_{k+1}} \abs{\widehat{S}(n)} \leq C \sup_{|n|>N_{k+1}} w(|n|) \leq C w(N_{k+1}) = C2^{-k-1}\varepsilon_{k+1}\leq \varepsilon_{k+1},
\]
whenever $k$ is sufficient large. Therefore, $\sup_{\zeta \in \T} \dist{\zeta}{\supp{S}} \leq 2^{-k}$ holds for all sufficiently large $k$, and since the support is a closed set, we conclude that $\supp{S}=\T$.

\end{proof}

Below, we shall also mention a one-sided version of the K\"orner's Theorem, which can easily be derived as a consequence of the Cartwright--Levinson Theorem, and goes as follows:
\begin{prop}[Unilateral K\"orner-type Theorem] There exists a decreasing sequence $\{\Phi(n)\}_{n=0}^\infty$ of positive numbers satisfying 
\[
\sum_n \Phi(n)^2 = \infty,
\]
but any non-trivial measure $\mu \in M(\T)$ satisfying
\[
\abs{\widehat{\mu}(n)} \leq \Phi(n), \qquad n=1,2,3,\dots
\]
must have full support in $\T$.
\end{prop} 

\begin{proof}
    Note that in view of Cartwright--Levinson Theorem, it suffices to construct a decreasing sequence of positive numbers $(\Phi(n))_n$ with
    \[
    \sum_{n\geq 1} \frac{\log \Phi(n)}{n^2} =- \infty, \qquad \sum_{n\geq 1} \Phi(n)^2 = +\infty.
    \]
   To this end, let $(N_j)_j$ be positive integers and $(\varepsilon_j)_j$ positive numbers shrinking to zero, both to be specified momentarily. Now define the sequence of positive real numbers $(\Psi(n))_n$ by
   \[
   \Psi(n) := \varepsilon_j, \qquad N_j \leq n < N_{j+1}, \qquad j=1,2, 3, \dots
   \]
    Note that
    \[
    \sum_n \Psi(n)^2 \asymp \sum_n \varepsilon^2_j (N_{j+1}-N_j), \qquad \sum_j \frac{\log \Psi(n)}{n^2} \asymp  \sum_j \log \varepsilon_j \left( \frac{1}{N_j} - \frac{1}{N_{j+1}} \right) 
    \]
    For instance, if we pick
    \[
    N_{j+1}=2^{N_j}, \qquad  \varepsilon_j = 2^{-N_j}, \qquad j=1,2,3,\dots
    \]
    then it easily follows that both the above series diverge. A desirable sequence $(\Phi(n))_n$ is obtain by only slightly modifying $(\Psi(n))_n$ so that it becomes strictly decreasing on each blocks $[N_j, N_{j+1})$, but slowly enough that it is almost constant there.
    
    A different but conceptually similar proof is outlined in \cite{havinbook}.
\end{proof}

\subsection{Supporting sets for Fourier majorants}

Here we shall establish \thref{THM:LIMIMCOMP}, which asserts that for any positive numbers $w=(w(n))_{n=0}^\infty$  which shrink to zero, one can construct wild sets $E$ which do not support measures $\mu$ with Fourier majorant $w$. The strengths of the result lies in the facts that $w(n)\downarrow 0$ as slow as possible, and that it is unilateral.

The proof is based on so-called simultaneous approximation argument, which we summarize as a lemma.

\begin{lemma} \thlabel{LEM:SAl1w} Let $w=(w(n))_{n=0}^\infty$ be positive real numbers with $w(n) \downarrow 0$, and $0<\delta<1$. Then there exists a compact set $E\subset \T$ with $m(E)>1-\delta$, and analytic functions $(f_j)_j$ in $\D$ which are smooth in $\cD$ such that the following holds:
\[
\sum_{n\geq 0} \abs{\widehat{f}_j(n)} w(n) \to 0, \qquad \sup_{E} \abs{f_j-1} \to 0.
\]
    
\end{lemma}
\begin{proof}
Let $I_\delta \subset \T$ be a closed arc of length $1-\delta$. For fixed $0<\gamma <1$, pick a smooth real-valued function $\psi_{\gamma, \delta}$ on $\T$ which satisfies the properties:
\[
\int_{\T} \psi_{\gamma, \delta} dm =0, \qquad \psi_{\gamma, \delta} \equiv \log \gamma, \qquad \text{on} \, \, \, I_\delta.
\]
Form the outer function $F_{\gamma, \delta}$ in $\D$ defined by 
\[
F_{\gamma, \delta}(z) = 1-\exp \left( \int_{\T} \frac{\zeta +z}{\zeta -z} \psi_{\gamma, \delta}(\zeta) dm(\zeta) \right) , \qquad z \in \D.
\]
These functions are bounded holomorphic in $\D$, extend smoothly up to $\T$, and satisfy: 
\begin{enumerate}
    \item[(i)] $F_{\gamma, \delta}(0)=0$.
    \item[(ii)] $\abs{F_{\gamma, \delta}(\zeta)-1} = \gamma$, \qquad $\zeta \in I_{\delta}$.

\end{enumerate}

Let $(N_j)_j$ be an increasing sequence of positive integers with $N_j \to \infty$ to be chosen in a moment, and consider the functions
\[
f_{j,\gamma, \delta}(z) := F_{\gamma,\delta}(z^{N_j}), \qquad z\in \D.
\]
Let $E_{j,\delta}:=\{\zeta\in \T: \zeta^{N_j} \in I_\delta \}$ denote the pre-image of $I_{\delta}$ under the monomial $z^{N_j}$, and note that $m(E_{j,\delta}) = m(I_{\delta})=1-\delta$. It follows from (ii) that 
\begin{equation}\label{EQ:fnE}
\abs{f_{j,\gamma, \delta}(\zeta)-1}= \gamma, \qquad \zeta \in E_{j,\delta}.
\end{equation}
Recalling that $f_{n,\gamma, \delta}(0)=0$, we have the following estimate:
\begin{equation}\label{EQ:l1w}
\sum_{n\geq0} \abs{\widehat{f}_{j,\gamma, \delta}(n)}w(n) = \sum_{n\geq 1} \abs{\widehat{F}_{\gamma, \delta}(n)}w(N_j n) \leq w(N_j) \sum_{n\geq 1} \abs{\widehat{F}_{\gamma, \delta}(n)},
\end{equation}
where the monotonicity assumption on $(w(n))_n$ was utilized in the last step. 

Now, pick $(\delta_j)_j$ of positive numbers with $\sum_j \delta_j \leq \delta$ and let $(\gamma_j)_j$ be any sequence of positive numbers tending to zero. Set $E := \cap_j E_{j, \delta_j}$, and note that the assumption on the $\delta_j$'s imply that $m(E) \geq 1-\delta$. Note that we can choose the positive integers $(N_j)_j$ to be large enough so that 
\[
w(N_j) \sum_{n\geq 1} \abs{\widehat{F}_{\gamma_j, \delta_j}(n)} \leq \gamma_j.
\]
Set $f_j := f_{j,\gamma_j, \delta_j}$. It follows from \eqref{EQ:fnE} that $f_j \to 1$ uniformly on $E$, while \eqref{EQ:l1w} implies that 
\[
\sum_{n\geq0} \abs{\widehat{f}_j(n)}w(n) \leq \gamma_j \to 0.
\]
\end{proof}

We are now ready to prove the main result in this subsection.

\begin{proof}[Proof of \thref{THM:LIMIMCOMP}] Let $E$ and $(f_j)_j$ be as in \thref{LEM:SAl1w}, and pick an arbitrary $\mu \in M(E)$ with 
\[
C(\mu,w):=\sup_{n\geq 0} \frac{\abs{\widehat{\mu}(n)}}{w(n)} < \infty.
\]
By means of substituting $d\mu(\zeta)$ with $\zeta^N d\mu(\zeta)$ for some integer $N\geq 0$, we may without loss of generality assume that $\widehat{\mu}(0)\neq 0$. However, the properties of $(f_j)_j$ yields
\begin{multline*}
|\widehat{\mu}(0)| = \lim_j \Big \lvert \int_{\T} \conj{f_j} d\mu \Big \rvert   = \lim_j \Big \lvert \sum_{n\geq 0} \widehat{\mu}(n) \conj{\widehat{f}_j(n)}\Big \rvert \leq C(\mu,w) \limsup_j \sum_{n\geq 0} \abs{\widehat{f}_j(n)} w(n) =0.
\end{multline*}
    
\end{proof}

\subsection{Further results}
Let $w=(w(n))_{n=0}^\infty$ be a bounded sequence of positive real numbers and denote by $\ell^2(w)$ the Hilbert space of distributions $S$ on $\T$ satisfying
\[
\sum_{n} \abs{\widehat{S}(n)}^2 w(|n|) < \infty.
\]
Recently, the author obtained the following summable analogue of the Ivashev--Musatov Theorem in \cite{limani2025generic2}.

\begin{thm}[Limani] Let $w=\{w(n)\}_{n=0}^\infty$ be positive real numbers with
\begin{enumerate}
    \item[(i)] $\liminf_n w(n)=0$,
    \item[(ii)] either $w$ is non-increasing, or there exists a constant $C>1$ such that for any integer $n\geq 1$:
    \[
    C^{-1} w(n) \leq w(k) \leq C w(n), \qquad n\leq k \leq 2n.
    \]
\end{enumerate}
Then there exists a probability measure $\sigma$ supported on a set of Lebesgue measure zero with $\{\widehat{\sigma}(n)\}_n \in \ell^2(w)$.
\end{thm}

Under additional regularity condition on $w$, this theorem can essentially be proved using abstract potential theory developed by L. Carleson in \cite{carleson1967selected}, hence the novelty is the minimalistic hypothesis. In a similar vain, it was also proved in \cite{limani2025generic} that the condition $(i)$ alone, is far from enough.

T. W. K\"orner also proved that if one wishes to merely construct a measure $\mu$ which is singular wrt $dm$, then the condition $(ii)$ can substantially relaxed, \cite{korner1987uniqueness}. K\"orner attributes the following result to J-P. Kahane, see \cite[p. 223]{kahane1963ensembles}

\begin{thm}[Kahane-K\"orner] Let $(w(n))_{n=0}^\infty$ be a sequence of positive numbers satisfying the condition
\[
\sum_{n=1}^\infty 2^n \min \left\{ w(k)^2: 2^n\leq k < 2^{n+1} \right\} = +\infty. 
\]
Then there exists a singular probability measure $\mu$ on $\T$ which satisfies $\abs{\widehat{\mu}(n)}\leq w(n)$, for all integers $n\neq 0$.
\end{thm}

In particular, it is worth noting that the K\"orner-Meyer ensures that there exists a decreasing sequence of positive numbers $w=(w(n))_n$ with 
\[
\sum_n w(n)^2 = \infty,
\]
such that any non-trivial distribution in $\T$ with Fourier majorant $w$ has full support. However, the Kahane-K\"orner Theorem ensures that we can still always exhibit a singular probability measure with Fourier majorant $w$. 

Another beautiful result due to T. W. K\"orner is a topological analogue of the Ivashev--Musatov Theorem, where the notion "small" in the sense of zero Lebesgue measure is substituted by small in topological sense.
\begin{thm}[K\"orner] Let $(w(n))_{n=0}^\infty$ be a sequence of positive numbers which satisfy the following properties:
\begin{enumerate}
    \item[(i)] $ \sum_n w(n) = + \infty$,
    \item[(ii)] there exists a constant $C>1$ such that for any integer $n\geq 1$:
    \[
    C^{-1} w(n) \leq w(k) \leq C w(n), \qquad n\leq k \leq 2n.
    \]
\end{enumerate}
There there exists a positive function $f\in L^\infty(\T)$ such that $\supp{fdm}$ has no interior, but 
\[
\abs{\widehat{f}(n)} \leq w(|n|), \qquad n=\pm1, \pm 2, \dots.
\]
\end{thm}

Note again that in view of the K\"orner-Meyer Theorem, the additional regularity condition in $(ii)$ cannot simple be removed. The topological Ivashev-Musatov Theorem essentially touches on the problem of finding the threshold of Fourier majorants $(w(n))_n$, for which not every function $f$ with $\abs{\widehat{f}(n)}\leq w(|n|)$ is continuous on $\T$. Conversely, this begs the question of how wild the Fourier majorants of continuous functions can be. It turns out that this question can be settled in a surprisingly sharp way.

\begin{thm}[de Leeuw--Kahane--Katznelson]\thlabel{THM:LEEUWKAHKATZ} Let $(a_n)_n$ be a sequence of complex numbers with $\sum_n |a_n|^2 < \infty$. Then there exists $f\in C(\T)$ such that 
\[
\abs{\widehat{f}(n)} \geq |a_n|, \qquad n=0,\pm 1, \pm 2, \dots
\]
\end{thm}
For a short beautiful proof of this result, which differs from its original one in \cite{deleeuw1977surles}, and involves probabilistic techniques, we refer the reader to Y. Katznelson in \cite{katznelson2004introduction}.

\section{The Beurling--Malliavin Theorem(s)}

Here we consider a line of work originating in the profound contributions of A. Beurling and P. Malliavin, which may justifiably be regarded as among the deepest developments in twentieth-century harmonic analysis. The theory is both subtle and far-reaching, and its influence is reflected, for instance, in the extensive treatment given by P. Koosis in his two volumes devoted to the subject, \cite{koosis1998logarithmic,MR1195788}.
In these notes, we shall primarily focus on the so-called Beurling--Malliavin multiplier theorem. As is customary, and for brevity, we will refer to it simply as the BM-1 theorem.

\subsection{The Beurling-Malliavin multiplier Theorem}
Let $w$ be a bounded weight function on $\R$ and $a>0$ be a real number. We say that the weight $w$ is $a$-admissible if there exists a non-trivial element $f \in L^1(\R)$ such that 
\begin{equation}\label{EQ:BMCOND}
(a.) \, \text{diam} \, \supp{\widehat{f}} < 2a, \qquad (b.) \, \abs{f(x)}\leq w(x), \qquad x \in \R.
\end{equation}
In other words, $a$-admissible weights are majorants of integrable functions in $\R$ with Fourier spectrum contained on an interval of length $2a$. If $w$ is $a$-admissible for some $a>0$, we simple refer to it as a \emph{BM-majorant}, short for Beurling--Malliavin majorant. The problem is to find "optimal" conditions of $w$ so that it becomes a BM-majorant. The first observation is that the logarithmic integrability
\[
\int_{\R} \frac{\log w(x)}{1+x^2}dx >- \infty
\]
is necessary in order for $w$ to be a BM-majorant. Indeed, this essentially a consequence of the F. and M. Riesz in the upper-half plane, which characterizes majorants of functions with semi-bounded Fourier spectra.

\begin{prop}[F. and M. Riesz Theorem] There exists a non-trivial $f \in L^1(\R)$ with semi-bounded spectrum satisfying $\abs{f}\leq w$ if and only if $\int_{\R} \log w dP >- \infty$.
\end{prop}
\begin{proof}
First, if $w\in L^\infty(\R)$ and $\log w \in L^1(dP)$, then we can form the outer function $f \in H^\infty(\C_+)$ with $\abs{f}=w$ a.e on $\R$. Conversely, if there exists a non-trivial $f\in L^1(\R)$ with semi-bounded spectrum and such that $\abs{f} \leq w$ a.e on $\R$. Then there exists exists a translate of $f$ which agree with boundary values of an $H^1(\R)$-function a.e on $\R$. But then by the F. and M. Riesz Theorem, we have
\[
\int_{\R} \log w dP \geq \int_{\R} \log |f| dP >- \infty.
\]

\end{proof}
However, it turns out that the log-integrability of $w$ is far from sufficient in order for $w$ to be a BM-majorant. Indeed, the Paley-Wiener Theorem implies that if the Fourier transform $\widehat{f}$ is compactly supported, then $f$ defines an entire analytic function in $\C$, thus any BM-majorant $w$ cannot vanish along a sequence of real numbers which accumulate, nor can it have a zero of infinite order. 

On the other hand, it is not difficult to construct admissible majorants for measures with semi-bounded spectrum that have these properties. Roughly speaking, these observations indicate that BM-majorants cannot be zero too often, but the main result essentially asserts that the principal enemy is the oscillation of $w$. One of the cleanest forms (i.e involving a quantitative condition) of the Beurling--Malliavin multiplier Theorem can phrased as follows.

\begin{thm}[Strong BM1] \thlabel{THM:BM1FULL} Let $w$ be a bounded weight on $\R$ with $\int_{\R} \log w dP >- \infty$ and $\log w $ is uniformly continuous. Then $w$ is $a$-admissible for any $a>0$.
\end{thm}
In the work of V. Havin, J. Mashreghi, F. Nazarov in \cite{mashreghi2006beurling} the BM-theorem is phrased under the weaker and more quantitative assumption that 
\[
\sup_{|x-y|\leq 1} \frac{w(x)}{w(y)} < \infty.
\]
This roughly asserts that $\log w$ has uniformly bounded oscillation on intervals of fixed length. In particular, the condition is satisfied if $\log w$ is uniformly continuous in $\R$. More commonly, the BM-Theorem is stated under the hypothesis that $\log w$ is Lipschitz continuous on $\R$. However, \thref{THM:BM1FULL} can readily be deduced from that assumption, by means of smoothing the majorant with an appropriate approximate of the identity. 

It is well-known that the BM1 Theorems are close to sharp in several respects, and it is therefore not completely surprising that their proofs are quite difficult and long, despite both joint and individual efforts of several mathematicians, see \cite{havinbook,mashreghi2006beurling,koosis1998logarithmic} and references therein. In this section, we shall prove certain milder forms of the BM1-Theorem, under other slightly different regularity conditions on $w$, and briefly explain how the strong BM1 Theorem can be attained from its milder counterparts.

 
\subsection{Non-admissible BM-majorants and a mild BM1 Theorem}
As mentioned in the previous subsection, the Paley-Wiener Theorem infers some additional restrictions for bounded and log-integrable weights to be admissible BM-majorants. Our main objective in this subsection is to prove a mild version of the BM1-Theorem, under the hypothesis that we remove the ability for $w$ to oscillate. In order to increase the readers appetite for these results, we shall strategically first examine two instructive scenarios that illustrate how certain sophisticated oscillatory behavior of $w$ can prevent it from being a $BM$-majorant. Our first example involves a weight which decreases way too fast at infinity. 

\begin{example}[Rapid unilateral decay at infinity] Fix a positive integer $n\geq 1$ and let $(t^n_j)_{j\geq 1}$ be an enumeration of the dyadic points $\pi k 2^{-n}$ for integers $k\geq n2^n$ located on the positive real-axis $(0,\infty)$. Pick any increasing positive function $\theta$ on $[0,\infty)$ with 
\[
\sum_{j=1}^{\infty} \frac{\theta(j)}{j^2} = +\infty.
\]
Define a bounded weight $w$ on $\R$ which interpolates the values
\[
w(t^n_j) = e^{-\theta(j)}, \qquad j=1,2,3,\dots
\]
and such that the graph of $w$ looks like a zig-zag function with spikes at the points $(t^n_j)$, in such a way that the we maintain finite logarithmic integral:
\[
\int_{\R} \log w \, dP >- \infty.
\]
We leave it to the reader so verify that such a construction is possible, recalling that we have the freedom to even further specify $\theta$. Now suppose $f \in L^1(\R)$ with the properties that $\supp{\widehat{f}}\subseteq [-2^n, 2^n]$ and such that $\abs{f} \leq W$ a.e on $\R$. Then $\widehat{f}$ is a continuous compactly supported function on $\R$, thus $f$ is also continuous on $\R$. In particular, this implies that $\mu \in M(\R)$ defined by $d\mu = \widehat{f}dx$ satisfies 
\[
 \text{diam} \, \supp{\mu} < 2^{n+1}, \qquad \widehat{\mu}(\pi t^n_j) = \mathcal{O}\left(e^{-\theta(j)}\right), \qquad j \to \infty.
\]
By the Cartwright--Levinson Theorem, we must have that $\mu \equiv 0$, hence $w$ cannot be $2^n$-admissible. However, nothing forbids $w$ from being $a$-admissible for some $a> 2^n$. This example illustrates that $BM$-majorants may suddenly fail to be $a$-admissible at smaller scales of $a>0$.
\end{example}

Our next example involves the construction of bounded continuous weights with wild oscillations, yet are log-integrable wrt $dP$, but 
are non-admissible at any scale.

\begin{example}[Wild oscillations] We shall construct a positive bounded continuous weight $w$ on $\R$ with $\log w \in L^1(dP)$, which is not a BM-majorant. We primarily note that if $f\in L^1(\R)$ with $\supp{\widehat{f}}\subseteq[-a,a]$, then $f' \in L^\infty(\R)$ with
\[
\abs{f'(x)} \leq \int_{-a}^{a} |\xi| \abs{\widehat{f}(\xi)} d\xi \leq a \norm{\widehat{f}}_{L^1} =: L(a) \qquad x \in \R,
\]
hence $f$ is uniformly Lipschitz on $\R$. Furthermore, if we also assume that $\abs{f}\leq w$ on $\R$, we see that on each interval $[x_k,x_{k+1}] \subset(0, \infty)$ we have
\[
\abs{f(x)} \leq \abs{f(x_k)} + L(a)(x_{k+1}-x_k) \leq w(x_k) + L(a) \ell_k, \qquad x\in [x_k, x_{k+1}]
\]
where $\ell_k := x_{k+1}-x_k>0$. With observation at hand, we get
\[
\int_{x_k}^{x_{k+1}} \frac{\log \abs{f(x)}}{x^2} dx \leq \log \left( w(x_k) + L(a) \ell_k \right) \left( \frac{1}{x_k}- \frac{1}{x_{k+1}} \right) \leq \log \left( w(x_k) + L(a) \ell_k \right) \frac{\ell_k}{x^2_k}.
\]
Expressing these requirements in terms of $(\ell_k)_k$, we see that they point towards the strategy of finding a positive numbers $(\ell_k)_k$ satisfying
\[
0<\ell_k \leq 1, \qquad \sum_k \ell_k = +\infty, \qquad \sum_k \frac{\ell_k \log \ell_k}{\left( \sum_{j=1}^{k} \ell_j \right)^2 } = -\infty.
\]
For instance, one easily checks that $\ell_k = 1/k$ does the job. We now define an even positive bound continuous weight $w$ on $\R$ with $w(x_k)= \ell_k$ and modify the rest of its values in $\R$ so that $\int_{\R} \log w dP>-\infty$. One should think $w$ being very oscillatory, whose graph roughly illustrates a landscape of peaks and hills. It follows from the above observations that any $f\in L^1(\R)$ with $\supp{\widehat{f}}\subseteq [-a,a]$ and $\abs{f}\leq w$ on $\R$ fails to be logarithmically integrable wrt $dP$, hence $f\equiv 0$ by the F. and M. Riesz Theorem. Therefore such oscillatory $w$ can never be BM-majorants, even though they are positive bounded continuous weights, logarithmically integrable wrt $dP$.

\end{example}

With the above examples in mind, we now prove that log-integrability in conjunction with complete lack of oscillation ensures admissible BM-majorants.

\begin{thm}[Mild BM1]\thlabel{THM:BM1Mild} Let $w$ be an even function on $\R$ which is decreasing on $[0,\infty)$, and satisfies $\int_{\R} \log w dP >-\infty$. Then $w$ is $a$-admissible for any $a>0$.
\end{thm}
\begin{proof}
Set $W(x)= cw(2x) e^{-|2x|^{1/2}}$, where $c>0$ constant to be determined momentarily. Since $\log w$ is Poisson-integrable, so is $\log W$, and the rapid-decay also ensures that $W \in L^1(\R)$. We can therefore find an outer function $F\in H^1(\R)$, such that $\abs{F}=W$ a.e on $\R$. For any given $a>0$, we may by means of multiplying $F$ with $e^{i\sigma x}\in H^\infty(\R)$ for some $\sigma>0$, we may assume that $\widehat{F}(a)\neq 0$. Now put 
\[
g(x) = \widehat{F}(a-x) \widehat{F}(a+x), \qquad x \in \R,
\]
which satisfies $g(0)=(\widehat{F})^2(a)\neq 0$, $\supp{gdm}\subseteq [-a,a]$, and $g \in C^\infty(\R)$. Now, a straightforward calculation involving the Fourier inversion formula, and the monotonicity assumption on $W$ implies
\begin{multline*}
\abs{\widehat{g}(\xi)} \leq \int_{\R} \abs{F(\xi-t)F(-t)} dt= \int_{\R} W(\xi-t) W(t) dt \\ = \left(\int_{t\leq \xi/2} + \int_{t>\xi/2}\right) W(\xi-t) W(t) dt 
\leq 2W(\xi/2)\norm{W}_{L^1} \leq w(\xi), \qquad \xi \in \R,
\end{multline*}
provided $c>0$ small enough. This proves that $w$ is $a$-admissible majorant for any given $a>0$.
\end{proof}

Note that the above proof was fairly simple and required no additional tools than the definition of an outer function. However, rather strong regularity assumptions on $W$ were imposed, which on the upside are fairly practical to verify.

\subsection{The BM-Theorem via subharmoncicity and the $\conj{\partial}$-equation}

Here, we shall prove yet another version of the Beurling--Malliavin multiplier Theorem, which is more accessible than the strong BM1--Theorem. Here, we follow a slightly different route than the proofs presented in \cite{havinbook}, and also in \cite{mashreghi2006beurling}. Instead, we shall here carry out a proof which was recently suggested in the work of A. Cohen in \cite{cohen2025fractal}, whose cental idea is attributed to J. Bourgain. Here we recall if $\mu \in M(\R)$ the Hilbert transform of $\mu$ is defined as 
\[
H(\mu)(x):= \lim_{\varepsilon \to 0+} \frac{1}{\pi}\int_{|t-x|>\varepsilon} \frac{d\mu(t)}{t-x} , \qquad x \in \R.
\]
Our principal effort in this subsection will be to prove the following Theorem.
\begin{thm}[Medium BM1] \thlabel{THM:BM1MED} Let $w$ be a bounded weight on $\R$ with $\int_{\R} \log w dP > -\infty$. If $\log w$ is uniformly Lipschitz continuous on $\R$ and if 
\[
H((\log w)') \in L^\infty(\R),
\]
then $w$ is $a$-admissible for all $a> \norm{H((\log w)')}_{L^\infty(\R)}$.
\end{thm}

The proof proceeds in three principal steps:
\begin{enumerate}
\item[1.)] Reformulate the problem as an obstacle-type problem in complex analysis,
\item[2.)] Solve a subharmonic version of the BM1-Theorem,
\item[3.)] Upgrade the solution using Hörmander's $\conj{\partial}$-method.
\end{enumerate}
In the present setting, the first step is straightforward, owing to the availability of the Paley–Wiener Theorem. However, in attempting to extend the Beurling--Malliavin Theorem to contexts where no Theorem of Paley–Wiener-type is available, one encounters a substantial obstruction. This roughly explain why such generalizations are notoriously difficult beyond the classical framework of functions of finite bandwidth. We return to this point in the subsection on further discussions.
Once the problem has been recast via the Paley–Wiener Theorem, the second step becomes almost canonical. Variants of this approach already appear in \cite{koosis1998logarithmic}. Having established a subharmonic version of the BM1-Theorem, it is then natural to seek a holomorphic solution via $\conj{\partial}$-methods. The implementation of Hörmander’s $\conj{\partial}$-theory, however, requires some care and constitutes the main technical difficulty.

\proofpart{1}{Reformulation with Paley-Wiener:}
To thid end, let us invoke the Paley-Wiener Theorem, in order to reformulate the BM1-Theorem. For any number $a>0$, we would like to find an entire analytic function $f_a$ which satisfies the following properties:
\begin{enumerate}
    \item[(i)] $\abs{f_a} \leq w$ on $\R$,
    \item[(ii)] $\abs{f_a(x+iy)} \leq C e^{a|y|}$ for some $C>0$ and all $x,y \in \R$,
    \item[(iii)] $f_a(0)\neq 0$.
\end{enumerate}
Note that the third condition ensures that $f_a$ is non-trivial, and is indeed by a simple translation argument seen to be both necessary and sufficient. Technically, the first condition only needs to hold a.e on $\R$, but since $f_a$ is continuous and we can apply a convolution of $w$ with an appropriate smooth approximate of the identity, in order to pass to a smooth weight $w$, which is pointwise comparable to $w$. 

\proofpart{2}{A subharmonic $BM$-Theorem:}
Note that the above growth conditions on $f_a$ readily translate to the corresponding subharmonic function $u_a := \log \abs{f_a}$ on $\C$. This makes it rather natural to consider the following weaker problem on subharmonic functions, often referred to as the subharmonic BM-problem. More precisely, we would like to find a subharmonic function $u_a$ on $\C$ with the following properties:
\begin{enumerate}
    \item[(i)] $\abs{u_a} \leq \log w$ on $\R$,
    \item[(ii)] $\abs{u_a(x+iy)} \leq C + a|y|$ for some $C>0$ and all $x,y \in \R$,
    \item[(iii)] $u_a(0)>-\infty$.
\end{enumerate}
It turns out that this problem can be solved rather exactly, under certain additional regularity assumption(s) on $w$.

\begin{prop}[Exact subharmonic BM-Problem] \thlabel{PROP:SUBHARBM} Let $w$ be a bounded weight on $\R$ with $\int_{\R} \log w dP >- \infty$. Assume that $\log w$ is Lipschitz continuous on $\R$ and $H((\log w)')\in L^\infty(\R)$. Then for any $a>\norm{H((\log w)')}_{L^\infty}$, there exists a subharmonic function $u_a$ on $\C$, which is Lipschitz continuous in $\C$, and satisfies  
\[
u \lvert_{\R} = \log w, \qquad u(x+iy) \leq C + a|y|, \qquad x,y\in \R.
\]
    
\end{prop}
\begin{proof}
The proof is essentially a matter of verification, once a suitable candidate for $u_a$ has been found. Indeed, we shall look for a solutions of the form
\[
u(z):= P_{|y|}(\log w)(x) + C |y| = \int_{\R} \log w(t) \frac{|y|}{(x-t)^2 + y^2} \frac{dt}{\pi} + C|y|, \qquad z= x+iy \in \C,
\]
where $P_y$ denotes the Poisson kernel in $\C_{\pm}$ and $C>0$ is a constant to be determined later. Note that $P_{|y|}(\log w)(x)$ is well-defined since $\log w\in L^1(dP)$, and it is harmonic in $\C \setminus \R$. The perturbed term $C|y|$ is also trivially harmonic in $\C \setminus \R$, hence $u$ is harmonic on $\C \setminus \R$. Furthermore, since $P_{|y|}(\log w)(x)$ is symmetric about the real line, extends continuously to $\C$ and equals to $\log w$ on $\R$, we conclude that $u$ is continuous on $\C$. We can therefore define its Laplacian in the sense of distributions on $\C$. Note that $\supp{ \Delta u} \subseteq \R$. Now pairing with test functions and applying Green's Theorem, one can show that
\[
\Delta u = 2\left(\partial_y P_{|y|}(\log w) \lvert_{y=0+}\, + \, C \right) \delta_{\R}.
\]
A noteworthy remark related to term appearing above is that the map 
\[
\omega \mapsto \partial_y E(\omega) \lvert_{y=0+}
\]
is in the language of PDE typically refereed to as the Dirichlet-to-Neumann operator on $\C \setminus \R$, as it transforms Dirichlet data for the harmonic extensions to Neumann data. Moving forward, we simply utilize the Cauchy Riemann equations in order to rewrite
\begin{equation}\label{EQ:IBPCR}
\partial_y E(\log w) \lvert_{y=0+} = - H ((\log w)')(x).
\end{equation}
Therefore, we obtain
\[
\Delta u = 2 \left( - H((\log w)') + C \right) \delta_{\R}.
\]
Now this where the subtle assumption $H((\log w)') \in L^\infty(\R)$ enters the picture, allowing us to choose the number $C>0$ so that it exceed its supremum norm, and thus ensuring that $u$ is subharmonic on $\C$ with 
\[
u=\log w \ \, \text{on}  \, \, \R \qquad \text{and} \qquad  u(x+iy)\leq a|y| + A, \qquad x,y\in \R,
\]
whenever $a$ exceeds the supremum norm of $H((\log w)')$ in $\R$, and with $A= \sup_{\R}\log w$, since $w$ is assumed to be bounded. To see that $u$ is Lipschitz, it suffices to show that the partial derivatives 
\[
\partial_x P_{|y|}(\log w)(x),\, \, \partial_y  P_{|y|}(\log w)(x) \in L^\infty(\C).
\]
To this end, an integration by parts yields
\[
\partial_x P_{|y|}(\log w)(x)= P_{|y|}\left( (\log w)'\right)(x), \qquad z=x+iy\in \C
\]
while the same calculation as in \eqref{EQ:IBPCR} shows
\[
\partial_y P_{|y|}(\log w)(x) = - P_{|y|}\left( H((\log w)') \right)(x), \qquad z=x+iy\in \C.
\]
Since both $(\log w)'$ and its Hilbert transform $ H((\log w)')$ belong to $L^\infty(\R)$, it follows that $E(\log w)$ has uniformly bounded gradient, thus $u$ is uniformly Lipschitz continuous in $\C$.
\end{proof}

Let us now gather a few remarks. The good thing with the above approach is that we could actually solve the subharmonic BM-problem exactly, using a fairly explicit ansatz, instead of merely finding a sub-solution: $u\leq \log w$ on $\R$. However, a drawback is that we had to crucially rely on the rather obscure and impractical additional assumption $H((\log w)')\in L^\infty(\R)$, but which appears rather intrinsically. In classical works of Beurling and Malliavin \cite{beurling1962fourier}, and Koosis \cite{koosis1998logarithmic}, a sub-solution to the subharmonic BM-problem was found. Beurling--Malliavin used a variational argument similar to the solution of the Dirichlet problem, while Koosis developed a Perron-type method to exhibit a sub-solution. In the next subsection, we shall provide a different proof of the Medium BM1-Theorem following the work of Havin-Mashreghi-Nazarov in \cite{mashreghi2006beurling}, and see that the same condition $H((\log w)')\in L^\infty(\R)$ enters the picture, but from a different perspective. 

\proofpart{3}{Upgrading subharmonic BM1:}
Having solved a subharmonic BM-problem, it now essentially remains to find an entire analytic function $f_a$ such that $\log \abs{f_a} \leq u_a + C$, on $\C$ with $f_a(0)=1$, say. Historically, problems of this type have been approached by looking for $f$ expressed as am infinite Weierstrass product formula over its zeros, which gives
\[
\log \abs{f(x)} = \sum_k c_k\log (a_k- z), \qquad z\in \C.
\]
Now since our $u$ already satisfies $\Delta u = \mu$ with data $\mu = (-H((\log w)') + C) \delta_{\R}$, we can write 
\[
u(z) = \int_{\R} \log |t-z| d\mu(t), \qquad z\in \C.
\]
These observations suggest that one may attempt to approximate $\mu$ by a sum of Dirac deltas. Indeed, it turns out that one can carry out such an approach in detail, as was done in \cite{?}.

However, we shall here instead follow a different route which was suggested by J. Bourgain, mentioned in \cite{cohen2025fractal}. The proof will instead rely on the following solution to the $\conj{\partial}$-equation in weighted $L^2$-spaces by L. H\"ormander. 

\begin{thm}[H\"ormander] \thlabel{THM:HÖRM} Let $f \in L^2_{loc}(\C)$ and let $\varphi: \C \to \R$ be a continuous function which is strictly subharmonic function in $\supp{f}$, and satisfies
\[
A:=\int_{\C} \abs{f(z)}^2 \frac{e^{-\varphi(z)}}{\Delta \varphi(z)} dA(z) < \infty.
\]
Then there exists $F \in L^2_{loc}(\C)$ which solves $\conj{\partial} F= f$ in distributional sense on $\C$, and satisfies 
\[
\int_{\C} \abs{F(z)}^2 e^{-\varphi(z)} dA(z) \leq 4A.
\]
\end{thm}
A surprising fact is that the above result seems to have been noticed first only in several complex variables. For a detailed proof, and additional historical remarks on H\"ormander's Theorem, we refer the reader to the excellent notes of B. Berndtsson \cite{berndtsson1995l2}.

We are now ready to solve the final piece of the problem.

\begin{prop}[Upgrading to an analytic solution] \thlabel{PROP:ANSUB} 
Let $u: \C \to \R$ be a non-trivial subharmonic function satisfying the hypotheses of \thref{PROP:SUBHARBM}. Then there exist a constant $A>0$ and an entire function $f$ such that 
\[
\log |f(z)| \le A + u(z), \qquad z \in \C, \qquad f(0)=1.
\]
\end{prop}

\begin{proof}
The proof is based on solving a non-homogeneous $\conj{\partial}$-equation and applying Hörmander’s $L^2$-method. Our data will be conveniently expressed in terms of a smooth compactly supported function $\psi \in C^\infty(\C)$ which satisfies 
\[
\psi \equiv 1 \qquad  \text{on} \qquad \{|z|\le 1\}, \qquad  \supp \psi \subseteq \{|z|\le 2\}.
\]
Below, we shall form simplicity allow the constants to change from line to line.
\smallskip

\proofpart{1}{A non-homogeneous $\conj{\partial}$-equation:}
Consider a smooth compactly supported function $\psi$ on $\C$ with 
\[
\Delta \chi (z) \geq 1, \qquad 1\leq |z|\leq 2.
\]
Now consider the following perturb function of $u$:
\[
v(z) := u(z) + \chi(z), \qquad z \in \C.
\]
Note that the property on $\chi$ ensures that $v$ is strictly subharmonic in the annulus 
\[
\left\{z\in \C: 1\leq |z|\leq 2 \right\}.
\]
We now seek a solution to the equation
\[
\conj{\partial} F(z) = \conj{\partial} \psi(z), \qquad z\in \C.
\]
Since $\conj{\partial}\psi$ is supported in $\{1 \le |z| \le 2\}$ and $\Delta v \gtrsim 1$, Hörmander’s \thref{THM:HÖRM} allows us to exhibit a solution $F \in L^2_{\mathrm{loc}}(\C)$ which satisfies the weighted estimate:
\[
\int_{\C} |F(z)|^2 e^{-v(z)}\, dA(z)
 \leq
C \int_{1\leq |z|\leq 2} |\conj{\partial}\psi(z)|^2 \frac{e^{-v(z)}}{\Delta v(z)}\, dA(z) \lesssim 1.
\]
We shall now argue that the desired candidate is simply given by the difference
\[
f := \psi - F + F(0),
\]
Indeed, since $\conj{\partial}f=0$, Weyl's Lemma ensures that $f$ is an entire analytic function in $\C$. Moreover, since $\psi$ and $f$ are smooth, so is $F$. In particular, this ensures that $F(0)$ is well-defined. Furthermore, we also have 
\[
f(0)=\psi(0)-F(0)+F(0) =1.
\]
Thus the required normalization holds.

\proofpart{2}{The estimate on $f$:}
It now only remains to show that for all $|\lambda| \gg 1$:
\[
\abs{f(\lambda)} \leq C + u(\lambda).
\]
Let $B(\lambda)$ denote the disc of radius $1$ centered at point $\lambda \in \C$, and choose $\lambda$ so that $B(\lambda)$ does not meet $\supp{\chi}\cup \supp{\psi}$. Since $|f|^2$ is subharmonic, and $f=-F + F(0)$ on $B(\lambda)$, we have
\[
|f(\lambda)|^2 \leq C \int_{B(\lambda)} |F(z)|^2\, dA(z) + C|F(0)|^2.
\]
Now since $u$ appearing in \thref{PROP:SUBHARBM} was assume to be Lipschitz in $\C$, we have 
\[
u(z) \leq u(\lambda) + L, \qquad z\in B(\lambda).
\]
Using this estimate on $u$, the fact that $\chi=0$ in $B(\lambda)$, we get 
\begin{multline*}
|f(\lambda)|^2 \leq \exp\left(C+u(\lambda) \right) \int_{B(\lambda)} |F(z)|^2 e^{-v(z)} dA(z) \\
\leq  \exp\left(C+u(\lambda) \right) \int_{\C} |F(z)|^2 e^{-v(z)} dA(z).
\end{multline*}
Since the integral on the right-hand side in view of H\"ormander's weighted estimate, hence we conclude that 
\[
\log \, \abs{f(\lambda)} \leq C + u(\lambda), \qquad \lambda \in \C.
\]

\end{proof}

The proof of \thref{THM:BM1MED} is now an immediate consequence of \thref{PROP:SUBHARBM} and \thref{PROP:ANSUB}, in that order. 

\subsection{A complex analytic perspective to the BM-Theorem}
In this subsection, we shall outline a somewhat more canonical approach towards proving the Medium BM1 Theorem, which involves the notion of \emph{outer functions} in Hardy spaces. But first, we shall attempt to describe the reason why it is typically referred to as a \emph{multiplier} Theorem. To this end, let $a>0$ be a number, and $\mathcal{C}_a$ denote the so-called Cartwright class of order $a$, consisting of entire analytic functions $f$ satisfying the following conditions:
\begin{enumerate}
    \item[(i)] $f$ is of exponential-type $a$: for any $\varepsilon>0$ there exists a constant $C(\varepsilon)>0$ such that 
\[
\abs{f(z)} \leq C(\varepsilon) \exp \left( (a+ \varepsilon) |z| \right), \qquad z \in \C.
\]
    \item[(ii)] 
    \[
    \int_{\R} \frac{\log^{+} \abs{f(x)}}{1+x^2} dx < \infty,
    \]
where $\log^+ x := \max\left(0, \log x\right)$.
\end{enumerate}
We shall refer to $\textbf{Cart} := \cup_{a>0} \mathcal{C}_a$ as the \emph{Cartwright class}, which forms an algebra under pointwise multiplication. A deep result by M. G. Krein (See Ch. 1.6 in \cite{mashreghi2006beurling}) asserts that $f \in \textbf{Cart}$ if and only if the restrictions $f_{\pm} := f \lvert_{\C_{\pm}}$ belong to the Nevanlinna classes $\mathcal{N}(\C_{\pm})$. In other words, functions in the Cartwright class are entire functions belonging to Nevanlinna classes of the upper and lower half-planes, and thus one actually has that $\log |f| \in L^1(dP)$ whenever $f \in \textbf{Cart}$. Now a classical result in the theory of Hardy spaces asserts that Nevanlinna class functions can always be factored as quotients of two bounded analytic functions, where the numerator is zero free. It turns out that in the symmetric framework of $\textbf{Cart}$, much more can be said. The following result is somewhat clarifies why the term \emph{multiplier} appears in the statement of the BM1 Theorem, and the original result of Beurling and Malliavin. 

\begin{thm}[The BM-multiplier Theorem] For any $f\in \textbf{Cart}$ and any $\varepsilon>0$, there exists a multiplier $m_\varepsilon \in B_\varepsilon$ such that $fm_\varepsilon$ is bounded on $\R$.
\end{thm}
Here $B_\varepsilon$ denotes the functions of exponential type $\varepsilon$, which are essentially bounded in $\R$, thus contained in $\mathcal{C}_\varepsilon$. The terms multiplier comes from the fact that multiplication of $f$ with $m_\varepsilon$ multiplies a function $f \in \mathcal{C}_a$ into the space $B_{a+ \varepsilon}$. Phrased differently, whenever $f\in \mathcal{C}_a$ and $\varepsilon>0$, one can always factorize 
\[
f= \frac{g}{m}, \qquad g \in B_{a+\varepsilon}, \qquad m \in B_\varepsilon. 
\]
Moving forward, we shall outline a second proof of the BM-1 \thref{THM:BM1MED}, which is somewhat more canonical and utilizes the factorization of functions in the Hardy spaces. The proof is principally taken from the work of V. Havin, J. Mashreghi and F. Nazarov in \cite{mashreghi2006beurling}, and further details and historical notes related to it can be found there.

\proofpart{1}{Reformulating using Paley-Wiener:} We shall first use the notion of Paley--Wiener spaces in order to reformulate the BM1-Problem. 

We start with a weight $w \in L^2(\R)$ with $\log w \in L^1(dP)$. This allows us to construct an outer function $W \in H^2(\C_+)$ with $\abs{W}=w$ a.e on $\R$. The main goal is to find a non-trivial function $f\in L^2(\R)$ with $\supp{\widehat{f}}\subseteq[0,a]$ such that $f/w \in L^\infty(\R)$, under the additional assumption on $w$.

Define the translated Paley-Wiener space by
\[
H^2_a :=\{f\in L^2(\R): \supp{\widehat{f}}\subseteq[0,a] \}, \qquad a>0.
\]
It is not difficult to show that $f \in H^2_a$ if and only if $f(x), \conj{f(x)}e^{2\pi i ax}$ are boundary elements of $H^2(\C_+)$-functions. In order to move forward, we shall need practical sufficient condition that allows us to identify the modulus of functions in $H^2_a$. This simple observation is a special case of a general Theorem by K. Dyakonov in \cite{dyakonov1990moduli}.

\begin{lemma}[Dyakonov's Lemma]\thlabel{LEM:ModH2a} Let $\psi \in L^2(\R)$ positive with $\int_{\R}\log \psi dP >- \infty$, and such that $\psi^2 e^{2\pi i ax}$ is outer. Then $\psi$ is the modulus of a function in $H^2_a$.
\end{lemma}
\begin{proof}
Let $F$ be an outer function in $H^2(\C_+)$ with $\abs{F}=\psi$ a.e on $\R$. Then $F^2$ is also an outer function whose modulus equals $\psi^2$ a.e on $\R$. Now by assumption, $\psi^2(x) e^{2\pi i ax}$ is also outer, and with modulus equal to $\psi^2$ a.e on $\R$, therefore there exists a unimodular constant $c$ such that
\[
cF^2(x) = F(x) \conj{F(x)} e^{2\pi i ax}, \qquad \text{a.e} \, \, x\in \R.
\]
Dividing by $F$, which is non-zero a.e on $\R$, we see that $F(x), \conj{F(x)} e^{2\pi i ax}=cF(x)  \in H^2(\C_+)$, hence we conclude that $F \in H^2_a$ with $\abs{F}= \psi$ a.e on $\R$. 
\end{proof}

With this observation at hand, we recall that our goal is not necessarily to exhibit a function $f\in H^2_a$ with $\abs{f}=w$ a.e on $\R$, as we settle for 
\[
\abs{f(x)} = m(x) w(x), \qquad \text{a.e} \, \,  x\in \R, 
\]
for some non-negative bounded function $m \in L^\infty(\R)$. Note that since both $f, w$ are assumed to have integrable logarithm wrt $dP$ on $\R$, we can reformulate the BM-problem by using Dyakonov's Lemma with $\psi= mw$, as follows: 

\textbf{Reformulating the BM1-Problem:} \\
Find a positive function $m \in L^\infty(\R)$ with $\log m \in L^1(dP)$, such that 
\[
mw \in L^2(\R), \qquad \text{and} \qquad  m^2 w^2 e^{2\pi i ax} \qquad \text{is outer.}
\]
\proofpart{2}{Reducing to a problem on conjugate transform:}
Now recall that $F$ is an outer function if and only if 
\[
\log F = \log |F| + i \widetilde{\log |F|},
\]
where $\widetilde{f}$ denotes the harmonic conjugate transform of $f$ (see explicit definition below). Using Dyakonov's lemma, it suffices to find $\log m \in L^1(dP)$ with $w \in L^\infty(\R)$ and positive, such that
\begin{multline}\label{EQ:BMFACT}
\log m^2(x) w^2(x) + 2\pi i ax = \log m^2(x) w^2(x) + i \widetilde{\log m^2 w^2}(x) + \, \text{constant} \qquad (\text{mod} \, \, 2\pi ) \\ \iff \\
\widetilde{\log m}(x) + \widetilde{\log w}(x) = \pi ax + \, \text{constant}  \qquad (\text{mod} \, \, \pi ). 
\end{multline}
It is natural to attempt removing the slightly "annoying" periodicity which is built-in in the above equation, and simple ask whether it is possible to solve the equation with one single constant in all of $\R$:
\begin{equation}\label{EQ:NAIVEQS}
\widetilde{\log m}(x) + \widetilde{\log w}(x) = \pi ax + \, \text{constant}, \qquad x\in \R.
\end{equation}
However, a classical result of A. Kolmogorov presents a significant obstacle to this approach.

\begin{thm}[Kolmogorov] For any $f\in L^1(dP)$, we have 
\[
\lambda \cdot P \left( \{ x\in \R: \abs{\widetilde{f}(x)} > \lambda \} \right) \to 0, \qquad \lambda \to +\infty.
\]
\end{thm}
Now if \eqref{EQ:NAIVEQS} where to hold, then invoking Kolmogorov's Theorem with $f= \log m + \log w$, we get 
\[
\lambda  \cdot P\left( \{x\in \R: \abs{\pi a x+ \, \text{const} } > \lambda  \} \right) = \lambda \cdot P \left( \{ x\in \R: \abs{\widetilde{f}(x)} > \lambda  \} \right) \to 0, \qquad \lambda \to + \infty
\]
but the corresponding quantity of the linear function $\pi a x + \text{constant}$ is explicitly seen to be only big-$O$ of $1/\lambda$ as $\lambda \to +\infty$, which makes \eqref{EQ:NAIVEQS} unattainable. Therefore, in order to solve the equation \eqref{EQ:BMFACT}, we also have to makes sure that we can choose a non-trivial function $k: \mathbb{R} \to \mathbb{Z}$ such that 
\[
\widetilde{\log m}(x) = \widetilde{\Omega}(x) + \pi ax  -\pi k(x), \qquad x\in \R,
\]
where $\Omega = \log \frac{1}{w}$. In order to compensate for the growth of $\widetilde{\Omega}(x) + \pi ax $, it is natural to choose $k(x)$ as the integer part of $\widetilde{\Omega}(x)/\pi + ax $, thus the problem reduces to finding a positive bounded function $m$ on $\R$ which satisfies the equation:
\begin{equation}\label{EQ:BMEQREDUC}
\widetilde{\log m}(x) = u(x) \iff 
m(x) = e^{-\widetilde{u}(x) }
\end{equation}
where $u(x) := \widetilde{\Omega}(x) + \pi a x -\pi k(x) - \pi/2$ satisfying $\norm{u}_{L^\infty(\R)}\leq \pi/2$. 

\proofpart{3}{Explicit estimates:}
We have now arrived at a stage where we can explicitly consider what the equation really entails, using the explicit formula for the conjugate transform:
\begin{multline*}
\log m(x)  = -\widetilde{u}(x) = \int_{\R} u(t) \left( \frac{1}{t-x} - \frac{t}{1+t^2}\right) \frac{dt}{\pi} = \\
 \int_{\{t:|x-t|<1\}} u(t) \frac{1}{t-x} \frac{dt}{\pi} - \int_{\{t:|x-t|<1\}} u(t)\frac{t}{1+t^2} \frac{dt}{\pi} + \int_{\{t:|x-t|\geq 1\}} u(t) \left( \frac{1}{t-x} - \frac{t}{1+t^2}\right) \frac{dt}{\pi} = I_1 + I_2 + I_3.
\end{multline*}
Now it simple to see that the integral $I_2$ is uniformly bounded on $\R$:
\[
\abs{I_2} \leq \norm{u}_{L^\infty(\R)} \int_{\{t: |t-x|\leq 1\} } \frac{\abs{t}}{1+t^2} \frac{dt}{\pi} \leq C_2, \qquad x\in \R.
\]
Although, the quantity $I_3$ is not necessarily bounded, one can still show that its growth is not too rapid:
\[
\abs{I_3} \leq \frac{\norm{u}_{L^\infty(\R)}}{\pi} \int_{\{t:|x-t|\geq 1\}} \abs{\frac{1}{x-t} + \frac{t}{1+t^2}} dt \leq C_3 \log (1+|x|), \qquad x\in \R.
\] 
Now the integral $I_1$ is the piece that will require imposing an additional regularity assumption on $u$, as we can rewrite it as follows:
\[
I_1 = \int_{\{t: |t-x|<1 \}} \frac{u(t)-u(x)}{t-x} \frac{dt}{\pi}
\]
A natural condition on $u$ which makes $I_1$ bounded is that $u$ is (uniformly) Lipschitz continuous on $\R$. The following proposition summarizes the technical aspects of this discussion, and the translates these conditions to $\Omega$.

\begin{prop}\thlabel{PROP:BMFINAL} Let $0<w\leq 1$ be weight with $\Omega :=-\log w \in L^1(dP)$ and which satisfies $\sup_{x\in \R} \abs{\widetilde{\Omega}'(x)} \leq \pi a$. Then the function $m$ defined in \eqref{EQ:BMEQREDUC} satisfies the estimate
\[
\log m(x) \leq 4a + C_2+ C_3 \log (1+|x|), \qquad x\in \R,
\]
and $m^2 w^2 e^{2\pi i a x}$ is an outer function.
\end{prop}
\begin{proof}
Note that from the above estimates, we have 
\[
\log m(x) \leq \int_{\{t: |t-x|<1 \}} \frac{u(x)-u(t)}{x-t} \frac{dt}{\pi} + C_2 + C_3 \log (1+|x|), \qquad x\in \R.  
\]
Now the estimate $\sup_{x\in \R} \abs{\widetilde{\Omega}'(x)} \leq \pi a$ ensures that $u(x):= \widetilde{\Omega}(x) + \pi ax $ is a non-decreasing uniformly Lipschitz function on $\R$, hence 
\[
k(x) := \text{integer part} \, \, \frac{\widetilde{\Omega}(x)}{\pi} + ax
\]
is also increasing, and thus recalling that $u(x) = \widetilde{\Omega}(x) + \pi ax - \pi k(x) - \pi/2$, we obtain 
\begin{multline*}
\int_{\{t: |t-x|<1 \}} \frac{u(x)-u(t)}{x-t} \frac{dt}{\pi} = \int_{\{t: |t-x|<1 \}} \frac{\widetilde{\Omega}(x) -\widetilde{\Omega}(t)}{x-t} \frac{dt}{\pi} + 2a - \int_{\{t: |t-x|<1 \}}  \frac{k(x)-k(t)}{x-t} \frac{dt}{\pi} \leq 4a.
\end{multline*}
This establishes the pointwise estimate on $\log m$. The claim that $m^2 w^2 e^{2\pi a x}$ is outer follows form \eqref{EQ:BMFACT}.
\end{proof}
Now comparing \thref{PROP:BMFINAL} with Dyakonov's Lemma, we also need to ensure that $mw \in L^2(\R)$ in order to conclude that there exists an outer function $f\in H^2_a$ with $\abs{f}=mw$ on $\R$, and that $m$ is bounded in order to conclude that $\abs{f}\leq w$, thus showing that $w$ is $a$-admissible. However, it turns out that these requirements are not difficult to rectify, and one should simply apply the previous argument to the modified weight
\[
w_1(x) = \frac{w(x)}{(x^2+l^2)^A}
\]
where $A,l>0$ are constants to be specified shortly. The following lemma is simple, and contains the necessary ingredients for continuing with the proof.

\begin{lemma} \thlabel{LEM:PROPgamma} The function $\gamma(x)= A\log(l^2+x^2)$ satisfies the following properties:
\begin{enumerate}
    \item[(i)] $\gamma \in L^1(dP)$,
    \item[(ii)] $\widetilde{\gamma} \in C^1(\R)$,
    \item[(iii)] $\sup_{x\in \R} \abs{\widetilde{\gamma}'(x) } \leq  2A/l $
\end{enumerate}
    
\end{lemma}
\begin{proof}
The claim in $(i)$ is obvious to verify, hence we restrict our attention to $(ii)$ and $(iii)$. Set $L(x) = \log (1+x^2)$ and note that since $L$ is an even function, we have
\[
\widetilde{L}(x) = \lim_{ \varepsilon \to 0+} \int_{\{t:|x-t|>\varepsilon\}} \frac{L(t)}{x-t} \frac{dt}{\pi}, \qquad x \in \R.
\]
Integrating by parts, applying a partial fraction decomposition, and using symmetry, we obtain
\begin{multline*}
\frac{d}{dx}\widetilde{L}(x)=   \lim_{ \varepsilon \to 0+} \int_{\{t:|x-t|>\varepsilon\}} \frac{l(t)}{(x-t)^2} \frac{dt}{\pi} =  \lim_{ \varepsilon \to 0+} 2\int_{\{t:|x-t|>\varepsilon\}}  \frac{2t}{(x-t)(1+t^2)} \\
= -\frac{2}{1+x^2} \int_{\R} \frac{1}{1+t^2} \frac{dt}{\pi} = - \frac{-2}{1+x^2}.
\end{multline*}
Now since $\gamma(x)= 2A\log l + A\cdot L(x/l)$, we get 
\[
\left(\widetilde{\gamma}\right)'(x) = \frac{A}{l} \cdot \frac{d}{dx}\widetilde{L}(x/l) = - \frac{2A}{l} \frac{1}{(1+(x/l)^2)}.
\]
This proves the claim.
\end{proof}

\proofpart{4}{Finalizing the proof:}
We can now complete the second proof of the medium version of the BM-1 Theorem.

\begin{proof}[Second proof of \thref{THM:BM1MED}]
Assume that $0<w \leq 1$ is a weight which satisfies the conditions $\Omega:=\log w \in L^1(dP)$ and $\norm{(\widetilde{\Omega})'}_{L^\infty(\R)} < \pi a$, for some $a>0$. Set
\[
\Omega_1(x) := - \log w_1(x) = \Omega(x) + \gamma(x), \qquad x\in \R.
\]
where $\gamma(x) = A \log(l^2 + x^2)$, and note that according to \thref{LEM:PROPgamma}, we may for any $A>0$, pick a number $l=l_A>0$ large enough so that 
\[
\norm{(\widetilde{\Omega_1})'}_{L^\infty(\R)} \leq \norm{(\widetilde{\Omega})'}_{L^\infty(\R)} + \norm{(\widetilde{\gamma})'}_{L^\infty(\R)} \leq \pi a.
\]
Invoking \thref{PROP:BMFINAL}, we can exhibit a function $m_1 \geq 0$ on $\R$, such that $m^2_1 (x) w^2_1(x) e^{2\pi i a x}$ is outer and 
\[
\log m_1(x) \leq 4 a + C_3 \log (1+|x|), \qquad x \in \R.
\]
Now this implies that 
\[
w_1(x)m_1(x) \leq e^{4a +C_2} \frac{w(x)(1+|x|)^{C_3}}{(x^2+l^2)^A} \leq w(x) e^{4a +C_2} \frac{1}{(1+x^2)^{A-C_3}} \leq  \frac{w(x)}{1+x^2} \in L^2(\R),
\]
if $A>0$ large enough. Invoking Dyakonov's lemma, we can find $f\in H^1_a$ such that  
\[
\abs{f(x)} = w_1(x)m_1(x) \leq w(x), \qquad x\in \R,
\]
which readily implies that $w$ is $a$-admissible, hence the second proof of the BM1-Theorem is complete.
\end{proof}

As was also mentioned in the previous subsection, the statement of \thref{PROP:BMFINAL} has the drawback that $a$-admissibility of $w$ depends on the supremum norm of the (uniform) Lipschitz constant of $\widetilde{\Omega}$, a quantity that is extremely difficult to confirm in practice. The reader may therefore object that this result sits rather far from the strong BM-Theorem, whose hypotheses involve only the qualitative Lipschitz continuity of $\Omega$, and for which admissibility actually holds for every $a>0$. 

It was proved in \cite{mashreghi2006beurling} (see Theorem 2 therein) that the strong BM1 Theorem can actually be obtained from \thref{PROP:BMFINAL}, relying on the following striking assertion, one whose plausibility is far from obvious and demands some courage to even conjecture.
\begin{thm}[Nazarov's Majorization Theorem] \thlabel{THM:MAINLEMMABM} Let $0<\omega\leq 1$ with $\log \omega \in L^1(dP)$ which is uniformly Lipchitz continuous on $\R$. For any $a>0$, there exists a function $0<\omega_a \leq1$ with the following properties:
\begin{enumerate}
    \item[(i)] $\omega_a \leq \omega$,
    \item[(ii)] The function $\Omega_a := \log \frac{1}{\omega_a} \in L^1(dP)$,
    \item[(iii)] $\Omega_a$ and $\widetilde{\Omega_a}$ are both uniformly Lipschitz continuous on $\R$ with  
    \[
    \sup_{x\in \R} \, \abs{\Omega_a '(x)}\leq a, \qquad \sup_{x\in \R}\,  \abs{\widetilde{\Omega_a} '(x)}\leq a.
    \]
\end{enumerate}

\end{thm}

As indicated, this result is admittedly technical and requires some intricate properties of the Hilbert transform, which we shall not include here. We therefore refer the reader to \cite{mashreghi2006beurling}. However, the reader should have no problem on reducing the strong BM1 Theorem to that of \thref{PROP:BMFINAL}, using Nazarov's Majorization Theorem.

    



\subsection{The Second BM-Theorem on completeness of exponentials} 
The second Beurling-Malliavin Theorem bears a crucial piece of the puzzle of the ambitious program of determining Fourier uniqueness pairs, as we encountered in Chapter 2, namely when do we have the following manifestation of the uncertainty principle:
\[
f\in L^2(\R), \qquad \text{ess supp}(f) \subseteq[-a,a], \qquad \widehat{f}\lvert_{\Lambda} = 0 \qquad \implies f\equiv 0,
\]
where $\Lambda$ is a subset of $\R$ and $a>0$. Note that the compact support of $f$ makes $\widehat{f}$ into an entire analytic functions, hence to avoid redundancy, one restricts the attention to discrete subsets of real numbers $\Lambda$. Of course, problem remains unaltered if the roles of $f$ and $\widehat{f}$ are interchanged. We shall now illustrate two classical reformulations of this problem. Fix $a>0$ and note that since $f\in L^2(\R)$, the Plancherel's Theorem allows us to rephrase the problem as for which discrete sets $\Lambda \subset \R$ do we have
\[
\widehat{f}(\lambda) = \int_{-a}^{a} \widehat{f}(\xi) e^{-i\xi \lambda} d\xi = 0 \qquad \forall\lambda \in \Lambda \qquad \implies \widehat{f} \equiv 0?
\]
In other words, we are asking for which pairs of $a>0$ and discrete subsets of real numbers $\Lambda$ is the linear span of the associated family of exponentials 
\[
\mathcal{E}(\Lambda)=\text{Span}\{e^{i\lambda t}: \lambda \in \Lambda \}
\]
dense in $L^2([-a,a])$? This implies that the problem is equivalent to a certain \emph{completeness problem} of exponentials in $L^2$ of a segment on the real line. \\

Using the Paley-Wiener Theorem, one can easily show that the Fourier transform maps $f\in L^2(\R)$ with $\text{ess supp}(f) \subseteq[-a,a]$ unitarily on the space of entire analytic functions $\widehat{f} \in L^2(\R)$ of exponential-type at most $a>0$. Formally, we identity 
\[
\textbf{PW}_{a} := \{f \in L^2(\R): \text{ess supp}(\widehat{f}) \subseteq[-a,a] \}
\]
as a closed subspace of $L^2(\R)$, but which according to Paley-Wiener Theorem consists of entire analytic functions of exponential-type at most $a$. In this framework, our problem can be reformulated as characterizing the discrete subsets $\Lambda \subset \R$ for which the following holds:
\[
\widehat{f}\in \textbf{PW}_a \qquad \widehat{f}\lvert _{\Lambda} = 0 \qquad \implies \qquad f \equiv 0.
\]
This is readily a \emph{uniqueness problem} for analytic functions in $\textbf{PW}_a$, and these two seemingly different perspectives makes the initial problem structurally rich. It turns out that the problem becomes simpler if we drop the strict prescription of $a>0$ and instead consider the quantity:
\[
\rho(\Lambda) := \sup \{ a>0: \, \mathcal{E}(\Lambda) \, \, \text{dense in} \, \, L^2([-a,a]) \}.
\]
We shall now introduce a more geometric quantity, which is intimately related to this problem. A sequence of disjoint intervals $\{I_n\}_n$ on the real line is said to be \emph{long} if 
\[
\sum_n \frac{\abs{I_n}^2}{1+ (\dist{I_n}{0})^2} = +\infty.
\]


For instance, if a system $\{I_n\}_n$ of disjoint intervals are uniformly bounded away from the origin, then one can show that $\{I_n\}_n$ is long if and only if their length relative to their distance to the origin fails to be square-summable:
\[
\sum_n \left( \frac{|I_n|}{\dist{I_n}{0}}\right)^2 = +\infty.
\]
On the contrary, a sequence of disjoint interval $\{I_n\}_n$ can never be long if their total lengths is too small
\[
\sum_n \abs{I_n}^2 < \infty.
\]
This somewhat justifies their label. With this notion at hand, we introduced the so-called \emph{upper BM-density} of a discrete sequence $\Lambda \subset \R$ as the quantity
\[
D^*(\Lambda) := \sup \left\{ d \geq 0: \, \exists \, \text{long}  \, \{I_n\}_n \, \text{such that} \, \#(\Lambda \cap I_n ) \geq d \abs{I_n}, \, \, \forall n \right\}.
\]
Roughly speaking, the upper BM-density of a discrete sequence $\Lambda \subset \R$ measures the largest possible density of $\Lambda$ among the family of long intervals on the real line. 

The second Beurling--Malliavin Theorem relates the completeness problem on exponentials with frequencies in $\Lambda$, or equivalently the uniqueness problem of in Paley-Wiener spaces, to a more metric/geometric density condition of $\Lambda$.
\begin{thm}[BM2-Theorem] For any discrete sequence $\Lambda \subset \R$ the following identity holds:
\[
\rho(\Lambda) = \pi D^*(\Lambda).
\]
\end{thm}
The proof of the BM2-Theorem hinges on the BM-Multiplier Theorem in the form of a multiplier, but it is beyond the scope of these notes. Besides the original source in \cite{beurling1967closure}, a careful treatment of the matter can be found in the book of P. Koosis in \cite{MR1195788}, and in V. Havin and B. J\"oricke in \cite[Ch. 4]{havinbook}. We also refer the reader to a survey of N. Makarov and A. Poltoratski in \cite{makarov2005meromorphic}, where a different approach to the BM2-Theorem, using the kernels of Toeplitz operators is provided.

\subsection{Further results}

There are also versions of the BM1-Theorem where uniform admissibility is substituted by admissibility in $p$-th mean. More specifically, we for which weights $w$ does there exists a non-trivial compactly supported $f\in L^2(\R)$ such that
\begin{equation}\label{EQ:BMpth}
\int_{\R} \frac{\abs{\widehat{f}(\xi)}^p}{w(\xi)^p} d\xi < \infty, \qquad p\geq 1.
\end{equation}
We declare such weights $w$ to be $BM(p)$-majorant. From this perspective, one can view the BM1-Theorem as the $p=\infty$ case of the above problem. It turns out that admissibility in $p$-th mean is slightly different than uniform admissibility. This is somewhat visible from the simple observation that while $w(\xi_0)=0$ forces any $\widehat{f}(\xi_0)$ in the uniform setting, while this is not necessarily the case in regime of $p$-th mean. It is not difficult to show that any BM-majorant $f$ is also a $BM(p)$-majorant, for any $p\geq 1$. Indeed, if $f\in L^2(\R)$ is non-trivial a compact supported function with
\[
\abs{f(\xi)} \leq w(\xi), \qquad \xi \in \R,
\]
then the function $f_\varepsilon (x) = f\ast 1_{[0,\varepsilon]} \ast 1_{[0,\varepsilon]}(x)$, with $\varepsilon>0$ has only slightly larger supported than $f$, and computing the Fourier transform one easily obtains

\[
\int_{\R} \frac{\abs{\widehat{f_\varepsilon}(\xi)}^p}{w(\xi)^p} d\xi = \int_{\R} \frac{\abs{\widehat{f}(\xi)}^p}{w(\xi)^p} \cdot \frac{\sin^{2p}(\xi/\varepsilon)}{\xi^{2p}} d\xi \leq \int_{\R} \frac{\sin^{2p}(\xi/\varepsilon)}{\xi^{2p}} d\xi < \infty.
\]
Now the main result of the original paper of Beurling and Mallivian in \cite{beurling1962fourier} goes as follows.

\begin{thm}[BM1-Mean] Let $w$ be a bounded weight in $\R$ with $\log w \in L^1(dP)$. Set $W(x) = (\log w(x) )/ x$ and assume that the following energy estimate holds:
\[
\tag{E} \int_{\R} \int_{\R} \frac{|W(x)-W(y)|^2}{|x-y|^2} dx dy < \infty.
\]
Then for every $p\geq 1$ and $a>0$, there exists a non-trivial function $f$ in $\R$ with
\[
\int_{\R} \frac{\abs{\widehat{f}(\xi)}^p}{w(\xi)^p} d\xi < \infty,
\]
such that $\supp{fdm}\subseteq[0,a]$.
\end{thm}
Note that the energy condition in $(E)$ is a global condition, which accounts for integrability at infinity, and is independent of $p$. The proof is, again, long, and we refer the reader to \cite{koosis1998logarithmic} and \cite{havin2006linear} for details. 

Despite the fact that a $BM(p)$-majorant may not necessarily be a $BM$-majorant, one still obtains the following relation.

\begin{thm}[Beurling--Malliavin, \cite{beurling1962fourier}] Let $w\in L^\infty(\R)$ be a weight satisfying the conditions of BM1-Mean. For any $\varepsilon>0$, the regularized weight $w_\varepsilon$ defined by 
\[
w_\varepsilon(x) := \exp \left( (P_{\varepsilon} \ast \log w) (x) \right), \qquad x\in \R,
\]
is a BM-majorant.
\end{thm}

Several dimensional analogues of the BM-multiplier theorem where obtained by I. Vasilyev in \cite{vasilyev2025beurling}. More recently, A. Bergman gave a simpler proof, which essentially reduced the  problem to the classical one-variable setting. The paper is short and elementary, see \cite{bergman2025remark}. We also mention the work of A. Cohen in \cite{cohen2025fractal}, where we utilizes a several variable BM-multiplier Theorem, in order to extend to aforementioned fractal uncertainty principle by J. Bourgain and S. Dyatlov, \cite{cohen2025fractal}. 

Other extensions to the BM-1 Theorem pertain to viewing the Paley-Wiener spaces $H^2_a(\R)$ as a special case of so called Model space. Given an inner function $\theta$ in $\C_+$, the model space $K^2_\theta$ in $\C_+$ is defined as the closed subspace of $f\in H^2(\C_+)$ with the property that 
\[
f, \theta \conj{f} \in H^2(\C_+).
\]
The generalized BM-problem then turns into describing admissible pairs $(\theta,w)$ for which one can exhibit a non-trivial function $f\in K^2_\theta$ with
\[
\abs{f(x)} \leq w(x), \qquad x\in \R.
\]
Note the classical BM-problem corresponds to the inner function $\theta(z)= e^{iaz}$ with $a>0$.

In this generality, the BM-problem becomes very difficult, as the model spaces are no longer characterized by a growth condition of Paley--Wiener type. However, several partial result special classes of meromorphic inner functions $\theta$ have been obtain. We refer the reader to the work of Y. Belov, V. Havin, and J. Mashreghi in \cite{belov2009model,belov2015beurling,havin2003admissible,havin2003admissible2}. A more novel approach involving Toeplitz kernels was realized by N. Makarov and A. Poltoratski in \cite{makarov2005meromorphic}, which was later surveyed in \cite{makarov2010beurling}.

We mention some recent important developments in harmonic analysis which stem from the BM-1 Theorem. A new refined version of the uncertainty principle for Fourier transforms supported on fractal sets was proved by J. Bourgain and S. Dyatlov in \cite{bourgain2018spectral}. Recall that a manifestation of the classical uncertainty principal asserts that for any pair of closed sets $X,Y \subset \R$ and any $f\in L^2(\R)$, we have
\[
\norm{1_Y \mathcal{F}(f 1_X) }_{L^2} \lesssim \abs{X}^{1/2} \abs{Y}^{1/2} \norm{f}_{L^2}.
\]
This estimate roughly quantifies the assertion that if a function is spatially localized in a small set $X$, then no small set $Y$ in the frequency domain can capture a large portion of its Fourier energy. This bound can in general not be improved, and is easily deduced from Cauchy-Schwarz inequality applied to the inner integral. 

Given a number $0<\delta<1$ and a closed set $X \subset \R$ is called $\delta$-regular on scales $a$ to $b$, if there is a probability measure $\mu_X$ supported in $X$ such that 
\[
\mu_X(I)\leq C|I|^\delta, \qquad a \leq \abs{I} \leq b, 
\]
and if in addition $I$ is centered in $X$, the opposite inequality also holds, $\mu_X(I)\geq C^{-1} |I|^\delta$. In view of Frostman's mass distribution principle, this roughly means that the set $X$ has Hausdorff dimension $\delta$ between scales $a$ to $b$. For instance, one can show that the any small $h$-neighborhood of the classical Cantor $1/3$-set is a $\log 2 / \log 3$-regular set on scales $h$ to $1$. 

\begin{thm}[The Fractal UP] Let $\delta \in (0,1)$ and assume $X \subseteq [−1, 1]$ be a $\delta$-regular set on scales $1/N$ to $1$, while $Y \subseteq [-N,N]$ is a $\delta$-regular set on scales $1$ to $N$. Then there is a number $\beta =\beta(\delta)>0$, such that 
\[
\norm{1_Y \mathcal{F}(1_X f)}_{L^2} \lesssim N^{-\beta} \norm{f}_{L^2}, \qquad \forall f\in L^2(\R).
\]
\end{thm}

Roughly speaking, the fractal uncertain principle asserts that a function cannot be both localized in position and frequency, near a fractal set. We also mention a gentle survey on the matter by S. Dyatlov, see \cite{dyatlov2019introduction}. Its known that the number $\beta$ has a rather wild dependence on $\delta$, and recent accounts on these matters were investigated in \cite{jin2020fractal}. \\

Then general completeness problem on exponentials 
\[
\mathcal{E}(\Lambda):= \{e^{ix\lambda}: \lambda \in \Lambda \}, \qquad \Lambda \subset \C
\]
in $L^2(\mu)$ for a general finite Borel measure $\mu$ on $\R$, was surprisingly solved by A. Poltoratski in \cite{poltoratski2013problem}. This problem and a certain related spectral gap problem of measures were long thought to be "transcendental", that is, not having  characterization. We also mention a recent interesting paper on connecting this problem to a uniqueness problem by Polya and Levinson, \cite{mitkovski2010polya}. We mention that an equivalence of the classical BM2 Theorem and the Mitkovski--Poltoratski Theorem was recently announced in \cite{baranov2017gap} for a class of separates sequences $\Lambda \subset \R$:
\[
d(\Lambda) := \inf_{\substack{\lambda,\lambda' \in \Lambda \\ \lambda \neq \lambda'}} \abs{\lambda- \lambda'}>0.
\]
For general discrete sequence $\Lambda \subset \R$, the work of A. Poltoratski in \cite{poltoratski2013problem} requires the different notion of energy.

\section{Appendix}

\subsection{Hardy spaces}

We summarize some key features of Hardy spaces. We begin with the unit disc $\D$, and later explain how the corresponding theory on the upper half-plane $\C_+$ is obtained via conformal mapping.

For $1 \leq p < \infty$, we denote by $H^p(\D)$ the space of analytic functions $f$ on $\D$ with
\[
\|f\|_{H^p}^p := \sup_{0<r<1} \int_{\T} |f(r\zeta)|^p \, dm(\zeta) < \infty.
\]
For $p=\infty$, we define $H^\infty(\D)$ as the Banach space of bounded analytic functions on $\D$ equipped with
\[
\|f\|_{H^\infty} := \sup_{z\in \D} |f(z)|.
\]

A central feature of Hardy spaces is that they admit both an analytic and a boundary-value interpretation. The following result clarifies on aspect of it.

\begin{prop}
Let $1 \leq p \leq \infty$. Then every function in $H^p(\D)$ arises as the Poisson extension of a function in
\[
H^p(\T) := \{f \in L^p(\T) : \widehat{f}(n)=0 \ \text{for } n<0\}.
\]
Conversely, for any $f \in H^p(\T)$, the Poisson integral
\[
P(f)(z) := \int_{\T} \frac{1-|z|^2}{|\zeta-z|^2} f(\zeta)\, dm(\zeta), \qquad z\in \D
\]
defines a function in $H^p(\D)$, and we have
\[
\|P(f)\|_{H^p}^p = \int_{\T} |f(\zeta)|^p dm(\zeta).
\]
\end{prop}

The Poisson extension solves the Dirichlet problem in $\D$ with boundary data $f \in L^p(\T)$, producing a harmonic function. The additional condition $\widehat{f}(n)=0$ for $n<0$ forces analyticity. Indeed, using the identity
\[
\frac{1-|z|^2}{|\zeta-z|^2}
=
\frac{1}{1-\overline{\zeta}z} + \frac{1}{1-\zeta \overline{z}} - 1,\qquad \zeta \in \T, \qquad z\in \D,
\]
one obtains
\[
P(f)(z)
= \int_{\T} \frac{f(\zeta)}{1-\conj{\zeta}z} dm(\zeta)=
\sum_{n=0}^\infty \widehat{f}(n) z^n, \qquad z\in \D.
\]

The following result shows that functions in $H^p(\D)$ admit well-defined boundary traces almost everywhere, and that $H^p(\T)$ may be viewed as the natural trace space of $H^p(\D)$.

\begin{thm}\thlabel{THM:NTLHARDY}
Let $F \in H^p(\D)$. Then for $dm$-almost every $\zeta \in \T$, the limit
\[
\lim_{\substack{z \to \zeta \\ |z-\zeta| < \alpha (1-|z|)}} F(z)
\]
exists (for any $\alpha>0$). This is called the \emph{non-tangential limit} of $F$ at $\zeta$. Furthermore, the resulting boundary function $f$ belongs to $H^p(\T)$, and $F$ is recovered as its Poisson extension:
\[
F(z) = P(f)(z).
\]
\end{thm}

It it is customary to identify $H^p(\D)$ and $H^p(\T)$ and simply write $H^p$ when no confusion arises. The following simple estimate is crucial in theory of Hardy spaces.

\begin{thm}[Jensen--Poisson inequality]\thlabel{COR:logPf}
If $f \in H^1(\T)$, then
\[
\log |P(f)(z)|
\leq
\int_{\T} \log |f(\zeta)| \frac{1-|z|^2}{|\zeta-z|^2} dm(\zeta),
\qquad z \in \D.
\]
\end{thm}

In particular, it follows that functions in $H^1$ have finite logarithmic integral on $\T$. More generally, one can show that this property holds for functions in the Nevanlinna class $\mathcal{N}$, defined by
\[
\sup_{0<r<1} \int_{\T} \log^+ |f(r\zeta)| dm(\zeta) < \infty.
\]
Another classical result asserts that every $f \in \mathcal{N}$ can be written as $f=g/h$ with $g,h \in H^\infty$ and $h$ zero-free in $\D$.

\begin{proof}[Sketch of proof of \thref{COR:logPf}]
Applying Jensen's lemma at the origin gives
\[
\log |P(f)(0)|
\leq
\int_{\T} \log |f(\zeta)| dm(\zeta).
\]
Composing with the automorphism
\[
\varphi_z(\lambda) = \frac{z-\lambda}{1-\overline{\lambda}z}
\]
and using invariance of the Poisson kernel yields the general estimate.
\end{proof}

The Jensen--Poisson inequality naturally leads to the inner--outer factorization. Given $f \in H^p$, define the associated \emph{outer factor} of $f$
\[
\mathcal{O}_f(z)
=
\exp\left(
\int_{\T} \frac{\zeta+z}{\zeta-z} \log |f(\zeta)|\, dm(\zeta)
\right), \qquad z\in \D.
\]
Then $\mathcal{O}_f$ is analytic and zero-free in $\D$, satisfies
\[
\log |\mathcal{O}_f(z)| = P(\log|f|)(z), \qquad z\in \D.
\]
Using Jensen's inequality for convex functions, and Fubini's Theorem we have 
\begin{multline*}
\int_{\T} \abs{\mathcal{O}_f(r\xi)}^p dm(\xi) = \int_{\T} \exp\left( p\int_{\T} \log|f(\zeta)| \frac{1-r^2}{\abs{\zeta-r\xi}^2} dm(\zeta) \right) dm(\xi) \\ 
\leq \int_{\T} \int_{\T} \frac{1-r^2}{\abs{\zeta-r\xi}^2} dm(\xi) \abs{f(\zeta)}^p dm(\zeta) = \norm{f}^p_{L^p}, \qquad 0<r<1.
\end{multline*}
This implies that $\mathcal{O}_f$ belongs to the Hardy space $H^p(\D)$, hence has finite non-tangential limits equal to $F\in H^p(\T)$. However, standard properties of Poisson kernels imply that 
\begin{equation}\label{EQ:INNERfF}
\abs{F(\zeta)} = \abs{f(\zeta)}, \qquad dm-\text{a.e} \, \, \zeta \in \T.
\end{equation}
Invoking the Jensen--Poisson Inequality, we get 
\[
\abs{f(z)} \leq \abs{\mathcal{O}_f (z)} \qquad z \in \D.
\]
This means that the quotient 
\[
\Theta_f (z) := \frac{f(z)}{\mathcal{O}_f(z)} \qquad z\in \D
\]
defines a bounded analytic function on $\D$ with $\sup_{z\in \D} |\Theta_f(z)|\leq 1$,  and whose boundary values are unimodular:
\[
\lim_{r\to 1-} \abs{\Theta_f(r\zeta)}=1, \qquad dm-\text{a.e} \, \, \, \zeta \in \T.
\]
Such functions are called \emph{inner functions}. We record the following result, which is essentially a polar decomposition of functions in $H^p$.

\begin{thm}[Inner--outer factorization]
Every $f \in H^p(\D)$ admits a factorization
\[
f = \Theta_f \cdot \mathcal{O}_f
\]
on $\D$, where $\Theta_f$ is inner and $\mathcal{O}_f$ is outer. This factorization is unique up to multiplication by a unimodular constant.
\end{thm}

Inner functions admit a further canonical decomposition which separates their “zero structure” from their “boundary singularities.” More precisely, any inner function $\Theta$ on $\D$ can be written uniquely (up to a unimodular constant) as
\[
\Theta(z) = B_\Lambda(z)\, S_\mu(z), \qquad z \in \D,
\]
where the two factors encode distinct phenomena.

The first factor $B_\Lambda$ is the \emph{Blaschke product}, determined by the zeros $\Lambda \subset \D$ of $\Theta$. It is given by
\[
B_\Lambda(z)
=
\prod_{\lambda \in \Lambda}
\frac{|\lambda|}{\lambda}
\frac{z-\lambda}{1-\overline{\lambda}z} \qquad z\in \D.
\]
Each factor is a Möbius transformation mapping $\D$ to itself and having modulus one on $\T$. The infinite product converges to a non-trivial analytic function if and only if the zeros satisfy the Blaschke condition
\[
\sum_{\lambda \in \Lambda} (1-|\lambda|) < \infty.
\]
This Blaschke condition expresses that the zeros do not accumulate too rapidly toward the boundary. The function $B_\Lambda$ is itself inner, and it is precisely the part of $\Theta$ that accounts for its zeros inside $\D$.

The second factor $S_\mu$ is the \emph{singular inner function}, which is zero-free in $\D$ and captures the remaining boundary behavior. It is given by the formula
\[
S_\mu(z)
=
\exp\left(
- \int_{\T} \frac{\zeta+z}{\zeta-z}\, d\mu(\zeta)
\right), \qquad z\in \D
\]
where $\mu$ is a positive finite Borel measure on $\T$ that is singular with respect to Lebesgue measure.

The terminology reflects the fact that $S_\mu$ arises from a “singular distribution of mass” on the boundary. In comparison with the outer factor $\mathcal{O}_f$ of $f$, the singular inner factor $S_\mu$ is also zero-free on $\D$, but now $\log|S_\mu|$ is the Poisson extension of a positive singular measure, while $\log |\mathcal{O}_f|$ was the Poisson extension of an integrable function $f\in L^1(\T)$.

Thus, the factorization
\[
\Theta = B_\Lambda \cdot S_\mu
\]
separates the contributions of interior zeros and singular boundary behavior. In particular, an inner function is completely determined by the location of its zeros in $\D$ together with a singular measure on $\T$.

\medskip

We now turn to the upper half-plane $\C_+= \{z\in \C: \Im(z)>0 \}$. For $1 \leq p < \infty$, we define $H^p(\C_+)$ as the space of analytic functions $F$ on $\C_+$ such that
\[
\sup_{y>0} \int_{\R} |F(x+iy)|^p dx < \infty.
\]
Functions in $H^p(\C_+)$ arise as Poisson extensions
\[
P(f)(x+iy)
=
\int_{\R} f(t)\, \frac{y}{(x-t)^2+y^2} \frac{dt}{\pi}, \qquad x+iy\in \C_+
\]
of an element in $f$ in
\[
H^p(\R):= \{f\in L^p(\R): \, \widehat{f}(\xi)=0 \, \, \text{a.e} \,  \, \, \xi<0 \, \}.
\]
If $p>2$, the Fourier transform may need to be interpreted in the sense of tempered distributions. Conversely, every $F\in H^p(\C_+)$ has limit along every cone in $\C_+$ with vertex at a.e $x\in \R$, and the limit gives rise to an unique element $f\in H^p(\R)$, for which $P(F)=f$. Hence, as in the disc, one has non-tangential boundary limits almost everywhere, and an isometric and isomorphic identification $H^p(\C_+) \cong H^p(\R)$.

The Hardy space theory on $\C_+$ is easily obtained from that on $\D$ via the conformal map
\[
\varphi: \D \to \C_+ \qquad \phi(z) = i\,\frac{1+z}{1-z}, \qquad z\in \D.
\]

In this setting, non-trivial functions $f$ in $H^p$ satisfy the logarithmic integrability condition:
\[
\int_{\R} \log \abs{f(t)} \frac{dt}{1+t^2} >- \infty.
\]

Under this correspondence, the inner and outer factorizations transfer as follows. Every $F \in H^p(\C_+)$ admits a factorization
\[
F(z) = B(z)\, S(z)\, \mathcal{O}(z), \qquad z\in \C_+
\]
where:

- $B$ is the associated Blaschke product over the zeros $\{z_k\} \subset \C_+$ of $F$ 
\[
B(z)= \prod_k \frac{z-z_k}{z-\conj{z_k}} \qquad z\in \C_+,
\]
satisfying the Blaschke condition:
\[
\sum_k \frac{\Im z_k}{|z_k|^2+1} < \infty.
\]
- $S$ is the associated singular inner function
\[
S(z)
=
\exp\left(
- i \int_{\R} \frac{1+tz}{t-z}\, d\mu(t)
\right), \qquad z\in \C_+
\]
where $\mu$ is a positive Borel measure singular wrt $dx$ and finite wrt the Poisson measure:
\[
\int_{\R} \frac{d\mu(x)}{1+x^2}< \infty.
\]
- $\mathcal{O}$ is an outer function of the form
\[
\mathcal{O}(z)
=
\exp\left(
\frac{1}{\pi i}
\int_{\R}
\frac{1+tz}{t-z}
\log |f(t)|\, dt
\right), \qquad z\in \C_+.
\]

In summary, the inner--outer factorization persists in $\C_+$, with the geometry of the boundary $\T$ replaced by $\R$ and the Poisson kernel adapted accordingly. For further details, we refer the reader to \cite{garnett,koosis}.

\newpage

\bibliographystyle{siam}
\bibliography{References.bib}


\end{document}